\documentclass{elsarticle}

\usepackage{geometry}
\usepackage{graphicx}
\usepackage{amssymb}
\usepackage{amsmath}
\usepackage{amsthm}
\usepackage{empheq}
\usepackage{mdframed}
\usepackage{booktabs}
\usepackage{lipsum}
\usepackage{graphicx}
\usepackage{color}
\usepackage{psfrag}
\usepackage{pgfplots}
\usepackage{bm}
\usepackage{url}
\usepackage{array}
\usepackage{float}
\usepackage{mathrsfs}
\usepackage{tikz}
\usepackage{subcaption}

\graphicspath{{fig/}}

\DeclareMathAlphabet{\mathpzc}{OT1}{pzc}{m}{it}

\biboptions{sort&compress} 
\journal{ }
\usepackage[colorlinks=true]{hyperref}
\AtBeginDocument{
  \hypersetup{
    linkcolor=red,
    citecolor=green,
    urlcolor=magenta,
  }
}

\definecolor{ocre}{RGB}{243,102,25}
\definecolor{mygray}{RGB}{243,243,244}
\definecolor{deepGreen}{RGB}{26,111,0}
\definecolor{shallowGreen}{RGB}{235,255,255}
\definecolor{deepBlue}{RGB}{61,124,222}
\definecolor{shallowBlue}{RGB}{235,249,255}

\newtheoremstyle{mytheoremstyle}{3pt}{3pt}{\normalfont}{0cm}{\rmfamily\bfseries}{}{1em}{{\color{black}\thmname{#1}~\thmnumber{#2}}\thmnote{\,--\,#3}}
\newtheoremstyle{myproblemstyle}{3pt}{3pt}{\normalfont}{0cm}{\rmfamily\bfseries}{}{1em}{{\color{black}\thmname{#1}~\thmnumber{#2}}\thmnote{\,--\,#3}}
\theoremstyle{mytheoremstyle}
\newmdtheoremenv[linewidth=1pt,backgroundcolor=shallowGreen,linecolor=deepGreen,leftmargin=0pt,innerleftmargin=20pt,innerrightmargin=20pt,]{theorem}{Theorem}[section]
\theoremstyle{mytheoremstyle}
\newmdtheoremenv[linewidth=1pt,backgroundcolor=shallowBlue,linecolor=deepBlue,leftmargin=0pt,innerleftmargin=20pt,innerrightmargin=20pt,]{definition}{Definition}[section]
\theoremstyle{myproblemstyle}
\newmdtheoremenv[linecolor=black,leftmargin=0pt,innerleftmargin=10pt,innerrightmargin=10pt,]{problem}{Problem}[section]

\newtheorem{thm}{Theorem}[section]

\newtheorem{lem}{Lemma}[section]

\newtheorem{rmk}{Remark}[section]
\theoremstyle{definition}

\theoremstyle{remark}
\newcommand{\ben}{
\begin{eqnarray}}
  \newcommand{\een}{
\end{eqnarray}}
\newcommand{\bea}{
\begin{array}}
  \newcommand{\eea}{
\end{array}}

\numberwithin{equation}{section}
\usepgfplotslibrary{colorbrewer}
\pgfplotsset{width=8cm,compat=1.9}

\def\bv{\mathbf{v}}

\def\b0{\mathbf{0}}
\def\bh{\mathbf{h}}
\def\bx{\mathbf{x}}

\def\bI{\mathbf{I}}

\def\bD{\mathbf{D}}

\def\bQ{\mathbf{Q}}
\def\bG{\mathbf{G}}

\title{ Efficient Numerical Schemes for a Two-Phase Hydrodynamical
Model of Active Liquid Crystals and Solids}

\begin{document}

\begin{frontmatter}
	\author{Xuelong Gu$^{a}$}
	\author{Guanghua Ji$^{b}$}
	\author{and Qi Wang$^{a,*}$}

	\address[1]{Department of Mathematics, University of South
		Carolina, Columbia, SC, 29208, USA}

	\address[2]{Laboratory of Mathematics and Complex Systems (Ministry
		of Education), School of Mathematical Sciences, Beijing Normal
		University, Beijing 100875, P.R. China\\ \vspace{-1cm}}

	\begin{abstract}
		We propose several linear, fully decoupled numerical schemes with
		first- and second-order temporal accuracy for a novel
		${\bf Q}$-tensor-based two-phase hydrodynamic model that
		describes the coupling of active nematic liquid crystal solutions
		with isotropic solid substrates. The model is derived from the
		generalized Onsager principle and includes nontrivial terms that
		contribute zero to the total free energy dissipation. We prove
		that the proposed decoupled linear schemes are thermodynamically
		consistent at the discrete level. In the passive limit, the
		SGE-BDF1 and SGE-PDG schemes are unconditionally energy stable,
		while the SGE-BDF2 scheme is energy stable with respect to a
		modified energy under a standard boundedness assumption and a
		sufficiently large stabilization parameter. We then perform
		extensive numerical simulations to investigate how activity and
		other model parameters affect active-nematic fluid–solid
		interactions. Finally, we analyze the physical mechanisms
		underlying the observed behaviors, providing deeper insight into
		the dynamics of soft, confined active nematic fluids.

	\end{abstract}


\end{frontmatter}

\begin{figure}[b]
	\small \baselineskip=10pt
	\rule[2mm]{1.8cm}{0.2mm} \par
	$^{*}$Corresponding author.\\
	E-mail address: \url{qwang@math.sc.edu} (Q. Wang).
\end{figure}

\section{Introduction}

Active matter encompasses a broad class of systems composed of
self-driven entities that consume energy to generate motion or
mechanical stresses, giving rise to behaviors fundamentally distinct
from those observed in passive matter systems. Examples range from
collective dynamics in bacterial suspensions
\cite{Dombrowski_2004_self} to fish schools
\cite{Becco_2006_experimental}, aerial flocks
\cite{Bialek_2012_statistical, Cavagna_2010_scale}, and cytoskeleton
dynamics \cite{Sanchez_2012_spontaneous,Kierfeld_2008_active}.

Among active matter systems, active liquid crystals represent a
unique subclass in which self-propelled anisotropic particles
exhibit liquid crystalline order \cite{Majumdar_2014_Perspectives}.
These systems arise naturally in biology-for example, in microtubule
bundles \cite{Sanchez_2012_spontaneous}, cytoskeletal filaments
driven by molecular motors \cite{Kierfeld_2008_active}, migrating
cell layers \cite{Poujade_2007_collective, petitjean_2010_velocity},
and dense suspensions of microswimmers \cite{Wensink_2012_meso}. They
also provide novel design principles for synthetic active materials with
potential technological applications. Based on the direction of force
couples, active liquid crystals are further classified as
contractile (forces directed inward along the particle axis) or
extensile (forces directed outward) \cite{MC_Marchetti_2013_hydrodynamics}.

The theoretical foundations of active liquid crystals build on
decades of work in active matter physics. Key advances include
Ramaswamy’s and Marchetti et al.’s development of hydrodynamic
theories \cite{S_Ramaswamy_2010_mechanics,
	MC_Marchetti_2013_hydrodynamics}, the Toner-Tu model, which
establishes that long-range order is possible in two-dimensional
active systems \cite{TT}, and the Vicsek model, which demonstrates
order-disorder transitions in collections of self-propelled particles
\cite{Vicsek1995}. These frameworks, complemented by computational
studies, have revealed phenomena such as spontaneous flow generation,
turbulence-like chaotic states, and emergent collective motion
resembling bacterial suspensions \cite{Dombrowski_2004_self,
	Wensink_2012_meso} and cytoskeletal dynamics \cite{Sanchez_2012_spontaneous}.

A particularly rich platform for exploring these behaviors is the
active nematic droplet: a soft-confined mixture of active nematic
liquid crystal and isotropic fluid. Early computational studies by
Giomi et al. \cite{Giomi_2014_spontaneous} demonstrated
activity-driven elongation and interfacial instabilities, while later
work by Doostmohammadi et al. \cite{A_Doostmohammadi_2016_defect} and
Ruske et al. \cite{LJ_Ruske_2021_morphology} explored defect dynamics
and morphology in two and three dimensions.

Existing continuum descriptions of these systems extend passive
liquid crystal hydrodynamics by incorporating activity as a
perturbation \cite{SR_Aditi_2002_hydrodynamic,
	MC_Marchetti_2013_hydrodynamics}. However, these formulations often
lack a variational structure, limiting both their analytical
tractability and the design of energy-stable numerical methods. To
address this gap, in our recent work
, we developed a
binary phase-field
model for the two-phase mixture of active liquid crystals and viscous
fluids. The model is derived using the variational and generalized
Onsager principles
\cite{Wang2020_Onsager,Joanny&J&K&P2007,Prost&J&J2015}, ensuring energy
dissipation in the absence of activity. This guarantees thermodynamic
consistency, enables rigorous analysis, and facilitates the
development of stable numerical algorithms. The resulting system couples
an Allen–Cahn type equation for the nematic order parameter with
incompressible Navier–Stokes equations for fluid motion. Given the
strong nonlinear interactions among the hydrodynamic variables, the design of
efficient and robust computational methods presents both a
significant challenge and a promising opportunity.

In addition to immiscible mixtures of active liquid crystals and
isotropic viscous fluids, another important class of two-phase systems
has received relatively little attention: active liquid crystals
interfaced with solid structures, either as inclusions or bounding
substrates. Boundary anchoring is known to play a central role in
dictating both the dynamics and the equilibrium configuration of
liquid crystal systems. Past studies have examined such boundary
effects, but analytical and computational limitations have restricted
investigations to simple, regular geometries. To date, no systematic
study has addressed the interaction between liquid crystals and
arbitrarily shaped solid inclusions or liquid crystals confined
within domains of
arbitrary geometries.

Motivated by our continuum framework for active liquid
crystal-viscous fluid mixtures, we propose a continuum model for
active liquid crystals interfaced with solid substrates of arbitrary shape.
Analogous to the droplet model, the new formulation is derived via
the generalized Onsager principle, ensuring thermodynamic
consistency. The solid phase is represented by a momentum balance
equation with a large friction coefficient and a viscoelastic
constitutive relation. Distinct phases are described by a phase-field
variable $\phi \in [0,1]$, in close analogy with hydrodynamic
phase-field models for droplet systems. The continuum model is a bona
fide extension of the phase field Navier-Stokes system for binary
fluids (e.g., Allen-Cahn-Navier-Stokes or Cahn-Hilliard-Navier-Stokes
equations). We note that when the activity is suppressed, the model
reduces to a passive liquid crystal and solid substrate system.

A wide array of numerical strategies exist for components of the
Cahn-Hilliard-Allen-Cahn (CH-AC) or Cahn-Hilliard-Navier-Stokes
(CH-NS) systems. For incompressible Navier-Stokes equations,
projection methods remain widely used
\cite{JL_Guermond_2006_overview, JL_Guermond_2000_projection,
	JL_Guermond_2009_splitting}. Phase-field equations have been
addressed through linear stabilization
\cite{RH_Nochetto_2016_diffuse, J_Shen_2010_discrete, Zhao&W&Y2016},
convex splitting \cite{DJ_Eyre_1998_unconditionally,
	Z_Guan_1014_second}, and quadratization-based approaches such as
invariant energy quadratization (IEQ) \cite{X_Yang_2016_linear,
	Yang&Z&W2017, Zhao&Y&G&W2017, X_Yang_2018_efficient} and scalar
auxiliary variable (SAV) methods \cite{J_Shen_2018_scalar,
	J_Shen_2019_new}. More recently, the supplemental variable method
(SVM) has been introduced as a general framework for constructing
thermodynamically consistent approximations of energy-based models
\cite{Hong&L&W2020, Gong&H&W2021, Hong&W&G2023}.

A central challenge in designing fully decoupled, second-order
schemes lies in handling the nonlinear coupling between phase-field
variables and the velocity field. Yang \cite{X_Yang_2021_numerical}
proposed a decoupling strategy using nonlocal auxiliary variables
governed by ODEs, while SAV-type variants \cite{SAV-NS} provide
related approaches. However, such methods typically preserve only a
modified energy that is weakly connected to the original. In our
earlier work \cite{Gu&W2025}, we clarified the essence of these
techniques by reformulating the governing system into a generalized
gradient flow with respect to a modified energy. Applying classical
discretizations to this reformulation, with explicit implementation
of the coefficient matrix, yields linearly implicit, decoupled
schemes that dissipate the modified energy.

Building on this insight, we introduce a skew-gradient embedding
(SGE) framework, which reformulates generalized dissipative systems
with “zero-energy-contribution” (ZEC) terms into a generalized
gradient flow structure with respect to the original energy. All ZEC
terms are absorbed into a skew-symmetric coefficient matrix in the
form of an exterior 2-form. This embedding permits classical
discretizations with fully explicit treatment of ZEC terms, thereby
leading to decoupled schemes. Crucially, because the reformulation
preserves the original energy structure, we can design both
original-energy-dissipative schemes either
modified-energy-dissipative schemes that remain closely tied to the
physical energy.

Within this framework, we develop three fully decoupled, linearly
implicit time-integration schemes:
\begin{itemize}
	\item SGE-BDF1, a first-order scheme rigorously proven to preserve
	      the original energy dissipation law;

	\item SGE-BDF2, a second-order scheme that is energy-stable under
	      additional solution regularity assumptions and requires a
	      sufficiently large stabilization parameter;

	\item SGE-PDG, a second-order scheme based on the polarized
	      discrete gradient method \cite{pdg_jcam}, which is
	      unconditionally energy-stable without restrictive assumptions and
	      requires only a mild stabilization condition. The SGE-PDG scheme
	      dissipates a polarized energy that remains tightly connected to
	      the physical energy.
\end{itemize}
These schemes enable efficient and accurate simulations of active
nematic fluids across a broad range of activity and environmental
parameters. The model incorporates two distinct activity parameters:
\begin{itemize}
	\item $\xi$, quantifying the effect of molecular activity on
	      nematic ordering and elastic stress;

	\item $\chi$, characterizing the contribution of activity to the
	      total stress (active stress).
\end{itemize}
Substrate friction is introduced into the bulk momentum balance
equation through a damping term $-b\mathbf{v}$ to approximate the
solid phase. In principle, the coefficient $b$ should be allowed to
vary spatially to reflect the spatial nonhomogeneity of the solid substrate.

The remainder of this paper is organized as follows. In \S2, we
review the model for active nematic liquid crystal droplets. In \S3,
we introduce the proposed numerical scheme, proving its stability in
the absence of activity. In \S4, we propose a rough description of
the spatial discretization and its energy stability. \S5 presents
numerical convergence tests and showcases diverse hydrodynamic
behaviors of active nematic droplets under various conditions.

\section{Hydrodynamic Model for Two-Phase  Active
  Liquid Crystals and Isotropic Solids }
\subsection{Notations}

In this paper, we use the
classical lowercase letters such as $f, g$ to represent scalar
fields, boldface lowercase letters $\mathbf{u}, \mathbf{v}$ to denote
vector fields; and boldface Greek letters $\bm{\tau}$ or boldface
uppercase letters $\mathbf{Q}$ to denote second-order tensor fields.
	{\color{blue} Let $\Omega \subset \mathbb{R}^d$ be a square
		computational domain, where
		$d=2$ or $3$ denotes the spatial dimension, with boundary $\partial \Omega$}.
All of these variables depend on space and time $(\mathbf{x}, t)$ .
We define the dot or inner product between vectors by $\mathbf{u}
	\cdot \mathbf{v} = \sum_{i=1}^d u_i v_i$ and the inner product
between second order tensors by $\mathbf{P} : \mathbf{Q} = {\rm
		tr}(\mathbf{P}^\top \mathbf{Q}) = \sum_{i=1}^d \sum_{j=1}^d
	\mathbf{P}_{ij} \mathbf{Q}_{ij}$, where $(\bullet)^\top$ denotes the
transpose of tensor $(\bullet)$.

For functions defined in $\Omega$ spatially, we
introduce the inner-products for scalar, vector, tensor fields,
respectively, as follows,
\begin{equation*}
	(f, g) = \textstyle\int_{\Omega} f g d\mathbf{x}, \ (\mathbf{u},
	\mathbf{v}) = \int_{\Omega} \mathbf{u} \cdot \mathbf{v}
	d\mathbf{x}, \ (\mathbf{P}, \mathbf{Q}) = \int_{\Omega}
	\mathbf{P}:\mathbf{Q} d\mathbf{x},
\end{equation*}
with the induced norm, $\|\bullet\| = \sqrt{(\bullet, \bullet)}$. The
divergence of second order tensor $\bm{\tau}$ is defined
componentwise by $[{\rm div} \bm{\tau}]_i = \partial_j
	\mathbf{\tau}_{ij}$. Here, we employ the Einstein summation
convection by omitting the summation symbol when indices are
repeated. This convention will be
used throughout the paper.

	{\color{deepGreen}
		We define the  second-order symmetric tensor
		space, $\mathcal{M}$, and its trace-free subspace, $\mathcal{M}_0$,
		as follows:
		\begin{equation*}
			\begin{aligned}
				 & \mathcal{M} = \{ {\bf M} \in \mathbb{R}^{d, d}, \mathbf{M} =
				\mathbf{M}^\top \},                                              \\
				 & \mathcal{M}_0 = \{ {\bf Q} \in \mathbb{R}^{d, d}, {\bf Q} \in
				\mathcal{M}, \ {\rm tr}({\bf Q}) = 0 \}.
			\end{aligned}
		\end{equation*}
	}
	{\color{deepGreen}
		\subsection{Governing system of equations}
		We consider a two-phase materials system in isothermal conditions
		consisting of an active nematic liquid
		crystal fluid and an isotropic solid substrate. The solid substrate
		is effectively modeled as a highly dissipative,  viscoelastic
		fluid. The active nematic liquid crystal fluid is
		composed of active liquid crystal solutions, with mass density
		$\rho_1$ and velocity $\mathbf{v}_1$, while the solid or effective
		very viscoelastic fluid is
		characterized by mass density $\rho_2$ and velocity $\mathbf{v}_2$.
		To describe the mixture, we define the total mass density $\rho =
			\rho_1 + \rho_2$ and introduce the mass-averaged velocity $\mathbf{v}
			= \tfrac{\rho_1 \mathbf{v}_1 + \rho_2 \mathbf{v}_2}{\rho}$. Here, we
		assume the two-phase system is incompressible  so that $\rho = {
					\text{const} }$ and
		$\nabla \cdot \mathbf{v} = 0$. The mass   and momentum conservation
		take the form
		\begin{equation} \label{eq:incompressible-conservation}
			\left\lbrace
			\begin{aligned}
				 & \mathbf{\rho} \tfrac{\partial \mathbf{v}}{\partial t} = \nabla
				\cdot (\bm{\sigma} - \rho \mathbf{v}\mathbf{v}) - b(\phi) \mathbf{v}, \\
				 & \nabla \cdot \mathbf{v} = 0,
				\\
			\end{aligned}
			\right.
		\end{equation}
		where $[\nabla \mathbf{v}]_{ij} = \partial_j v_i$ is the velocity
		gradient and ${\bm \sigma}$ the total stress tensor, the mass
		fraction of the active component $\phi = \frac{\rho_1}{\rho}$ is used
		to distinguish active nematic fluid from passive solid components,
		$b(\phi)\geq 0$ is the drag coefficient taking a large value in the
		solid and zero value in the active liquid crystal region. Specifically,
		\begin{itemize}
			\item $\phi = 1$ corresponds to the region fully occupied by the
			      active nematic fluid,
			\item $\phi = 0$ corresponds to the isotropic solid,
			\item $0 < \phi < 1$ represents the interface between the two phases.
		\end{itemize}
		Generally, $\phi$ may depend on both spatial and temporal variables.
		In this work, we fix $\phi$ as a prescribed spatially varying
		function $\phi = \phi(\mathbf{x})$, which defines the spatial
		distribution of the active nematic phase. The model is valid though
		even if we prescribe $\phi$ as time-dependent as well. $b(\phi) >
			0$ is a friction
		coefficient. By choosing a large value for $b$, we model the solid
		as a very dissipative viscoelastic fluid, fully separated from the
		active liquid crystal.

		To characterize the orientational order of active liquid crystals, we
		employ a second-order, symmetric and traceless tensor order parameter
		${\bf Q}$ \cite{Doi&E1986,LCBook-1993}, perceived as the
		traceless component of the second moment tensor of the orientation
		distribution probability function of liquid crystals. In two spatial
		dimensions, $\bf Q$ is given by  the uniaxial
		representation ${\bf Q} = 2q\left({\bf \hat{n}} {\bf \hat{n}}^T -
			\tfrac{1}{2} {\bf I}\right)$, where $q$ denotes the scalar degree of
		nematic order and ${\bf \hat{n}}$ is a unit vector field known as the
		director, representing the average local orientation of the active
		particles. In 3D, ${\bf Q}$ is generally biaxial when flows  or
		strong anchoring boundary conditions are present. We note that a 3-D uniaxial
		state only shows up in relaxed equilibrium states or strong
		uniaxial elongational flows in passive liquid crystals
		\cite{Wang1997,Zhou&W&W&F2007}.
			{\color{ocre}

				Throughout this work, we focus on  the following boundary
				conditions:
				\begin{equation}\label{eq:bc-v}
					\mathbf{v}\big|_{\partial\Omega} = \mathbf{0},
				\end{equation}
				and
				\begin{equation}\label{eq:bc-Q}
					\mathbf{Q}\big|_{\partial\Omega} = \mathbf{Q}_D \ \text{or} \
					\partial_\mathbf{n} \mathbf{Q} \big|_{\partial\Omega} = 0,
				\end{equation}
				where $\mathbf{n}$ denotes the outward unit normal vector on
				$\partial\Omega$. In both case, the boundary condition is required
				to be compatible with the constraint $\mathbf{Q} \in \mathcal{M}_0$.
				These boundary
				conditions play a crucial role
				in the derivation of energy estimates and energy stability. Other boundary
				conditions, such as periodic boundary conditions, can be treated in a
				similar manner, and the corresponding conclusions can be extended
				accordingly.
			}

		Within domain $\Omega$, the Helmholtz free energy functional depends on the
		density of active particles and the nematic order, given by
		\begin{equation*}
			F_\phi(\mathbf{Q}, \nabla \mathbf{Q}) = \int_\Omega
			f_\phi(\mathbf{Q}, \nabla \mathbf{Q}) d\mathbf{x},
		\end{equation*}
		where $f_\phi(\cdot)$ denotes the free energy density. We define the
		total free energy, including the macroscopic kinetic energy and the
		Helmholtz free energy,  by
		\begin{equation}\label{eq:total-F}
			F^{\text{total}}_\phi(\mathbf{v}, \mathbf{Q}, \nabla \mathbf{Q}) =
			\tfrac{\rho}{2} \|\mathbf{v}\|^2 + F_\phi(\mathbf{Q}, \nabla \mathbf{Q}).
		\end{equation}
		Under isothermal conditions, the rate of energy dissipation is given by
		\begin{equation}\label{eq:dt-total-F}
			\tfrac{d F^{\rm total}_\phi}{d t} = \int_\Omega[ \rho \mathbf{v}
				\cdot \mathbf{v}_t - \mathbf{H}: \mathbf{Q}_t - \hat{r} \hat{\mu}]
			d\mathbf{x} + \int_{\partial \Omega} \mathbf{n} \cdot \mathbf{h} ds,
		\end{equation}
		where, $\mathbf{H} = - \frac{\delta F^{\rm total}_\phi}{\delta
				\mathbf{Q}}$ defines the molecule field conjugate to the nematic
		order tensor $\mathbf{Q}$, $\hat{r}$ is the number of active
		particles per unit time and unit volume, $\hat{\mu}$ is the energy
		gain or generated per unit particle, the boundary energy flux is
		$\mathbf{h}$, $\mathbf{n}$ the unit outward normal, and
		\ben
		\bh=(\sigma-\rho \bv \bv -p\bI)\cdot \bv+\frac{\partial
			F_{\phi}}{\partial \nabla \bQ}:\bQ_t.
		\een
		In \eqref{eq:dt-total-F}, the first term represents the change of the
		kinetic energy, the second encapsulates the change in free
		energy through identity $\tfrac{d F_\phi}{dt} =\int_{\Omega}
			\mathbf{H} : \mathbf{Q}_t d\bx$, and the third term describes the
		energy reduction rate of the biological energy corresponding to the
		consumption of energy  by active particles. In the theory, we set
		$\hat{r}$ as a prescribed constant and treat $\hat{\mu}$ as a
		fundamental active parameter.

		By substituting \eqref{eq:incompressible-conservation} into
		\eqref{eq:dt-total-F}, the first contribution becomes
		\begin{equation}\label{eq:expansion-fenergy1}
			\begin{aligned}
				\int_\Omega \rho \mathbf{v} \mathbf{v}_t d\mathbf{x} & =
				\int_\Omega  \mathbf{v} \cdot \nabla \cdot (\bm{\sigma}- \rho
				\mathbf{v} \mathbf{v}) - b(\phi)|\mathbf{v}|^2 d\mathbf{x},
				\\
				                                                     & = \int_\Omega -\nabla \mathbf{v}: (\bm{\sigma} - \rho
				\mathbf{v} \mathbf{v}) - b(\phi) |\mathbf{v}|^2 d\mathbf{x} +
				\int_{\partial \Omega} \mathbf{v} \cdot (\bm{\sigma} - \rho
				\mathbf{v} \mathbf{v}) \cdot \mathbf{n} ds.
			\end{aligned}
		\end{equation}
		Notice that in the case of Dirichlet boundary conditions, ${\bf v}
			= 0$ on the boundary, the last surface integration is zero;
		otherwise, one must retain the surface contribution to the total
		free energy dissipation.

		In assessing the second contribution in \eqref{eq:dt-total-F}, we
		introduce the corotational time derivative for tensor $\bf Q$:
		\begin{equation}\label{eq:rotation_deriv_Q}
			\mathring{\bf Q} = \mathbf{Q}_t + \mathbf{v} \cdot \nabla
			\mathbf{Q} + \mathbf{Q} \cdot \bm{\Omega} - \bm{\Omega} \cdot \mathbf{Q}.
		\end{equation}
		where $\mathbf{D} = \tfrac{1}{2} (\nabla \mathbf{v} + \nabla
			\mathbf{v}^\top)$ is the symmetric rate of strain tensor and
		$\bm{\Omega} = \tfrac{1}{2} (\nabla \mathbf{v} -
			\nabla\mathbf{v}^\top)$ the antisymmetric vorticity tensor.
		Exploiting the symmetry of $\mathbf{Q}, \mathbf{H}$ alongside the
		antisymmetry of $\bm{\Omega}$ yields
		\begin{equation}\label{eq:transf}
			\begin{aligned}
				\mathbf{H} : \mathbf{Q}_t & = \mathbf{H} : (\mathring{\mathbf{Q}}
				- \mathbf{v} \cdot \nabla \mathbf{Q} - \mathbf{Q} \cdot
				\bm{\Omega} + \bm{\Omega} \cdot \mathbf{Q})
				\\
				                          & = \mathbf{H} : \mathring{\mathbf{Q}} - \mathbf{H}:(\mathbf{v}
				\cdot \nabla \mathbf{Q}) - \nabla \mathbf{v}: (\mathbf{Q} \cdot
				\mathbf{H} - \mathbf{H} \cdot \mathbf{Q})
			\end{aligned}
		\end{equation}
		We then decompose the stress into three parts: the antisymmetric part
		of the stress $\bm{\sigma}^a$, the Ericksen stress $\bm{\sigma}^e$,
		and the symmetric stress $\bm{\sigma}^s$, respectively,
		\begin{equation} \label{eq:stress-aes}
			\begin{aligned}
				\bm{\sigma}^a & = \mathbf{Q} \cdot \mathbf{H} - \mathbf{H} \cdot
				\mathbf{Q},                                                      \\
				\bm{\sigma}^e & = f_\phi \mathbf{I} - \nabla \mathbf{Q} :
				\tfrac{\partial f}{\partial \nabla \mathbf{Q}},                  \\
				\bm{\sigma}^s & = \bm{\sigma} - \rho \mathbf{v} \mathbf{v} -
				\bm{\sigma}^e - \bm{\sigma}^a.
			\end{aligned}
		\end{equation}
		Note that symmetry of $\bm{\sigma}^s$ can be derived from
		\cite{X_Yang_2014_active} and the following Gibbs-Duhem relation
		\begin{equation}\label{eq:gibbs}
			\nabla \cdot \bm{\sigma}^e = - \mathbf{H} : \nabla \mathbf{Q}.
		\end{equation}
		Combining \eqref{eq:dt-total-F}, \eqref{eq:expansion-fenergy1},
		\eqref{eq:rotation_deriv_Q}, \eqref{eq:transf}, \eqref{eq:stress-aes}
		and \eqref{eq:gibbs}, we rewrite the time rate of change of the
		total free energy
		as follows \cite{Joanny&J&K&P2007,X_Yang_2014_active}:
		\begin{equation}
			\begin{aligned}
				\tfrac{d F_\phi^{\rm total}}{dt} = -\int_\Omega [ \mathbf{D}:
					                                         \bm{\sigma}^s  + \mathbf{H}:\mathring{\bf Q} + b(\phi)
					                                         |\mathbf{v}|^2 + \hat{r} \hat{\mu} ]d\mathbf{x}.
			\end{aligned}
		\end{equation}

		Notice that this is an inner product between generalized fluxes ($\bD$,
		$\mathring{\bf Q}$, $\hat{r}$) and the corresponding forces
		($\bm{\sigma}^s$, $\bf H$, $\hat{\mu}$). Since time parity of $\bf D$
		is different from other fluxes and time parity of ${\bm \sigma}^s$ is
		different from other forces, we distinguish between the components of
		the fluxes that show the same behavior under time inversion as the
		dissipative conjugate forces or fluxes, and those that show the
		opposite behavior called reactive conjugate forces or fluxes. We
		denote the various components by superscripts ``$\mathpzc{d}$'' and
		``$\mathpzc{r}$'', respectively \cite{Joanny&J&K&P2007,X_Yang_2014_active}.

		Following the generalized Onsager principle, a set of constitutive
		equations for the dissipative currents is obtained
		\cite{Joanny&J&K&P2007,X_Yang_2014_active}
		\begin{equation}
			\begin{pmatrix}
				\sigma^{s, \mathpzc{d}}_{\alpha \beta}
				\\ \mathring{Q}_{\alpha\beta}^\mathpzc{d} \\ \hat{r}^\mathpzc{d}
			\end{pmatrix}
			=
			\begin{pmatrix}
				2 \eta(\phi) \delta_{\alpha k} \delta_{\beta l}           & 0
				                                                          & 0       \\
				0                                                         &
				\Gamma_\mathbf{Q}(\phi)\delta_{\alpha k} \delta_{\beta l} &
				\zeta_{10}(\phi) Q_{\alpha \beta}                                   \\
				0                                                         &
				\zeta_{10}(\phi) Q_{kl}                                   & \lambda
			\end{pmatrix}
			\begin{pmatrix}
				D_{kl} \\ H_{kl} \\ \hat{\mu}
			\end{pmatrix}.
		\end{equation}
		Here, $\eta(\phi)$ denotes the shear viscosity. For thermodynamic
		consistency, the diagonal parameters $\eta(\phi)$,
		$\Gamma_{\mathbf{Q}}(\phi)$ must be non-negative. For fluxes of the
		opposite time parity, the antisymmetric mobility, i.e., the reactive
		terms are given by

		\begin{equation}
			\begin{pmatrix}
				\sigma^{s, \mathpzc{r}}_{\alpha \beta}
				\\ \mathring{Q}_{\alpha\beta}^\mathpzc{r} \\ \hat{r}^\mathpzc{r}
			\end{pmatrix}
			=
			\begin{pmatrix}
				0                            & - \mathcal{A}_{\alpha\beta k l} &
				\mathcal{B}_{\alpha \beta}                                         \\
				\mathcal{A}_{\alpha\beta kl} & 0                               &
				0                                                                  \\
				-\mathcal{B}_{kl}            & 0                               & 0
			\end{pmatrix}
			\begin{pmatrix}
				D_{kl} \\ H_{kl} \\ \hat{\mu}
			\end{pmatrix},
		\end{equation}
		where
		\begin{equation*}
			\begin{aligned}
				\mathcal{A}_{\alpha\beta kl} & = \tfrac{2a}{d}  \delta_{\alpha k}
				\delta_{\beta l} + a (Q_{\alpha k} \delta_{\beta l} +
				\delta_{\alpha k} Q_{\beta l}) - 2a Q_{kl} (Q_{\alpha \beta} +
				\tfrac{1}{d}\delta_{\alpha \beta}) + \theta_1
				\delta_{kl}\delta_{\alpha \beta},                                 \\
				\mathcal{B}_{\alpha \beta}   & = \zeta_0 (\phi) - \zeta_2
				\delta_{\alpha \beta}.
			\end{aligned}
		\end{equation*}
		We rewrite $\zeta_{10}(\phi) \hat{\mu}$ as $\xi(\phi)$,
		$\zeta_0(\phi) \hat{\mu}$ as $-\chi(\phi)$, where $\xi(\phi)$ and
		$\chi(\phi)$ quantify the
		particle's activity via a linear dependence on $\hat{\mu}$. The
		positive (negative) value of $\chi$
		corresponds to extensile (contractile) particles
		\cite{SR_Aditi_2002_hydrodynamic, D_Marenduzzo_2007_steady,
			R_Voituriez_2005_spontaneous}.

		Adopting the expressions for dissipative and reactive fluxes in the
		above constitutive equations, we summarize the governing system of
		equations for the two-phase system as follows:
		\begin{equation}\label{eq:governing}
			\left\lbrace
			\begin{aligned}
				 & \rho (\mathbf{v}_t + \mathbf{v}\cdot \nabla \mathbf{v}) =
				\nabla \cdot (\bm{\sigma}^a + \bm{\sigma}^s) - \mathbf{H}:\nabla
				\mathbf{Q} - b(\phi) \mathbf{v},                             \\
				 & \nabla \cdot \mathbf{v} = 0,
				\\
				 & \mathbf{Q}_t + \mathbf{v} \cdot \nabla \mathbf{Q} =
				\Gamma_\mathbf{Q}(\phi) \mathbf{H} + \mathbf{S}+ \xi(\phi) \mathbf{Q},
			\end{aligned}
			\right.
		\end{equation}
		where,
		\begin{equation*}
			\begin{aligned}
				\bm{\sigma}^a & = \mathbf{Q} \cdot \mathbf{H} - \mathbf{H} \cdot
				\mathbf{Q},
				\\
				\bm{\sigma}^s & = -p \mathbf{I} + 2 \eta(\phi) \mathbf{D} -
				a(\mathbf{Q} \cdot \mathbf{H} + \mathbf{H} \cdot \mathbf{Q}) -
				\tfrac{2a}{d} \mathbf{H} + 2a (\mathbf{Q}:\mathbf{H})(\mathbf{Q}
				+ \tfrac{1}{d}\mathbf{I}) - \chi(\phi) \mathbf{Q},                                                                                                   \\
				\mathbf{S}    & = {\bf \Omega \cdot Q - Q \cdot \Omega} + a ({\bf
						                                                          Q\cdot D + D \cdot Q}) + \tfrac{2a}{d} {\bf D} - 2a ({\bf D :
						                                                                                                             Q})({\bf Q } + \tfrac{1}{d} {\bf I}),
			\end{aligned}
		\end{equation*}
		and $p$ is the hydrostatic pressure.


		We further decompose the Helmholtz free energy of the liquid
		crystal system into  three components: the conformational entropy,
		bulk and anchoring
		energy as follows \cite{LCBook-1993}
		\begin{equation}\label{eq:fenergy}
			F_\phi(\mathbf{Q}, \nabla \mathbf{Q})
			=
			F_{\text{el}}(\nabla \mathbf{Q})
			+
			F_{\text{bulk}}(\mathbf{Q})
			+
			F_{\text{anch}}(\mathbf{Q}).
		\end{equation}
		The conformational entropy  is given by
		\begin{equation}
			F_{\text{el}}(\nabla\mathbf{Q})
			=
			\int_\Omega
			\tfrac{K(\phi)}{2}\,|\nabla \mathbf{Q}|^2\, d\mathbf{x},
		\end{equation}
		where the coefficient $K(\phi)$ depends on phase-field $\phi$
		and thus modulates the energetic penalty for spatial distortions of
		$\mathbf Q$ according to the local material state.
		The bulk free energy takes the classical Landau-de Gennes form
		\cite{LCBook-1993},
		\begin{equation}
			F_{bulk}(\mathbf{Q})
			=
			\int_\Omega
			\left[
				\tfrac{\alpha(\phi)}{2}\operatorname{tr}(\mathbf{Q}^2)
				-\tfrac{\beta(\phi)}{3}\operatorname{tr}(\mathbf{Q}^3)
				+\tfrac{\gamma(\phi)}{4}\bigl(\operatorname{tr}(\mathbf{Q}^2)\bigr)^2
				\right]
			d\mathbf{x},
		\end{equation}
		where the coefficients depend on $\phi$ through the functions,
		\[
			\alpha(\phi)=A\!(1-\tfrac{N(\phi)}{3}),
			\qquad
			\beta(\phi)=\gamma(\phi)=A\,N(\phi).
		\]
		{\color{red} Here, $A > 0$ sets the bulk energy scale, and $N(\phi)
				\geq 0$ is a
			scalar dimensionless concentration
			\cite{LJ_Ruske_2021_morphology, Doi&E1986,Beris&E1994}.}

		To model weak anchoring effects at the fluid-solid
		interface, the anchoring energy is proposed as follows
		\begin{equation}
			F_{\text{anch}}(\mathbf{Q})
			=
			\int_\Omega
			\tfrac{W(\phi)}{2}\,|\mathbf{Q}-\mathbf{Q}_\star|^2\, d\mathbf{x},
		\end{equation}
		where $W(\phi)=W_0|\nabla\phi|$ localizes the anchoring contribution to the
		interfacial region of finite thickness, $W_0$ parameterizes the
		anchoring strength,
		and $\mathbf{Q}_\star$ represents the preferred configuration imposed
		by the boundary anchoring (a traceless second order tensor).

			{\color{ocre}
				\begin{rmk}
					The governing system of equations is derived for the two-phase
					active liquid crystal
					solutions and highly dissipative, viscoelastic fluid system. We
					use the highly dissipative viscoelastic fluid part to describe
					the isotropic solid approximately
					by enforcing a large friction term $-b(\phi) \bv$, assuming a
					large viscosity coefficient, specifying other viscoelastic materials
					parameters, and annihilating active parameters.
					$\mathbf{Q}_\star$ denotes a prescribed, three-dimensional, anchored
					$\mathbf{Q}$-tensor field, which
					admits a uniaxial form,
					\begin{equation}\label{eq:uniaxial-form}
						\mathbf{Q}_\star
						= S_{eq}\!\left(\mathbf{p}\mathbf{p}^\top -
						\tfrac{1}{3}\mathbf{I}\right),
					\end{equation}
					where $\mathbf{p}$ is a unit vector representing the preferred molecular
					orientation at the interface. In this study, $\mathbf{p}$
					corresponds to the unit
					outward normal
					vector in the case of normal anchoring, or to the unit tangential
					vector along
					the solid boundary in the case of tangential anchoring.

					Under conditions $\alpha < 0$ and $\beta, \gamma <
						0$,
					\[
						S_{eq}
						= \frac{\beta + \sqrt{\beta^2 - 24 \alpha \gamma}}{4\gamma}.
					\]
					This uniaxial equilibrium structure and the expression for
					$S_{eq}$ are standard
					results in the Landau--de~Gennes theory for nematic liquid
					crystals; see, for
					example, \cite{LCBook-1993} for a details.
				\end{rmk}
			}

		We introduce characteristic length scale $l_0$ and time scale $t_0$ to
		nondimensionlize the physical variables as follows
		\begin{equation}\label{eq:dimensionless-01}
			\tilde{t} = \tfrac{t}{t_0}, \ \tilde{\mathbf{x}} =
			\tfrac{\mathbf{x}}{l_0},  \ \tilde{\mathbf{H}} = \tfrac{t_0^2
				\mathbf{H}}{\rho l_0^2}.
		\end{equation}
		Then, the following dimensionless parameters arise
		\begin{equation}\label{eq:dimensionless-02}
			\tilde{A} = \tfrac{t_0^2 A}{\rho l_0^2},
			\ \tilde{\Gamma}_{\mathbf{Q}} = \tfrac{\rho l_0^2
				\Gamma_{\mathbf{Q}}}{t_0}, \ \tilde{K} = \tfrac{t_0^2 K}{\rho},
			\ \tilde{a} = a, \ \tilde{\eta} = \tfrac{t_0 \eta}{\rho l_0^2},
			\ \tilde{p} = \tfrac{t_0^2}{\rho l_0^2} p, \ \tilde{\xi} = t_0
			\xi , \ \tilde{\chi} = \tfrac{t_0^2 \chi}{\rho l_0^2}.
		\end{equation}
		Utilizing \eqref{eq:dimensionless-01} and
		\eqref{eq:dimensionless-02},  we
		obtain the dimensionless governing equations after dropping the
		$\tilde{}$ over the parameters:
	}

\begin{equation}\label{eq:governing-dimensionless}
	\left\lbrace
	\begin{aligned}
		 & \mathbf{v}_t + \mathbf{v} \cdot \nabla \mathbf{v} = - \nabla p
		+ 2 \nabla \cdot (\eta(\phi) \mathbf{D}) + \nabla \cdot
		\bm{\Sigma} - \mathbf{H}:\nabla \mathbf{Q} - \nabla \cdot
		(\chi(\phi) \mathbf{Q}) - b(\phi) \mathbf{v},                                                                                                       \\
		 & \nabla \cdot \mathbf{v} = 0,
		\\
		 & \mathbf{Q}_t + \mathbf{v} \cdot \nabla \mathbf{Q} =
		\Gamma_\mathbf{Q}(\phi) {\bf H} + {\bf S} + \xi(\phi) {\bf Q},
		\\
		 & {\bf \Sigma} = {\bf Q \cdot H - H \cdot Q} - a({\bf Q \cdot H +
				                                                H \cdot Q}) - \tfrac{2a}{d} {\bf H} + 2a ({\bf Q : H})({ \bf Q} +
		\tfrac{1}{d} {\bf I}),
		\\
		 & \mathbf{S}  = {\bf \Omega \cdot Q - Q \cdot \Omega} + a ({\bf
				                                                         Q\cdot D + D \cdot Q}) + \tfrac{2a}{d} {\bf D} - 2a ({\bf D :
				                                                                                                            Q})({\bf Q } + \tfrac{1}{d} {\bf I}),
		\\
		 & \mathbf{H} = \nabla \cdot (K(\phi) \nabla \mathbf{Q}) -
		\left[\alpha(\phi) \mathbf{Q} - \beta(\phi) (\mathbf{Q}^2 -
			\tfrac{{\rm tr}(\mathbf{Q}^2)}{d} \mathbf{I} ) + \gamma(\phi)
			{\rm tr}(\bf{Q}^2){\bf Q}\right] - W(\phi)(\mathbf{Q} - \mathbf{Q}_\star),
	\end{aligned}
	\right.
\end{equation}

{
	In this study, we prescribe the model parameters in
	\eqref{eq:governing-dimensionless}
	as follows:
	\begin{equation}\label{eq:model-parameters}
		\begin{aligned}
			 & b(\phi)=b_{solid}(1-\phi), \ \eta(\phi)=\eta_{fluid} \phi +
			\eta_{solid}(1-\phi),                                          \\
			 & \xi(\phi)=\xi_{fluid} \phi, \ \chi(\phi)=\chi_{fluid} \phi,
			\ K(\phi) = K_{fluid} \phi,                                    \\
			 & N(\phi)=N_{fluid} \phi + N_{solid} (1-\phi),
			\ \Gamma_{\mathbf{Q}}=\Gamma_{fluid} \phi + \Gamma_{solid}(1 - \phi).
		\end{aligned}
	\end{equation}
	Here, the ${(\bullet)}_{fluid}$ and ${(\bullet)}_{solid}$ are
	prescribed model parameters in the active liquid crystal fluid and
	solid, respectively.
}

The total {\color{ocre} free energy} of
\eqref{eq:governing-dimensionless} is given by
\begin{equation}\label{eq:total-F}
	F^{\text{total}}_\phi(\mathbf{v}, \mathbf{Q}, \nabla \mathbf{Q}) =
	\tfrac{\rho}{2} \|\mathbf{v}\|^2 + F_\phi(\mathbf{Q}, \nabla \mathbf{Q}).
\end{equation}
It
yields {\color{ocre} following energy dissipation rate:}
\begin{equation}\label{eq:energy-continuous}
	\tfrac{d F_\phi^{total}}{dt} = - \int_\Omega [2\eta(\phi)
		|\mathbf{D}|^2 + \Gamma(\phi) |\mathbf{H}|^2 + b(\phi)
		|\mathbf{v}|^2 ]d\mathbf{x} + (\mathbf{Q}, \chi(\phi) \mathbf{D} -
	\xi(\phi) \mathbf{H}).
\end{equation}
If active parameters $\chi(\phi)$, $\xi(\phi)$
vanish, we have $\tfrac{d F^{total}_\phi}{d t} \leq 0$.
	{\color{ocre} In this case, the governing system
		\eqref{eq:governing-dimensionless} describes a two-phase system
		consisting of passive liquid crystals and
		solids whose total free energy dissipative.}

\begin{rmk}
	Note that  $\mathbf{H}$ corresponds to the
	variational derivative in subspace $\mathcal{M}_0$.
	Specifically, ${\bf H} = - \mathcal{P}_{\mathcal{M}_0}
		\tfrac{\delta F^{\rm total}_\phi}{\delta \mathbf{Q}}$, where
	$\mathcal{P}_{\mathcal{M}_0}$ is the orthogonal projection operator from
	$\mathcal{M}$ to $\mathcal{M}_0$, given by
	\begin{equation*}
		\mathcal{P}_{\mathcal{M}_0} (\mathbf{M}) = {\bf M} -
		\tfrac{\bf{I}}{d} {\rm tr}({\bf M}).
	\end{equation*}
\end{rmk}

\section{Thermodynamically consistent, decoupled, and linearly implicit
  numerical algorithms}
In this section, we present several thermodynamically consistent
discretizations of system \eqref{eq:governing-dimensionless}, i.e.,
the resulting numerical schemes preserve a discrete analog of
\eqref{eq:energy-continuous}. In particular, when the liquid crystal
system is passive, the proposed algorithms preserve the discrete
energy dissipation rate, leading to energy stability. We remark that
the energy dissipation rate preservation is stronger property than
simply energy stability.

\subsection{SGE reformulation of the system}
{\color{blue} We} reformulate \eqref{eq:governing-dimensionless}
using SGE formulated in
\cite{Gu&W2025}. The central idea of the SGE approach is to express
the ``zero-energy-contribution'' (ZEC) term in a low-rank
skew-gradient form. By unveiling this intrinsic gradient structure,
one can seamlessly adapt classical techniques to construct
energy dissipation rate preserving or energy stable schemes.

We define the following intermediate variables:
\begin{equation*}
	\begin{array}{llll}
		                                & \mathbf{g}_\mathbf{v} = \mathbf{v},         & \mathbf{c}_\mathbf{v} = -
		\mathbf{v} \cdot \nabla \mathbf{v} + \nabla \cdot {\bf \Sigma} -
		\mathbf{Q} : \nabla \mathbf{H}, & \mathbf{f}_\mathbf{v} = -\nabla
		\cdot (\chi(\phi) \mathbf{Q}),                                                                                                                     \\\\
		                                & \mathbf{G}_{\bf Q} = -\mathbf{H},           & \mathbf{C}_{\bf Q} = -
		                                                                                                     {\bf v \cdot \nabla \mathbf{Q}} + \mathbf{S},
		                                & \mathbf{F}_{\bf Q} = \xi (\phi) \mathbf{Q},
	\end{array}
\end{equation*}
{\color{blue} Then, \eqref{eq:governing-dimensionless} can be recast into}
\begin{equation}\label{eq:governing-SGE}
	\left\lbrace
	\begin{aligned}
		 & \mathbf{v}_t = -\nabla p + 2 \nabla \cdot (\eta(\phi)
		\mathbf{D}) - b(\phi)\mathbf{v} +
		\tfrac{(\mathbf{g}_{\mathbf{v}}, \mathbf{g}_{\mathbf{v}}) +
			(\mathbf{G}_\mathbf{Q},
			\mathbf{G}_\mathbf{Q})}{\mathcal{G}(\mathbf{g}_\mathbf{v},
			\mathbf{G}_\mathbf{Q})} \mathbf{c}_\mathbf{v} -
		\tfrac{(\mathbf{g}_{\mathbf{v}}, \mathbf{c}_{\mathbf{v}}) +
			(\mathbf{G}_\mathbf{Q},
			\mathbf{C}_\mathbf{Q})}{\mathcal{G}(\mathbf{g}_\mathbf{v},
			\mathbf{G}_\mathbf{Q})} \mathbf{g}_\mathbf{v} + \mathbf{f}_\mathbf{v}, \\
		 & \nabla \cdot \mathbf{v} = 0,
		\\
		 & \mathbf{Q}_t = \Gamma_{\mathbf{Q}}(\phi)\mathbf{H} +
		\tfrac{(\mathbf{g}_\mathbf{v}, \mathbf{g}_\mathbf{v}) +
			(\mathbf{G}_\mathbf{Q},
			\mathbf{G}_\mathbf{Q})}{\mathcal{G}(\mathbf{g}_\mathbf{v},
			\mathbf{G}_\mathbf{Q})} \mathbf{C}_\mathbf{Q} -
		\tfrac{(\mathbf{g}_\mathbf{v}, \mathbf{c}_\mathbf{v}) +
			(\mathbf{G}_\mathbf{Q},
			\mathbf{C}_\mathbf{Q})}{\mathcal{G}(\mathbf{g}_\mathbf{v},
			\mathbf{G}_\mathbf{Q})} \mathbf{G}_\mathbf{Q} + \mathbf{F}_\mathbf{Q}.
	\end{aligned}
	\right.,
\end{equation}
{\color{blue}
	where
	\begin{equation}\label{den:mathcalg}
		\mathcal{G}(\mathbf{g}_\mathbf{v}, \mathbf{G}_\mathbf{Q}) =
		\|\mathbf{g}_\mathbf{v}\|^2 + \|\mathbf{G}_\mathbf{Q}\|^2.
	\end{equation}
}

Note that the SGE reformulation, \eqref{eq:governing-SGE}, is
equivalent to the original system,
\eqref{eq:governing-dimensionless}, via the following identities
\begin{equation}\label{eq:identity-equivalent}
	\left\{
	\begin{aligned}
		 & (\mathbf{g}_\mathbf{v}, \mathbf{g}_\mathbf{v}) +
		(\mathbf{G}_\mathbf{Q}, \mathbf{G}_\mathbf{Q}) =
		\mathcal{G}(\mathbf{g}_\mathbf{v}, \mathbf{G}_\mathbf{Q}), \\\\
		 & (\mathbf{g}_\mathbf{v}, \mathbf{c}_\mathbf{v}) +
		(\mathbf{G}_\mathbf{Q}, \mathbf{C}_\mathbf{Q}) = 0.
	\end{aligned}\right.
\end{equation}
For the second identity, readers please refer to
\cite{J_Zhao_JSC_Qtensor} for details.

\begin{rmk}
	\color{blue}
	Note that $\mathcal{G}(\mathbf{g}_\mathbf{v},
		\mathbf{G}_\mathbf{Q})$ is defined as the sum of the squared $L^2$ norms
	of the free-energy gradient components \eqref{den:mathcalg}. The condition
	$\mathcal{G}(\mathbf{g}_\mathbf{v},
		\mathbf{G}_\mathbf{Q}) = 0$ therefore corresponds to the system
	reaching a critical point on the energy surface. Away from such
	critical points, one generally has
	$\mathcal{G}(\mathbf{g}_\mathbf{v}, \mathbf{G}_\mathbf{Q}) \neq
		0$. To ensure that the formulation remains well defined in both
	situations, the degenerate case is treated in a manner consistent
	with the general case: if $(\mathbf{g}_\mathbf{v},
		\mathbf{G}_\mathbf{Q}) = 0$, the coefficients in front of
	$\mathbf{c}_\mathbf{v}, \mathbf{C}_\mathbf{Q}$ are set to $1$,
	while those multiplying $\mathbf{g}_\mathbf{v}$ and
	$\mathbf{G}_\mathbf{Q}$ are set to $0$.
	Since the main difficulty in constructing an energy-stable scheme
	arises in the case where the free-energy gradient does not vanish, i.e.,
	$\mathcal{G}(\mathbf{g}_{\mathbf v}, \mathbf{G}_{\mathbf Q}) \neq 0$,
	we primarily focus on this regime in the subsequent analysis.
\end{rmk}
\begin{rmk}
	By introducing $\mathbf{\Psi} = (\mathbf{v}, \mathbf{Q})$, system
	\eqref{eq:governing-SGE}
	can be further written in the form,
	\begin{equation*}
		\mathbf{\Psi}_t = \mathscr{M}(\bm{\Psi}) \delta_{\bm{\Psi}}
		F^{total} + \mathscr{J}(\mathbf{\Psi}) \delta_{\mathbf{\Psi}}
		F^{total} + \mathscr{F}(\mathbf{\Psi}),
	\end{equation*}
	where $\delta_{\mathbf{\Psi}}$ denotes the variational derivative
	of the total energy functional with respect to $\bm{\Psi}$ in
	space $\mathbf{V} \times \mathcal{M}_0$. For velocity fields
	subject to the zero Dirichlet boundary condition, the admissible
	solution space is defined as {\color{blue} $\mathbf{V} = \{ \mathbf{v} \in
				[L^2(\Omega)]^d | \nabla \cdot \mathbf{v} = 0, \ \mathbf{v} \cdot
				\mathbf{n}|_{\partial \Omega} = 0 \}$}. The specific forms of the
	operators are given by
	\begin{equation}\label{eq:SGE-reformulation-matrix}
		\mathscr{M}(\mathbf{\Psi}) = {\rm diag}(\mathcal{B},
		-\Gamma_\mathbf{Q}), \ \mathscr{J}(\bm{\Psi}) =
		\begin{pmatrix}
			\mathbf{v} \\ -\mathbf{H}
		\end{pmatrix}
		\wedge
		\begin{pmatrix}
			\mathcal{P}_{\mathbf{V}} \mathbf{c}_\mathbf{v}
			\mathcal{P}_{\mathbf{V}} \\
			\mathcal{P}_{\mathcal{M}_0} \mathbf{C}_\mathbf{Q}
			\mathbf{P}_{\mathcal{M}_0}
		\end{pmatrix},
		\mathscr{F}(\mathbf{\Psi}) =
		\begin{pmatrix}\mathcal{P}_\mathbf{V} \mathbf{f}_\mathbf{v}
			\\ \mathbf{F}_\mathbf{Q}
		\end{pmatrix}.
	\end{equation}
	Here, $\delta_{\bm{\Psi}} F^{total} = (\mathbf{v}, -\mathbf{H}) =
		(\mathbf{g}_\mathbf{v}, \mathbf{G}_\mathbf{Q})$,
	$\mathcal{P}_\mathbf{V}$ denotes the Hemholtz-Leray projection,
	and $\mathcal{B} = \mathcal{P}_\mathbf{V} \nabla \cdot (\nabla
		(\bullet))$, represents a Stokes operator with variable
	coefficients. It is self-adjoint with non-positive spectra when
	$\eta > 0$, ensuring that $\mathscr{M}$ is a self-adjoint and
	non-positive mobility operator. The operator
	$\mathscr{J}(\mathbf{\Psi})$ defines a skew-symmetric operator on
	$\mathbf{V} \times \mathcal{M}_0$ {\color{ocre} and therefore
		generates the reversible part of the dynamics, contributing zero
		dissipation to the total free energy, } as shown in our previous work
	\cite{Gu&W2025}. This is an activity-driven gradient flow system,
	where $ \mathscr{F}(\mathbf{\Psi})$ represents
	the generalized force due to the activity.

	The reformulation,
	\eqref{eq:SGE-reformulation-matrix}, together with its equivalent
	form \eqref{eq:governing-SGE}, provides a convenient framework
	for constructing energy-stable schemes using some well-tested tools,
	such as stabilization techniques
	\cite{Lu-Stabilization,Wang-CH-Stabilization,Feng-2013,Tang-2016,Tang-2020},
	and discrete gradient methods \cite{AVF1,AVF2,AVF3,AVF4,AVF5}.
	In addition, since all convective terms are encapsulated within
	operator $\mathscr{J}(\bm{\Psi})$, they can be
	discretized explicitly independent of the spatial discretization
	without destroying the underlying energy dissipative structure. Notably, the
	resulting system can also be solved efficiently in a fully
	decoupled manner due to the low-rank (rank-2) structure in
	$\mathscr{J}(\bm{\Psi})$.
\end{rmk}

\subsection{First and second-order schemes based on stabilization
	techniques}
In this section, we construct first- and second-order numerical
schemes based on the stabilization technique and the SGE
reformulation of governing system \eqref{eq:governing-SGE}. By
applying these two techniques, we break down the discretized,
coupled hydrodynamic system into decoupled systems: a Stokes
problem for the velocity-pressure pair and a sequence of decoupled
Poisson equations. Each time step thus requires only a Stokes
solver and a Poisson solver, respectively, greatly simplifying the
overall computation.

Suppose we consider the problem over time interval $(0, T]$. Let
$N_t$ be a positive integer representing the total number of time
steps. We partition the interval into a uniform temporal grid with
time step size $\tau = \tfrac{T}{N_t}$. The discrete time levels
are denoted by $t_n = n\tau$ for $n = 0, 1, \dots, N_t$, and the
approximation of a physical variable $f$ at time $t_n$ is denoted by $f^n$.

We next introduce stabilized backward-difference schemes of first
and second order (SGE-BDF1 and SGE-BDF2) for
\eqref{eq:governing-SGE}, respectively. At each time step the
update takes the unified form as follows,
\begin{equation}\label{eq:sge-bdf1}
	\left\lbrace
	\begin{aligned}
		 & \tfrac{\alpha_k \mathbf{v}^{n+1} - A_k(\mathbf{v}^n)}{\tau} =
		-\nabla p^{n+1} + 2 \nabla \cdot (\eta \mathbf{D}^{n+1}) - b
		\mathbf{v}^{n+1} + \zeta^{n+1} B_k (\mathbf{c}_\mathbf{v}^n) -
		\omega^{n+1} B_k(\mathbf{g}_{\mathbf{v}}^n) +
		B_k(\mathbf{f}_\mathbf{v}^n),
		\\
		 & \nabla \cdot \mathbf{v}^{n+1} = 0,
		\\
		 & \tfrac{\alpha_k \mathbf{Q}^{n+1} - A_k (\mathbf{Q}^n)}{\tau}
		= \Gamma_{\bf Q} \mathbf{H}^{n+1} + \zeta^{n+1}
		B_k(\mathbf{C}_\mathbf{Q}^n) - \omega^{n+1}
		B_k(\mathbf{G}_\mathbf{Q}^n) + B_k(\mathbf{F}_\mathbf{Q}^{n}),
		\\
		 & \zeta^{n+1} = \tfrac{(\mathbf{g}_\mathbf{v}^{n+1},
			                 B_k(\mathbf{g}_\mathbf{v}^n)) + (\mathbf{G}_\mathbf{Q}^{n+1},
			                 B_k(\mathbf{G}_\mathbf{Q}^n))}{\mathcal{G}(B_k(\mathbf{g}_\mathbf{v}^n),
			                 B_k(\mathbf{G}_\mathbf{Q}^n))}, \ \omega^{n+1} =
		\tfrac{(\mathbf{g}_\mathbf{v}^{n+1},
			B_k(\mathbf{c}_\mathbf{v}^n)) + (\mathbf{G}_\mathbf{Q}^{n+1},
			B_k(\mathbf{C}_\mathbf{Q}^n))}{\mathcal{G}(B_k(\mathbf{g}_\mathbf{v}^n),
			B_k(\mathbf{G}_\mathbf{Q}^n))},
	\end{aligned}
	\right.
\end{equation}
where $k = 1$ yields the first-order scheme and $k = 2$ the
second-order variant, The weights $\alpha_k$, $A_k$ and $B_k$ arise
from a Taylor expansion and are specified as follows:
\begin{itemize}
	\item SGE-BDF1 (first-order, $k = 1$): $\alpha_1 = 1, \ A_1(f^n)
		      = f^n, \ B_1(f^n) = f^n$,
	\item SGE-BDF2 (second-order, $k = 2$): $\alpha_2 = \tfrac{3}{2},
		      \ A_2(f^n) = 2f^n - \tfrac{1}{2} f^{n-1}, \ B_2(f^n) = 2 f^n - f^{n-1}$.
\end{itemize}
In addition, we set $\mathbf{G}_\mathbf{Q}^{n+1} = -\mathbf{H}^{n+1}$ with
\begin{equation}\label{eq:sge-bdf-H}
	\mathbf{H}^{n+1} =
	\left\lbrace
	\begin{aligned}
		 & \nabla \cdot (K \nabla \mathbf{Q}^{n+1}) - W(\mathbf{Q}^{n+1}
		- \mathbf{Q}_\star) - \kappa (\mathbf{Q}^{n+1} - \mathbf{Q}^n)
		- f_{\rm bulk}(\mathbf{Q}^n) , \ \text{(first-order)},
		\\
		 & \nabla \cdot (K \nabla \mathbf{Q}^{n+1}) - W(\mathbf{Q}^{n+1}
		- \mathbf{Q}_\star) -  \tau \kappa (\alpha_2 \mathbf{Q}^{n+1} -
		A_2(\mathbf{Q}^n)) - B_2(f_{\rm bulk}(\mathbf{Q}^n)),
		\ \text{(second-order)}.
	\end{aligned}
	\right.
\end{equation}
At first glance, \eqref{eq:sge-bdf1} appears only weakly
coupled-through parameters $\zeta^{n+1}$ and $\omega^{n+1}$-despite
being linearly implicit. In the next section, we show that the
update can in fact be carried out in two fully decoupled steps,
eliminating any potential coupling.

By inspecting linear system \eqref{eq:sge-bdf1} and formulation
\eqref{eq:sge-bdf-H} at time level $n+1$, we split each unknown
into three independent components
\begin{equation}\label{eq:bdf1-semi-decomposition}
	\left\lbrace
	\begin{aligned}
		\mathbf{v}^{n+1} & = \mathbf{v}_1 + \zeta^{n+1} \mathbf{v}_2 -
		\omega^{n+1} \mathbf{v}_3,                                     \\
		\mathbf{Q}^{n+1} & = \mathbf{Q}_1 + \zeta^{n+1} \mathbf{Q}_2 -
		\omega^{n+1} \mathbf{Q}_3,                                     \\
		\mathbf{H}^{n+1} & = \mathbf{H}_1 + \zeta^{n+1} \mathbf{H}_2 -
		\omega^{n+1} \mathbf{H}_3,                                     \\
	\end{aligned}
	\right.
\end{equation}
Each triplet $(\mathbf{v}_i, \mathbf{Q}_i, \mathbf{H}_i)$ is then
computed from fully decoupled subproblems:
\begin{itemize}
	\item {\bf Generalized Stokes equations. } For $i = 1, 2, 3$,
	      find $\mathbf{v}_i$ by solving:
	      \begin{equation}
		      \left\lbrace
		      \begin{aligned}
			       & (\tfrac{\alpha_k}{\tau} + b)\mathbf{v}_i^{n+1} - 2 \nabla
			      \cdot (\eta \mathbf{D}(\mathbf{v}_i^{n+1})) + \nabla
			      p^{n+1} = \mathbf{f}_i,                                      \\
			       & \nabla \cdot \mathbf{v}_i^{n+1} = 0.
		      \end{aligned}
		      \right.
	      \end{equation}
	\item {\bf Decoupled tensor Poisson equations. } For $i = 1, 2,
		      3$, find $\mathbf{Q}_i$ from
	      \begin{equation}
		      (\tfrac{\alpha_k}{\tau} + \Gamma_{\mathbf{Q}} (\kappa +
		      W))\mathbf{Q}_i^{n+1} - \Gamma_{\mathbf{Q}} \nabla \cdot (K
		      \nabla \mathbf{Q}_i^{n+1}) = \mathbf{F}_i.
	      \end{equation}
\end{itemize}
The right-hand sides contain all explicit  terms:
\begin{equation*}
	\left\lbrace
	\begin{aligned}
		 & \mathbf{f}_1 = \tfrac{1}{\tau}A_k(\mathbf{v}^n) +
		B_k(\mathbf{f}_\mathbf{v}^n), \ \mathbf{f}_2 =
		B_k(\mathbf{c}_\mathbf{v}^n), \ \mathbf{f}_3 =
		B_k(\mathbf{g}_\mathbf{v}^n),                              \\
		 & \mathbf{F}_1 = \tfrac{1}{\tau} A_k (\mathbf{Q}^n) + B_k
		(\mathbf{F}_\mathbf{Q}^n), \ \mathbf{F}_2 =
		B_k(\mathbf{C}_\mathbf{Q}^n), \ \mathbf{F}_3 =
		B_k(\mathbf{G}_\mathbf{Q}^n).
	\end{aligned}
	\right.
\end{equation*}
Next, each $\mathbf{H}_i^{n+1}$ is obtained by substituting
\eqref{eq:bdf1-semi-decomposition} into \eqref{eq:sge-bdf-H} and
collecting like terms with respect to $\zeta^{n+1}$ and
$\omega^{n+1}$. As an example, for the SGE-BDF1 scheme, this reads
\begin{equation*}
	\left\lbrace
	\begin{aligned}
		 & \mathbf{H}_1^{n+1} = \nabla \cdot (K \nabla
		\mathbf{Q}_1^{n+1}) -  W(\mathbf{Q}_1^{n+1} - \mathbf{Q}_\star)
		- \kappa (\mathbf{Q}_1^{n+1} - \mathbf{Q}^n) -
		f_{\rm bulk}(\mathbf{Q}^n),                    \\
		 & \mathbf{H}_2^{n+1} = \nabla \cdot (K \nabla
		\mathbf{Q}_2^{n+1}) - (W + \kappa) \mathbf{Q}_2^{n+1},
		\\
		 & \mathbf{H}_3^{n+1} = \nabla \cdot (K \nabla
		\mathbf{Q}_3^{n+1}) - (W + \kappa) \mathbf{Q}_3^{n+1}.
		\\
	\end{aligned}
	\right.
\end{equation*}
To close the time step, we solve for the two scalars $\zeta^{n+1}$
and $\omega^{n+1}$. Noting that $\mathbf{G}_\mathbf{Q}^{n+1} =
	-\mathbf{H}^{n+1}$ and letting $\mathbf{G}_i = -\mathbf{H}_i$, we
take the inner-product of the velocity split in
\eqref{eq:bdf1-semi-decomposition} with
$B_k(\mathbf{g}_\mathbf{v}^n)$ and $B_k(\mathbf{c}_\mathbf{v}^n)$,
and of the $\mathbf{Q}$-split \eqref{eq:bdf1-semi-decomposition}
with $-B_k(\mathbf{G}_\mathbf{Q}^n)$ and
$-B_k(\mathbf{C}_\mathbf{Q}^n)$. This produces four equations:
\begin{equation}\label{eq:semi-before-scalar}
	\begin{aligned}
		 & (\mathbf{g}_\mathbf{v}^{n+1}, B_k(\mathbf{g}_\mathbf{v}^n)) =
		(\mathbf{v}_1, B_k(\mathbf{g}_\mathbf{v}^n)) + \zeta^{n+1}
		(\mathbf{v}_2, B_k(\mathbf{g}_\mathbf{v}^n)) - \omega^{n+1}
		(\mathbf{v}_3, B_k(\mathbf{g}_\mathbf{v}^n)),                    \\
		 & (\mathbf{g}_\mathbf{v}^{n+1}, B_k(\mathbf{c}_\mathbf{v}^n)) =
		(\mathbf{v}_1, B_k(\mathbf{c}_\mathbf{v}^n)) + \zeta^{n+1}
		(\mathbf{v}_2, B_k(\mathbf{c}_\mathbf{v}^n)) - \omega^{n+1}
		(\mathbf{v}_3, B_k(\mathbf{c}_\mathbf{v}^n)),                    \\
		 & (\mathbf{G}_\mathbf{Q}^{n+1}, B_k(\mathbf{G}_\mathbf{Q}^n)) =
		(\mathbf{G}_1, B_k(\mathbf{G}_\mathbf{Q}^n)) + \zeta^{n+1}
		(\mathbf{G}_2, B_k(\mathbf{G}_\mathbf{Q}^n)) - \omega^{n+1}
		(\mathbf{G}_3, B_k(\mathbf{G}_\mathbf{Q}^n)),                    \\
		 & (\mathbf{G}_\mathbf{Q}^{n+1}, B_k(\mathbf{C}_\mathbf{Q}^n)) =
		(\mathbf{G}_1, B_k(\mathbf{C}_\mathbf{Q}^n)) + \zeta^{n+1}
		(\mathbf{G}_2, B_k(\mathbf{C}_\mathbf{Q}^n)) - \omega^{n+1}
		(\mathbf{G}_3, B_k(\mathbf{C}_\mathbf{Q}^n)).                    \\
	\end{aligned}
\end{equation}
By summing the first and third equations-and likewise the second
and fourth-in \eqref{eq:semi-before-scalar} to eliminate two
unknowns $(\bv_{\bv}, \bG_{\bQ})$, we arrive at the following $2
	\times 2$ systems for $\zeta^{n+1}$, $\omega^{n+1}$:
\begin{equation}\label{eq:semi-2by2}
	\left\lbrace
	\begin{aligned}
		 & [\mathcal{G}(B_k(\mathbf{g}_\mathbf{v}^n),
			   B_k(\mathbf{G}_\mathbf{Q}^n)) - (\mathbf{v}_2,
			   B_k(\mathbf{g}_\mathbf{v}^n)) - (\mathbf{G}_2,
			   B_k(\mathbf{G}_\mathbf{Q}^n))] \zeta^{n+1} + [(\mathbf{v}_3,
			                                   B_k(\mathbf{g}_\mathbf{v}^n)) + (\mathbf{G}_3,
			                                   B_k(\mathbf{G}_\mathbf{Q}^n))] \omega^{n+1}                \\
		 & \quad = (\mathbf{v}_1, B_k(\mathbf{g}_\mathbf{v}^n)) +
		(\mathbf{G}_1, B_k(\mathbf{G}_\mathbf{Q}^n)),                                                    \\
		 & [(\mathbf{v}_2, B_k(\mathbf{c}_\mathbf{v}^n)) +
			   (\mathbf{G}_2, B_k(\mathbf{C}_\mathbf{Q}^n))] \zeta^{n+1} -
		                                                  [\mathcal{G}(B_k(\mathbf{g}_\mathbf{v}^n),
			                                                  B_k(\mathbf{G}_\mathbf{Q}^n)) + (\mathbf{v}_2,
			                                                  B_k(\mathbf{c}_\mathbf{v}^n)) + (\mathbf{G}_2,
			                                                  B_k(\mathbf{C}_\mathbf{Q}^n))] \omega^{n+1} \\
		 & \quad = -(\mathbf{v}_1, B_k(\mathbf{c}_\mathbf{v}^n)) -
		(\mathbf{G}_1, B_k(\mathbf{C}_\mathbf{Q}^n)).
	\end{aligned}
	\right.
\end{equation}
{\color{blue}For sufficiently small time steps, the determinant of
	the coefficient matrix in system \eqref{eq:semi-2by2} is nonzero,
	and hence} system \eqref{eq:semi-2by2} admits a unique solution for
$\zeta^{n+1}$ and $\omega^{n+1}$. Substituting them back into the
decompositions,  \eqref{eq:bdf1-semi-decomposition}, yields
solutions $\mathbf{v}^{n+1}$ and $\mathbf{Q}^{n+1}$ at the next time level.

Notice  that at each time step, the update decouples into just
three Stokes type equations, multiple
scalar Poisson equations, and a single $2\times 2$ linear
system. These yield a linearly implicit, fully decoupled system {\color{blue}
		in which equations arising at each time step can be solved efficiently}.

We now turn to the method’s structure‐preserving property.
First, we note that if the initial $\mathbf{Q}$-tensor is
traceless, this property is exactly retained at each subsequent time.
\begin{thm}
	Suppose the initial condition $\mathbf{v}_0$ and $\mathbf{Q}_0$
	satisfies $\nabla \cdot \mathbf{v}_0 = 0$ and ${\rm
			tr}(\mathbf{Q}_0) = 0$, then for any $n = 1, \cdots, N_t$,
	solutions $\mathbf{Q}^n$ solved from \eqref{eq:sge-bdf1}
	satisfies ${\rm tr}(\mathbf{Q}^n) = 0$.
\end{thm}
\begin{proof}
	We prove it by induction. Assume ${\rm tr}(\mathbf{Q}^0) = 0$.
	For the SGE-BDF2 scheme, we also assume that the computed
	$\mathbf{Q}^1$ is traceless. Now, suppose that for any $1 \leq i
		\leq n$, the sequence $\{\mathbf{Q}^i\}_{i=1}^n$ solved from
	\eqref{eq:sge-bdf1} satisfies ${\rm tr} (\mathbf{Q}^i) = 0$. We
	next prove ${\rm tr}(\mathbf{Q}^{n+1}) = 0$. Applying the trace
	operator on the third equation of \eqref{eq:sge-bdf1} leads to
	\begin{equation}\label{eq:trace-bdf}
		\tfrac{\alpha_k}{\tau}  {\rm tr}(\mathbf{Q}^{n+1}) =
		\Gamma_\mathbf{Q} {\rm tr}(\mathbf{H}^{n+1}) + \zeta^{n+1} {\rm
			tr} (B_k (\mathbf{C}_\mathbf{Q}^n)) - \omega^{n+1} {\rm tr}
		(B_k (\mathbf{G}_\mathbf{Q}^n)) + {\rm tr} (B_k
		(\mathbf{F}_\mathbf{Q}^n)).
	\end{equation}
	{\color{ocre}
		Since $B_k$ is linear, it follows immediately that
		\begin{equation*}
			\begin{aligned}
				{\rm tr} (B_k (\mathbf{C}_\mathbf{Q}^n)) = B_k\left({\rm
					                                              tr}\left(\mathbf{C}_\mathbf{Q}^n\right)\right), \ {\rm tr}
				(B_k (\mathbf{G}_\mathbf{Q}^n)) = B_k\left({\rm
					                                     tr}\left(\mathbf{G}_\mathbf{Q}^n\right)\right), \  {\rm tr}
				(B_k (\mathbf{F}_\mathbf{Q}^n)) = B_k\left({\rm
					                                     tr}\left(\mathbf{F}_\mathbf{Q}^n\right)\right).
			\end{aligned}
		\end{equation*}
		We first show that
		\begin{equation}\label{eq:add_traceless_a1}
			{\rm tr}  (B_k (\mathbf{C}_\mathbf{Q}^n)) = {\rm tr} (B_k
			(\mathbf{G}_\mathbf{Q}^n)) = {\rm tr} (B_k
			(\mathbf{F}_\mathbf{Q}^n)) = 0.
		\end{equation}
		To do this, it suffices to prove
		\begin{equation*}
			{\rm tr}(\mathbf{C}_\mathbf{Q}^n) = {\rm
				tr}(\mathbf{G}_\mathbf{Q}^n) = {\rm tr}(\mathbf{F}_\mathbf{Q}^n)
			= 0
		\end{equation*}
		by induction, which immediately implies
		\eqref{eq:add_traceless_a1} for the case $k = 1$.  For $k = 2$,
		one also needs to verify
		\begin{equation*}
			{\rm tr}(\mathbf{C}_\mathbf{Q}^{n-1}) = {\rm
				tr}(\mathbf{G}_\mathbf{Q}^{n-1}) = {\rm
				tr}(\mathbf{F}_\mathbf{Q}^{n-1}) = 0,
		\end{equation*}
		and the proof proceeds in a completely analogous manner.

		By definition and assumption, we have
		\begin{equation*}
			\begin{aligned}
				 & {\rm tr}(\mathbf{F}_\mathbf{Q}^n) = \xi(\phi) {\rm
					                                       tr}(\mathbf{Q}^n) = 0,                                   \\
				 & {\rm tr} (\mathbf{C}_\mathbf{Q}^n) = - {\rm tr} (\mathbf{v}
				                                        \cdot \nabla \mathbf{Q}^n) + {\rm tr} (\mathbf{S}^n) = -
				\mathbf{v} \cdot \nabla {\rm tr}(\mathbf{Q}^n) + {\rm
					                 tr}(\mathbf{S}^n) = {\rm tr}(\mathbf{S}^n),                                    \\
				 & {\rm tr}(\mathbf{G}_\mathbf{Q}^n) = -{\rm tr} (\mathbf{H}^n).
			\end{aligned}
		\end{equation*}
		Therefore, it remains to show that
		\begin{equation*}
			{\rm tr}(\mathbf{S}^n)=0 \  \text{and} \ {\rm tr} (\mathbf{H}^n)=0.
		\end{equation*}
		A direct calculation yields
		\begin{equation*}
			\begin{aligned}
				{\rm tr} (\mathbf{S}^n) & = [{\rm tr}(\bm{\Omega}^n \cdot
					                            \mathbf{Q}^n - \mathbf{Q}^n \cdot \bm{\Omega}^n)]  + a [{\rm
							                                                                                  tr}(\mathbf{Q}^n \cdot \mathbf{D}^n + \mathbf{D}^n \cdot
						                                                                                  \mathbf{Q}^n)] + \tfrac{2a}{d} {\rm tr} (\mathbf{D}^n) - 2a
					                                                                                                                                             [(\mathbf{D}^n : \mathbf{Q}^n) {\rm tr}\left(\mathbf{Q}^n +
						                                                                                                                                             \tfrac{1}{d}\mathbf{I}\right)] \\
				                        & = [{\rm tr}(\bm{\Omega}^n \cdot
					                            \mathbf{Q}^n) - {\rm tr}(\mathbf{Q}^n \cdot \bm{\Omega}^n)] +  a [{\rm
							                                                                                            tr}(\mathbf{Q}^n \cdot \mathbf{D}^n) + {\rm tr}(\mathbf{D}^n \cdot
						                                                                                            \mathbf{Q}^n)]  - 2a
				(\mathbf{D}^n : \mathbf{Q}^n)
			\end{aligned}
		\end{equation*}
		where we have used  ${\rm tr}(\mathbf{D}^n) = \nabla \cdot
			\mathbf{v}^n = 0$ and ${\rm
				tr} (\mathbf{Q}^n) = 0$.
		Using the cyclic property of the trace and symmetry of
		$\mathbf{D}^n$, we obtain
		\begin{equation*}
			{\rm tr} (\mathbf{S}^n) =  2a {\rm tr} (\mathbf{D}^n \cdot
			\mathbf{Q}^n) - 2a {\rm
					tr} ((\mathbf{D}^n)^\top \cdot \mathbf{Q}^n) = 0.
		\end{equation*}
		Analogously, we have
		\begin{equation*}
			\begin{aligned}
				{\rm tr}(\mathbf{H}^n) & = {\rm tr} \left( \nabla \cdot \left(K
				\nabla \mathbf{Q}^n\right)  \right) - W {\rm tr}(
				\mathbf{Q}^n - \mathbf{Q}_\star ) - \kappa {\rm tr}
				(\mathbf{Q}^n - \mathbf{Q}^{n-1})                                                                                   \\
				                       & \quad - {\rm tr} \left[ \alpha(\phi)\mathbf{Q}^n - \beta(\phi)
					                                          \left((\mathbf{Q}^n)^2 - \tfrac{{\rm tr}((\mathbf{Q}^n)^2)}{d}
					                                          \right) \mathbf{I}  + \gamma(\phi) {\rm tr}((\mathbf{Q}^n)^2)
					                                          {\rm tr}(\mathbf{Q}^n) \right]            \\
				                       & = \nabla \cdot \left(K \nabla {\rm tr}(\mathbf{Q}^n)\right)
				- W {\rm tr}(\mathbf{Q}^n) + W {\rm tr}(\mathbf{Q}_\star) -
				\kappa {\rm tr}(\mathbf{Q}^n) + \kappa {\rm tr}(\mathbf{Q}^{n-1})                                                   \\
				                       & \quad - \alpha(\phi) {\rm tr}(\mathbf{Q}^n) + \beta(\phi)
				                                                                       {\rm tr} \left( (\mathbf{Q}^n)^2 -  {\rm tr}
				((\mathbf{Q}^n)^2 )\tfrac{\mathbf{I}}{d} \right) -
				\gamma(\phi) {\rm tr}((\mathbf{Q}^n)^2) {\rm tr}(\mathbf{Q}^n) = 0,
			\end{aligned}
		\end{equation*}
		where the last equality follows from the traceless property of
		$\mathbf{Q}^n$ and $\mathbf{Q}^{n-1}$. Consequently,
		\begin{equation*}
			{\rm tr}  (B_k (\mathbf{C}_\mathbf{Q}^n)) = {\rm tr} (B_k
			(\mathbf{G}_\mathbf{Q}^n)) = {\rm tr} (B_k
			(\mathbf{F}_\mathbf{Q}^n)) = 0.
		\end{equation*}
	}
	According to \eqref{eq:sge-bdf-H}, we have
	\begin{equation}\label{eq:trace-H-bdf}
		{\rm tr}(\mathbf{H}^{n+1}) =
		\left\lbrace
		\begin{aligned}
			 & \nabla \cdot (K \nabla {\rm tr} (\mathbf{Q}^{n+1})) - (W +
			          \kappa) {\rm tr}(\mathbf{Q}^{n+1}), \ \text{(SGE-BDF1)} \\
			 & \nabla \cdot (K \nabla {\rm tr} (\mathbf{Q}^{n+1})) - \tau
			(W + \kappa) \alpha_2 {\rm tr}(\mathbf{Q}^{n+1}), \ \text{(SGE-BDF2)}.
		\end{aligned}
		\right.
	\end{equation}
	Combining \eqref{eq:trace-bdf} and \eqref{eq:trace-H-bdf} yields
	a scalar Poisson problem for ${\rm tr} (\mathbf{Q}^{n+1})$. In
	the SGE-BDF1 case this takes the form
	\begin{equation}\label{eq:tr_poisson}
		\Big[\tfrac{\alpha_k}{\tau} + \kappa + W - \nabla \cdot \left(K
			\nabla (\bullet)\right)\Big] {\rm tr} (\mathbf{Q}^{n+1}) = 0.
	\end{equation}
	with the proposed boundary conditions. The unique solution is thus ${\rm
			tr}(\mathbf{Q}^{n+1}) = 0$. An identical argument applies to
	SGE-BDF2. Hence ${\rm tr} (\mathbf{Q}^{n+1}) = 0$ and the
	induction is complete.
\end{proof}
{\color{ocre}
	\begin{rmk}\label{rmk:boundary_tr}
		We remark that the unique solvability of \eqref{eq:tr_poisson}
		relies on imposed boundary conditions.
		For example, in the case of the Dirichlet boundary
		condition, equation
		\eqref{eq:tr_poisson} reduces to a Poisson equation with
		homogeneous Dirichlet boundary data, which admits a unique solution.

		For the homogeneous Neumann boundary condition
		imposed on $\mathbf{Q}$, equation \eqref{eq:tr_poisson} becomes
		a Poisson equation equipped with the homogeneous Neumann
		boundary condition. In this case, unique solvability
		is guaranteed in the sense
		\[
			\int_{\Omega} {\rm tr}(\mathbf{Q}^{n+1})\,{\rm d}\mathbf{x} = 0,
		\]
		which fixes the additive constant and ensures uniqueness of the solution.
	\end{rmk}
}

We now investigate energy stability of the algorithms.
We first prove that the SGE-BDF1 scheme is energy stable with a sufficiently large  stabilization parameter $\kappa$. This proof contains the main cancellation mechanism in the SGE
reformulation. {\color{magenta} The proof for the energy stability of the SGE-BDF2 algorithm involves
		several additional
		estimates for the extrapolated bulk
		molecular field. It is presented separately in
		\ref{app:sge-bdf2-stability}}.

The analyses rest upon {\color{blue} an assumption on the pointwise boundedness of
		$\bf Q^n$}:
\begin{equation}\label{eq:ass-stabilization}
	\max_{n=1, \cdots N_t} \max_{\mathbf{x} \in \overline{\Omega}}
	\|\mathbf{Q}^n\| \leq M,
\end{equation}
where $M>0$ is a constant.

	{   \color{blue}
		In our numerical solutions of the model, condition
		\eqref{eq:ass-stabilization} holds whenever
		the numerical solution doesn't blow up. Physically, it should not as an order parameter tensor for liquid crystals. For the continuous active liquid crystal model with the polynomial bulk free energy density, a proof for the pointwise boundedness of $\bQ$ model is not available. As the result,
		a fully rigorous proof at the discrete level would require a comprehensive convergence
		analysis using the inductive arguments analogous to those in
		\cite{SAV-NS, SAV-H2}, which is
		far  beyond the scope of this study. Nevertheless,
		extensive computations confirm that the stabilization strategy
		delivers both robustness and efficiency in many scenarios and the
		boundedness of $\bf Q^n$ is never an issue.
		\cite{Wang-CH-Stabilization, Lu-Stabilization, Feng-2013, Tang-2020}.}

\begin{lem}\label{lem:convex}
	Under the hypothesis, \eqref{eq:ass-stabilization},  energy
	functional $F_\kappa(\mathbf{Q})$ defined by
	\begin{equation}\label{eq:den_F_kap}
		\begin{aligned}
			F_\kappa(\mathbf{Q}) & = \int_\Omega \tfrac{\kappa}{2} {\rm
				                                     tr}(\mathbf{Q}^2) - \left[\tfrac{\alpha}{2} {\rm
					                                                          tr}(\mathbf{Q}^2) - \tfrac{\beta}{3} {\rm tr}(\mathbf{Q}^3) +
				                                                          \tfrac{\gamma}{4}{\rm tr}(\mathbf{Q}^2)^2  \right] d\mathbf{x} \\
			                     & = \tfrac{\kappa}{2} \|\mathbf{Q}\|^2 - F_{\text{bulk}}(\mathbf{Q})
		\end{aligned}
	\end{equation}
	is convex in $\Omega$  for sufficiently large $\kappa>0$.
\end{lem}
\begin{proof}
	To complete the proof, we only need to verify that the
	second-order variation of $F_\kappa(\mathbf{Q})$ is non-negative.
	Indeed, for any $\mathbf{P} \in \mathcal{M}_0$, a direct calculation gives
	\begin{equation*}
		D_2 F_\kappa(\mathbf{Q})[\mathbf{P}, \mathbf{P}] = \int_\Omega
		(\kappa - \alpha) {\rm tr}(\mathbf{P}^2) + 2\beta {\rm
			tr}(\mathbf{Q} \mathbf{P}^2) - \gamma [ {\rm
				tr}(\mathbf{Q^2}){\rm tr}(\mathbf{P}^2) + 2({\rm
				tr}(\mathbf{P}\mathbf{Q}))^2  ] d\mathbf{x}.
	\end{equation*}
	{\color{ocre} We note that $\mathbf{P} \in \mathcal{M}_0$
		and therefore admits the decomposition} $\mathbf{P} =
		\mathbf{W}^\top \Lambda \mathbf{W} = \lambda_i \mathbf{w}_i
		\mathbf{w}_i^\top$, where $\mathbf{w}_i, \ i = 1, \cdots, d$ are
	orthonormal eigenvectors. Using $|\mathbf{Q}| \leq M$, we find
	\begin{equation*}
		\begin{aligned}
			 & {\rm tr} (\mathbf{Q}\mathbf{P}^2) = {\rm tr( \lambda_i^2
				                                       \mathbf{Q} \mathbf{w}_i \mathbf{w}_i^\top )} = \lambda_i^2
			\mathbf{w}_i^\top \mathbf{Q} \mathbf{w}_i \leq M
			\sum\limits_{i=1}^d \lambda_i^2 = M {\rm tr}(\mathbf{P}^2),                                                          \\
			 & {\rm tr}(\mathbf{Q}^2){\rm tr}(\mathbf{P}^2) \leq M^2{\rm
					                                                     tr}(\mathbf{P}^2),  \quad [{\rm tr}(\mathbf{P} \mathbf{Q})]^2
			\leq |\mathbf{P}|^2 |\mathbf{Q}|^2 \leq M^2 {\rm tr}(\mathbf{P}^2).
		\end{aligned}
	\end{equation*}
	Consequently, the second variation satisfies
	\begin{equation*}
		D_2 F_\kappa(\mathbf{Q})[\mathbf{P}, \mathbf{P}] \geq [(\kappa
		- \alpha) - 2|\beta|M - 3|\gamma|M^2) \int_{\Omega} {\rm
			tr}(\mathbf{P}^2) d\mathbf{x}.
	\end{equation*}
	Hence, by choosing $\kappa \geq \alpha + 2 |\beta| M + 3 |\gamma|
		M^2$, the functional $F_\kappa(\mathbf{Q})$ is convex.
\end{proof}

\begin{lem}\label{lem:dg-bdf1}
	Assuming the stabilization parameter satisfies the requirement of
	Lemma~\ref{lem:convex}, we then have the following inequality:
	{\color{orange}
		\begin{equation*}
			F_{\text{bulk}}(\mathbf{Q}^{n+1}) - F_{\text{bulk}}(\mathbf{Q}^n)
			\leq (\kappa (\mathbf{Q}^{n+1} - \mathbf{Q}^n) +
			f_{\text{bulk}}(\mathbf{Q}^n), \mathbf{Q}^{n+1} - \mathbf{Q}^n),
		\end{equation*}
	}
\end{lem}
\begin{proof}
	Using the identity
	\begin{equation}\label{eq:dg-bdf1-linear}
		(\mathbf{Q}^{n+1}, \mathbf{Q}^{n+1} - \mathbf{Q}^n) =
		\tfrac{1}{2} \|\mathbf{Q}^{n+1}\|^2 -
		\tfrac{1}{2}\|\mathbf{Q}^n\|^2 + \tfrac{1}{2}\|\mathbf{Q}^{n+1}
		- \mathbf{Q}^n\|^2.
	\end{equation}
	By Lemma \ref{lem:convex} and Taylor expansion, there exits
	$\widetilde{\mathbf{Q}} = (1 - s) \mathbf{Q}^n +
		s\mathbf{Q}^{n+1}$, with $s \in [0, 1]$,  such that
		{\color{orange}
			\begin{equation}\label{eq:dg-bdf1-nlinear}
				\begin{aligned}
					F_\kappa(\mathbf{Q}^{n+1})- F_\kappa(\mathbf{Q}^n) & =
					(\kappa \mathbf{Q}^n - f_{\text{bulk}}(\mathbf{Q}^n),
					\mathbf{Q}^{n+1} - \mathbf{Q}^n) + D_2
					F_\kappa(\widetilde{\mathbf{Q}})[\mathbf{Q}^{n+1} -
						  \mathbf{Q}^n, \mathbf{Q}^{n+1} - \mathbf{Q}^n]                                                              \\
					                                                   & \geq (\kappa \mathbf{Q}^n - f_{\text{bulk}}(\mathbf{Q}^n),
					\mathbf{Q}^{n+1} - \mathbf{Q}^n).
				\end{aligned}
			\end{equation}
		}
		{\color{orange}
			Multiplying \eqref{eq:dg-bdf1-linear} by $\kappa$ and
			\eqref{eq:dg-bdf1-nlinear} by $-1$, and summing the resulting
			equations, we have
		}
		{\color{orange}
			\begin{equation}\label{eq:res_ene_pre}
				\begin{aligned}
					 & \left[\tfrac{\kappa}{2} \|\mathbf{Q}^{n+1}\|^2 -
						   F_\kappa(\mathbf{Q}^{n+1})\right] - \left[\tfrac{\kappa}{2}
						                                        \|\mathbf{Q}^{n}\|^2 - F_\kappa(\mathbf{Q}^{n})\right] +
					\tfrac{\kappa}{2} \|\mathbf{Q}^{n+1} - \mathbf{Q}^n\|^2                                           \\
					 & \quad \leq \left(\kappa (\mathbf{Q}^{n+1} - \mathbf{Q}^n) +
					f_{\text{bulk}}(\mathbf{Q}^n), \mathbf{Q}^{n+1} - \mathbf{Q}^n\right).
				\end{aligned}
			\end{equation}
			Combining \eqref{eq:res_ene_pre} and \eqref{eq:den_F_kap}, we
			obtain the desired result.
		}

\end{proof}
\begin{thm}\label{thm:semi-sge-bdf1}
	Provided that the stabilization parameter is sufficiently large
	so that the condition in Lemma~\ref{lem:convex} is satisfied, the
	total energy in the SGE-BDF1 scheme satisfies the following inequality
	\begin{equation*}
		\begin{aligned}
			E_{\rm BDF1}^{n+1} - E_{\rm BDF1}^n & \leq  -2 \tau
			\|\sqrt{\eta} \mathbf{D}^{n+1}\|^2 - \tau \|\sqrt{b}
			\mathbf{v}^{n+1}\|^2 - \tau \| \sqrt{\Gamma_\mathbf{Q}}
			\mathbf{H}^{n+1}\|^2
			\\
			                                    & \quad - \tfrac{\tau}{2} \|\mathbf{v}^{n+1} - \mathbf{v}^n\|^2
			- \tfrac{\tau}{2}\|\sqrt{K} \nabla (\mathbf{Q}^{n+1} -
			\mathbf{Q}^n)\|^2 + \tau (\mathbf{Q}^n, \chi \mathbf{D}^{n+1} -
			\xi \mathbf{H}^{n+1}),
		\end{aligned}
	\end{equation*}
	where
	\begin{equation*}
		E_{\rm BDF1}^n = \int_\Omega \left[\tfrac{1}{2} |\mathbf{v}^n|^2
			+ \tfrac{K}{2} |\nabla \mathbf{Q}^n|^2 + \tfrac{W}{2}
			|\mathbf{Q}^n - \mathbf{Q}^\star|^2 \right] d\mathbf{x} +
		F_{\text{bulk}}(\mathbf{Q}^n).
	\end{equation*}
	So, the scheme is energy stable for passive liquid crystal systems.
\end{thm}
\begin{proof}
	Taking the inner products of the first and third equations of
	\eqref{eq:sge-bdf1}, respectively, with
	$\mathbf{g}_\mathbf{v}^{n+1} = \mathbf{v}^{n+1}$ and
	$\mathbf{G}_\mathbf{Q}^{n+1} = -\mathbf{H}_{\rm BDF1}^{n+1}$,
	integrating by parts, and using $\nabla \cdot \mathbf{v}^{n+1} =
		0$, one obtains
	\begin{equation}\label{eq:ene-bdf1-01}
		\begin{aligned}
			\tfrac{1}{\tau}(\mathbf{v}^{n+1} - \mathbf{v}^n,
			\mathbf{v}^{n+1}) & = -2 \|\sqrt{\eta} \mathbf{D}^{n+1}\|^2 -
			\|\sqrt{b}\mathbf{v}^{n+1}\| + (\mathbf{f}_\mathbf{v}^n,
			\mathbf{v}^{n+1}),
			\\
			                  & \quad + \zeta^{n+1} (\mathbf{g}_\mathbf{v}^{n+1},
			\mathbf{c}_\mathbf{v}^n) - \omega^{n+1}
			(\mathbf{g}_\mathbf{v}^{n+1}, \mathbf{g}_\mathbf{v}^n),
		\end{aligned}
	\end{equation}
	\begin{equation} \label{eq:ene-bdf1-02}
		\begin{aligned}
			\tfrac{1}{\tau}(\mathbf{Q}^{n+1} - \mathbf{Q}^n,
			\mathbf{G}_\mathbf{Q}^{n+1}) & = - \|\sqrt{\Gamma_\mathbf{Q}}
			\mathbf{H}^{n+1}\|^2 - (\mathbf{F}_\mathbf{Q}^n,
			\mathbf{H}^{n+1})                                                               \\
			                             & \quad + \zeta^{n+1}(\mathbf{G}_\mathbf{Q}^{n+1},
			\mathbf{C}_\mathbf{Q}^n) -
			\omega^{n+1}(\mathbf{G}_\mathbf{Q}^{n+1}, \mathbf{G}_\mathbf{Q}^n).
		\end{aligned}
	\end{equation}
	A straightforward computation using Lemma~\ref{lem:dg-bdf1} yields
	\begin{equation} \label{eq:ene-bdf1-03}
		(\mathbf{v}^{n+1} - \mathbf{v}^n, \mathbf{v}^{n+1}) =
		\tfrac{1}{2}\|\mathbf{v}^{n+1}\|^2 -
		\tfrac{1}{2}\|\mathbf{v}^n\|^2 + \tfrac{1}{2}
		\|\mathbf{v}^{n+1} - \mathbf{v}^n\|^2,
	\end{equation}
	\begin{equation} \label{eq:ene-bdf1-04}
		\begin{aligned}
			(\mathbf{Q}^{n+1} - \mathbf{Q}^n,
			\mathbf{G}_\mathbf{Q}^{n+1}) & = - (\mathbf{Q}^{n+1} -
			\mathbf{Q}^n, \nabla \cdot (K\nabla \mathbf{Q}^{n+1})) +
			W(\mathbf{Q}^{n+1} - \mathbf{Q}^\star, \mathbf{Q}^{n+1} -
			\mathbf{Q}^n)                                                                           \\
			                             & \quad + (\mathbf{Q}^{n+1} - \mathbf{Q}^n, \kappa
			(\mathbf{Q}^{n+1} - \mathbf{Q}^n) +
			f_\mathcal{\text{bulk}}(\mathbf{Q}^n))
			\\
			                             & \geq \tfrac{1}{2}\|\sqrt{K}\nabla \mathbf{Q}^{n+1}\|^2 -
			\tfrac{1}{2}\|\sqrt{K}\nabla \mathbf{Q}^n\|^2
			+\tfrac{1}{2}\|\sqrt{K} \nabla (\mathbf{Q}^{n+1} -
			\mathbf{Q}^n) \|^2                                                                      \\
			                             & \quad  + \tfrac{1}{2} \|\sqrt{W}(\mathbf{Q}^{n+1} -
			\mathbf{Q}_\star)\|^2 - \tfrac{1}{2}\|\sqrt{W}(\mathbf{Q}^{n}
			- \mathbf{Q}_\star)\|^2 + \tfrac{1}{2}
			\|\sqrt{W}(\mathbf{Q}^{n+1} - \mathbf{Q}^n)\|^2                                         \\
			                             & \quad + F_{\text{bulk}}(\mathbf{Q}^{n+1}) -
			F_{\text{bulk}}(\mathbf{Q}^n)
		\end{aligned}
	\end{equation}
	\begin{equation}\label{eq:ene-bdf1-05}
		\begin{aligned}
			 & \zeta^{n+1} (\mathbf{g}_\mathbf{v}^{n+1},
			\mathbf{c}_\mathbf{v}^n) - \omega^{n+1}
			(\mathbf{g}_\mathbf{v}^{n+1}, \mathbf{g}_\mathbf{v}^n) +
			\zeta^{n+1} (\mathbf{G}_\mathbf{Q}^{n+1},
			\mathbf{C}_\mathbf{Q}^n) -
			\omega^{n+1}(\mathbf{G}_\mathbf{Q}^{n+1}, \mathbf{G}_\mathbf{Q}^n) \\
			 & \quad = \mathcal{G}(\mathbf{g}_\mathbf{v}^n,
			\mathbf{G}_\mathbf{Q}^n)\zeta^{n+1} \omega^{n+1} -
			\mathcal{G}(\mathbf{g}_\mathbf{v}^n, \mathbf{G}_\mathbf{Q}^n)
			\omega^{n+1} \zeta^{n+1} = 0.
		\end{aligned}
	\end{equation}
	Summing \eqref{eq:ene-bdf1-01} and \eqref{eq:ene-bdf1-02} and
	substituting the indentities
	\eqref{eq:ene-bdf1-03}-\eqref{eq:ene-bdf1-05}, one directly
	arrives at the discrete energy inequality.
\end{proof}

\begin{rmk}
	\color{magenta}
	The condition on the stabilization parameter $\kappa$ in Theorem~\ref{thm:semi-sge-bdf1} is a sufficient condition for ensuring energy stability of the algorithm at the theoretical level. Because stabilization methods are intrinsically connected to convex-splitting approaches, such conditions are commonly imposed to guarantee the convexity of the bulk free energy required in the analysis. They are, however, generally not necessary in practical computations.

	In practice, the choice of $\kappa$ must be balanced against the time-step size $\tau$. A smaller $\kappa$ may lead to a loss of convexity in the bulk free energy and thereby induce possible energy growth, but this effect can often be compensated for by taking a sufficiently small $\tau$. Conversely, a larger $\kappa$ enhances stability and permits larger time steps. However, choosing a large $\kappa$ together with a large time-step size may perturb the discrete system substantially away from the continuum model, leading to considerable numerical error.

	Thus, the interplay between $\kappa$ and $\tau$ is primarily a practical implementation issue, and these parameters should be calibrated with care.  Figure~\ref{fig:energy_kappa_dt} shows that energy stability   can  be achieved even when $\kappa$ is below the theoretical bound,  provided that $\tau$ is chosen appropriately.




\end{rmk}

\subsection{Second-order linearly implicit schemes based on the
	discrete gradient scheme}
The SGE-BDF1 and SGE-BDF2 schemes are both robust and efficient;
however, their energy stability proofs require a relatively
large stabilization parameter. To address this limitation, we adopt the
polarized discrete gradient (PDG) approach
\cite{pdg_siam,pdg_jcam}. In this method, a mild stabilization
parameter $\kappa$ would warrant energy stability {\color{blue}
		while compared with the
		SGE-BDF1 scheme}. The role of the stabilization parameter is to
guarantee solvability of the
$\bQ$-tensor system. Specifically, it suffices to take $\kappa \geq
	\max\{0, -\alpha\}$. Unlike the previous schemes, no additional
assumption on
$\kappa$ is required to establish energy stability. The resulting
second-order SGE-PDG scheme is formulated as follows:
\begin{equation}\label{eq:scheme-PDG}
	\left\lbrace
	\begin{aligned}
		 & \tfrac{\mathbf{v}^{n+1} - \mathbf{v}^{n-1}}{2\tau} = - \nabla
		p^{\overline{n}} + 2 \nabla \cdot (\eta \mathbf{D}^{\overline{n}})
		- b \mathbf{v}^{\overline{n}}
		+ \overline{\zeta}^{n} \mathbf{c}_{\mathbf{v}}^n
		- \overline{\omega}^{n} \mathbf{g}_{\mathbf{v}}^n +
		\mathbf{f}_\mathbf{v}^n,                                         \\
		 & \nabla \cdot \mathbf{v}^{\overline{n}} = 0,                   \\
		 & \tfrac{\bf Q^{n+1} - Q^{n-1}}{2\tau} = \Gamma_{\bf Q}
		\overline{\mathbf{H}}^{n} + \overline{\zeta}^n
		\mathbf{C}_{\mathbf{Q}}^n - \overline{\omega}^n
		\mathbf{G}^n_{\mathbf{Q}} + \mathbf{F}_{\mathbf{Q}}^n,           \\
		 & \overline{\zeta}^n =
		\tfrac{(\overline{\mathbf{g}}_\mathbf{v}^n,
			\mathbf{g}_\mathbf{v}^n) + (\overline{\mathbf{G}}_\mathbf{Q}^n,
			\mathbf{G}_\mathbf{Q}^n)}{\mathcal{G}(\mathbf{g}_\mathbf{v}^n,
			\mathbf{G}_\mathbf{Q}^n)} , \ \overline{\omega}^n =
		\tfrac{(\overline{\mathbf{g}}_\mathbf{v}^n,
			\mathbf{c}_\mathbf{v}^n) + (\overline{\mathbf{G}}_\mathbf{Q}^n,
			\mathbf{C}_\mathbf{Q}^n)}{\mathcal{G}(\mathbf{g}_\mathbf{v}^n,
			\mathbf{G}_\mathbf{Q}^n)},
	\end{aligned}
	\right.
\end{equation}
where $(\bullet)^{\overline{n}} = \tfrac{(\bullet)^{n+1} +
		(\bullet)^{n-1}}{2}$,
\begin{equation}\label{eq:polarized-H}
	\overline{\mathbf{H}}^n =\nabla \cdot (K \nabla
	\mathbf{Q}^{\overline{n}}) - W (\mathbf{Q}^{\overline{n}} -
	\mathbf{Q}_\star) - \kappa (\mathbf{Q}^{\overline{n}} -
	\mathbf{Q}^n) - 2 \overline{f}_{\text{bulk}} (\mathbf{Q}^{n-1},
	\mathbf{Q}^n, \mathbf{Q}^{n+1}).
\end{equation}
Here, $\overline{f}_{\text{bulk}}(\bullet, \bullet, \bullet)$ is the
PDG associated with the polarized energy
$\overline{F}_{\text{bulk}}(\bullet, \bullet)$ of
$F_{\text{bulk}}(\bullet)$. We call
$\overline{F}_{\text{bulk}}(\bullet, \bullet)$ a polarized functional
of $F_\text{bulk}$ if it satisfies
\begin{equation}
	\overline{F}_{\text{bulk}}(\mathbf{U}, \mathbf{U}) =
	F_{\text{bulk}}(\mathbf{U}), \ \overline{F}_{\text{bulk}}(\mathbf{U},
	\mathbf{V}) = \overline{F}_{\text{bulk}}(\mathbf{V}, \mathbf{U}).
\end{equation}
The first property ensures consistency of
$\overline{F}_{\text{bulk}}(\bullet, \bullet)$ with the original
energy, while the second enforces symmetry and thus guarantees the
second-order accuracy. We then say that
$\overline{f}_{\text{bulk}}(\bullet, \bullet, \bullet)$ is the
polarized discrete gradient of $\overline{F}_{\text{bulk}}$ if, for
all $\mathbf{U}$, $\mathbf{V}$, $\mathbf{W}$,
\begin{equation}\label{eq:den-pdg}
	2\overline{f}_\text{bulk}(\mathbf{U}, \mathbf{U}, \mathbf{U}) =
	f_\text{bulk}(\mathbf{U}),
	\ (\overline{f}_\text{bulk}(\mathbf{U}, \mathbf{V}, \mathbf{W}),
	\mathbf{W} - \mathbf{U})=  \overline{F}_\text{bulk}(\mathbf{V},
	\mathbf{W}) - \overline{F}_\text{bulk}(\mathbf{U}, \mathbf{V}).
\end{equation}
The key to the PDG method is the selection of an appropriate
polarized energy functional. With a judicious choice, one obtains a
linearly implicit scheme. Here, we define the polarization of the
bulk energy as follows:
\begin{equation} \label{eq:polarized-FN}
	\overline{F}_\text{bulk}(\mathbf{U}, \mathbf{V}) = \int_\Omega
	[\tfrac{\alpha}{2} \tfrac{{\rm tr}(\mathbf{U}^2) + {\rm
				tr}(\mathbf{V}^2)}{2} - \tfrac{\beta}{3} {\rm tr}(\mathbf{U}
		\mathbf{V} \tfrac{\mathbf{U} + \mathbf{V}}{2}) +
		\tfrac{\gamma}{4} {\rm tr} (\mathbf{U}^2) {\rm tr}(\mathbf{V}^2)]
	d\mathbf{x}.
\end{equation}
The polarized discrete gradient $\overline{f}_\mathcal{N}$ in
\eqref{eq:polarized-H} to \eqref{eq:polarized-FN} is as follows
\begin{equation} \label{eq:polarized-fN}
	\begin{aligned}
		\overline{f}_\text{bulk}(\mathbf{U}, \mathbf{V}, \mathbf{W}) & =
		\tfrac{\alpha}{2} \tfrac{\mathbf{U} + \mathbf{W}}{2} -
		\tfrac{\beta}{6} \Big[ \left(\mathbf{V}^2 + \mathbf{V}
			                 \tfrac{\mathbf{U} + \mathbf{W}}{2} +
			                 \tfrac{\mathbf{U}+\mathbf{W}}{2} \mathbf{V}\right) \\
			                 &\quad - \tfrac{1}{3} {\rm tr} \left(\mathbf{V} (\mathbf{U} +
			                 \mathbf{V} + \mathbf{W})\right) \mathbf{I}  \Big] +
		\tfrac{\gamma}{2} {\rm tr}(\mathbf{V}^2) \tfrac{\mathbf{U} +
			                                         \mathbf{W}}{2}.
	\end{aligned}
\end{equation}
\begin{rmk}
	In the scalar scenarios, one can define the PDG using the
	following Itoh-Abe formulation
	\begin{equation}\label{eq:Itoh-Abe}
		\overline{f}(u, v, w) = \tfrac{F(v, w) - F(u, v)}{w - u}.
	\end{equation}
	However, \eqref{eq:Itoh-Abe} cannot be extended directly to
	tensor variables, since $\mathbf{U}$, $\mathbf{W}$ and energy
	differences live in noncommutative tensor spaces, so the division
	is undefined.
\end{rmk}

We next demonstrate that if ${\rm tr}(\mathbf{Q}^0) = 0$ and the
solution $\mathbf{Q}^1$ is solved from a chosen startup scheme that is also
traceless, then the SGE-PDG scheme preserves ${\rm
		tr}(\mathbf{Q}^n) = 0$ at every future step.

\begin{thm}
	Suppose the SGE-PDG algorithm is initialized with ${\rm tr}({\bf
			Q}^0)$ and ${\rm tr}({\bf Q}^1) = 0$ both traceless and $\nabla
		\cdot \mathbf{v}^0$, $\nabla \cdot \mathbf{v}^1 = 0$ both
	divergence-free, and choose $\kappa \geq \max \{0, -\alpha\}$.
	Then every subsequent solution $\{\mathbf{v}^n, \mathbf{Q}^n
		\}_{n=2}^{N_t}$ produced by \eqref{eq:scheme-PDG} satisfies
	$\nabla \cdot \mathbf{v}^n = 0$, \ ${\rm tr} (\mathbf{Q}^n) = 0$.
\end{thm}
\begin{proof}
	The divergence-free condition follows immediately from the second
	formulation \eqref{eq:scheme-PDG} by a straightforward inductive
	argument. It remains to show that $\mathbf{Q}^{n+1}$ is
	traceless. By hypothesis, ${\rm tr}(\mathbf{Q}^0) = {\rm
			tr}(\mathbf{Q}^1) = 0$. Assume ${\rm tr}(\mathbf{Q}^k) = 0$ for
	$k = 2, \cdots, n$. Taking the trace of the third equation in
	\eqref{eq:scheme-PDG} and using the inductive assumption, we obtain
	\begin{equation}\label{eq:trQnp1}
		\begin{aligned}
			\tfrac{{\rm tr} (\mathbf{Q}^{n+1})}{2\tau} & =
			\Gamma_\mathbf{Q} {\rm tr}(\overline{\mathbf{H}}^n) +
			\overline{\zeta}^n {\rm tr} (\mathbf{C}_\mathbf{Q}^n) -
			\overline{\omega}^n {\rm tr}(\mathbf{G}_\mathbf{Q}^n),                                         \\
			                                           & = \Gamma_\mathbf{Q}{\rm tr}(\overline{\bf H}^n) +
			\overline{\zeta}^n {\rm tr}(\mathbf{S}^n)
		\end{aligned}
	\end{equation}
	By $\nabla \cdot \mathbf{v}^n = 0$ and the facts that
	$\mathbf{Q}^n$ is symmetric while $\bm{\Omega}^n$ is
	skew-symmetric, we have
	\begin{equation*}
		\begin{aligned}
			{\rm tr}(\mathbf{S}^n) & = {\rm tr}(a (\mathbf{Q}^n
			\mathbf{D}^n + \mathbf{D}^n \mathbf{Q}^n)) - \tfrac{2a}{d}
			                                             (\mathbf{Q}^n : \mathbf{D}^n) {\rm tr} (\mathbf{I}) \\
			                       & = 2a {\rm tr}(\mathbf{Q}^n \mathbf{D}^n) - 2a
			\mathbf{Q}^n:\mathbf{D}^n = 0.
		\end{aligned}
	\end{equation*}
	By straightforward calculations, one finds
	\begin{equation}\label{eq:tr-H}
		\begin{aligned}
			{\rm tr}(\overline{\mathbf{H}}^n) & = \tfrac{1}{2} \nabla
			\cdot \left(K \Delta {\rm tr}(\mathbf{Q}^{n+1})\right) -
			\tfrac{W + \kappa + \alpha}{2} {\rm tr} (\mathbf{Q}^{n+1}) -
			\tfrac{\gamma}{2} |\mathbf{Q}^n|^2 {\rm tr}(\mathbf{Q}^{n+1}).
		\end{aligned}
	\end{equation}
	Combining \eqref{eq:trQnp1} and \eqref{eq:tr-H}, we end up with
	the scalar Poisson problem with respect to ${\rm tr} (\mathbf{Q}^{n+1})$:
	\begin{equation*}
		\left[\tfrac{1}{2}(\tfrac{1}{\tau} + W + \kappa + \alpha +
			\gamma |\mathbf{Q}^n|^2) - \nabla \cdot ( K \nabla (\bullet) )
			\right] {\rm tr}(\mathbf{Q}^{n+1}) = 0.
	\end{equation*}
	Whenever $\kappa \geq \max \{0, -\alpha\}$, {\color{ocre} by
		imposing boundary conditions on
		${\rm tr}(\mathbf{Q}^{n+1})$ that are compatible with those prescribed for
		$\mathbf{Q}^{n+1}$ (see Remark~\ref{rmk:boundary_tr}) and by
		invoking the classical theory of
		elliptic equations,
		the above problem admits a unique solution, satisfying}
	\begin{equation*}
		{\rm tr} (\mathbf{Q}^{n+1}) = 0.
	\end{equation*}
	The proof is thus completed.
\end{proof}

Finally, we show that SGE-PDG preserves the energy dissipation rate
even in the active liquid crystal case.
\begin{thm}\label{thm:semi-sge-pdg}
	The discrete energy dissipation rate of the SGE-PDG scheme is given by
	\begin{equation*}
		\begin{aligned}
			E^{n+1}_{\rm PDG} - E^n_{\rm PDG} & = -2 \tau \|\sqrt{\eta}
			\mathbf{D}^{\overline{n}}\|^2 - \tau \|\sqrt{b}
			\mathbf{v}^{\overline{n}}\|^2 - \tau \|
			\sqrt{\Gamma_\mathbf{Q}} \overline{\mathbf{H}}^n \|^2 + \tau
			(\mathbf{Q}^n, \chi \mathbf{D}^{\overline{n}} - \xi
			\overline{\mathbf{H}}^n),
		\end{aligned}
	\end{equation*}
	where
	\begin{equation*}
		\begin{aligned}
			E^n_{\rm PDG} & = \int_\Omega \tfrac{1}{2}
			\tfrac{|\mathbf{v}^{n-1}|^2 + |\mathbf{v}^n|^2}{2} +
			\tfrac{K}{2} \tfrac{|\nabla \mathbf{Q}^{n-1}|^2 +  |\nabla
				             \mathbf{Q}^n|^2}{2} + \tfrac{W}{2} \tfrac{|\mathbf{Q}^{n-1} -
				                                                 \mathbf{Q}_\star|^2 + |\mathbf{Q}^n-\mathbf{Q}_\star|^2}{2} +
			\tfrac{\kappa}{4}|\mathbf{Q}^n - \mathbf{Q}^{n-1}|^2 d\mathbf{x}                  \\
			              & \quad + \overline{F}_\text{bulk}(\mathbf{Q}^{n-1}, \mathbf{Q}^n).
		\end{aligned}
	\end{equation*}
	For passive liquid crystal systems, it implies the scheme is
	unconditional energy stable.
\end{thm}
\begin{proof}
	Taking the inner-product of the first equation in
	\eqref{eq:scheme-PDG} with $\mathbf{v}^{\overline{n}}$ and using
	incompressibility from the continuity equation immediately yields
	\begin{equation}\label{eq:energy-pdg-01}
		\begin{aligned}
			 & \tfrac{1}{4\tau}\|\mathbf{v}^{n+1}\|^2 -
			\tfrac{1}{4\tau}\|\mathbf{v}^{n-1}\|^2                          \\
			 & = (p^{\overline{n}}, \nabla \cdot \mathbf{v}^{\overline{n}})
			- 2 (\eta \mathbf{D}^{\overline{n}},
			\mathbf{v}^{\overline{n}}) - (b \mathbf{v}^{\overline{n}},
			\mathbf{v}^{\overline{n}}) + \overline{\zeta}^n
			(\mathbf{c}_\mathbf{v}^n, \overline{\mathbf{g}}_\mathbf{v}^n)
			- \overline{\omega}^n (\mathbf{g}_\mathbf{v}^n,
			\overline{\mathbf{g}}_\mathbf{v}^n) + (\mathbf{Q}^n, \chi
			\nabla \mathbf{v}^{\overline{n}})                               \\
			 & = -2 \|\sqrt{\eta} \mathbf{D}^{\overline{n}}\|^2 -
			\|\sqrt{b} \mathbf{v}^{\overline{n}}\|^2 + \overline{\zeta}^n
			(\mathbf{c}_\mathbf{v}^n, \overline{\mathbf{g}}_\mathbf{v}^n)
			- \overline{\omega}^n (\mathbf{g}_\mathbf{v}^n,
			\overline{\mathbf{g}}_\mathbf{v}^n) + (\mathbf{Q}^n, \chi
			\mathbf{D}^{\overline{n}}).
		\end{aligned}
	\end{equation}
	Taking the inner-product of the third equation in
	\eqref{eq:scheme-PDG} with $\overline{\mathbf{G}}_\mathbf{Q}^n =
		-\overline{\mathbf{H}}^n$ yields
	\begin{equation}\label{eq:energy-pdg-02}
		-\tfrac{1}{2 \tau}(\mathbf{Q}^{n+1} - \mathbf{Q}^{n-1},
		\overline{\mathbf{H}}^n) = - \| \sqrt{\Gamma_\mathbf{Q}}
		\overline{\mathbf{H}}^n\|^2 + \overline{\zeta}^n
		(\mathbf{C}_\mathbf{Q}^n, \overline{\mathbf{G}}_\mathbf{Q}^n) -
		\overline{\omega}^n (\mathbf{G}_\mathbf{Q}^n,
		\overline{\mathbf{G}}_\mathbf{Q}^n) - (\mathbf{Q}^n, \xi
		\overline{\mathbf{H}}^n).
	\end{equation}
	Substituting the definition of $\overline{\mathbf{H}}^n$
	\eqref{eq:polarized-H} then leads to
	\begin{equation}\label{eq:energy-pdg-03}
		\begin{aligned}
			 & -\tfrac{1}{2}(\mathbf{Q}^{n+1} - \mathbf{Q}^{n-1},
			\overline{\mathbf{H}}^n) -
			(\overline{f}_\text{bulk}(\mathbf{Q}^{n-1}, \mathbf{Q}^n,
			\mathbf{Q}^{n+1}), \mathbf{Q}^{n+1} - \mathbf{Q}^{n-1})                                                               \\
			 & = -\tfrac{1}{2}(\mathbf{Q}^{n+1} - \mathbf{Q}^{n-1}, \nabla
			\cdot (K \nabla \mathbf{Q}^{\overline{n}}) - W
			(\mathbf{Q}^{\overline{n}} - \mathbf{Q}_\star) - \kappa
			(\mathbf{Q}^{\overline{n}} - \mathbf{Q}^n))                                                                           \\
			 & = \tfrac{1}{2} (K \nabla \mathbf{Q}^{\overline{n}}, \nabla
			\mathbf{Q}^{n+1} - \nabla \mathbf{Q}^{n-1}) +
			\tfrac{\kappa}{2} (\mathbf{Q}^{\overline{n}} - \mathbf{Q}^n,
			\mathbf{Q}^{n+1}-\mathbf{Q}^{n-1})                                                                                    \\
			 & \quad +  \tfrac{1}{2} ( W (\tfrac{\mathbf{Q}^{n+1} -
				                              \mathbf{Q}_\star}{2} +  \tfrac{\mathbf{Q}^{n-1} -
				                                                       \mathbf{Q}_\star}{2}), (\mathbf{Q}^{n+1} - \mathbf{Q}_\star )
			- (\mathbf{Q}^{n-1} - \mathbf{Q}_\star))                                                                              \\
			 & = \int_\Omega \tfrac{K}{4} (|\nabla \mathbf{Q}^{n+1}|^2 -
			|\nabla \mathbf{Q}^{n-1}|^2) + \tfrac{\kappa}{4}
			|\mathbf{Q}^{n+1} - \mathbf{Q}^n|^2 - \tfrac{\kappa}{4}
			|\mathbf{Q}^{n} - \mathbf{Q}^{n-1}|^2 d\mathbf{x},                                                                    \\
			 & \quad + \int_\Omega \tfrac{W}{4} (| \mathbf{Q}^{n+1} -
			\mathbf{Q}_\star |^2 - |\mathbf{Q}^{n-1} -
			\mathbf{Q}_\star|^2) d\mathbf{x}.
		\end{aligned}
	\end{equation}
	By the definition of the discrete gradient \eqref{eq:den-pdg}, we obtain
	\begin{equation}\label{eq:energy-pdg-04}
		(\overline{f}_\text{bulk}(\mathbf{Q}^{n-1}, \mathbf{Q}^n,
		\mathbf{Q}^{n+1}), \mathbf{Q}^{n+1}-\mathbf{Q}^{n-1}) =
		\overline{F}_\text{bulk}(\mathbf{Q}^n, \mathbf{Q}^{n+1}) -
		\overline{F}_\text{bulk}(\mathbf{Q}^{n-1}, \mathbf{Q}^n).
	\end{equation}
	Finally, combining
	\eqref{eq:energy-pdg-01}-\eqref{eq:energy-pdg-04} with the
	following identity, \eqref{eq:energy-pdg-05},
	\begin{equation}\label{eq:energy-pdg-05}
		\begin{aligned}
			 & \overline{\zeta}^n (\mathbf{c}_\mathbf{v}^n,
			\overline{\mathbf{g}}_\mathbf{v}^n) - \overline{\omega}^n
			(\mathbf{g}_\mathbf{v}^n, \overline{\mathbf{g}}_\mathbf{v}^n)
			+ \overline{\zeta}^n (\mathbf{C}_\mathbf{Q}^n,
			\overline{\mathbf{G}}_\mathbf{Q}^n) - \overline{\omega}^n
			(\mathbf{G}_\mathbf{Q}^n, \overline{\mathbf{G}}_\mathbf{Q}^n) \\
			 & = \mathcal{G}(\mathbf{g}_\mathbf{v}^n,
			\mathbf{G}_\mathbf{Q}^n) \overline{\zeta}^n
			\overline{\omega}^n - \mathcal{G}(\mathbf{g}_\mathbf{v}^n,
			\mathbf{G}_\mathbf{Q}^n) \overline{\omega}^n \overline{\zeta}^n  = 0,
		\end{aligned}
	\end{equation}
	completes the proof.
\end{proof}
The equations resulted from the SGE-PDG scheme are solved as follows.

\begin{itemize}
	\item {\bf Generalized Stokes equations. } For $i = 1, 2, 3$,
	      find $\mathbf{v}_i$ by solving:
	      \begin{equation}
		      \left\lbrace
		      \begin{aligned}
			       & \tfrac{1}{2}(\tfrac{1}{\tau} + b)\mathbf{v}_i^{n+1} -
			      \nabla \cdot (\eta \mathbf{D}(\mathbf{v}_i^{n+1})) + \nabla
			      p^{n+1} = \mathbf{f}_i,                                  \\
			       & \nabla \cdot \mathbf{v}_i^{n+1} = 0.
		      \end{aligned}
		      \right.
	      \end{equation}
	      {\color{blue}
		      Here,
		      \begin{equation*}
			      \begin{aligned}
				       & \mathbf{f}_1 = \tfrac{1}{2}\left(\tfrac{1}{\tau} -
				      b\right)\mathbf{v}^{n-1} + \nabla
				      \cdot (\eta \mathbf{D}(\mathbf{v}^{n-1})) +
				      \mathbf{f}_\mathbf{v}^n,                              \\
				       & \mathbf{f}_2 = \mathbf{c}_\mathbf{v}^n, \quad
				      \mathbf{f}_3 = \mathbf{g}_\mathbf{v}^n.               \\
			      \end{aligned}
		      \end{equation*}
	      }
	\item {\bf Coupled tensor-valued Poisson equations. } For $i = 1,
		      2, 3$, find $\mathbf{Q}_i$ from
	      \begin{equation}
		      \mathscr{L}(\mathbf{Q}^n) \mathbf{Q}^{n+1} = \mathbf{F}_i,
	      \end{equation}
\end{itemize}
where, the linear operator $\mathscr{L}$ is
\begin{equation*}
	\begin{aligned}
		\mathscr{L}(\mathbf{Q}^n)\mathbf{Q} & = \Big(\tfrac{1}{2\tau} +
		\tfrac{\kappa + W}{2}\Gamma_\mathbf{Q} - \tfrac{\alpha
			                                         \Gamma_\mathbf{Q}}{2} -
		\tfrac{\beta\Gamma_\mathbf{Q}}{6} \mathbf{Q}^n +
		\tfrac{\beta\Gamma_\mathbf{Q}}{2}{\rm tr}((\mathbf{Q}^n)^2) -
		\tfrac{\Gamma_\mathbf{Q}}{2}\nabla \cdot (K \nabla
		(\bullet))\Big) \mathbf{Q},                                                      \\
		                                    & \quad - \tfrac{\beta \Gamma_\mathbf{Q}}{6}
		\mathbf{Q}\mathbf{Q}^n + \tfrac{\beta
			                         \Gamma_\mathbf{Q}}{9} {\rm tr}(\mathbf{Q}^n \mathbf{Q}) \mathbf{I},
	\end{aligned}
\end{equation*}
and
	{
		\color{blue}
		\begin{equation*}
			\begin{aligned}
				 & \mathbf{F}_1 = \left(\tfrac{1}{2\tau} +
				\tfrac{\Gamma_\mathbf{Q}}{2} \nabla \cdot \left(K \nabla
				(\bullet)\right) - \tfrac{\kappa + W}{2}
				\Gamma_\mathbf{Q} - \tfrac{\Gamma_\mathbf{Q}\alpha(\phi)}{2} +
				\tfrac{\Gamma_\mathbf{Q}\beta(\phi)}{6} \mathbf{Q}^{n} -
				\tfrac{\Gamma_\mathbf{Q}\beta(\phi)}{2} {\rm
					tr}((\mathbf{Q}^n)^2) \right)
				\mathbf{Q}^{n-1}                                              \\
				 & \quad + \Gamma_\mathbf{Q} W
				\mathbf{Q}_\star + \kappa \Gamma_\mathbf{Q} \mathbf{Q}^n +
				\left(\tfrac{\Gamma_\mathbf{Q} \beta (\phi)}{3}\mathbf{Q}^n +
				\tfrac{\Gamma_\mathbf{Q} \beta(\phi)}{6}
				\mathbf{Q}^{n-1}\right) \mathbf{Q}^n                          \\
				 & \quad -
				\tfrac{\Gamma_\mathbf{Q} \beta(\phi)}{9} {\rm tr}
				((\mathbf{Q}^n)^2) \mathbf{I} - \tfrac{\Gamma_\mathbf{Q}
					                                \beta(\phi)}{9} {\rm tr}(\mathbf{Q}^n
				\mathbf{Q}^{n-1})\mathbf{I} + \mathbf{F}_\mathbf{Q}^n,        \\
				 & \mathbf{F}_2 = \mathbf{C}_\mathbf{Q}^n, \quad \mathbf{F}_3
				= \mathbf{G}_\mathbf{Q}^n.
			\end{aligned}
		\end{equation*}
		Using this algorithm, a coupled Poisson equation system
		needs to be solved, rather than three separate Poisson equations as in
		the previous SGE-BDF1 algorithm.
	}

\section{Spatial discretization and implementation}
The spatial discretization of the afore-developed semidiscrete
schemes is performed using finite difference schemes on a staggered
grid \cite{Hong-2024-JCP,Hong-2025-POF}. To ensure that the fully
discrete scheme preserves the energy dissipation properties stated
in Theorems~\ref{thm:semi-sge-bdf1} and \ref{thm:semi-sge-pdg}, it
is essential that the spatial discretization preserves discrete
summation-by-parts (SBP) identities (see \cite{Gong-2018-CHNS}), such as
\begin{equation}\label{eq:mac-sbp}
	\begin{aligned}
		 & (\nabla_h \cdot \mathbf{v}_h, p_h)_h = - (\mathbf{v}_h,
		\nabla p_h)_h,                                                  \\
		 & (\nabla_h \cdot (\eta\mathbf{D}_h) \mathbf{v}_h,
		\mathbf{u}_h)_h = -(\eta \mathbf{D}_h(\mathbf{v}_h),
		\mathbf{D}_h(\mathbf{u}_h))_h,                                  \\
		 & (\nabla_h \cdot( K \nabla_h \mathbf{Q}_h), \mathbf{P}_h)_h =
		- (K \nabla_h \mathbf{Q}_h, \nabla_h \mathbf{P}_h)_h.
	\end{aligned}
\end{equation}
Here, $(\cdot, \cdot)_h$ denotes the discrete inner product. The
discrete inner product may depend on the location of variables on
the staggered grid, but for simplicity, we do not introduce
multiple distinct inner-product notations. The operators $\nabla_h$
and $\nabla_h \cdot$ represent the discrete gradient and
divergence operation, respectively,
and $p_h, \mathbf{v}_h, \mathbf{Q}_h$ are the discrete counterparts
of the physical variables. Identities in \eqref{eq:mac-sbp} can be ensured by
employing a mark-and-cell (MAC) schemes, which warrants
Theorems~\ref{thm:semi-sge-bdf1} and \ref{thm:semi-sge-pdg}
after spatial discretizations.
\begin{rmk}
	In structure-preserving frameworks, achieving full energy dissipation at the
	discrete level also requires the preservation of fundamental structural
	properties {\color{ocre} under boundary conditions \eqref{eq:bc-v} and
			\eqref{eq:bc-Q}}:
	\begin{equation}\label{eq:sbp-conv-ns}
		(\mathbf{v} \cdot \nabla \mathbf{v}, \mathbf{v}) = 0,
	\end{equation}
	\begin{equation}\label{eq:sbp-conv-qtensor}
		(\mathbf{v} \cdot \nabla \mathbf{Q}, f^\prime(\mathbf{Q})) =
		(\mathbf{v}, \nabla f(\mathbf{Q}) ) = -(f(\mathbf{Q}), \nabla
		\cdot \mathbf{v}) = 0.
	\end{equation}
	To maintain \eqref{eq:sbp-conv-ns}, a common approach is to
	rewrite $\mathbf{v}\cdot \nabla \mathbf{v}$ as $B(\mathbf{v},
		\mathbf{v}) = \tfrac{1}{2}[\mathbf{v}\cdot \nabla \mathbf{v} +
			\nabla \cdot (\mathbf{v}\mathbf{v})]$, which leverages the
	divergence free condition, and can be discretized to preserve
	\eqref{eq:sbp-conv-ns} due to its skew-symmetry property.

	However, \eqref{eq:sbp-conv-qtensor} is generally not preserved
	after spatial discretization if $f(\cdot)$ is non-quadratic.
	Consequently, fully discrete energy-dissipative schemes cannot be
	constructed in the classical framework, even though the
	semi-discrete dissipation may still be achieved (see
	\cite{Zhao&W&Y2016, J_Zhao_JSC_Qtensor}).

	In contrast, under the SGE framework, both $\mathbf{v} \cdot
		\nabla \mathbf{v}$ and $\mathbf{v} \cdot \nabla \mathbf{Q}$ are
	inherently skew-symmetrized with respect to the energy gradient,
	allowing \eqref{eq:sbp-conv-ns} and \eqref{eq:sbp-conv-qtensor}
	to hold naturally, regardless of the spatial discretization.
\end{rmk}

In the practical implementation, all resulting linear systems are
solved using a multigrid method. For the generalized Stokes system,
we employ the Vanka smoother \cite{Vanka}. For scalar Poisson
equations, the classical Gauss-Seidel method is used as a smoother.
For vector-valued Poisson equations, a block Gauss-Seidel smoother
is adopted and an additional projection step is included in the
smoother to enforce the traceless property of $\bQ$.

\section{Numerical results}
\subsection{Convergence test}
We begin by evaluating the spatial and temporal convergence
properties of the three proposed schemes. In the absence of an
analytical solution, we employ the following functions as the
solution of governing system of equations
\eqref{eq:governing-dimensionless} with a forcing term,
\begin{equation*}
	\left\lbrace
	\begin{aligned}
		 & u = \cos(t^2) (-\cos{(2 \pi x)}\sin{(2\pi y)} + \sin{(2\pi y)}), \\
		 & v = \cos(t^2) (\sin{(2\pi x)} \cos{(2\pi y)} - \sin{(2\pi x)}),  \\
		 & p = \cos(t^2) \cos{(2\pi x)} \cos{(2\pi y)},                     \\
		 & \mathbf{Q} = \cos(t^2) \cos(4  \pi x) \cos(4 \pi y)
		\begin{pmatrix}
			0.5  & 0.15 & 0     \\
			0.15 & 0.25 & 0     \\
			0    & 0    & -0.75
		\end{pmatrix}.
	\end{aligned}
	\right.
\end{equation*}
In this test, all simulations are conducted on the unit square,
$\Omega = (0, 1)^2$, with $\Gamma = 10$, $K = 10^{-3}$ and all
remaining parameters are set to unity. Zero Dirichlet boundary
conditions are employed to the velocity field and zero Neumann
boundary conditions are imposed to all the components of the
$\mathbf{Q}$-tensor field. {\color{blue}
		In all numerical experiments, if there is no additional
		explanation, the stabilization parameter is set
		to $\kappa = 5$ for the SGE-BDF1 and SGE-BDF2 schemes, and to
		$\kappa = 1$ for the SGE-PDG scheme.
	} For the SGE-BDF1 scheme, a fixed spatial
grid of $h = \frac{1}{512}$ is employed such that the spatial error
is negligible while the time step is refined; {\color{blue} For the
	SGE-BDF2 and SGE-PDG schemes, the spatial and temporal discretizations are
	refined simultaneously in order to verify the second-order
	convergence of the scheme in both space and time.} At the
termination time $T
	= 1$, the following discrete errors of the velocity and
$\mathbf{Q}$-tensor are computed, and convergence rates are
extracted via the log-log regression. The discrete errors are calculated by
\begin{equation*}
	\text{Error}_\mathbf{v} = \left(\sum\limits_{i=1}^{N_x}
	\sum\limits_{j=1}^{N_y} \left|\mathbf{v}_{i,j}^{N_t} -
	\mathbf{v}(x_i, y_j, T)\right|^2\right)^{1/2}, \
	\text{Error}_\mathbf{Q} = \left(\sum\limits_{i=1}^{N_x}
	\sum\limits_{j=1}^{N_y} \left|\mathbf{Q}_{i,j}^{N_t} -
	\mathbf{Q}(x_i, y_j, T)\right|^2\right)^{1/2}.
\end{equation*}
\begin{figure}[H]
	\begin{center}
		\includegraphics[width=0.24\textwidth]{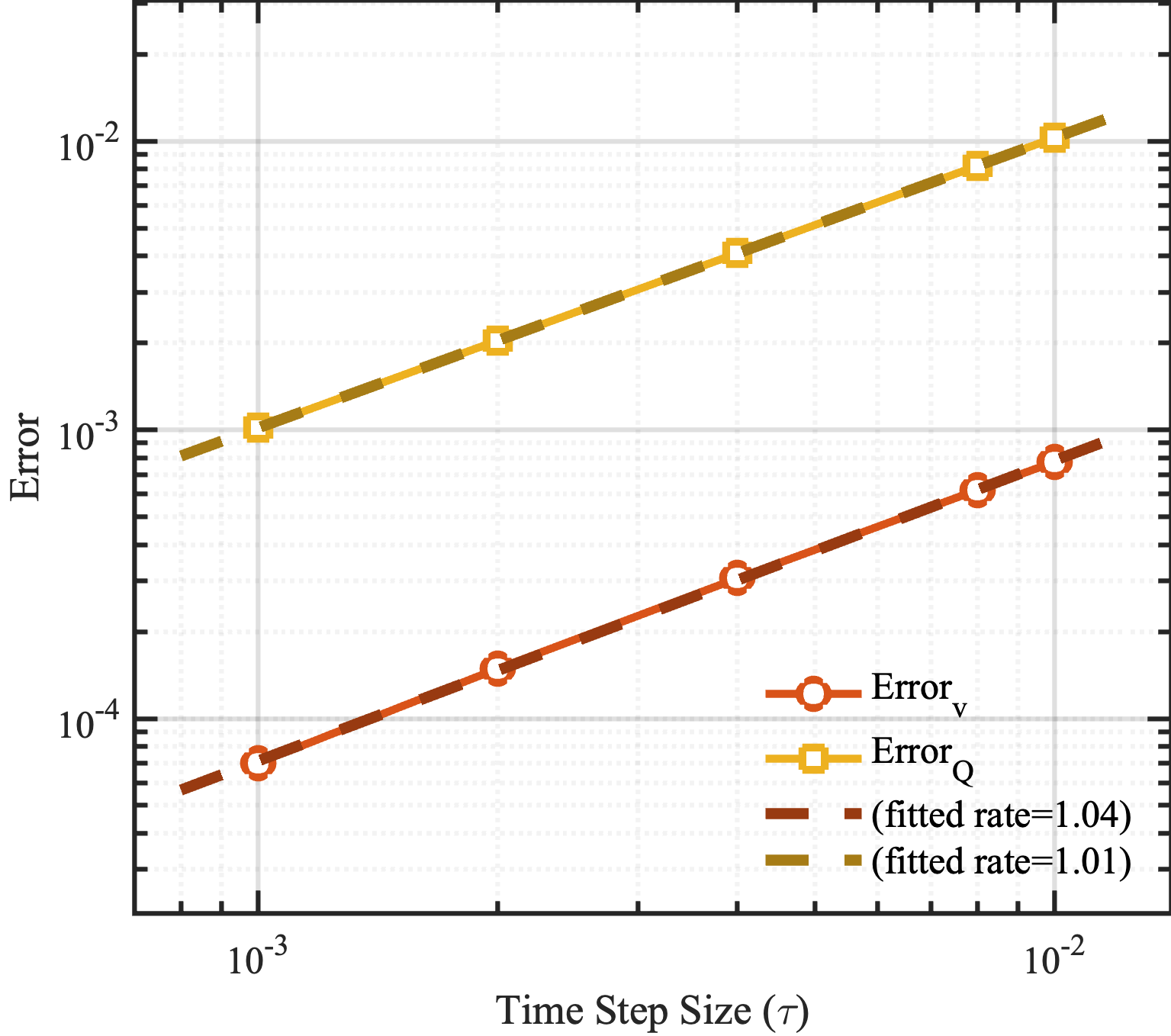}
		\includegraphics[width=0.24\textwidth]{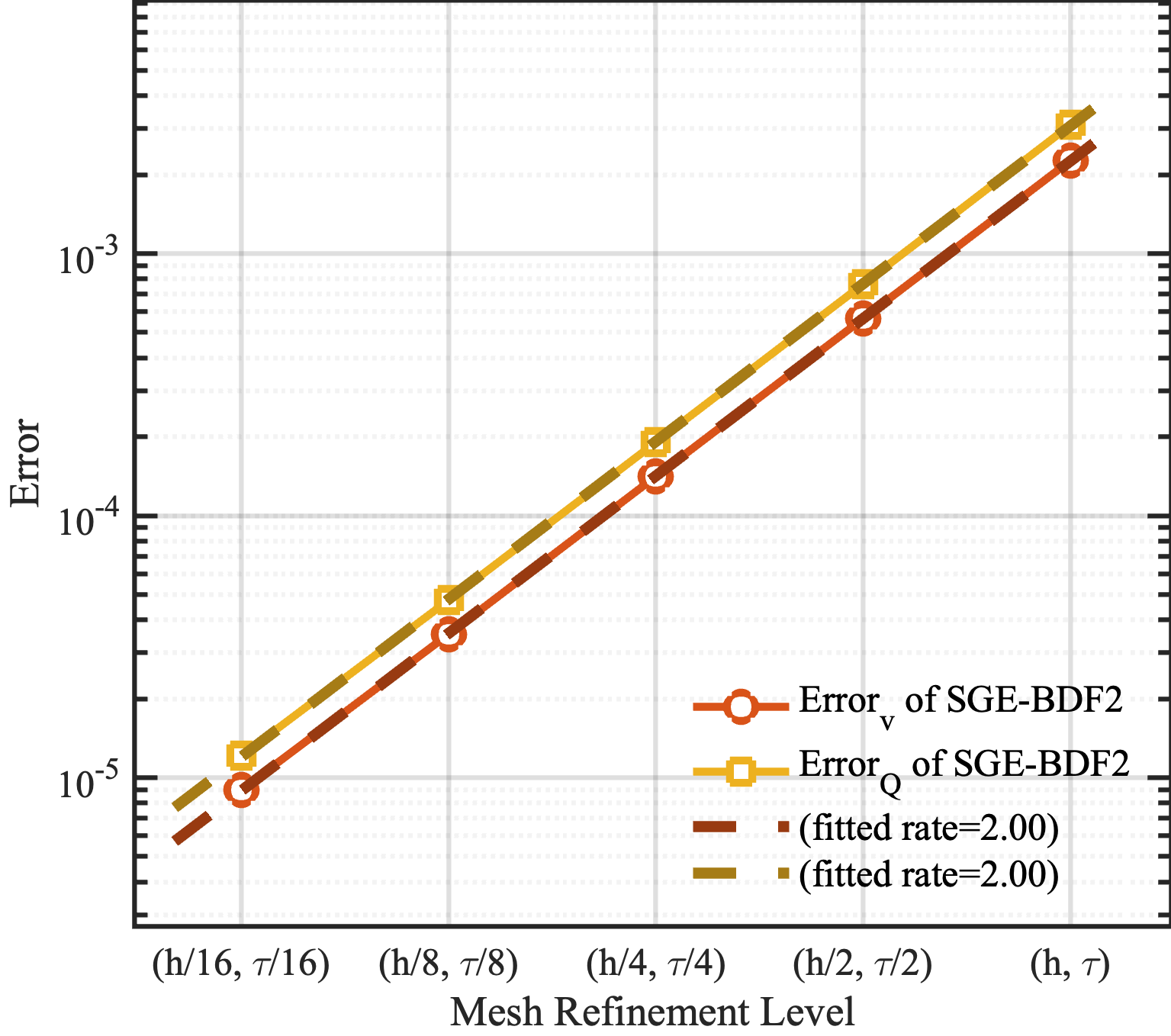}
		\includegraphics[width=0.24\textwidth]{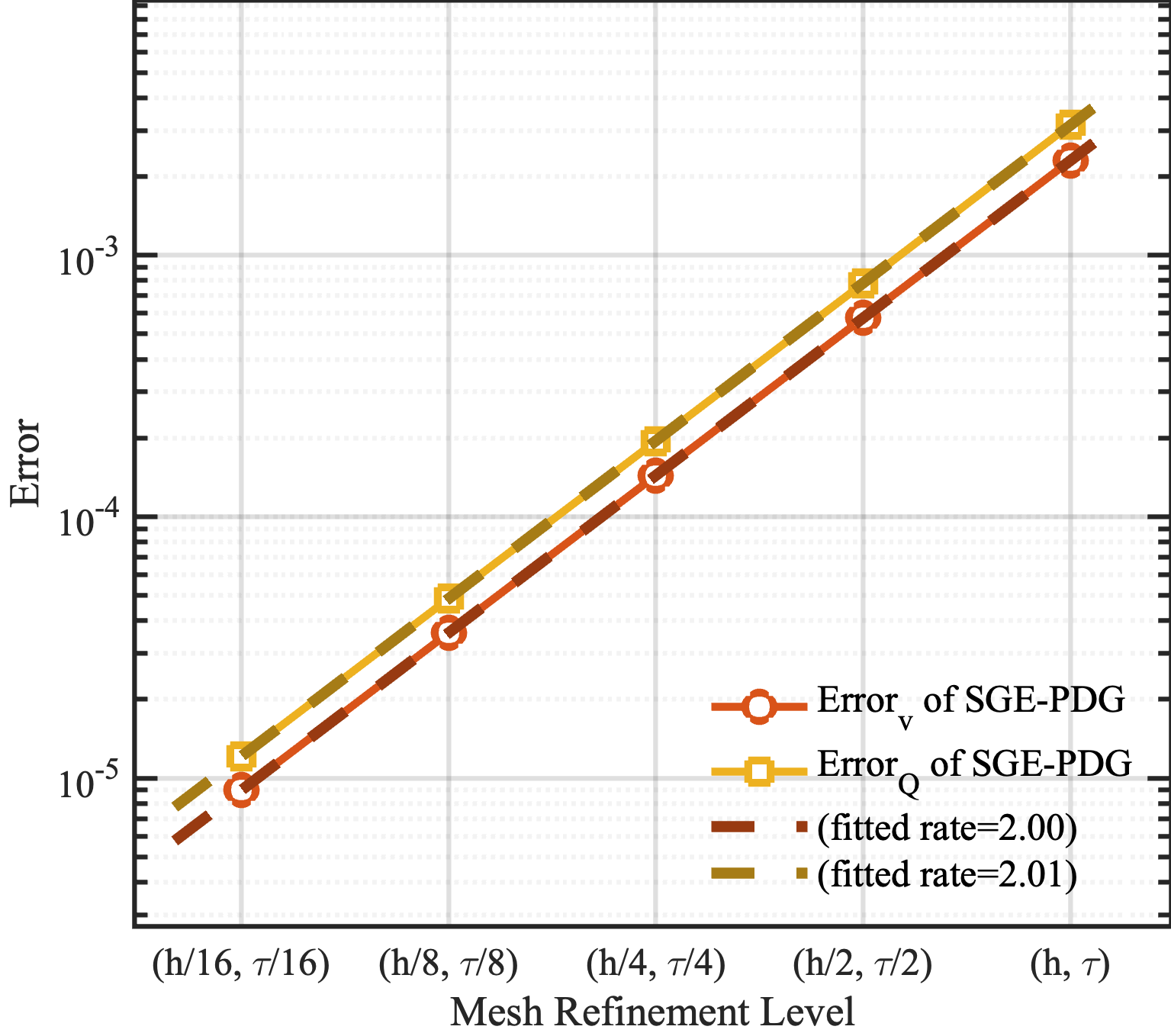}
	\end{center}
	\caption{Mesh refinement test of the proposed schemes. The error
		is calculated at $T = 1$. Linear fitting is used to compute the
		convergence rate.}\label{fig:convergence}
\end{figure}
Figure~\ref{fig:convergence} confirms first-order temporal
convergence of the SGE-BDF1 scheme and second-order convergence in
both spatial and temporal variables for the SGE-BDF2 and SGE-PDG schemes.

\subsection{Interaction of a pair of  $\pm\tfrac{1}{2}$ defects}
In this example, we investigate the interaction of a pair of  $\pm
	\tfrac{1}{2}$ defects, and examine how their behavior is influenced
by the presence of solid obstacles, the variation in activity
parameters, and the effect of the boundary anchoring. The
computational domain is defined as $\Omega = (0, 1)^2$.
The initial condition is given by
\begin{equation*}
	\mathbf{Q}_0 = S_{eq}\left(\tfrac{\bf n_0 n_0^\top}{\|\bf
		n_0\|^2} - \tfrac{\bf I}{3}\right), \quad \mathbf{n}_0 = ({\rm
			\cos \Theta}, {\rm \sin \Theta}, 0),
\end{equation*}
where
\begin{equation*}
	\Theta =
	\left\lbrace
	\begin{aligned}
		 & -0.5 {\rm atan2}(y - 0.5, x - 0.3), \ x < 0.5,                \\
		 & 0.5 {\rm atan2}(y - 0.5, x - 0.7) + 3\pi, \ \text{otherwise}. \\
	\end{aligned}
	\right.
\end{equation*}
The initial velocity is set to zero. Homogeneous Dirichlet boundary
conditions are imposed on the velocity field, while homogeneous
Neumann boundary conditions are applied to the $\mathbf{Q}$-tensor
field at the boundary of the computational domain.

In this simulation, the obstacle is defined as an interior
rectangle: $D = (0.48, 0.52) \times (0.2, 0.8)$. The model
parameters are given as follows
\begin{equation*}
	\begin{aligned}
		 & b_{solid} = 10^5, \ \eta_{fluid} = 1, \ \eta_{solid} = 10^3,
		\ K_{fluid} = 5\times 10^{-4},                                  \\
		 & N_{fluid} = 7.5, \ N_{solid} = 0.1, \ \Gamma_{fluid} = 10^2,
		\ \Gamma_{solid} = 10^3, \ A = \tfrac{2}{15}.
	\end{aligned}
\end{equation*}

\subsubsection{Results with the normal anchoring condition}
We first compare the case without the rectangular obstacle with the
case with the normal boundary anchoring at the obstacle.
\begin{figure}[H]
	\begin{center}
		\includegraphics[width=0.18\textwidth]{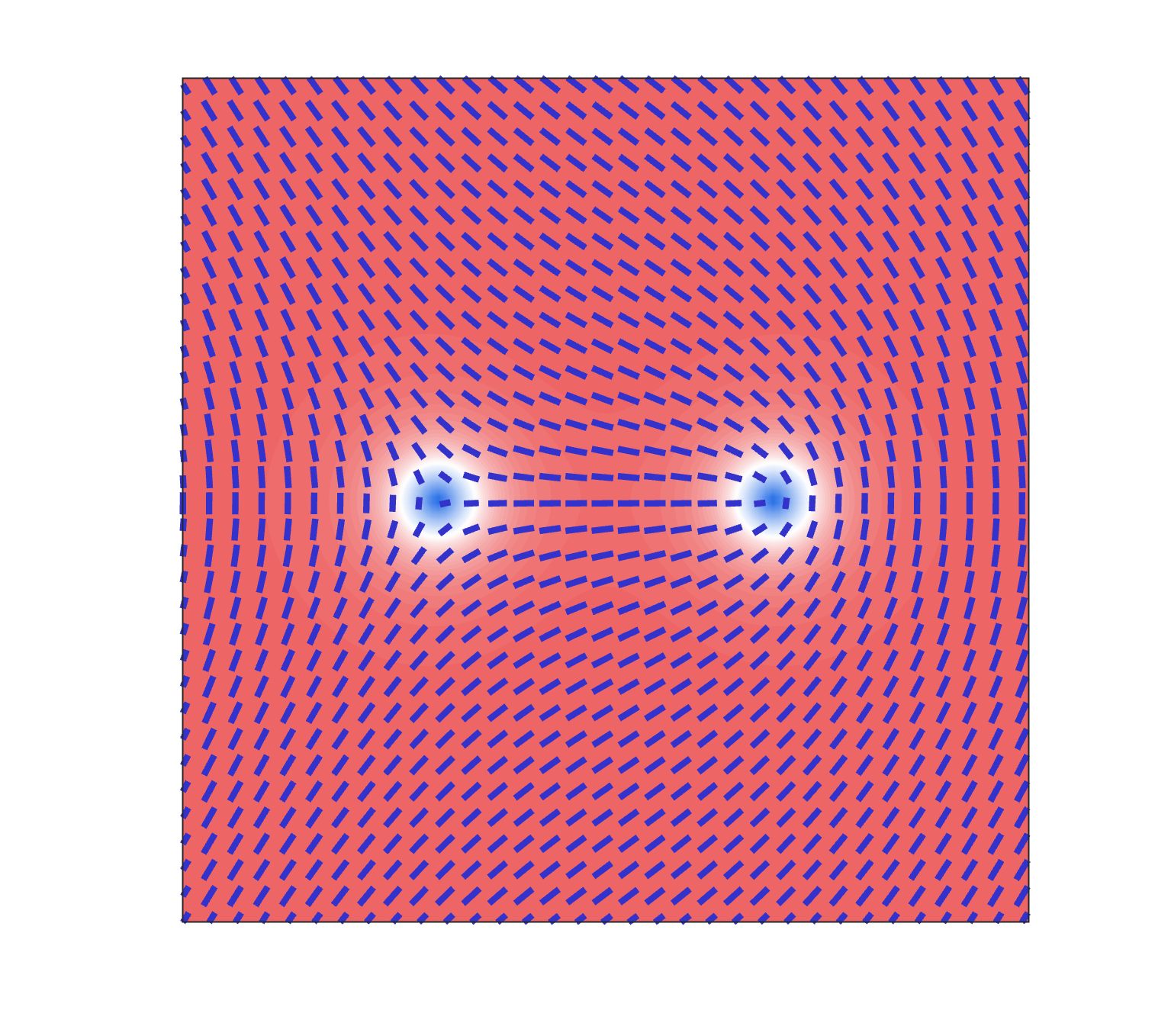}
		\includegraphics[width=0.18\textwidth]{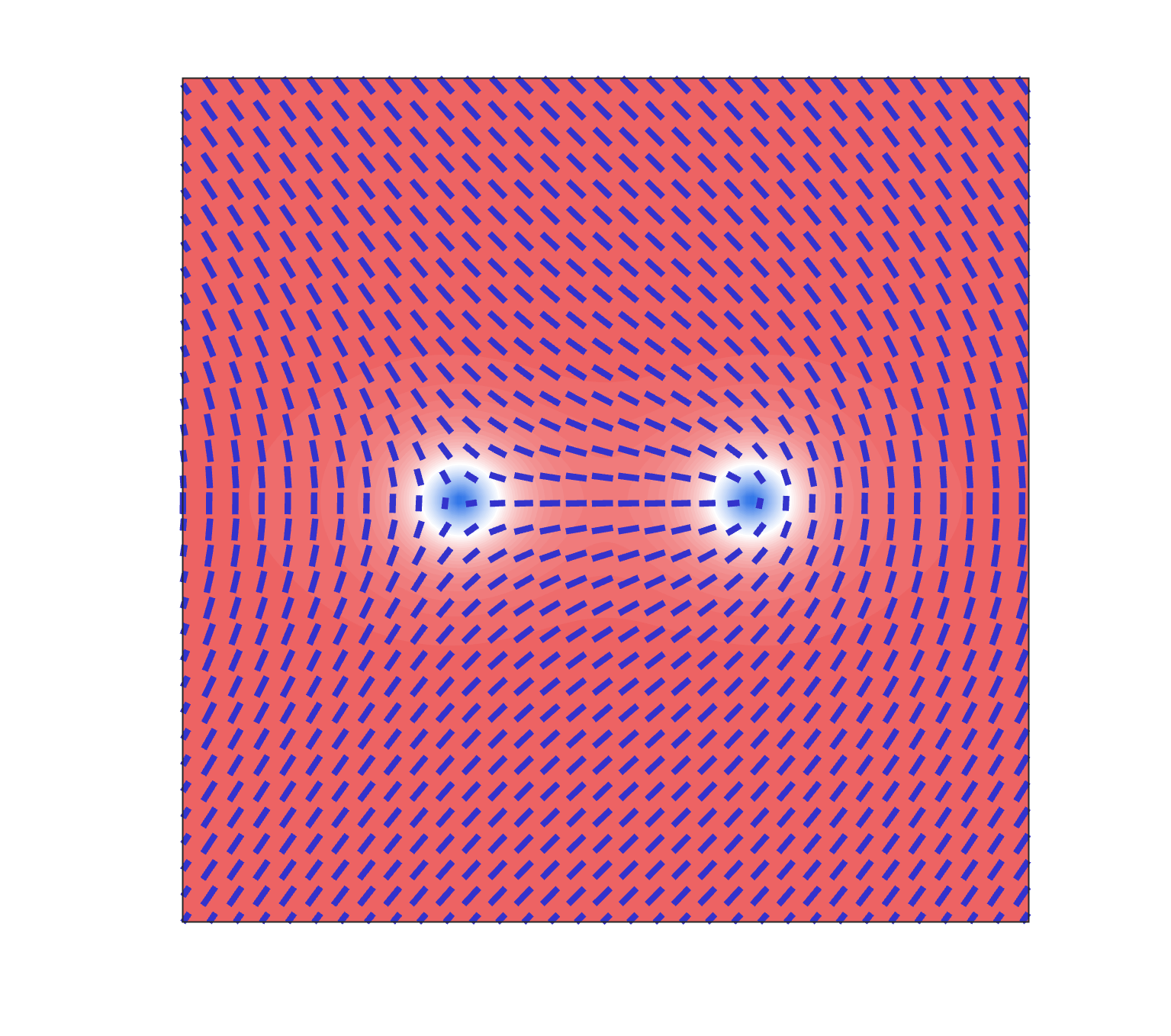}
		\includegraphics[width=0.18\textwidth]{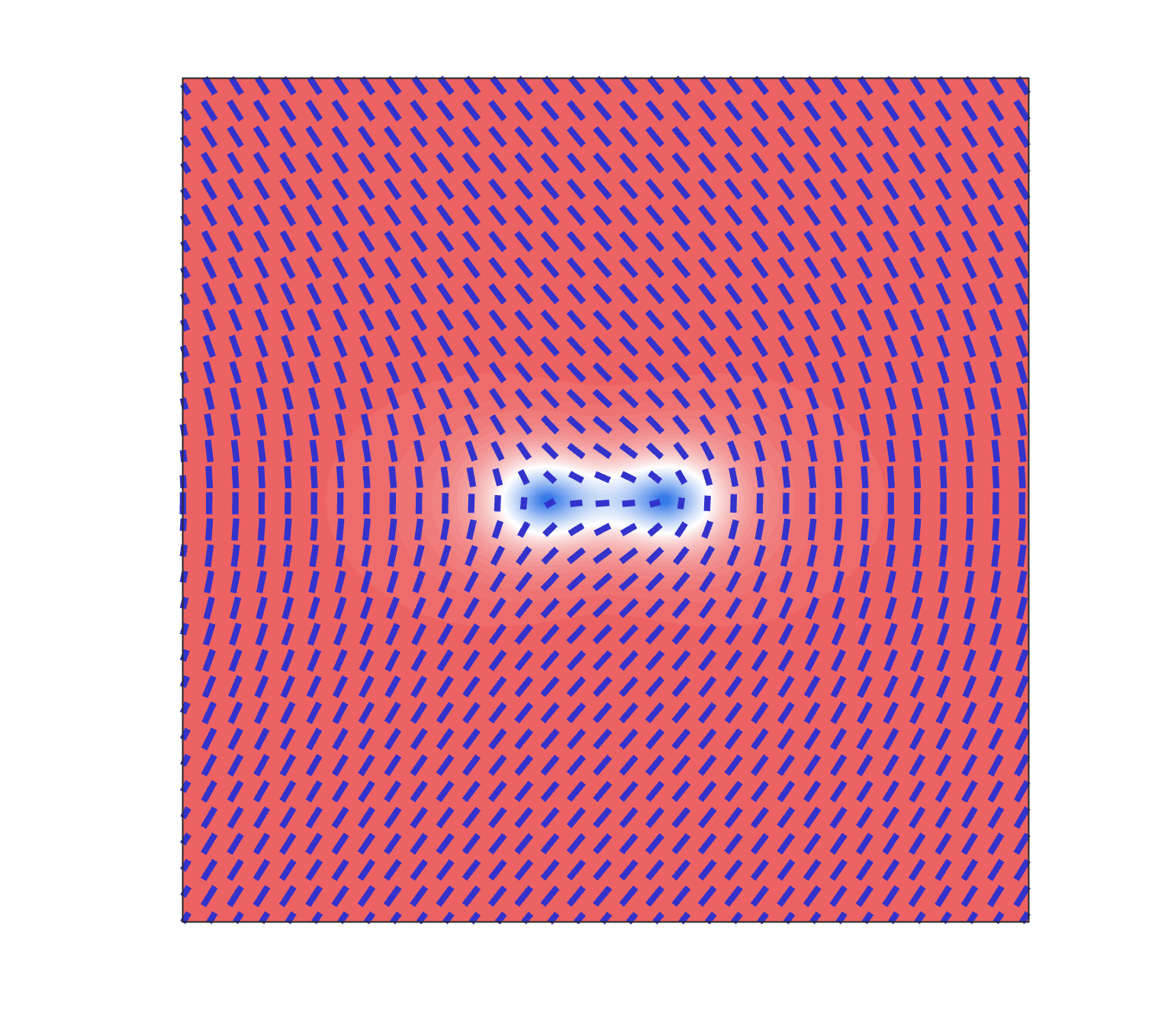}
		\includegraphics[width=0.18\textwidth]{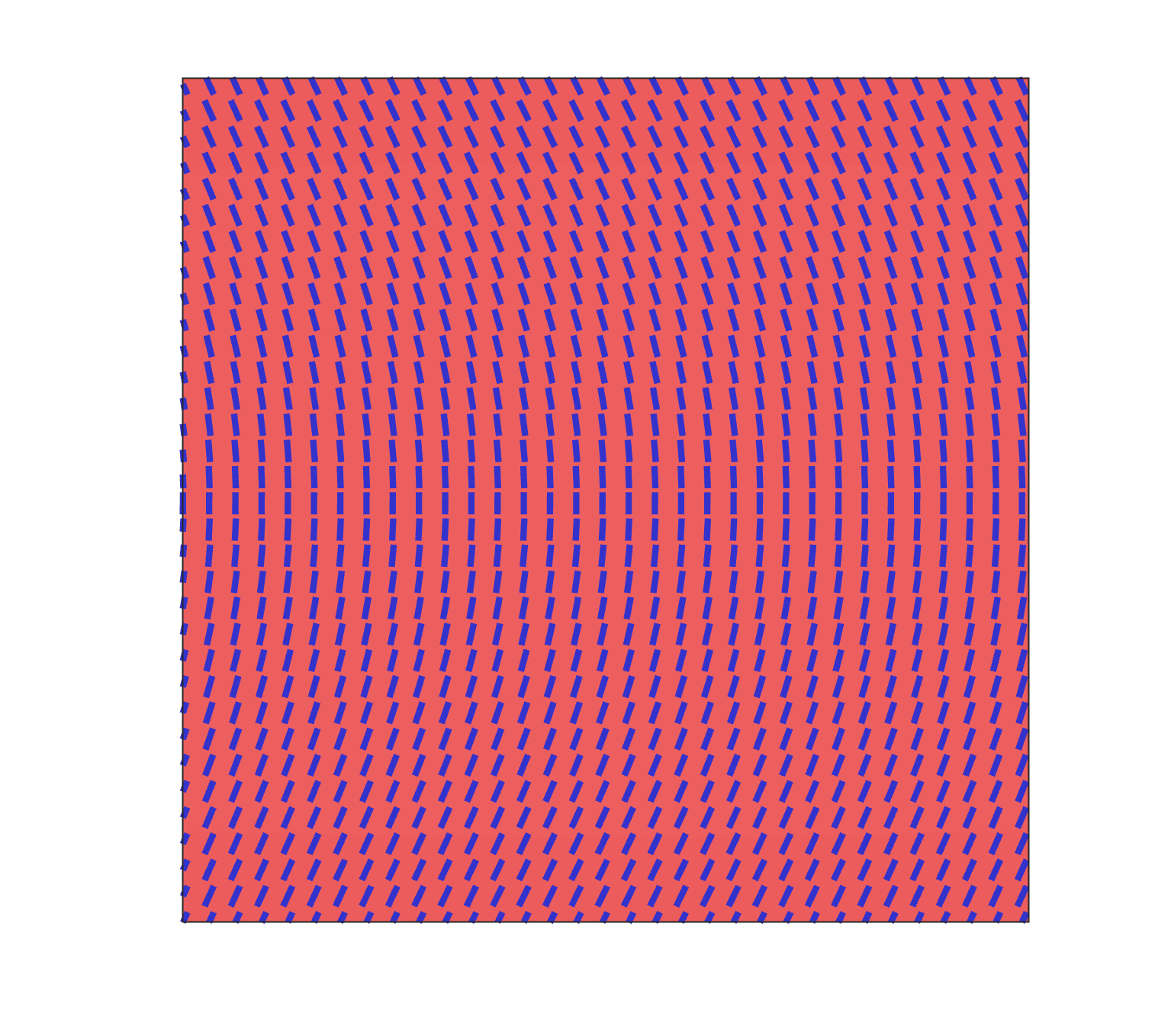}
		\includegraphics[width=0.18\textwidth]{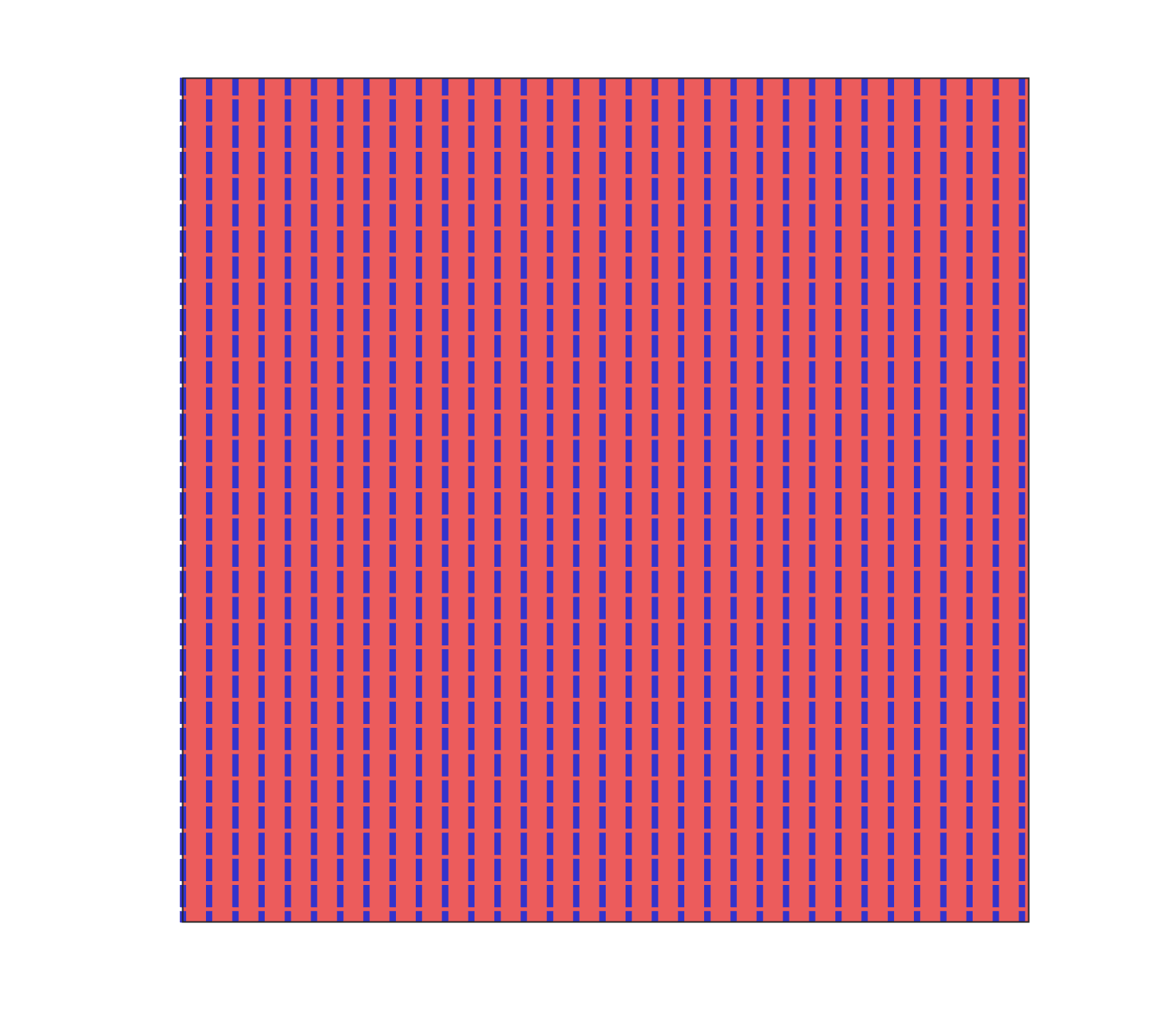}
	\end{center}
	\caption{Interaction of a pair of passive $\pm\tfrac{1}{2}$
		defects without obstacles at times $t = 0.1$, $1$, $2$, $3$,
		$20$, respectively. The pair annihilate each other after a while.
	}\label{fig:nobstacle}
\end{figure}

Figure~\ref{fig:nobstacle} shows the contour plots of the principal
eigenvalue difference and the corresponding director field for a
pair of passive $\pm \tfrac{1}{2}$ defects at different time
instances in the domain without any solid obstacles. Over time, the
two defects gradually move toward each other due to their opposite
topological charges. Eventually, they merge, leading to the
complete annihilation of the defect pair. In the steady state, no
defects remain in the system, indicating that the system has
relaxed to a defect-free configuration. This result depicts the
intrinsic energy-minimizing process of passive liquid crystals, in
which oppositely charged defects attract and annihilate each other,
ultimately yielding a uniform, topologically trivial state.
\begin{figure}[H]
	\begin{center}
		\includegraphics[width=0.3\textwidth]{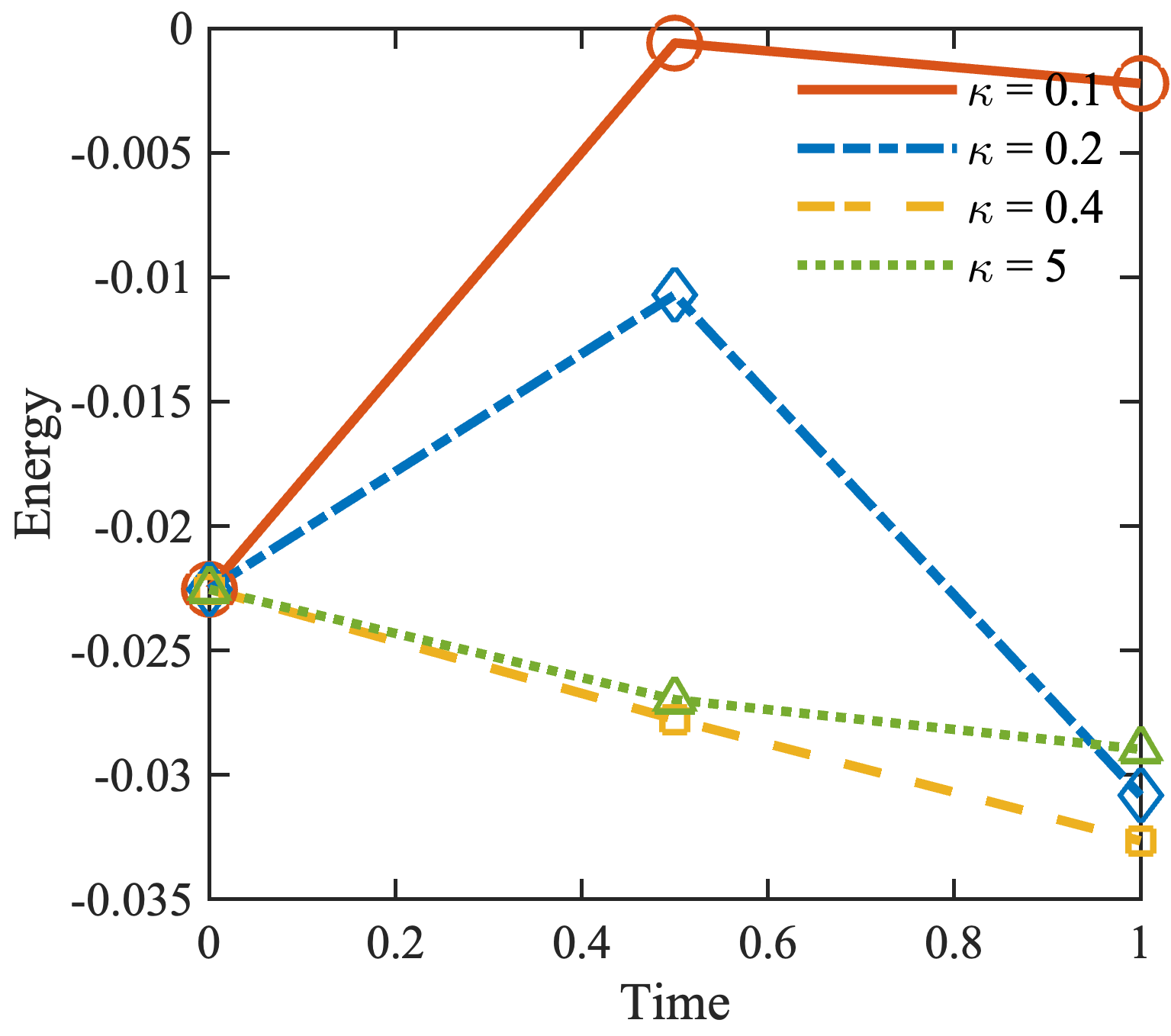}
		\includegraphics[width=0.3\textwidth]{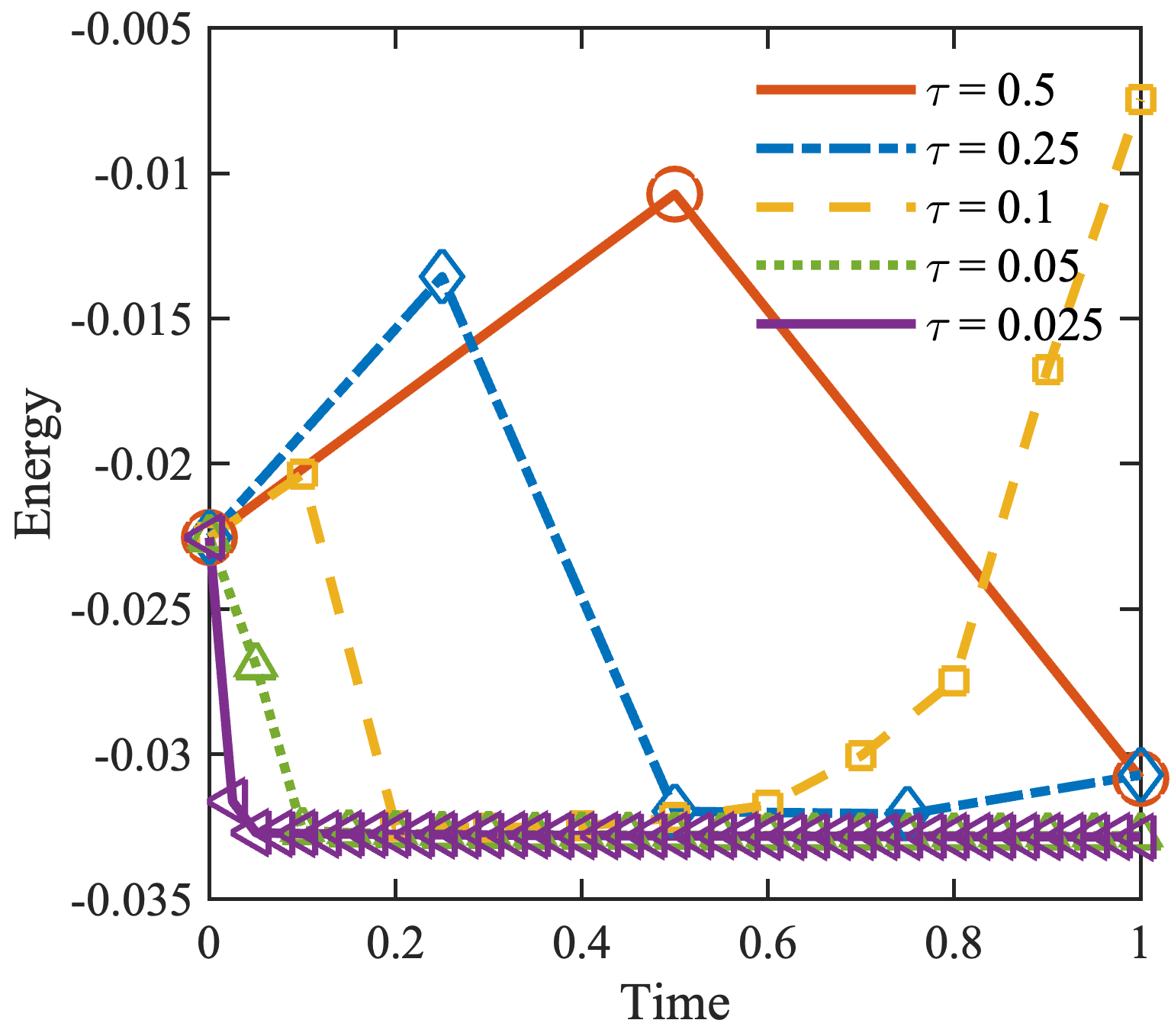}
		\includegraphics[width=0.3\textwidth]{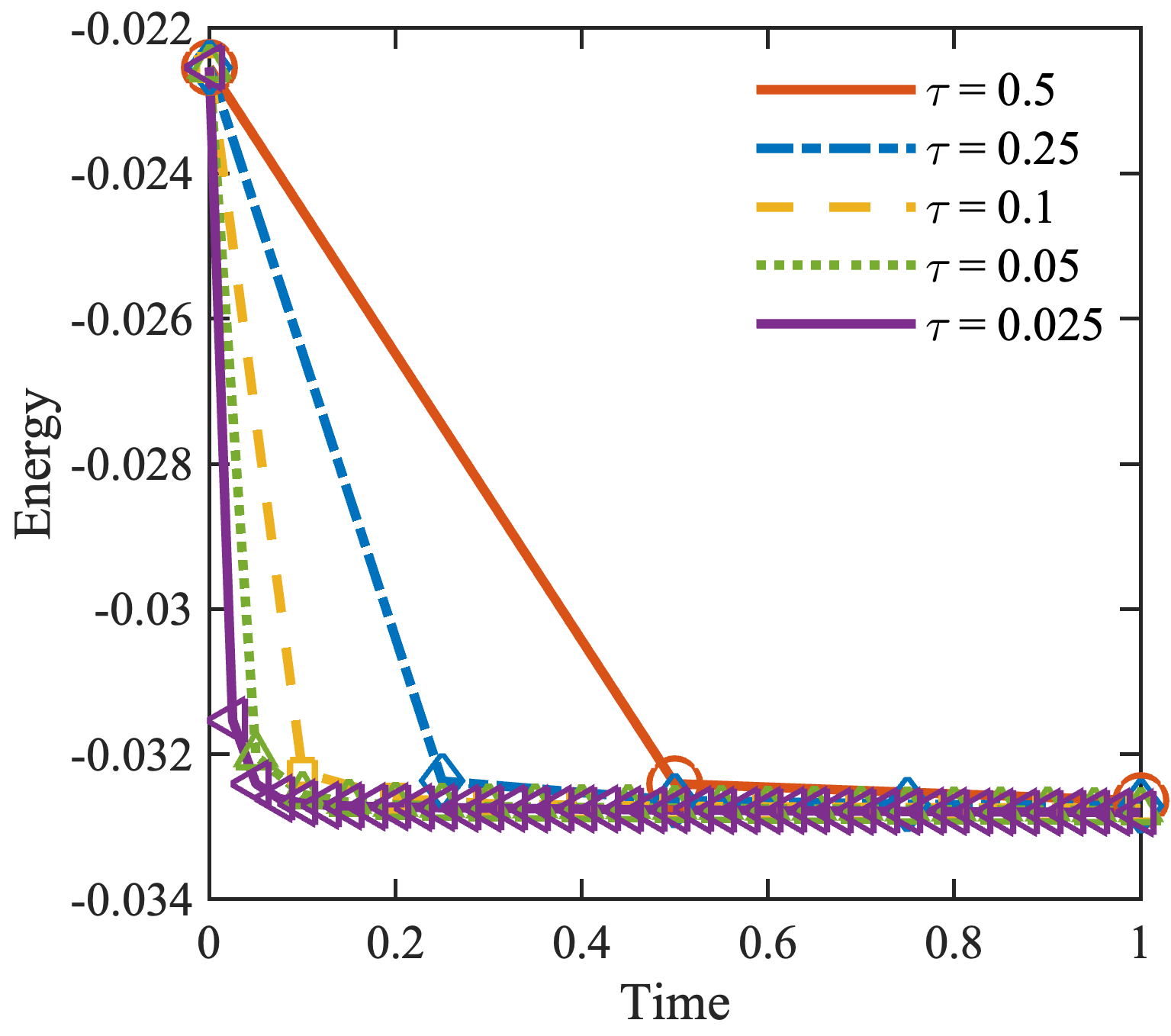}
	\end{center}
	\caption{Free-energy evolution of the SGE-BDF1 scheme under
		several choices of the
		stabilization parameter, $\kappa$, and the time step size, $\tau$.
		Left: fixed $\tau=0.5$ with varying $\kappa$.
		Middle: fixed $\kappa=0.2$ with varying $\tau$.
		Right: fixed $\kappa=1$ with varying $\tau$.
		The results show that either a large stabilization
		parameter or a small time step, $\tau$, enforces
		monotone energy decay.}\label{fig:energy_kappa_dt}
\end{figure}

{\color{magenta}
	Figure~\ref{fig:energy_kappa_dt} illustrates the interplay between the
	stabilization parameter, $\kappa$, and the time step size, $\tau$, in the
	SGE-BDF1 scheme. With $\tau$ fixed (left), a large
	$\kappa$ guarantees monotone energy decay. For smaller $\kappa$,
	this monotonicity
	may be lost.

	However, as shown in the middle panel, reducing $\tau$ can recover the
	expected energy dissipation even when $\kappa$ is below the theoretical
	bound. In the right panel, for a relatively large $\kappa$, the scheme
	remains stable for all tested time steps.

	These results indicate that the theoretical condition on $\kappa$ is
	sufficient but not necessary in practice, and that a suitable balance
	between $\kappa$ and $\tau$ can ensure the energy stable behavior.
}

\begin{figure}[H]
	\begin{center}
		\includegraphics[width=0.18\textwidth]{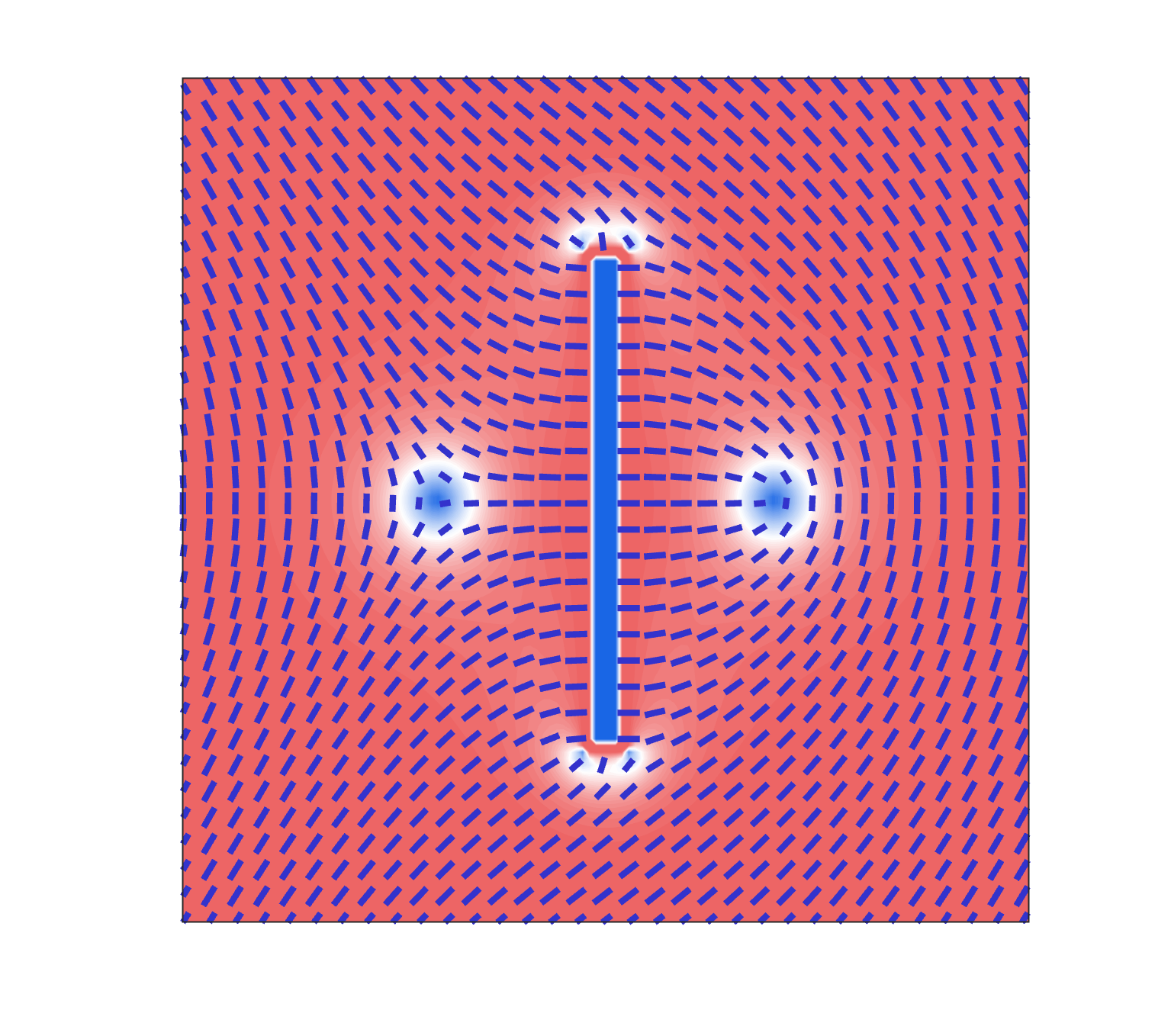}
		\includegraphics[width=0.18\textwidth]{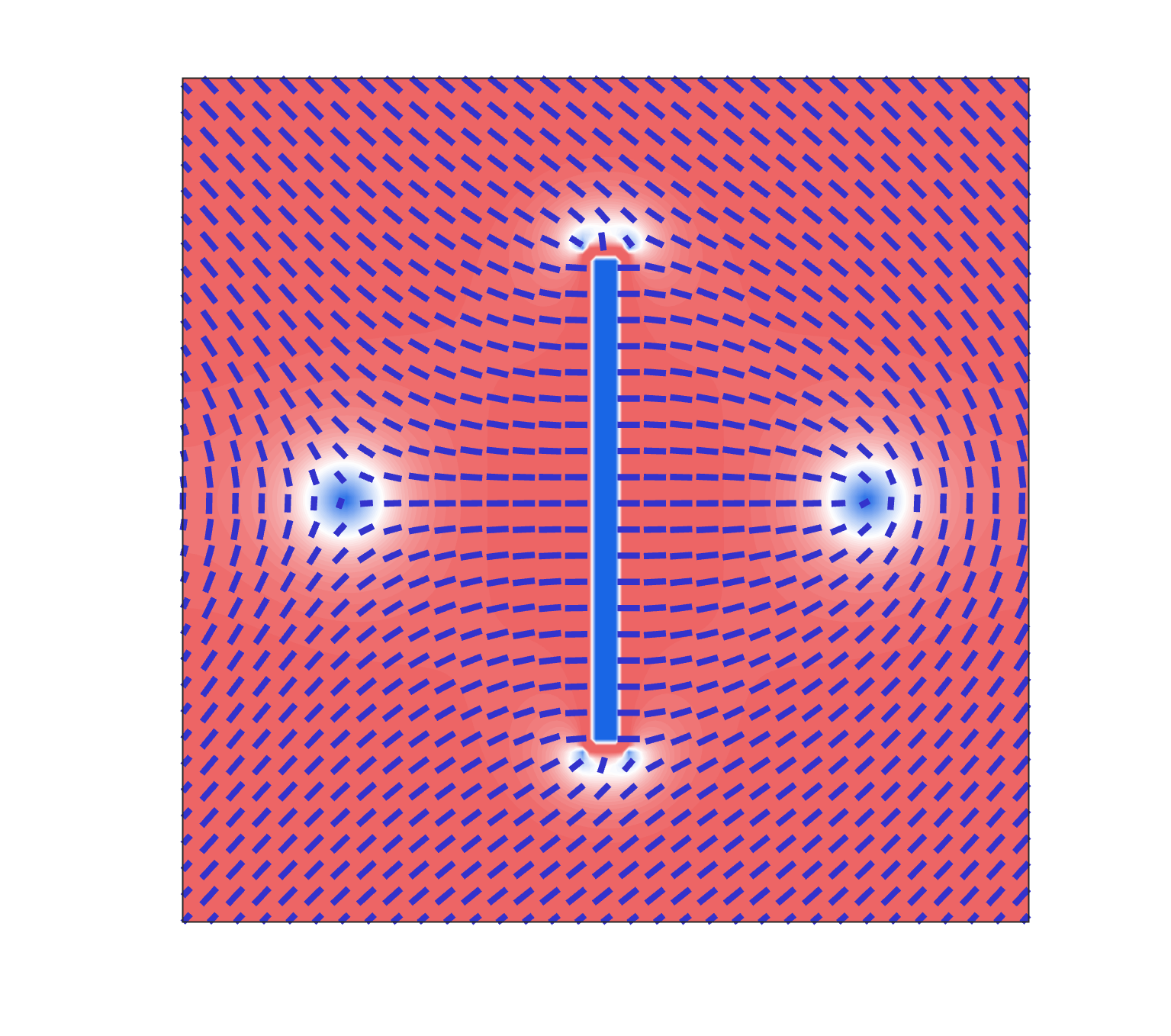}
		\includegraphics[width=0.18\textwidth]{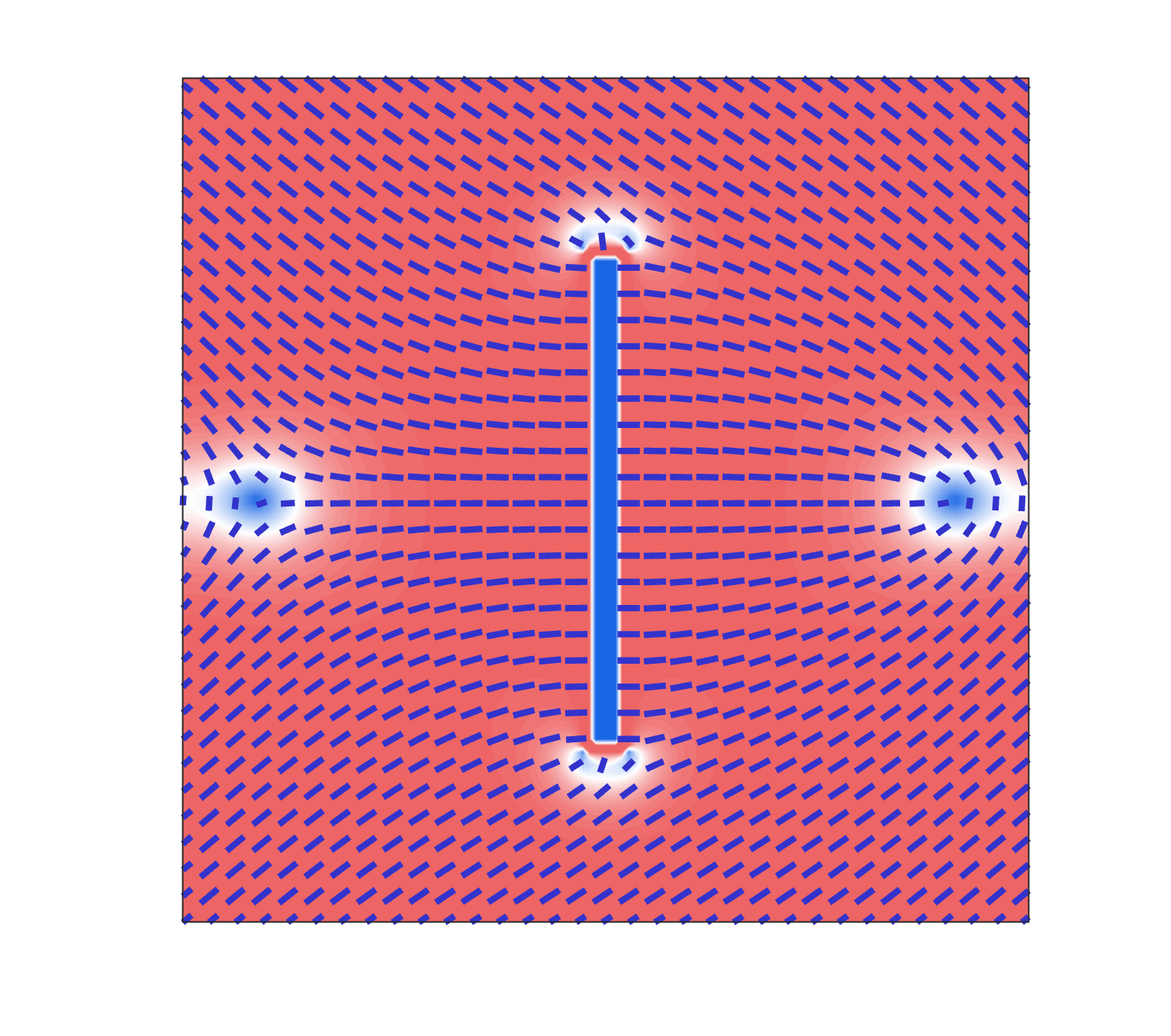}
		\includegraphics[width=0.18\textwidth]{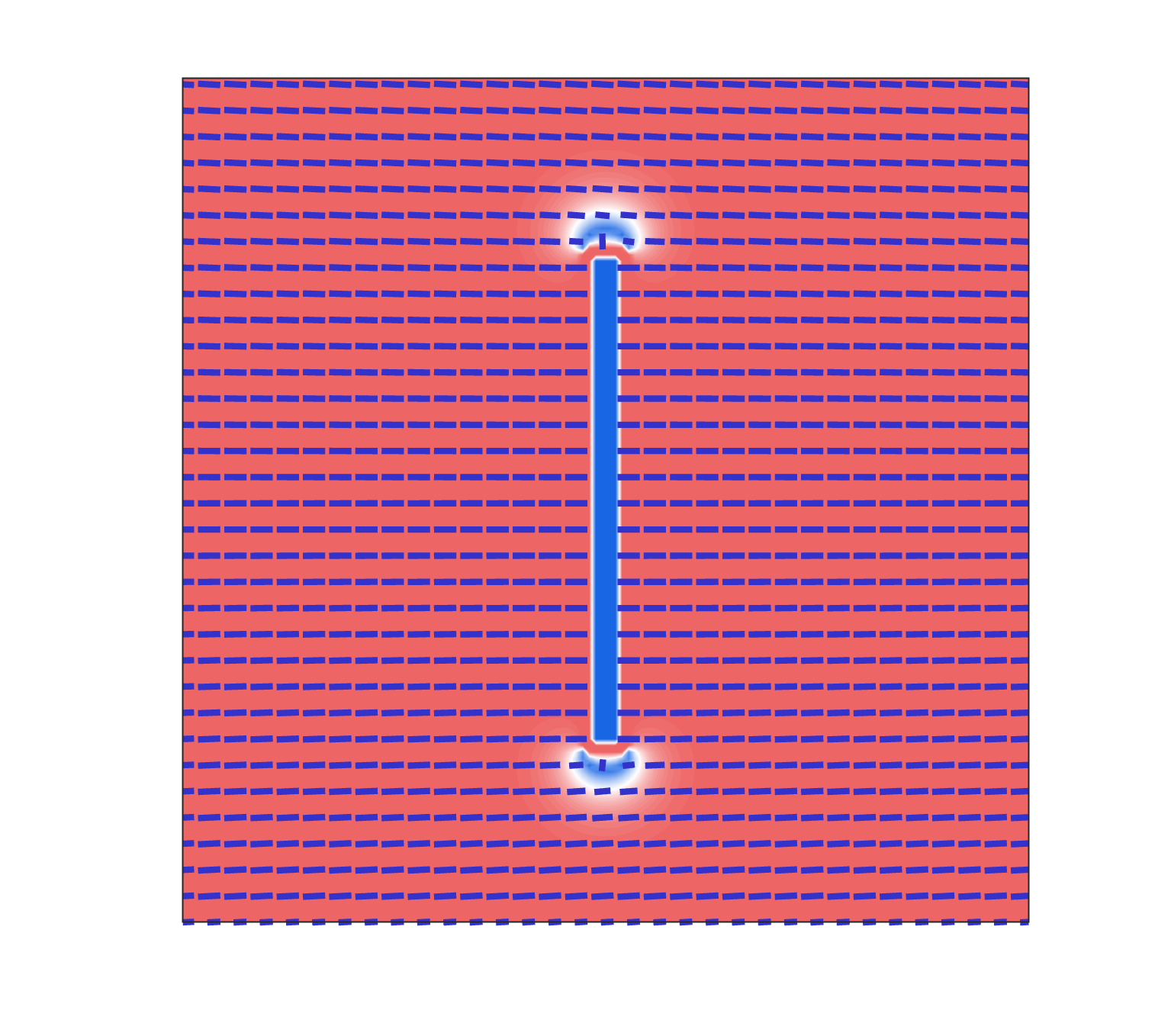}
		\includegraphics[width=0.18\textwidth]{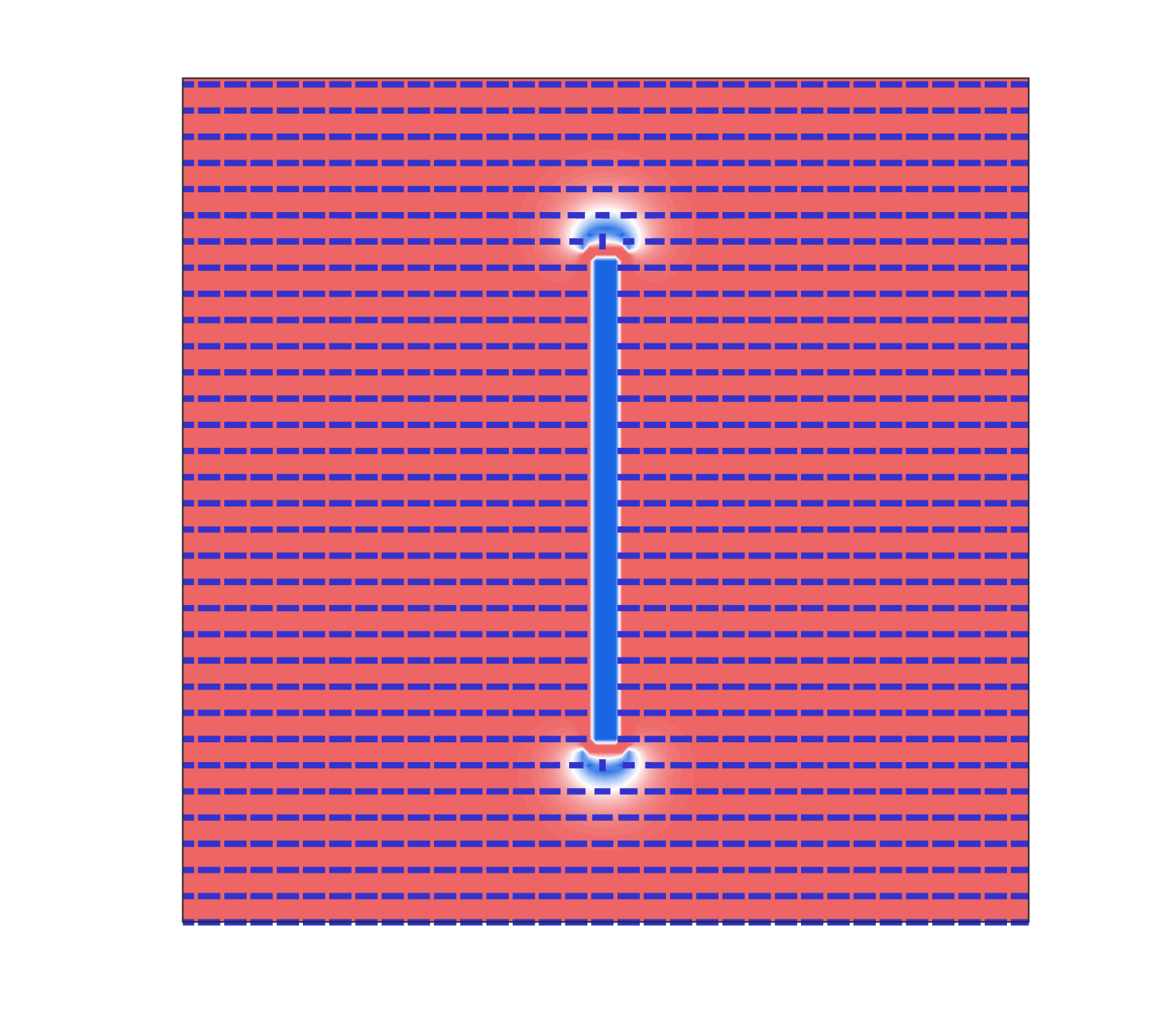}
	\end{center}
	\caption{Interaction of a pair of passive $\pm\tfrac{1}{2}$
		defects with a rectangular obstacle under normal anchoring at
		times $t = 0.1$, $2$, $3$, $7.8$, $20$, respectively. The pair
		moves out of the computational domain after a while. The
		$\bQ-$tensor is isotropic in the obstacle.}\label{fig:chi0}
\end{figure}

Figure~\ref{fig:chi0} illustrates the dynamics of a pair of passive
$\pm \tfrac{1}{2}$ defects in a domain containing a rectangular
obstacle positioned between them. Following a rapid initial
relaxation, the orientation tensor field is zero within the solid
at the time instances shown above, demonstrating the effectiveness
of the proposed obstacle treatment.

In contrast to the case without obstacles, the two defects remain
separated and symmetrically positioned with respect to the
obstacle, which acts as a barrier due to its anchoring effect.
Around $t = 3$, both defects exit the domain through the
boundaries, resulting in a defect-free configuration. The final
panel at $t = 20$ confirms that the director field within the
obstacle remains isotropic and stationary over long time evolution,
further validating the effective of the numerical approach.

\begin{figure}[H]
	\begin{center}
		\includegraphics[width=0.18\textwidth]{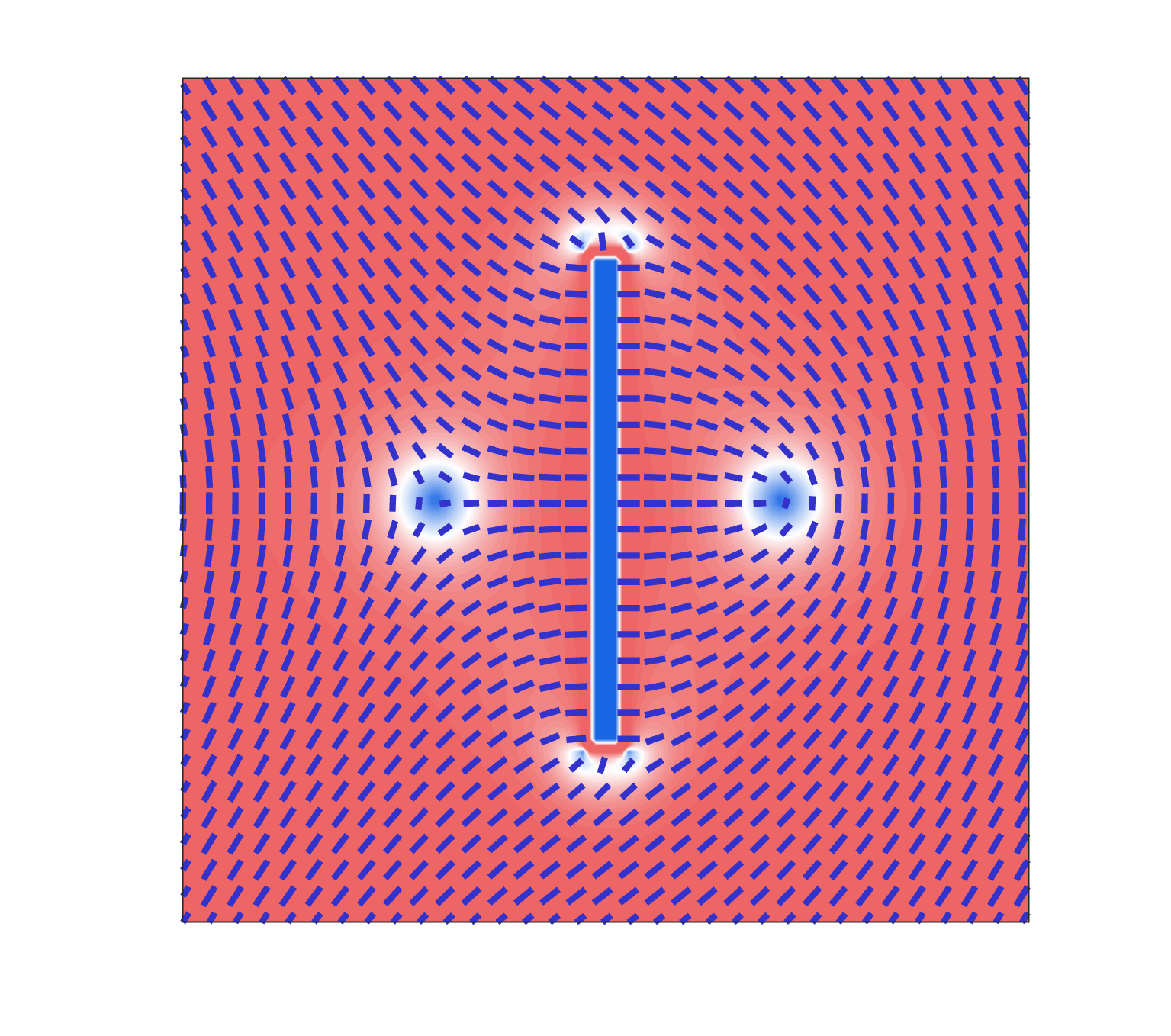}
		\includegraphics[width=0.18\textwidth]{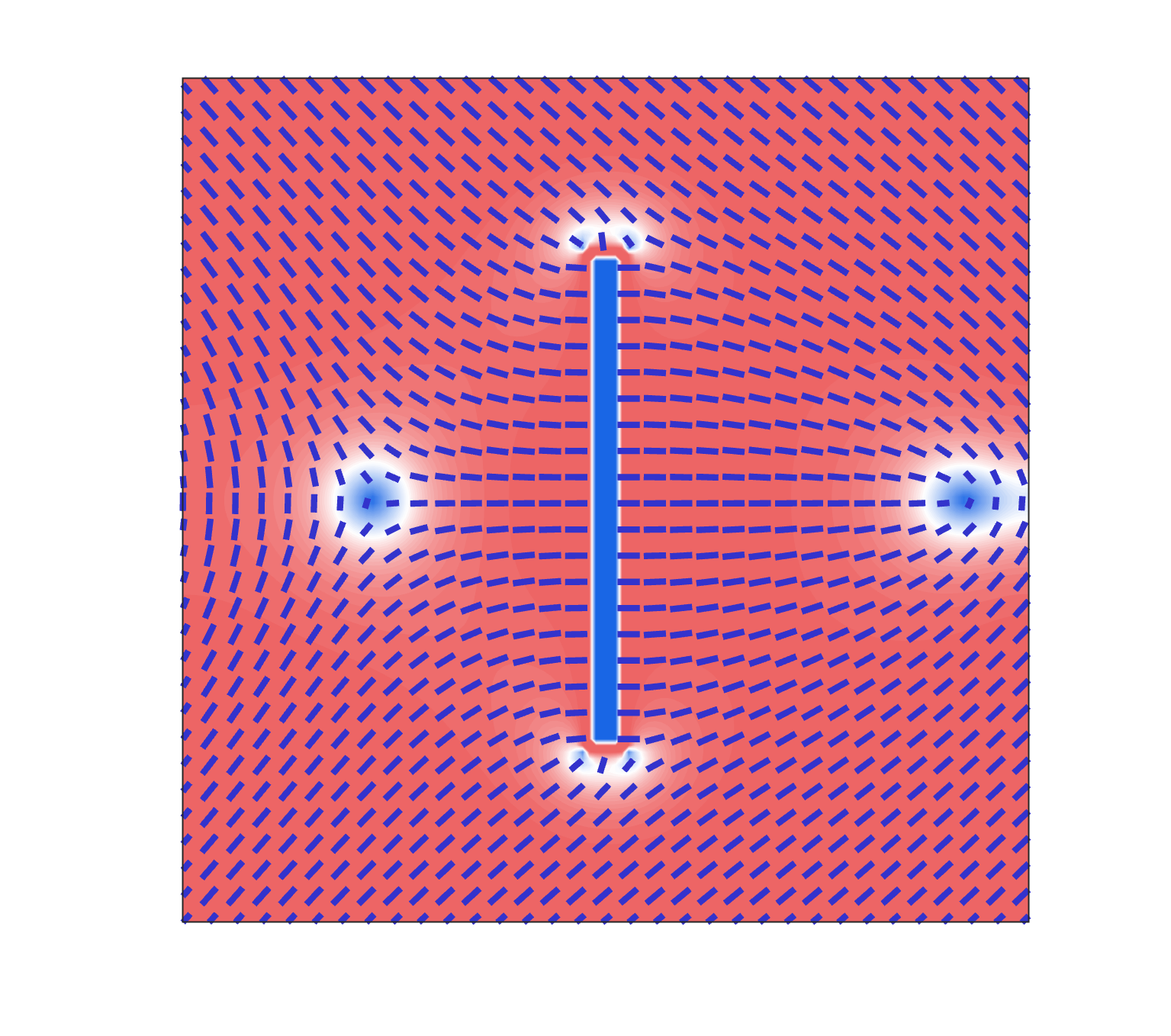}
		\includegraphics[width=0.18\textwidth]{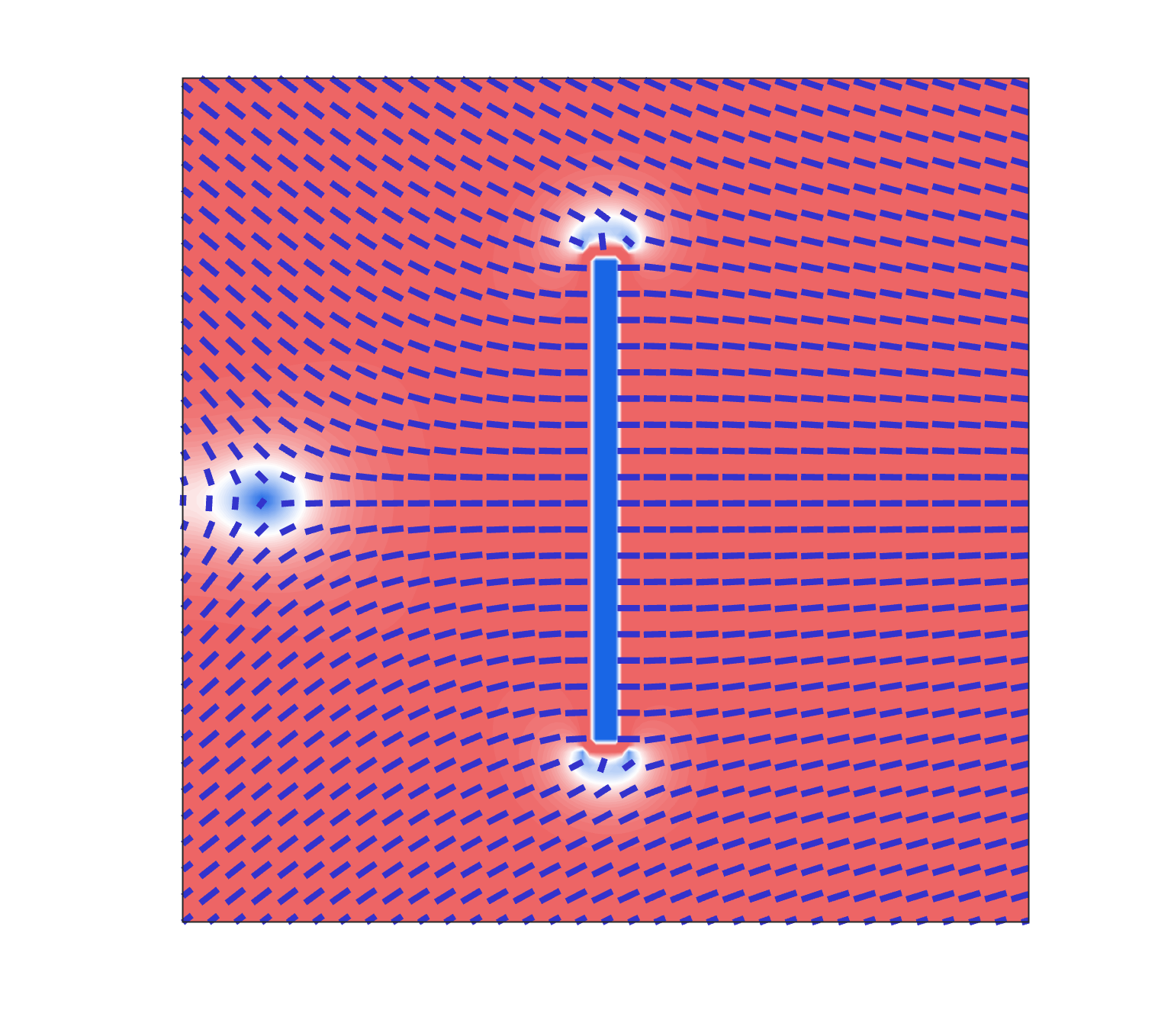}
		\includegraphics[width=0.18\textwidth]{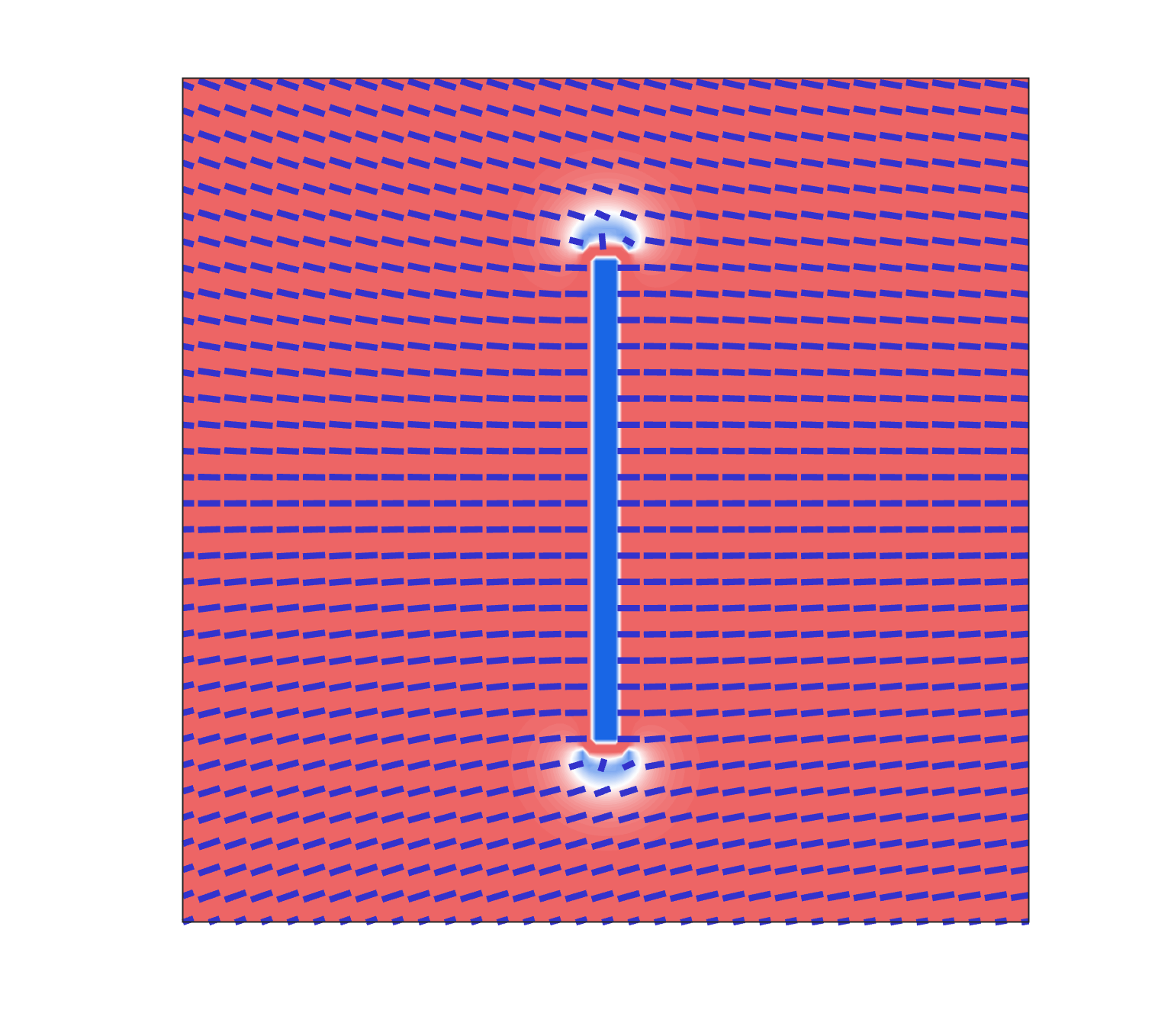}
		\includegraphics[width=0.18\textwidth]{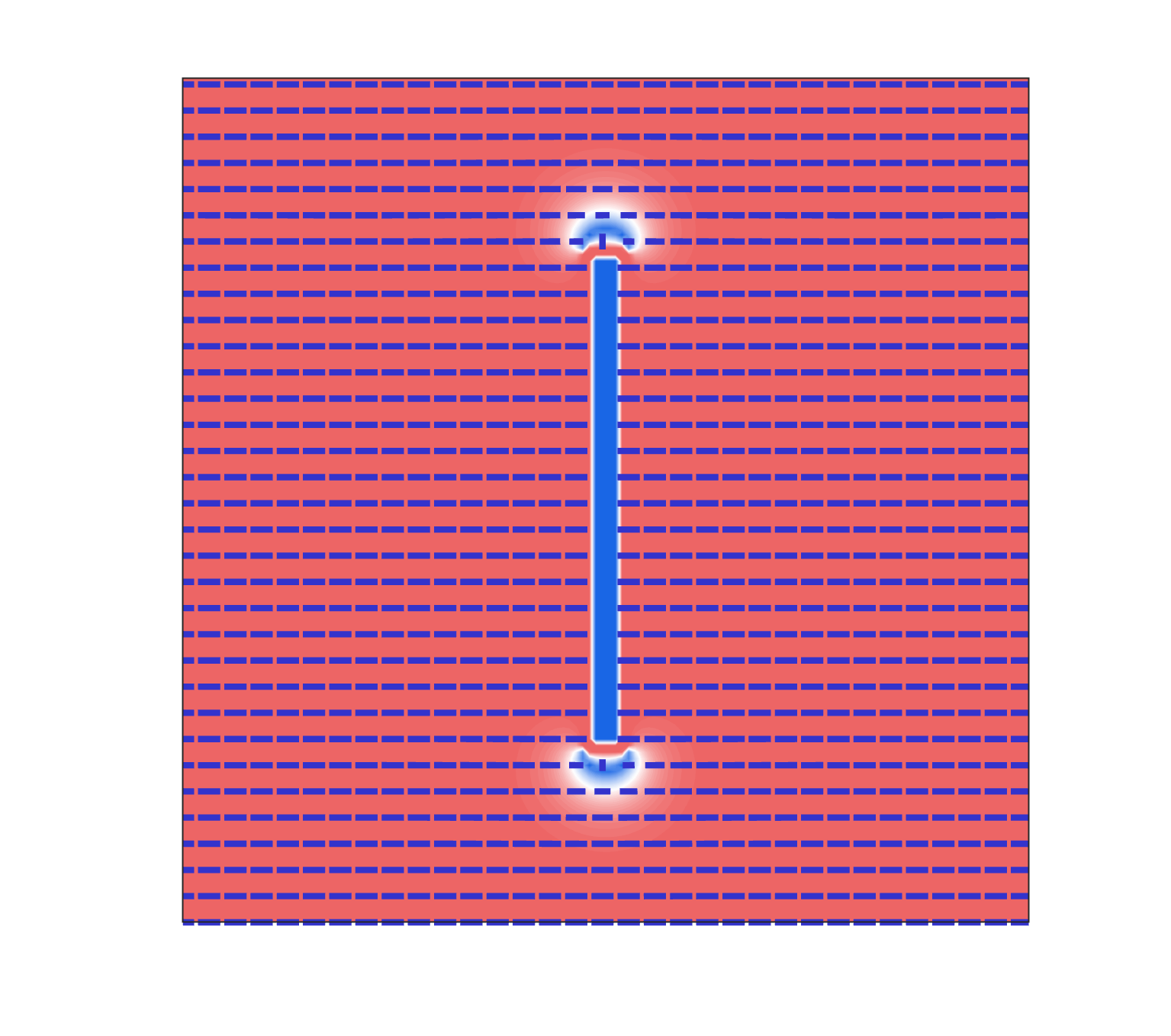}
	\end{center}
	\caption{Interaction of a pair of active $\pm \tfrac{1}{2}$
		defects with a rectangular obstacle under normal anchoring at
		times $t = 0.1$, $1.8$, $3.7$, $5$, $20$, respectively, with
		parameter values $\chi_{fluid} = 5$ and $\xi_{fluid} =
			-0.1$.}\label{fig:chipos}
\end{figure}

Figure~\ref{fig:chipos} shows the dynamics of a pair of active $\pm
	\tfrac{1}{2}$ defects in a domain containing a solid obstacle, with
activity parameters \eqref{eq:model-parameters} set at
$\chi_{fluid} = 5$ and $\xi_{fluid} = -0.1$. Similar to the passive
case, the defects are repelled from each other due to anchoring
effects imposed by both the obstacle and the domain boundary.
However, a key difference arises from the presence of activity: the
active $+\tfrac{1}{2}$ defect moves more rapidly along its head
direction compared to its passive counterpart, and therefore exits
the domain earlier.

For the $-\tfrac{1}{2}$ defect, the net effect of the active force
is negligible, resulting in a slower departure. It gradually drifts
away from the obstacle and eventually leaves the domain as well.
The final panel at $t = 20$ confirms that the director field within
the obstacle remains isotropic and stationary over long time
integration, consistent with earlier observations and validating
the robustness of the numerical scheme.

\begin{figure}[H]
	\begin{center}
		\includegraphics[width=0.18\textwidth]{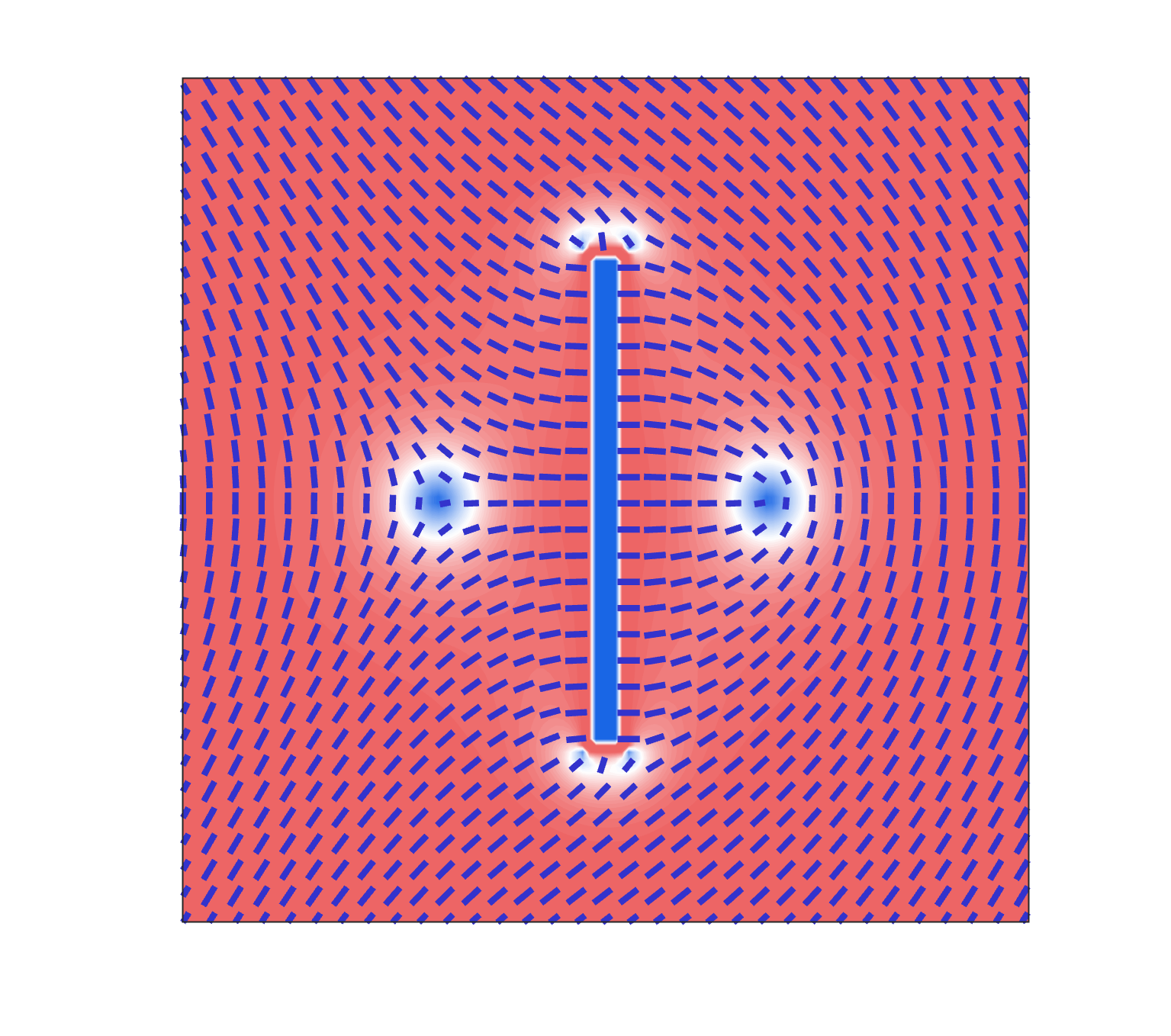}
		\includegraphics[width=0.18\textwidth]{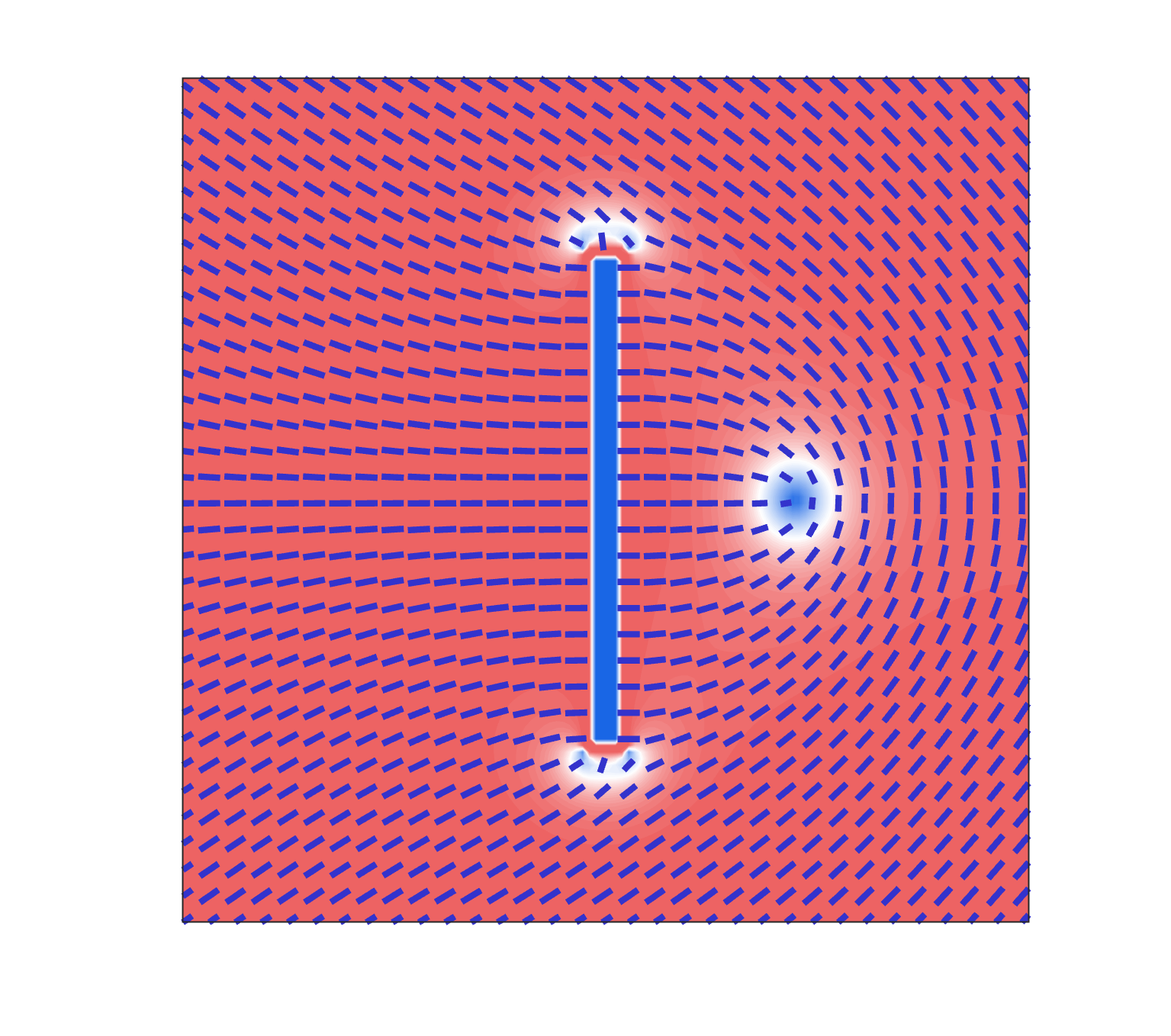}
		\includegraphics[width=0.18\textwidth]{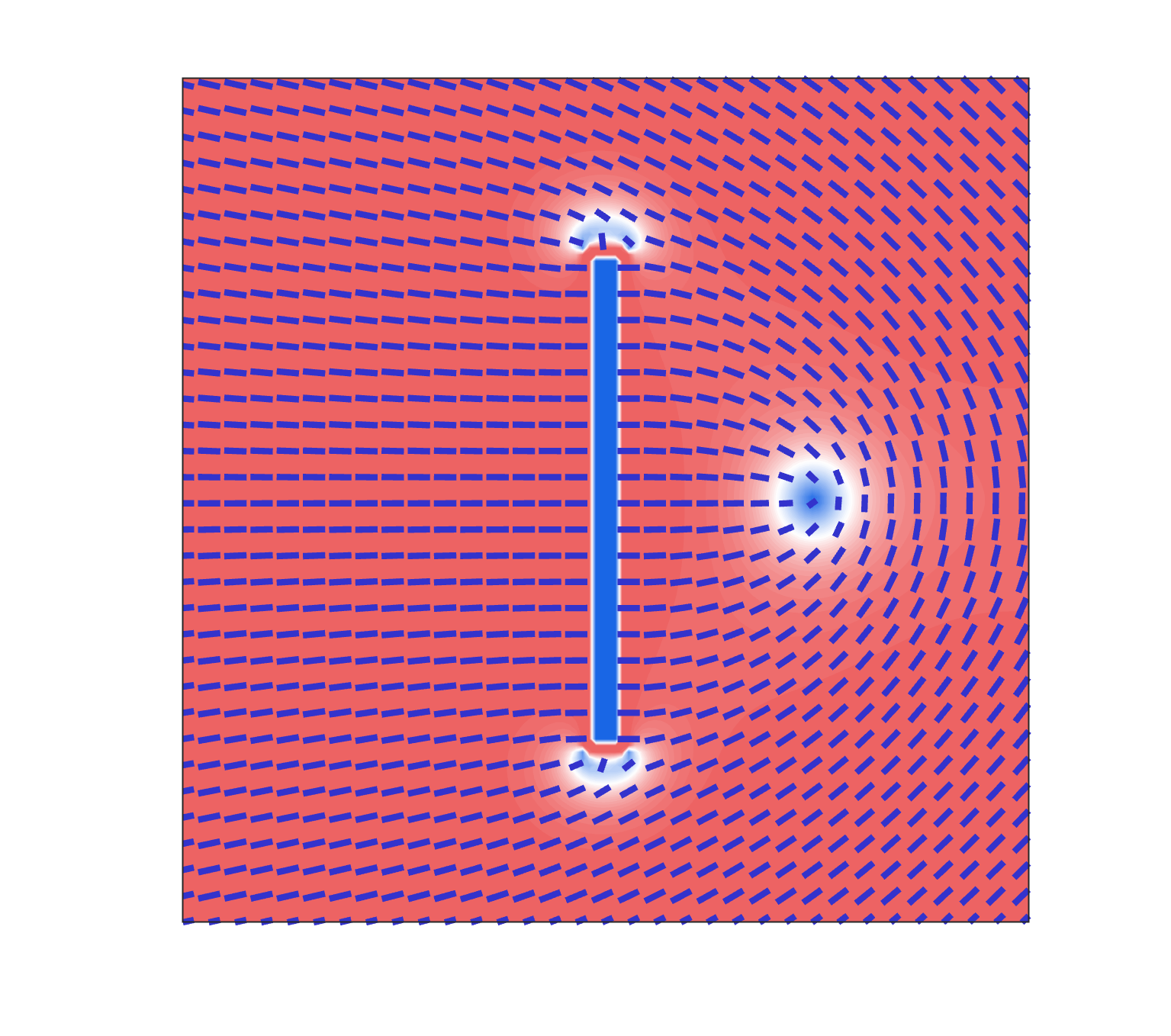}
		\includegraphics[width=0.18\textwidth]{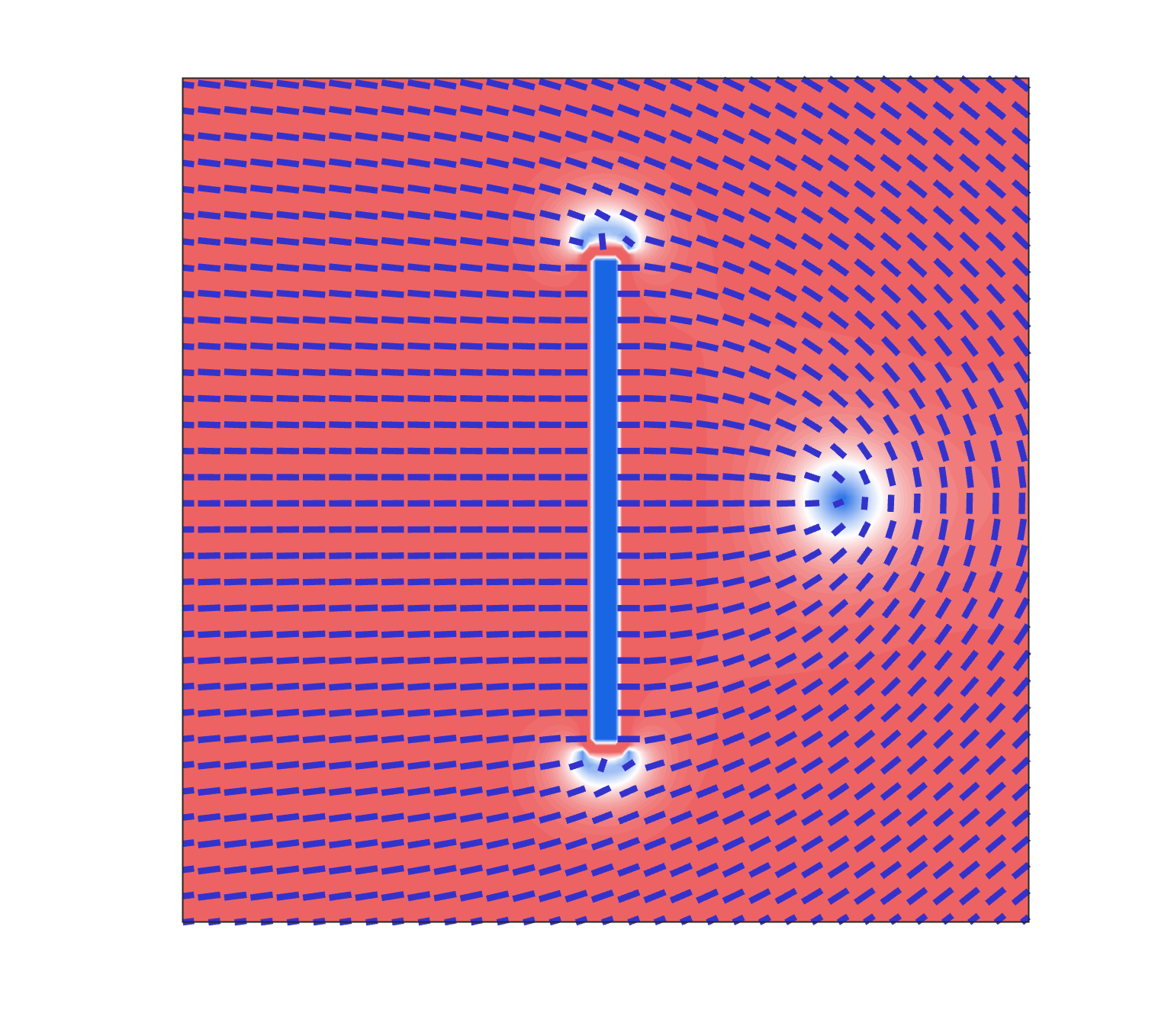}
		\includegraphics[width=0.18\textwidth]{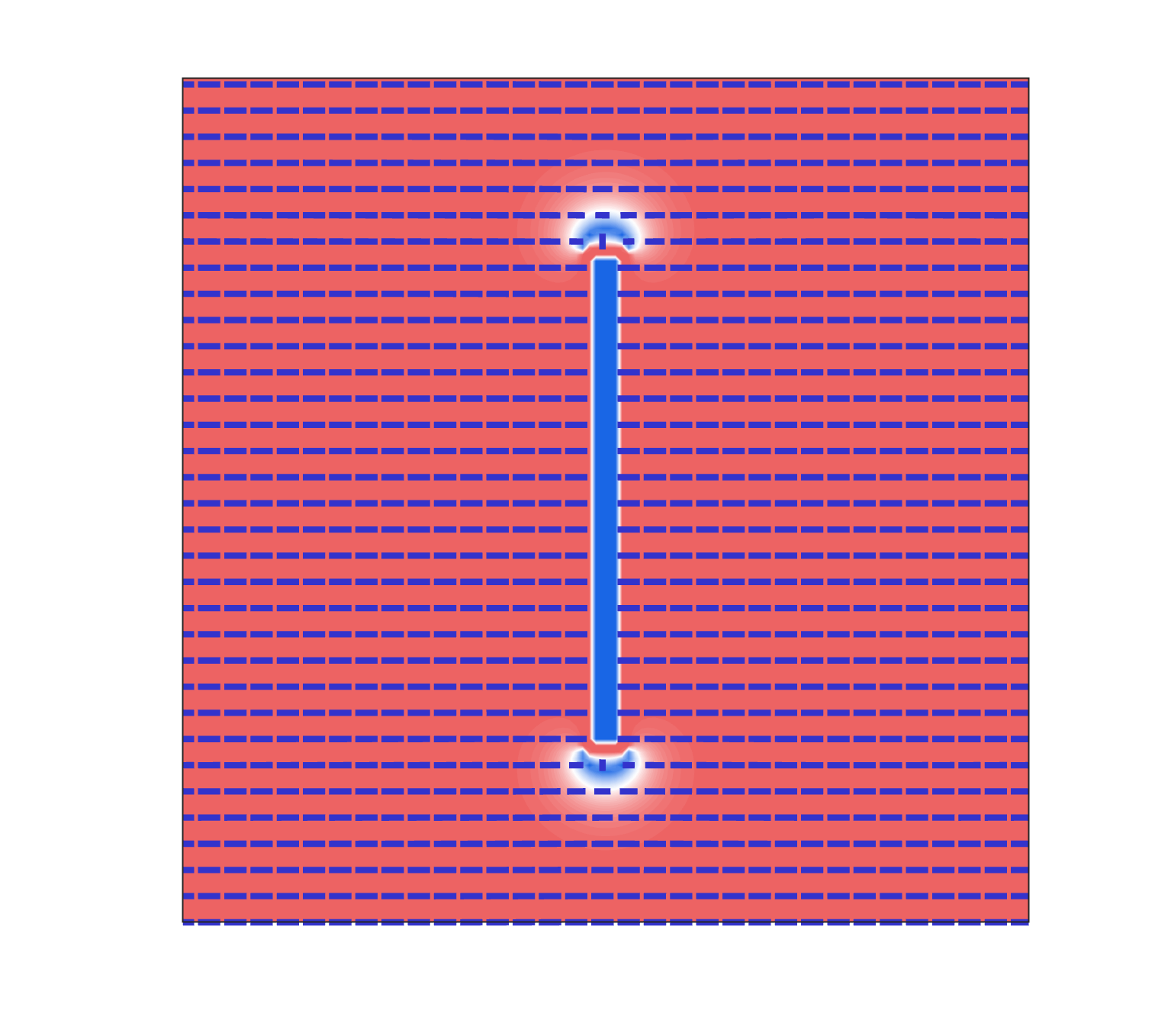}
	\end{center}
	\caption{Interaction of a pair of active $\pm \tfrac{1}{2}$
		defects with a rectangular obstacle under normal anchoring at
		different times $t = 0.1, 3, 5, 7, 20$,respectively,  with
		activity parameter values  $\chi_{fluid} = -5$ and $\xi_{fluid} =
			0.1$. The left defect moves out of the computational domain
		sooner than the right one. }\label{fig:chineg}
\end{figure}

Figure~\ref{fig:chineg} illustrates the dynamics of a pair of $\pm
	\tfrac{1}{2}$ defects in a domain containing a solid obstacle,
under activity parameter values $\chi_{fluid} = -5$ and
$\xi_{fluid} = 0.1$. The $-\tfrac{1}{2}$ defect gradually drifts
away from the obstacle and eventually exits the domain through the
boundary. In contrast, the motion of the $+\tfrac{1}{2}$ defect is
strongly influenced by the negative activity parameter
$\chi_{fluid}$ and the anchoring condition. A negative activity
parameter includes a force directed from head to tail, rather than
in its usual head-forward direction. As a result, the
$+\tfrac{1}{2}$ defect moves more slowly than the $-\tfrac{1}{2}$
in this scenario.
The final panel at $t = 20$ shows that both $\pm\tfrac{1}{2}$
defects have exited the domain, and the director field remains
isotropic and stationary within the solid region.


\begin{figure}[H]
	\begin{center}
		\includegraphics[width=0.24\textwidth]{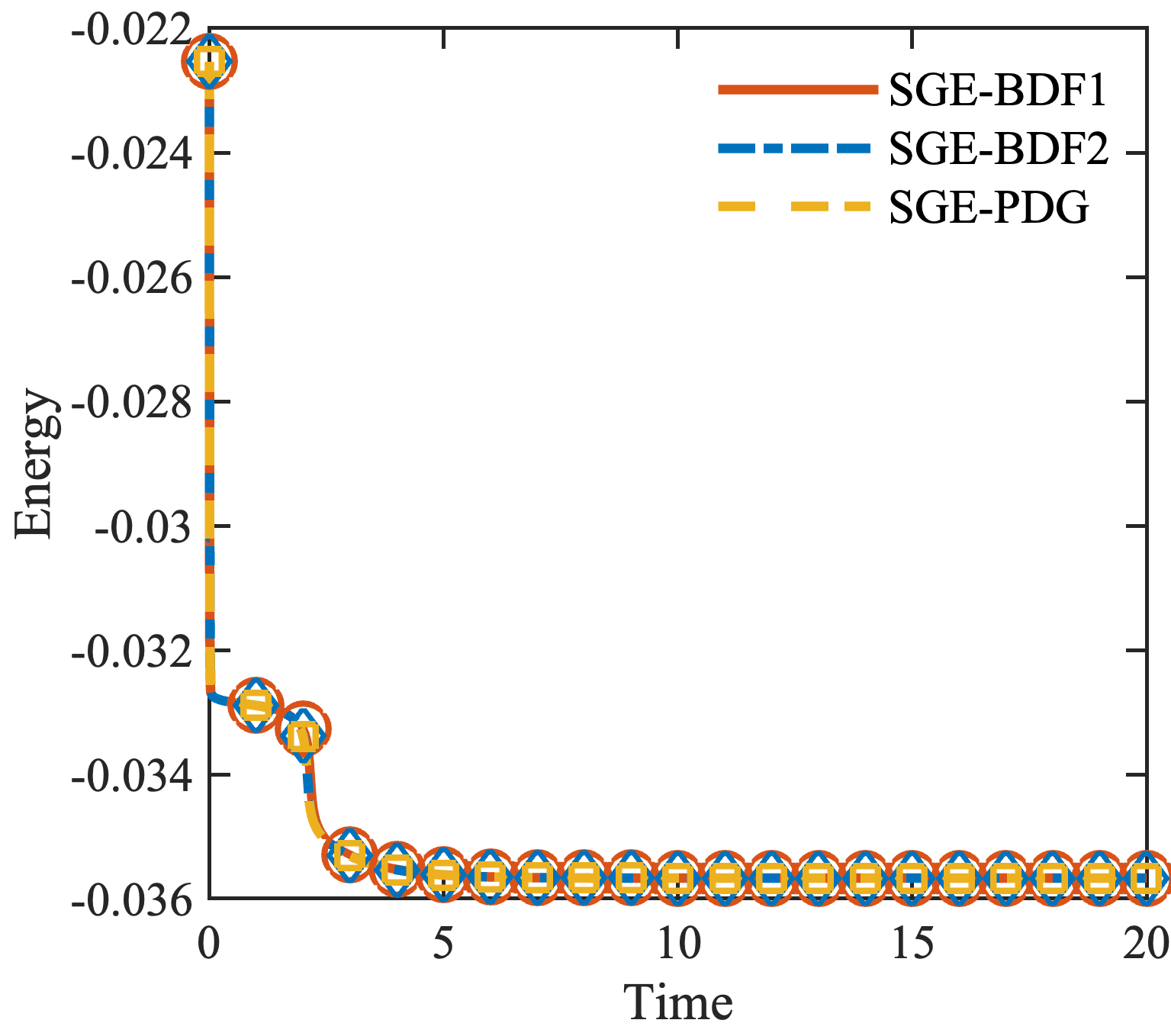}
		\includegraphics[width=0.24\textwidth]{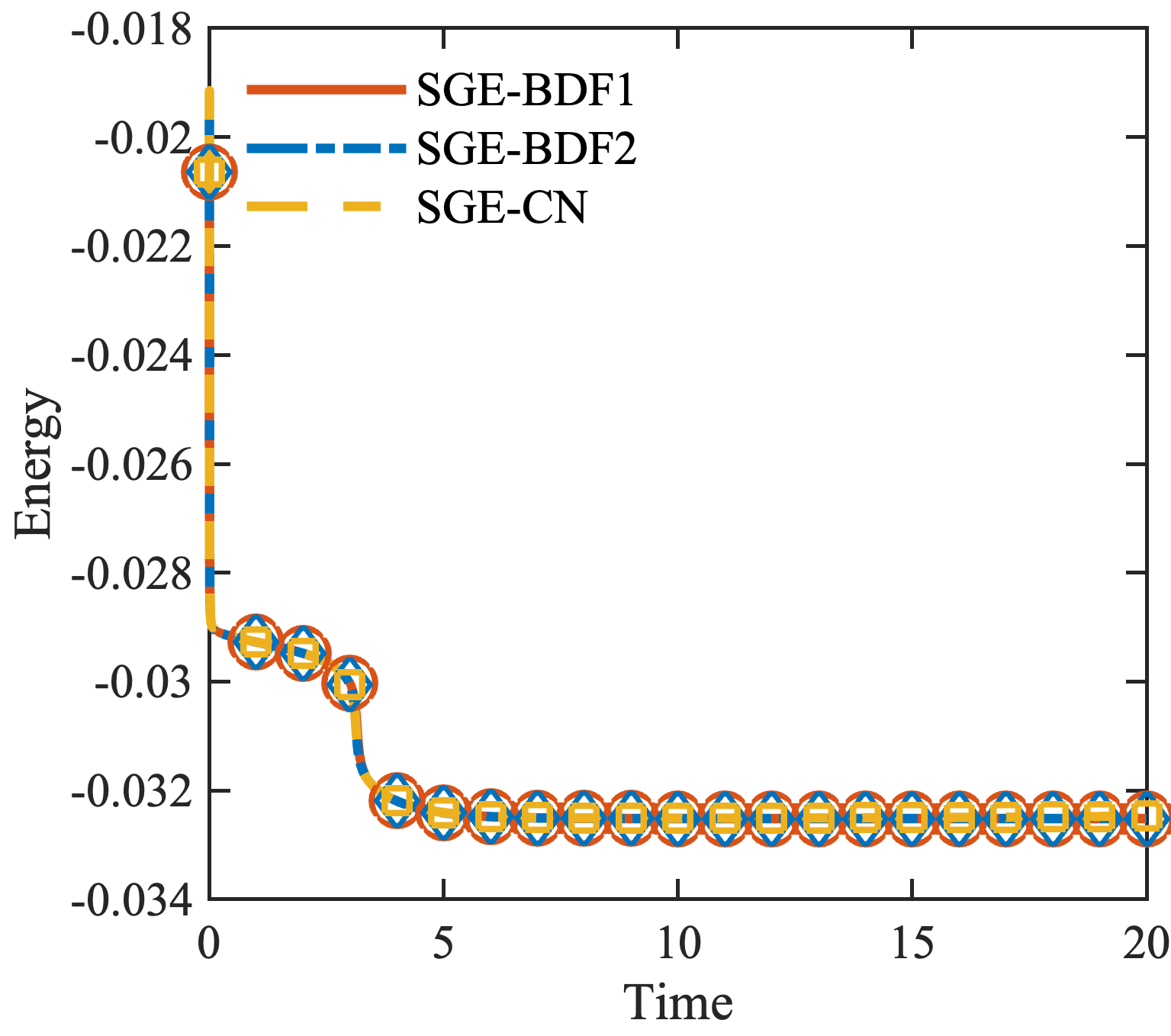}
		\includegraphics[width=0.24\textwidth]{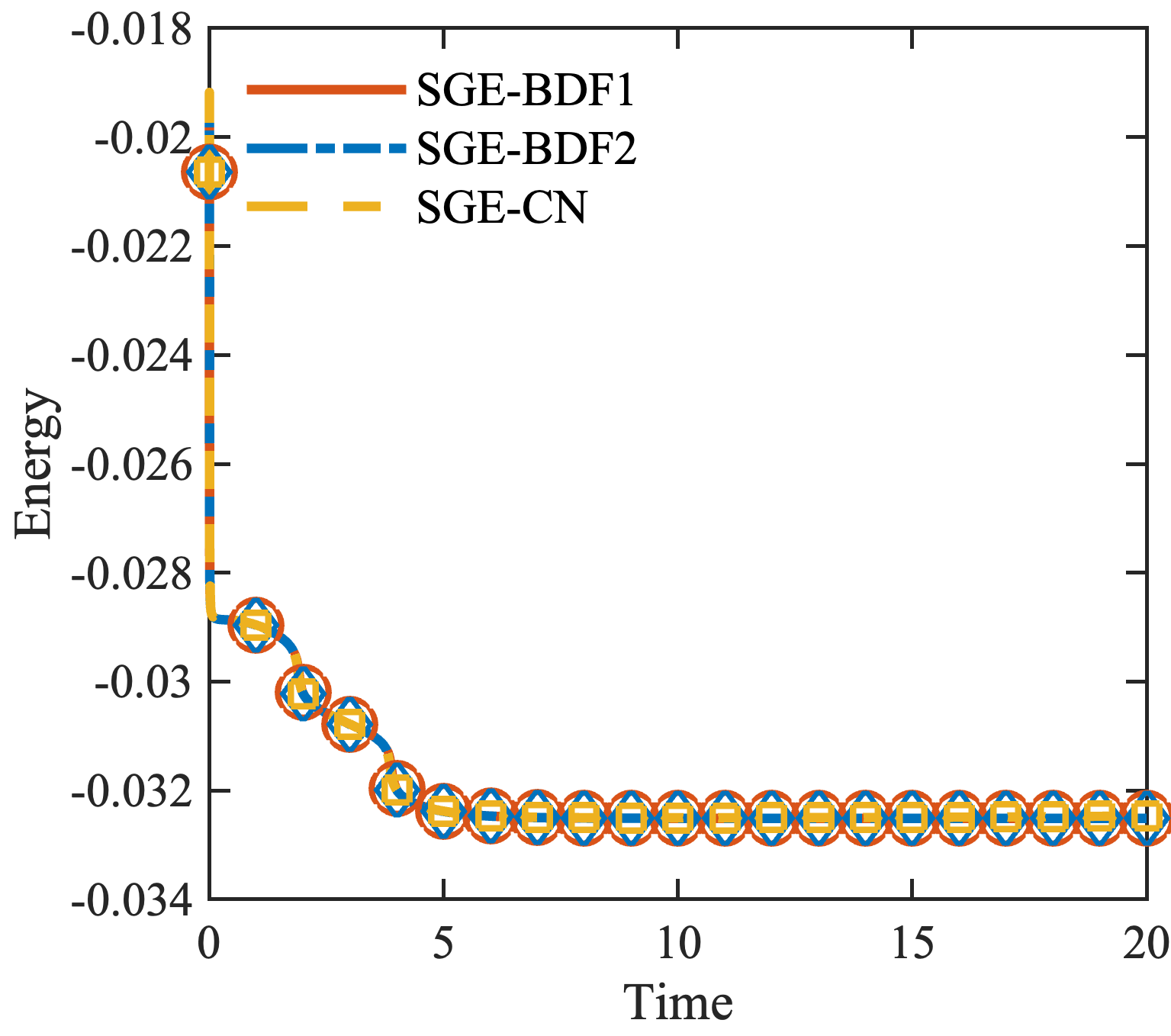}
		\includegraphics[width=0.24\textwidth]{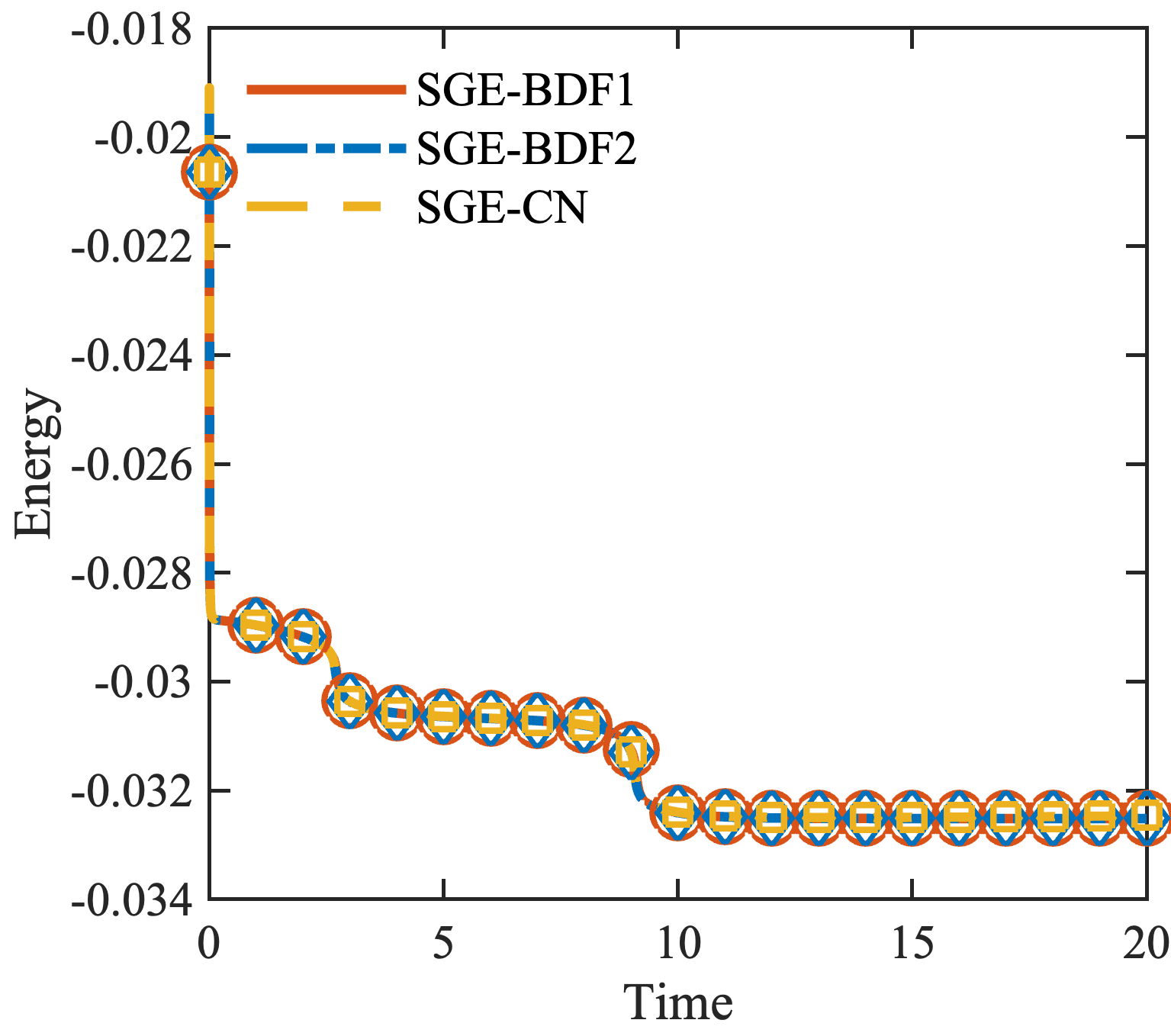}
	\end{center}
	\caption{Time evolution of the free energy in the four different
		test cases with the solutions solved using the three proposed
		schemes. From left to right, the plots correspond to: (1) the
		passive liquid crystal without an obstacle; (2) the passive
		liquid crystal with a rectangular obstacle; (3) the active liquid
		crystal with $\chi_{fluid} = 5$, $\xi_{fluid} = -0.1$; and (4)
		the active liquid crystal with $\chi_{fluid} = -5$, $\xi_{fluid}
			= 0.1$. In all cases, the energy decays in time. }\label{fig:ex2_fenergy}
\end{figure}

Figure~\ref{fig:ex2_fenergy} presents the temporal evolution of the
free energy in the four cases discussed above, computed using the
three proposed numerical schemes. The energy profiles produced by
the different schemes show excellent agreement, demonstrating the
consistency and stability of the schemes.

For the passive liquid crystal cases, the free energy decays
monotonically over time, as expected. In contrast, for active
systems, the energy does not necessarily decrease monotonically due
to the presence of active forcing. However, this non-monotonic
behavior is less pronounced in the present example, as the activity
parameters are relatively small. This phenomenon will be more
clearly demonstrated in subsequent numerical experiments.

Notably, each energy curve exhibits one or more sudden drops, which
correspond to events of the topological defects. These drops
typically indicate that defect pairs have annihilated or exited the
domain due to boundary anchoring. In the case with non-zero
activity parameters, two distinct energy drops are observed,
corresponding to the sequential disappearance of the
$+\tfrac{1}{2}$ and $-\tfrac{1}{2}$ defects. In the other
scenarios, only one energy drop is present, suggesting either
simultaneous annihilation of the $\pm \tfrac{1}{2}$ defects or
simultaneous exits of the defects.

\subsubsection{Solutions with different anchoring conditions}
We compare the solutions of the model under different boundary
anchoring conditions. In this section, the active parameters are
set at $\chi_{fluid} = -5$ and $\xi_{fluid} = 0.1$.

\begin{figure}[H]
	\begin{center}
		\includegraphics[width=0.18\textwidth]{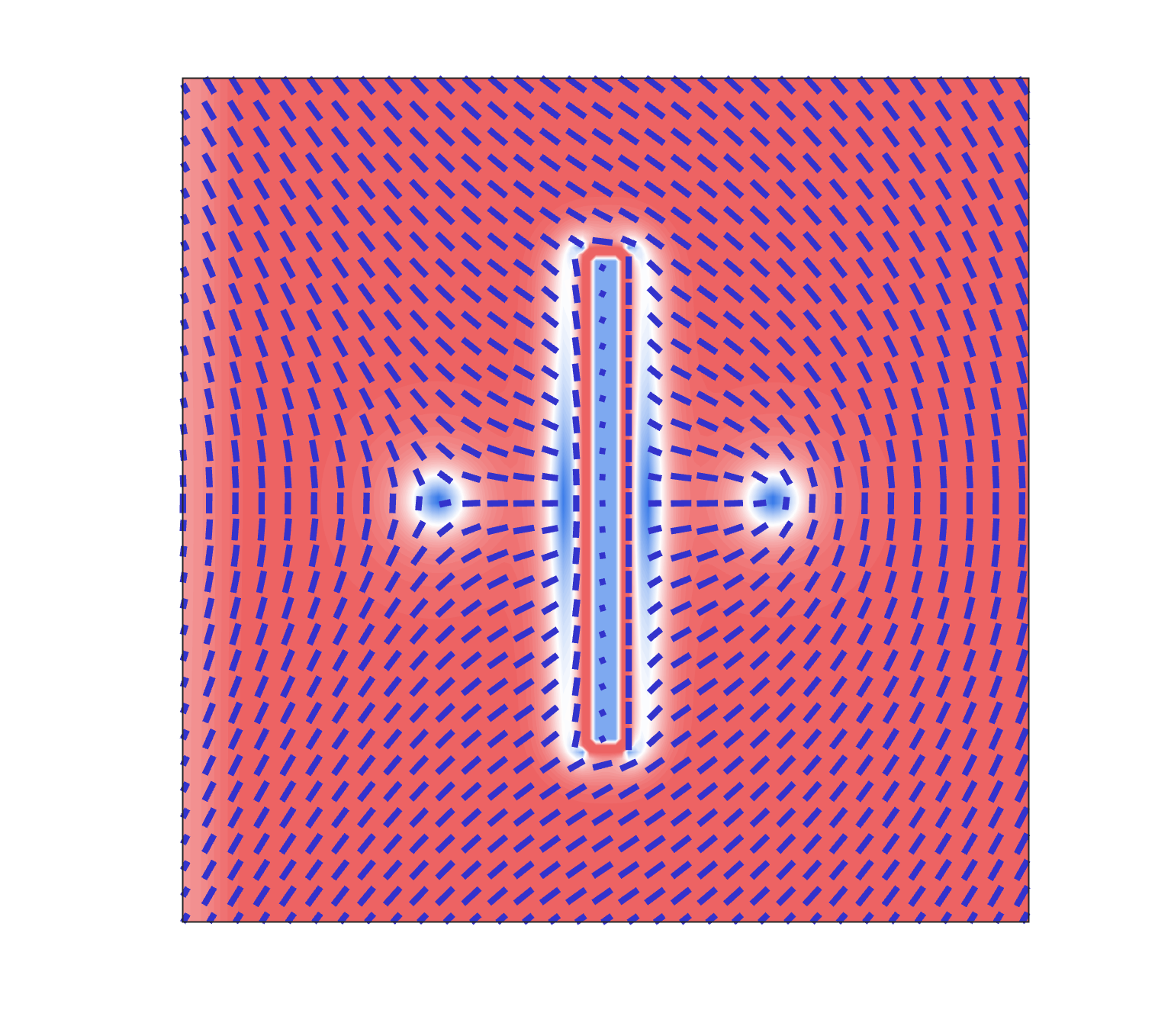}
		\includegraphics[width=0.18\textwidth]{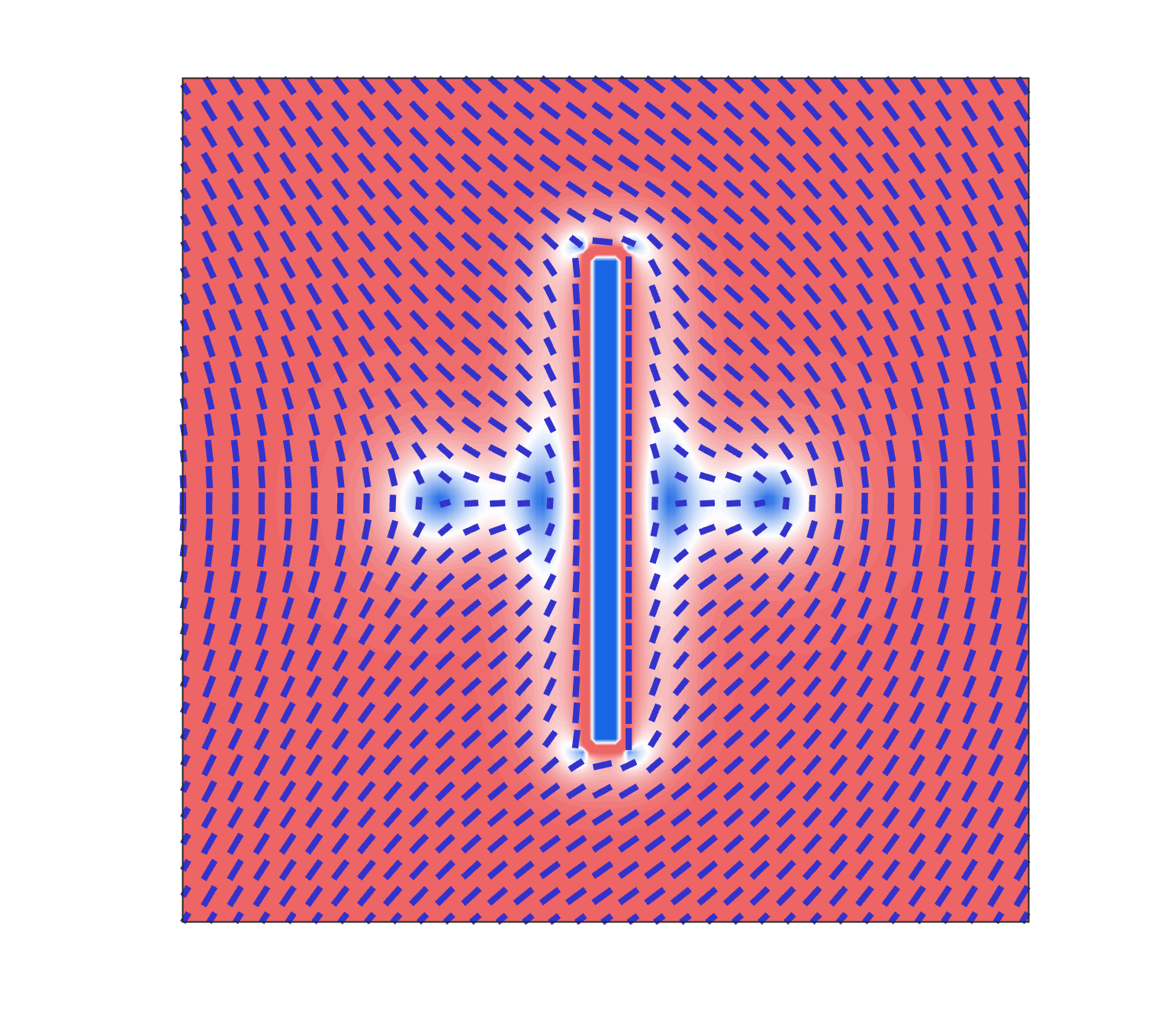}
		\includegraphics[width=0.18\textwidth]{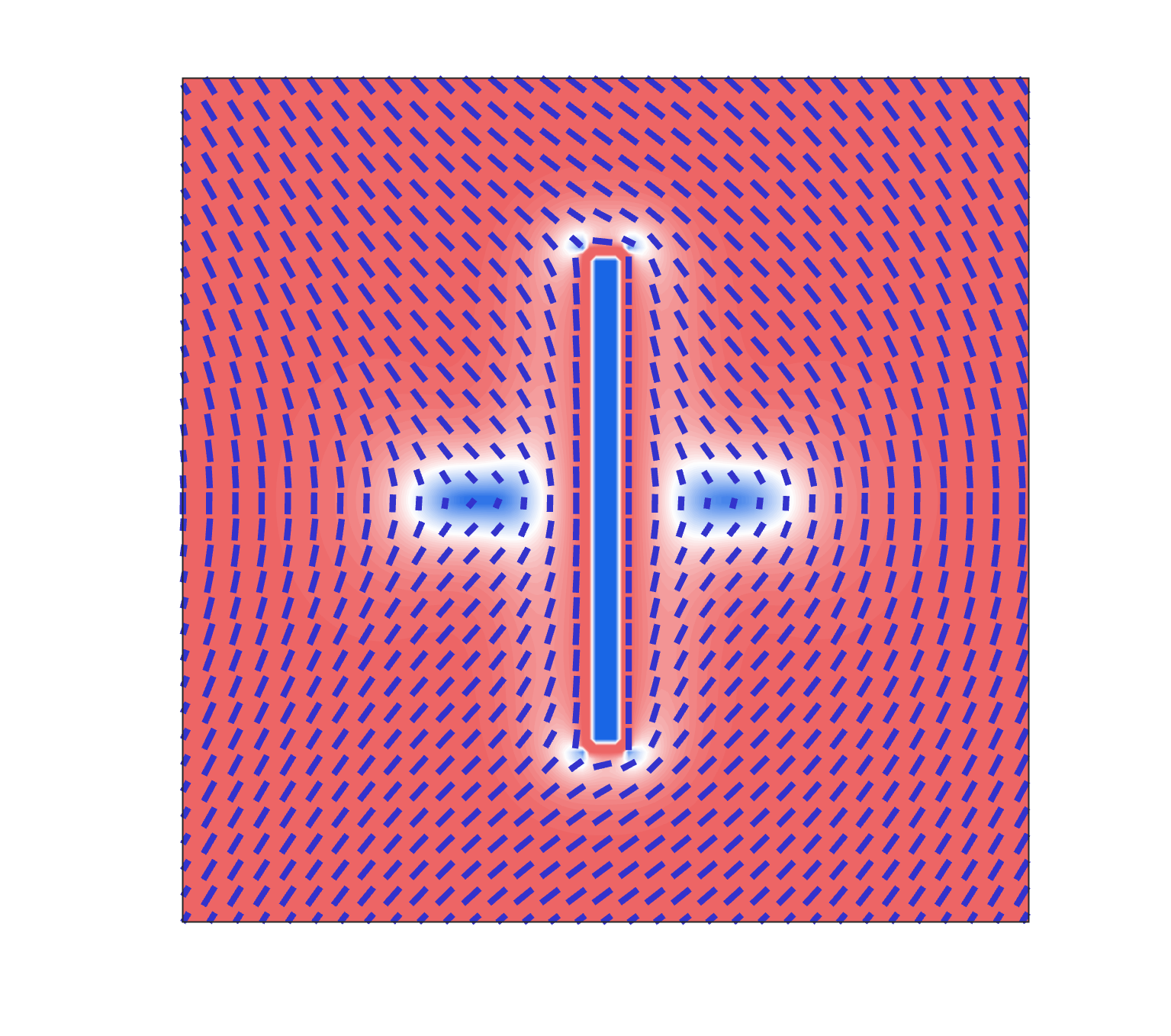}
		\includegraphics[width=0.18\textwidth]{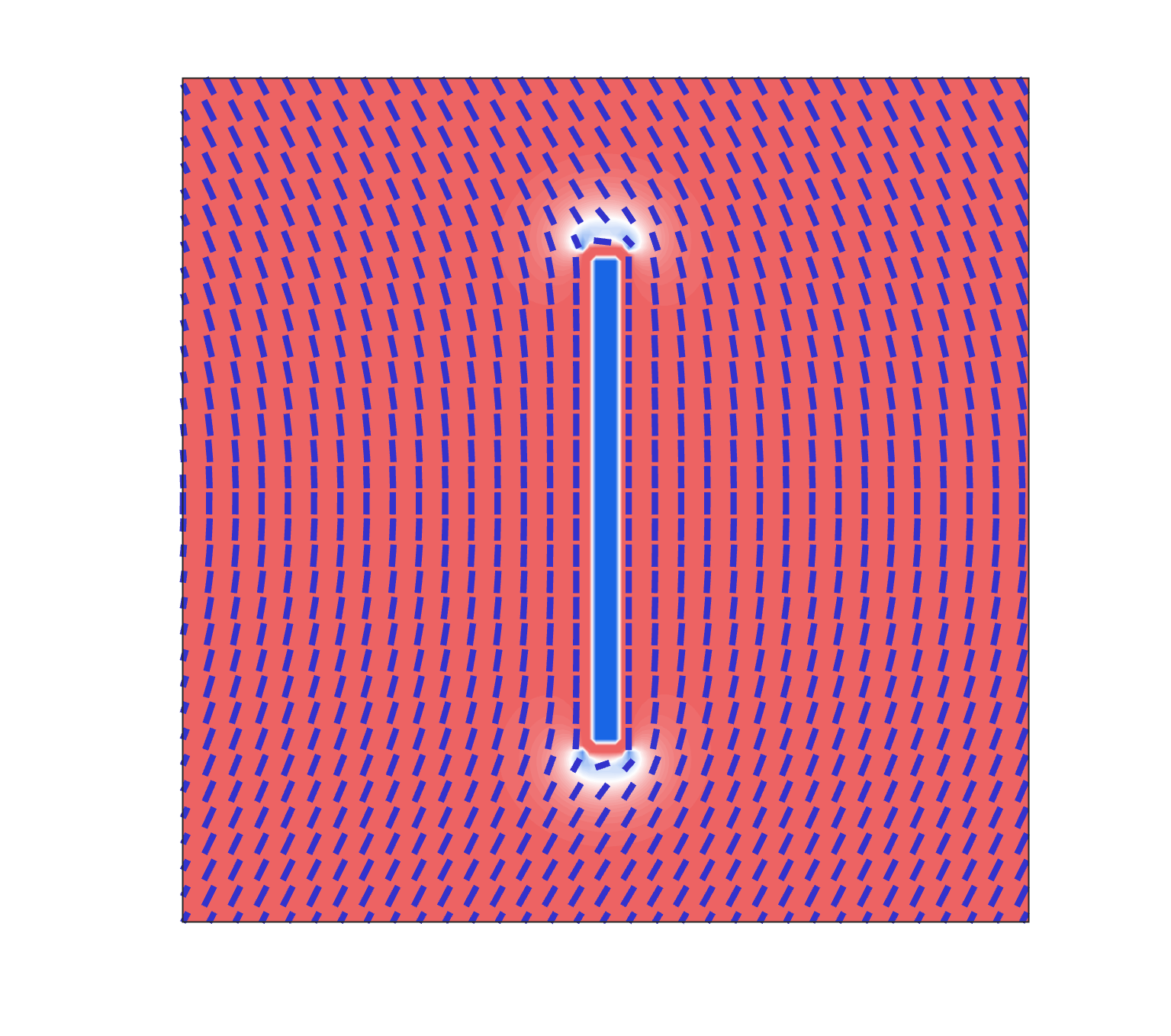}
		\includegraphics[width=0.18\textwidth]{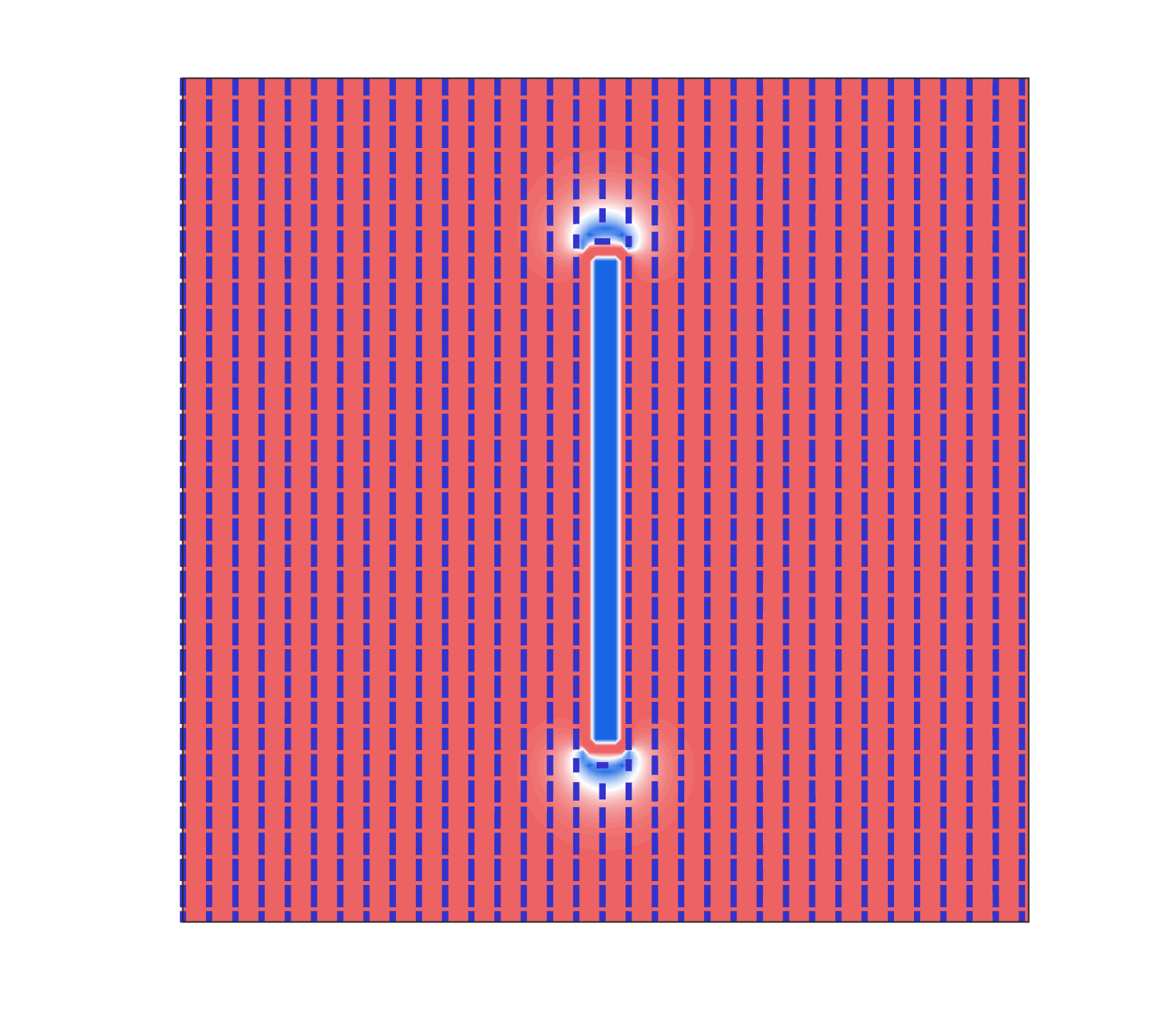}
	\end{center}
	\caption{Interaction of a pair of active $\pm \tfrac{1}{2}$
		defects with a rectangular obstacle with the tangential anchoring
		at $t = 0.01$, $0.05$, $0.1$, $1.0$, $20$, respectively, with
		activity parameter values  $\chi_{fluid} = -5$ and $\xi_{fluid} =
			0.1$. Defects disappear near the obstacle. }\label{fig:tangential_chineg}
\end{figure}

\begin{figure}[H]
	\begin{center}
		\includegraphics[width=0.18\textwidth]{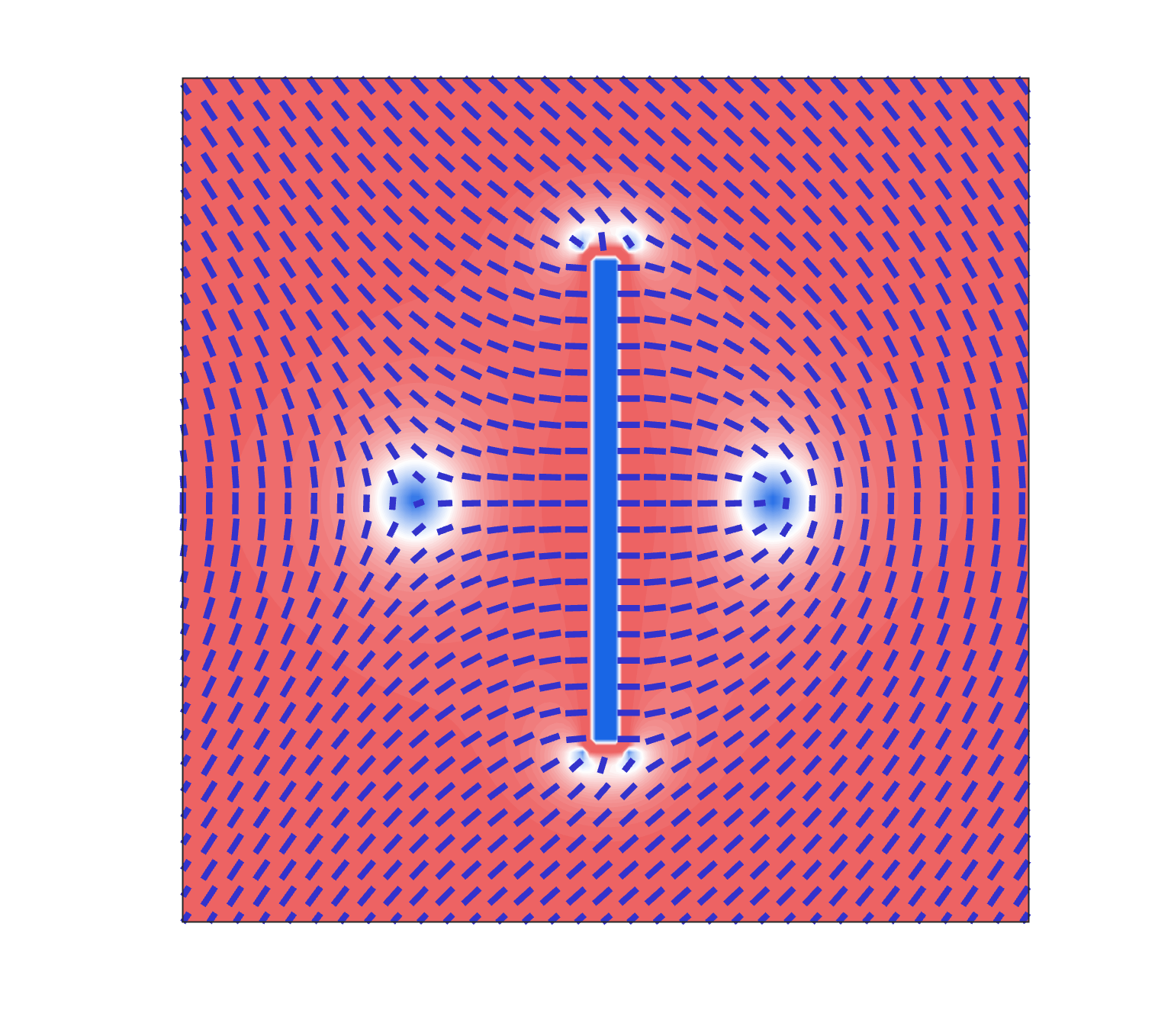}
		\includegraphics[width=0.18\textwidth]{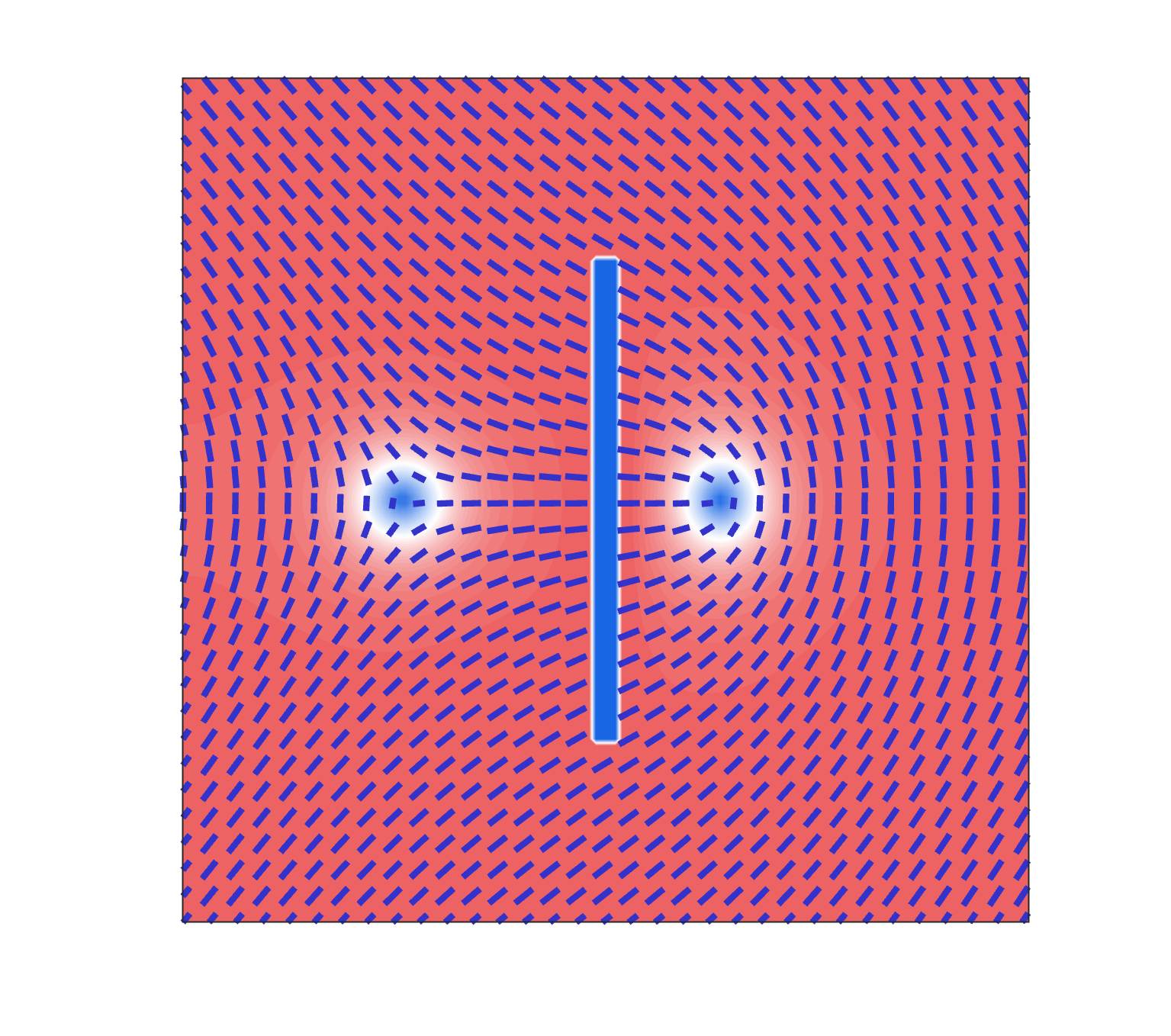}
		\includegraphics[width=0.18\textwidth]{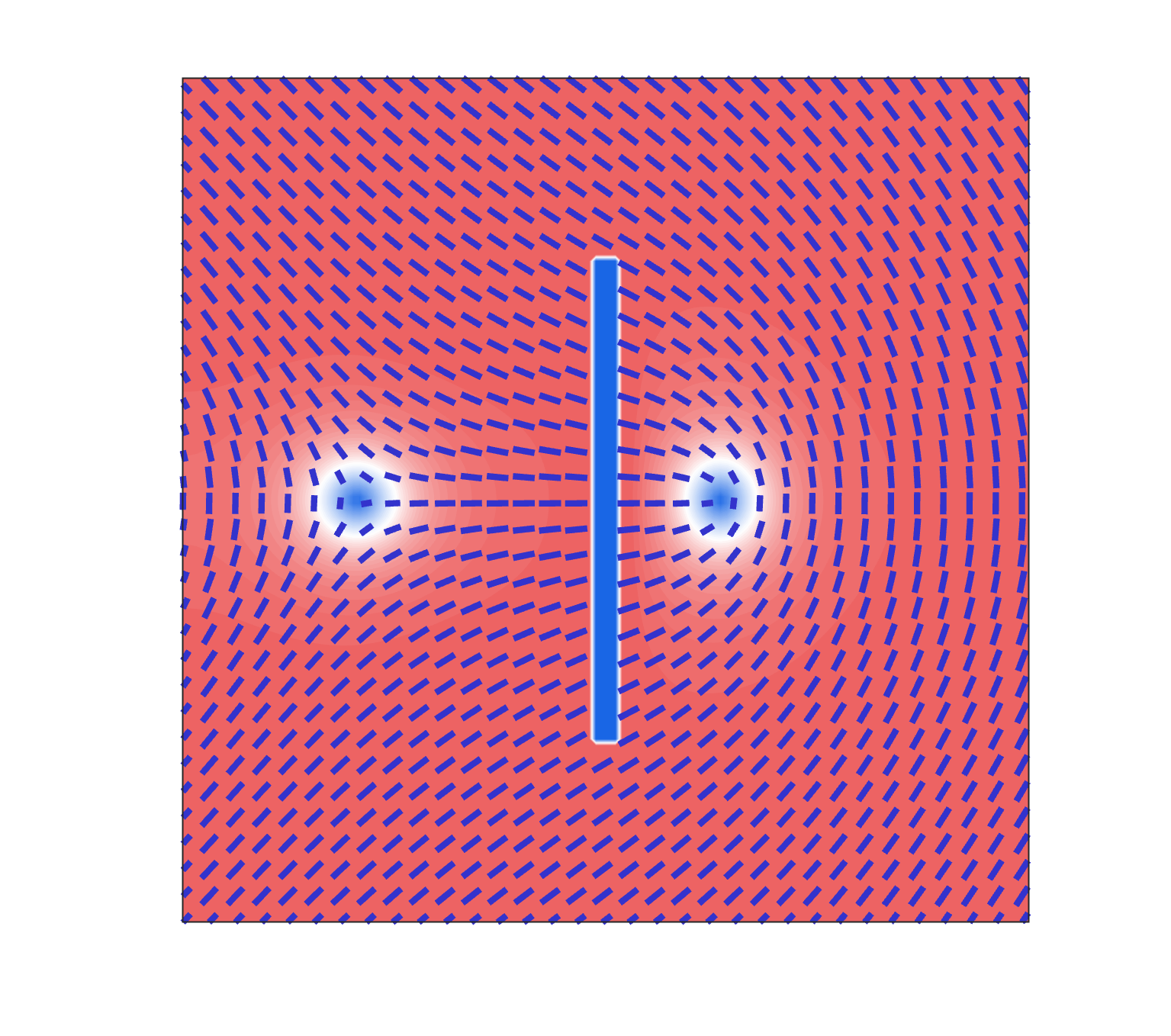}
		\includegraphics[width=0.18\textwidth]{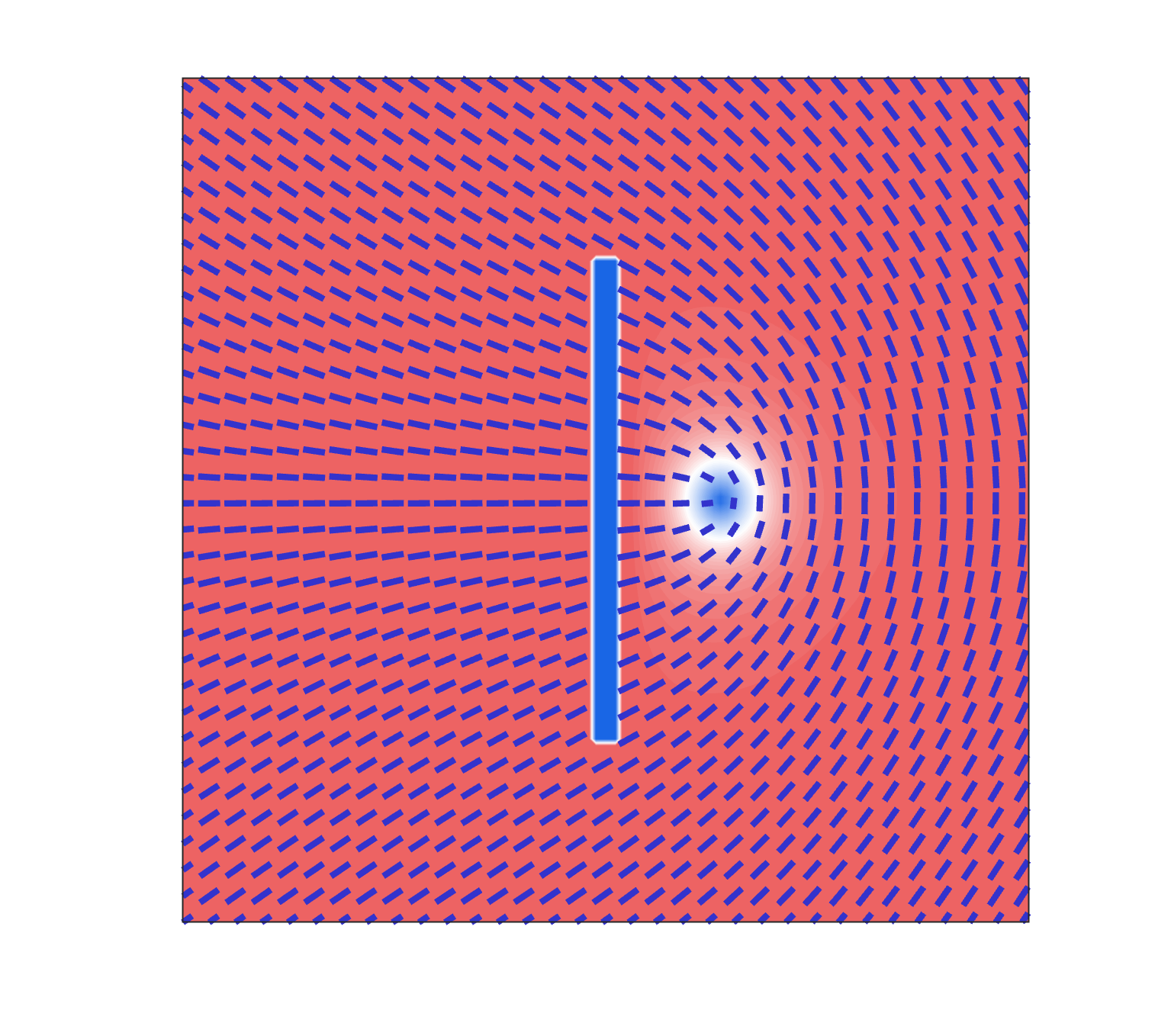}
		\includegraphics[width=0.18\textwidth]{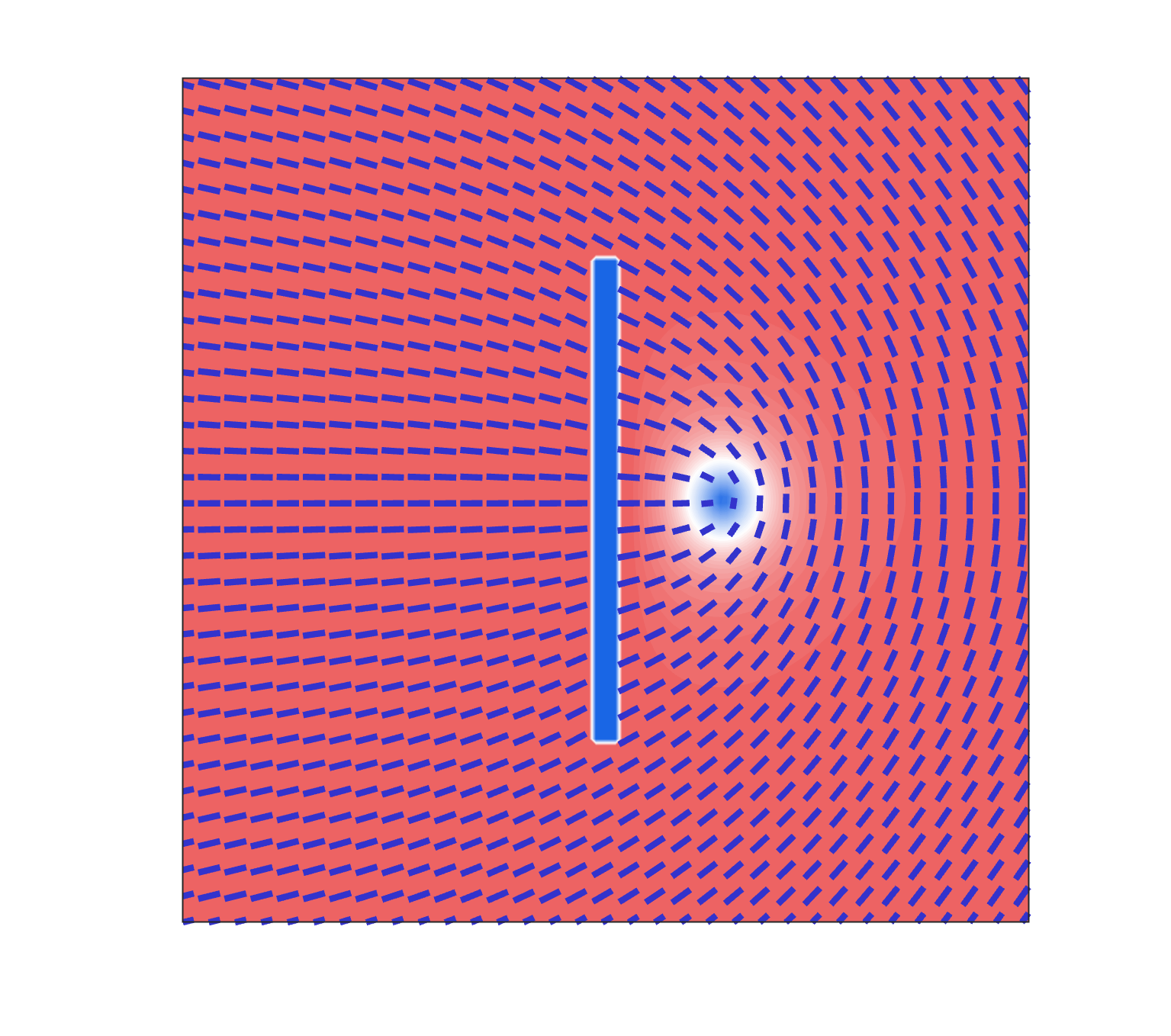}
	\end{center}
	\caption{Interaction of a pair of active $\pm \tfrac{1}{2}$
		defects with a rectangular obstacle under anchoring consistent
		with the initial condition at $t = 0.5$, $3$, $5$, $7$, $20$,
		respectively, with activity parameter values  $\chi_{fluid} = -5$
		and $\xi_{fluid} = 0.1$. The left defect leaves the computational
		domain after a while.}\label{fig:initial_chineg}
\end{figure}

\begin{figure}[H]
	\begin{center}
		\includegraphics[width=0.24\textwidth]{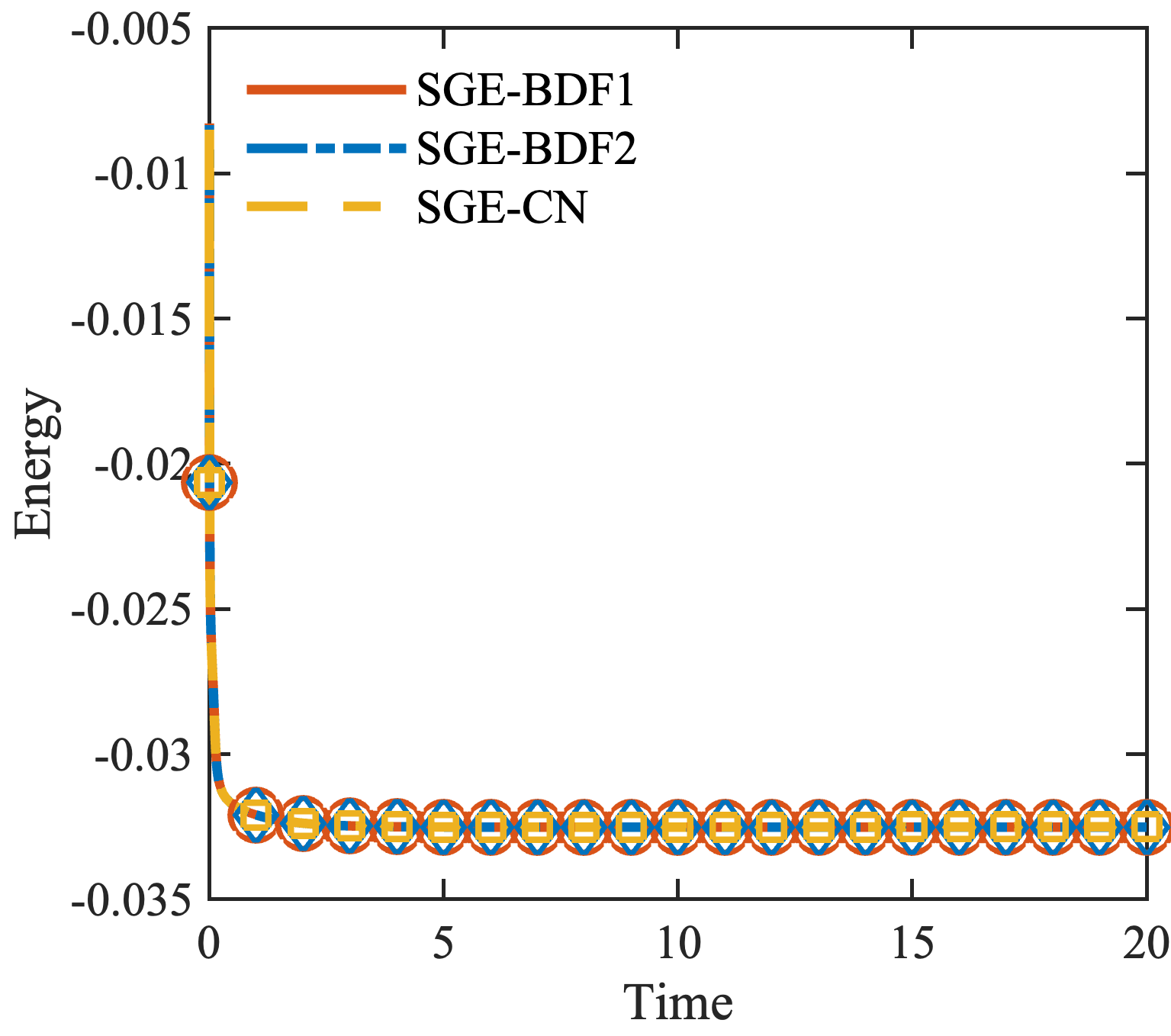}
		\includegraphics[width=0.24\textwidth]{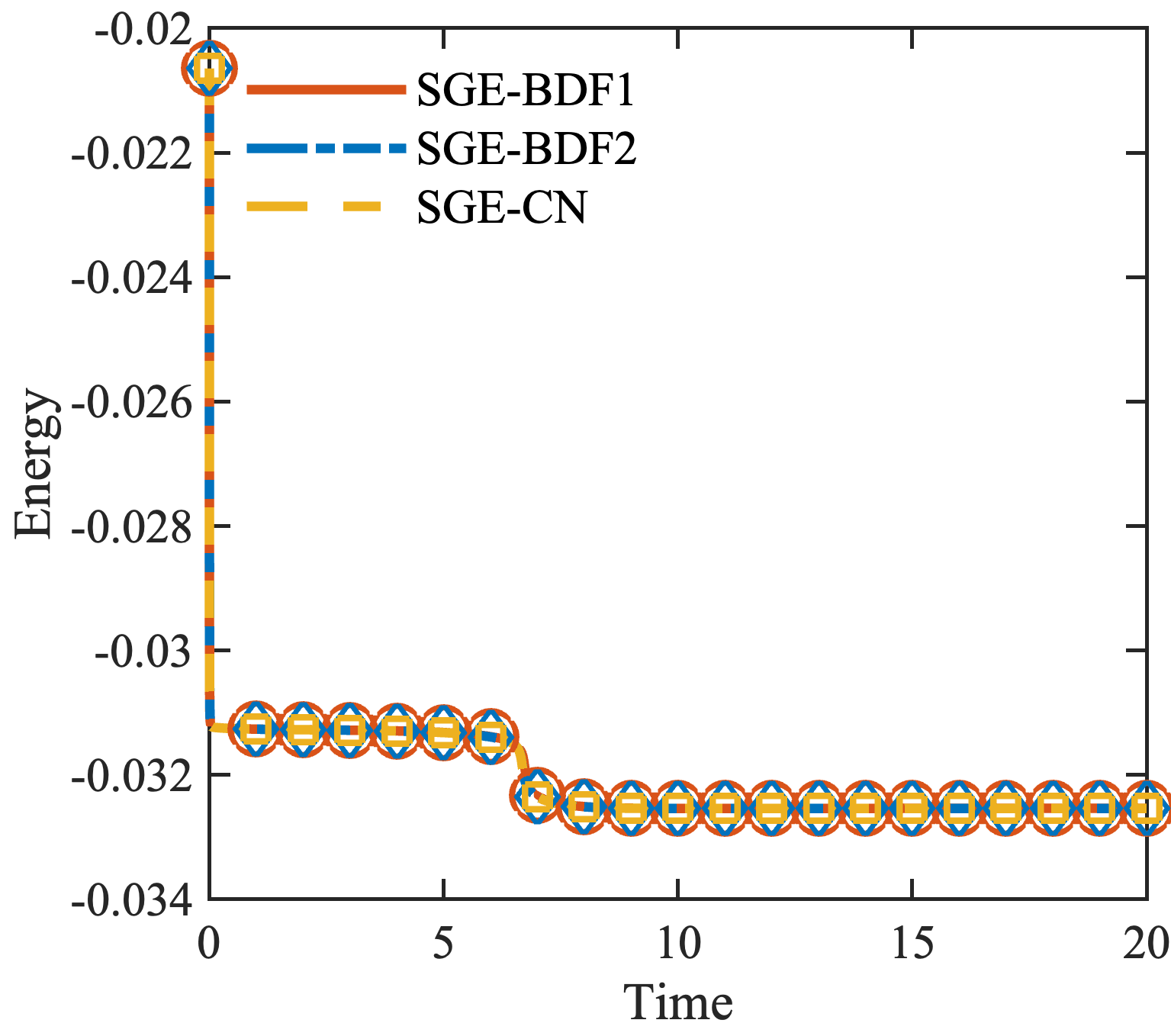}
	\end{center}
	\caption{Time evolution of the free energy for two different test
		cases solved using the three proposed schemes. From the left to
		right, the plots correspond to: (1) the active liquid crystal
		with $\chi_{fluid} = -5$ and $\xi_{fluid} = 0.1$ with the
		tangential anchoring; and (2) the active liquid crystal with
		$\chi_{fluid} = -5$ and $\xi_{fluid} = 0.1$ with the anchoring
		consistent with the initial condition.}\label{fig:ex2_fenergy_2}
\end{figure}

Figure~\ref{fig:tangential_chineg} depicts the time evolution of
the principal eigenvalue difference and the defect field for an
active liquid crystal with weak tangential anchoring at the solid
boundary. The defect is absorbed by the boundary due to the
anchoring effect, and the system eventually reaches a steady state.
Compared with the case of normal anchoring-where the steady-state
defect field is aligned horizontally-here the director field is
aligned vertically, reflecting the influence of the tangential
anchoring at the solid boundary.

Figure~\ref{fig:initial_chineg} depicts the time evolution of the
principal eigenvalue difference and the defect field for the active
liquid crystal under the anchoring condition
$$
	\mathbf{Q}|_{\partial D} = \mathbf{Q}_0|_{\partial D}.
$$
The steady-state behavior differs from the previous cases. The
$-\tfrac{1}{2}$ defect gradually drifts away from the obstacle and
eventually leaves the domain. In contrast, due to the active force
and the imposed anchoring, the $+\tfrac{1}{2}$ defect is driven
from head to tail, causing it to move toward the obstacle where it
is ultimately trapped. The final panel at $t = 20$ shows that
$+\tfrac{1}{2}$ remains trapped in the domain in this scenario.
Moreover, despite the continuous accumulation of active stress near
the obstacle, the director field remains isotropic and stationary
within the solid region, further validating the effectiveness of
the proposed treatment for obstacles in long-time simulations.

Figures~\ref{fig:ex2_fenergy_2} show the corresponding free energy
for the two cases. The energy curves obtained by the three
different schemes are in close agreement; they decrease
monotonically. These two cases illustrate the strong influence that
the boundary anchoring can have on the defect dynamics.

\subsection{Shape effect of the obstacle  on the dynamics of active
	liquid crystals}
In this example, we further investigate how the shape of the solid
obstacle influences the dynamics of active liquid crystals. The
computational domain is chosen as $\Omega = (0, 4)^2$.
	{\color{blue} Stabilization parameter is set to $\kappa = 5$ in all
		cases.} The initial velocity field is set to zero, and the initial
condition for the
$\mathbf{Q}$-tensor field is defined as

\begin{equation*}
	\mathbf{Q}_0 = S_{eq}\left(\tfrac{\mathbf{n}_0
		\mathbf{n}_0^\top}{\|\mathbf{n}_0\|^2} -
	\tfrac{\mathbf{I}}{3}\right), \quad \mathbf{n}_0 =
	\left\lbrace
	\begin{aligned}
		 & (1, 0, 0)^\top, \ \sqrt{(x - 2)^2 + (y - 2)^2} < 0.5, \\
		 & (0, 1, 0)^\top, \ \text{otherwise}.
	\end{aligned}
	\right.
\end{equation*}

Homogeneous Dirichlet boundary conditions are imposed on the
velocity field, while homogeneous Neumann boundary conditions are
applied to the $\mathbf{Q}$-tensor field at the boundary of the
computational domain. The model parameters are specified as follows
\begin{equation*}
	\begin{aligned}
		 & b_{solid} = 10^5, \ \eta_{fluid} = 1, \ \eta_{solid} = 10^3,
		\ K_{fluid} = 10^{-3},                                          \\
		 & N_{fluid} = 7.5, \ N_{solid} = 0.1, \ \Gamma_{fluid} = 10^2,
		\ \Gamma_{solid} = 10^3, \ A = \tfrac{2}{15}.
	\end{aligned}
\end{equation*}

We choose spatial step size as $h = \tfrac{1}{64}$ and time step
size as $\tau = 5 \times 10^{-4}$. In all the subsequent examples,
$\mathbf{Q}_\star$ is chosen to be consistent with the initial
condition, $\mathbf{Q}_\star = \mathbf{Q}_0$.

We examine the evolution of active liquid crystals confined within
solid obstacles of various shapes, under different signs of
activity parameters. The following obstacle geometries are
considered, each represented via indicator functions.

\begin{itemize}
	\item Circle:
	      \begin{equation*}
		      \phi = \left\lbrace
		      \begin{aligned}
			      1, & \ \sqrt{(x - 2)^2 + (y - 2)^2} < 1.4, \\
			      0, & \ \text{otherwise}.
		      \end{aligned}
		      \right.
	      \end{equation*}
	\item Horizontal ellipse:
	      \begin{equation*}
		      \phi = \left\lbrace
		      \begin{aligned}
			      1, & \ (\tfrac{x - 2}{1.4})^2 + (\tfrac{y - 2}{0.7})^2 < 1, \\
			      0, & \ \text{otherwise}.
		      \end{aligned}
		      \right.
	      \end{equation*}
	\item Vertical ellipse:
	      \begin{equation*}
		      \phi = \left\lbrace
		      \begin{aligned}
			      1, & \ (\tfrac{x - 2}{0.7})^2 + (\tfrac{y - 2}{1.4})^2 < 1, \\
			      0, & \ \text{otherwise}.
		      \end{aligned}
		      \right.
	      \end{equation*}
	\item Star:
	      \begin{equation*}
		      \phi = \left\lbrace
		      \begin{aligned}
			      1, & \ \sqrt{(x - 2)^2 + (y - 2)^2} < 1 + 0.4 \cos{\Big(5
				                                                    {\rm atan2}(y - 2, x - 2)\Big)}, \\
			      0, & \ \text{otherwise}.
		      \end{aligned}
		      \right.
	      \end{equation*}
\end{itemize}

Figures~\ref{fig:incircle-compare}--\ref{fig:star-compare} show the
contour plots of the principal eigenvalue and the corresponding
director field at selected time instances. In all cases, the
initially circular defect splits into four distinct defects-two of
type $+\tfrac{1}{2}$ and two of type $-\tfrac{1}{2}$. These defects
subsequently interact under the combined influence of boundary
anchoring, active stresses, and mutual defect interactions.

A primary observation is that the post-splitting behavior strongly
depends on the sign of the active parameter $\chi_{fluid}$. For
$\chi_{fluid} > 0$, the initial circular defect splits into four
defects, with two $-\tfrac{1}{2}$ defects aligned vertically and
two $+\tfrac{1}{2}$ defects aligned horizontally. Whereas when
$\chi_{fluid} < 0$, the alignment is reversed: the two
$-\tfrac{1}{2}$ defects appear horizontally, while the two
$+\tfrac{1}{2}$ defects are oriented vertically.

\begin{figure}[H]
	\centering

	\fbox{%
		\begin{minipage}{0.96\textwidth}
			\centering
			\textbf{$\chi_{fluid} = 10$, $\xi_{fluid} = -0.1$} \\[0.5em]
			\includegraphics[width=0.18\textwidth]{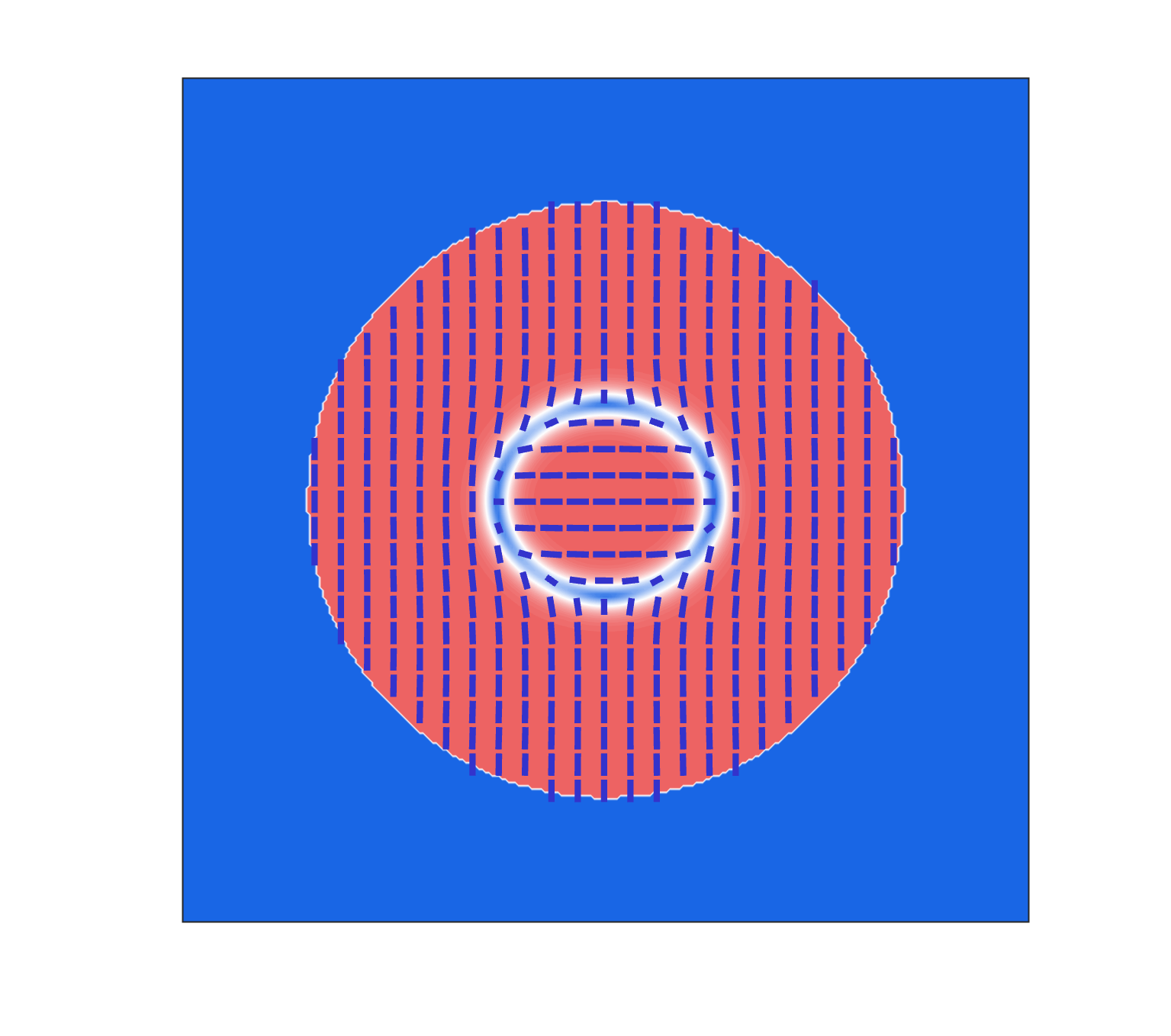}
			\includegraphics[width=0.18\textwidth]{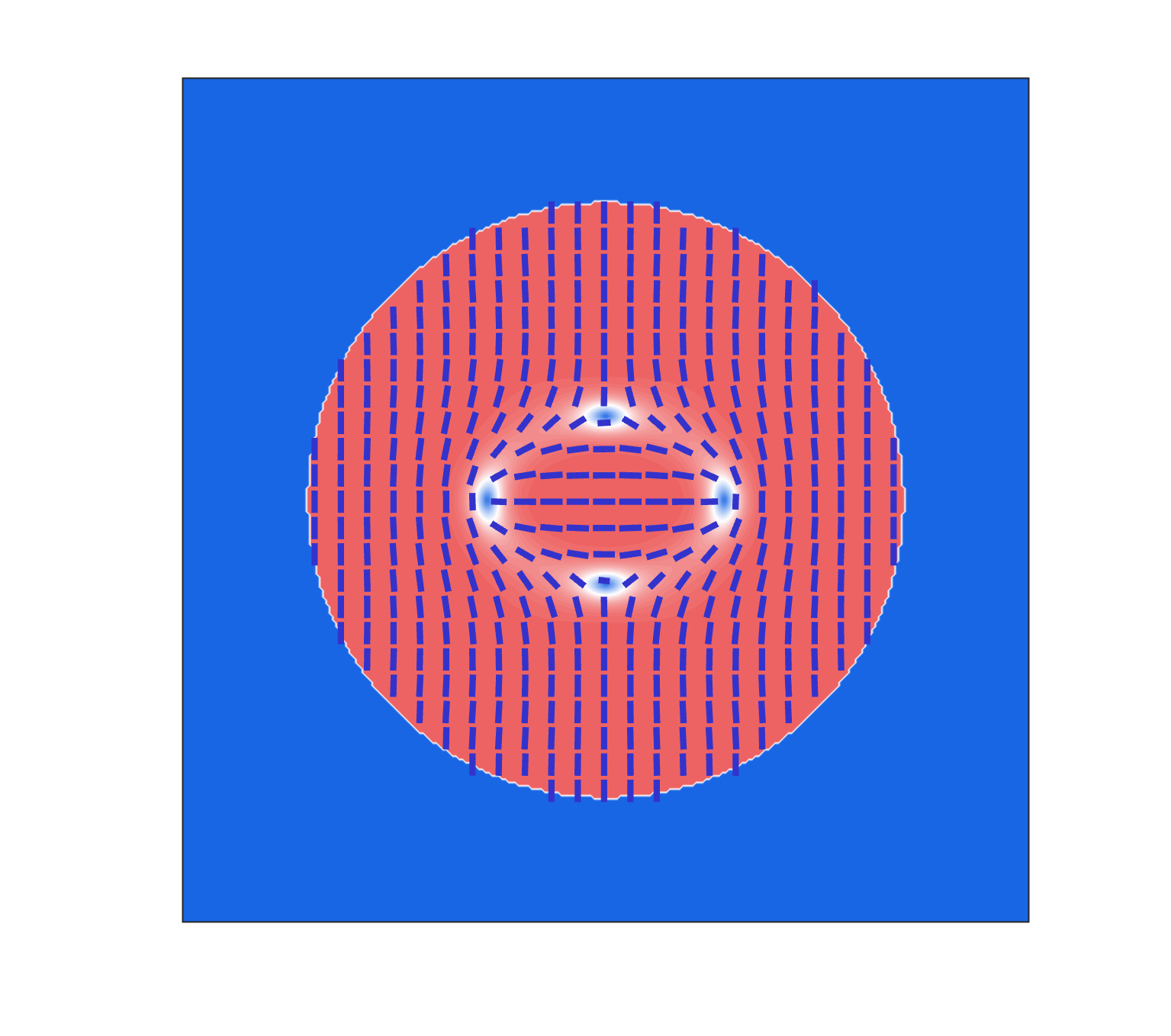}
			\includegraphics[width=0.18\textwidth]{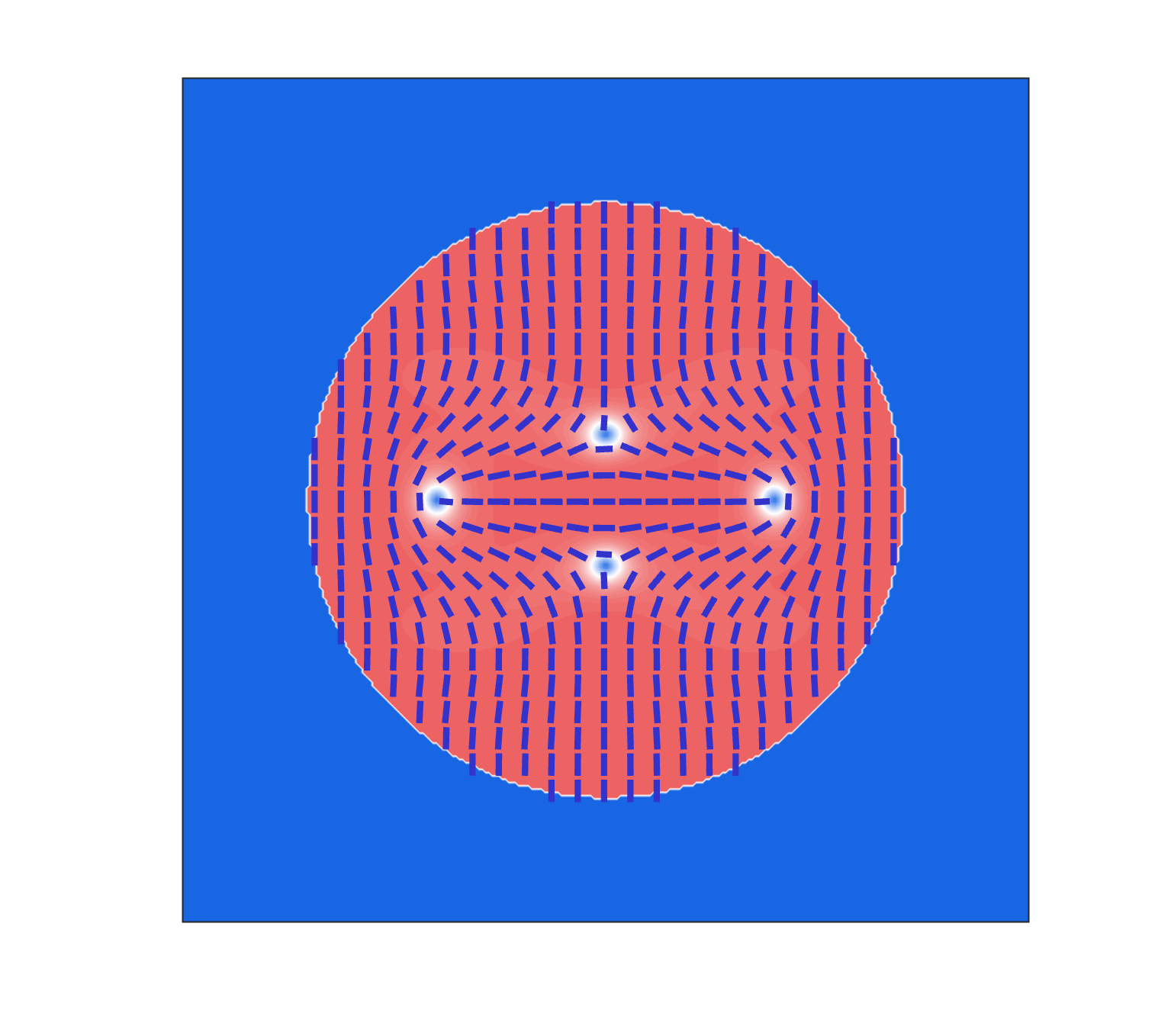}
			\includegraphics[width=0.18\textwidth]{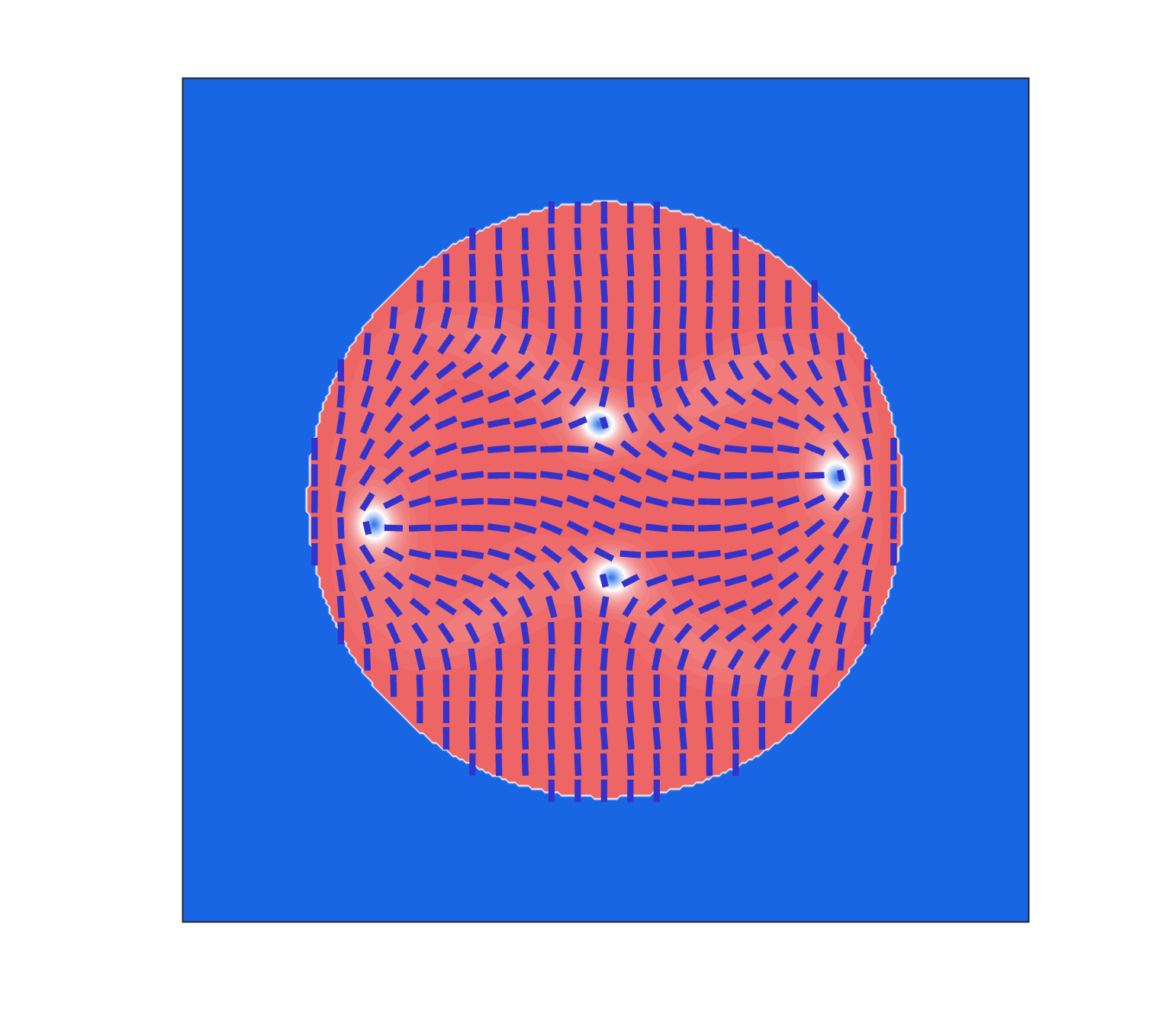}
			\includegraphics[width=0.18\textwidth]{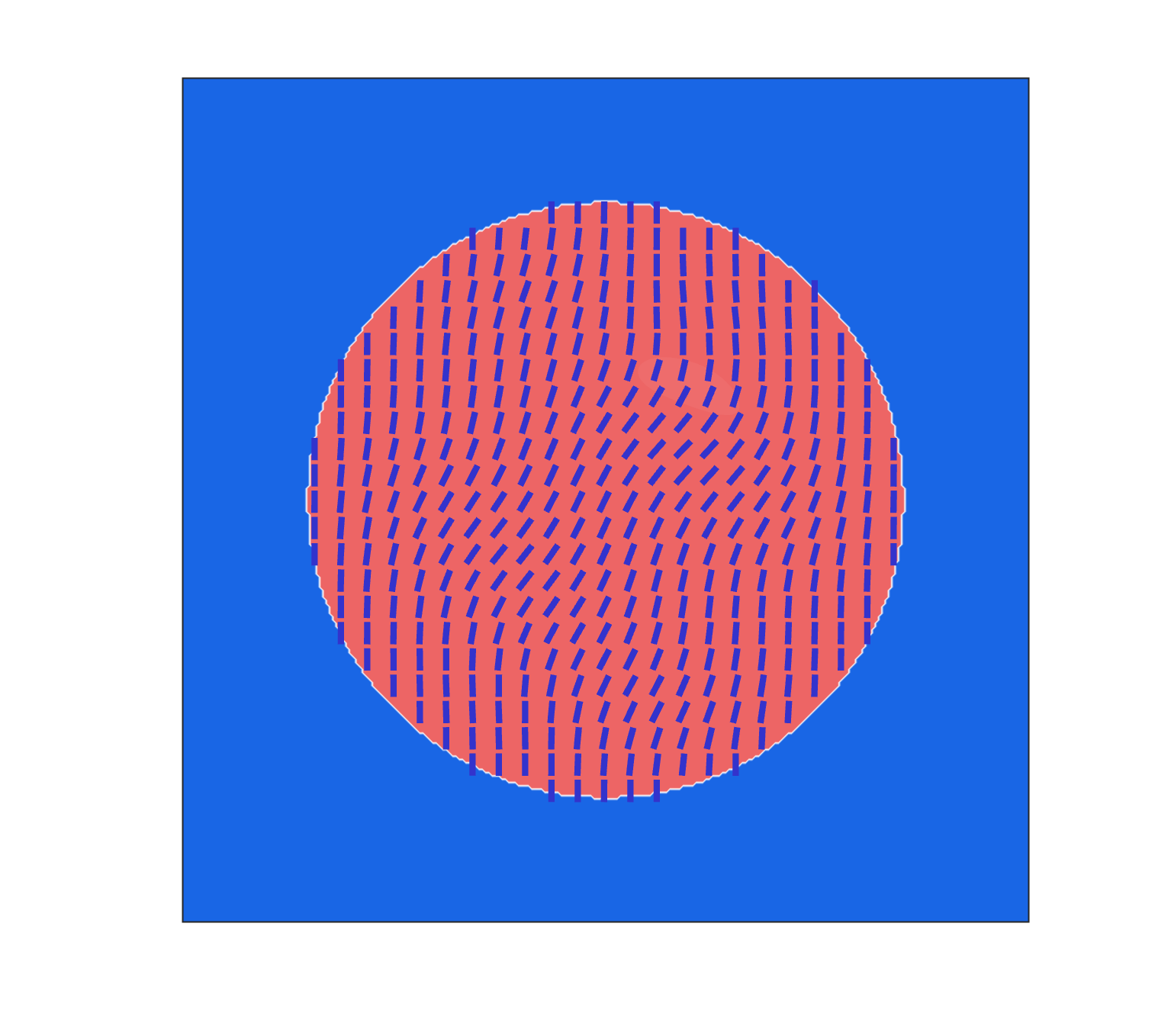}
			\includegraphics[width=0.18\textwidth]{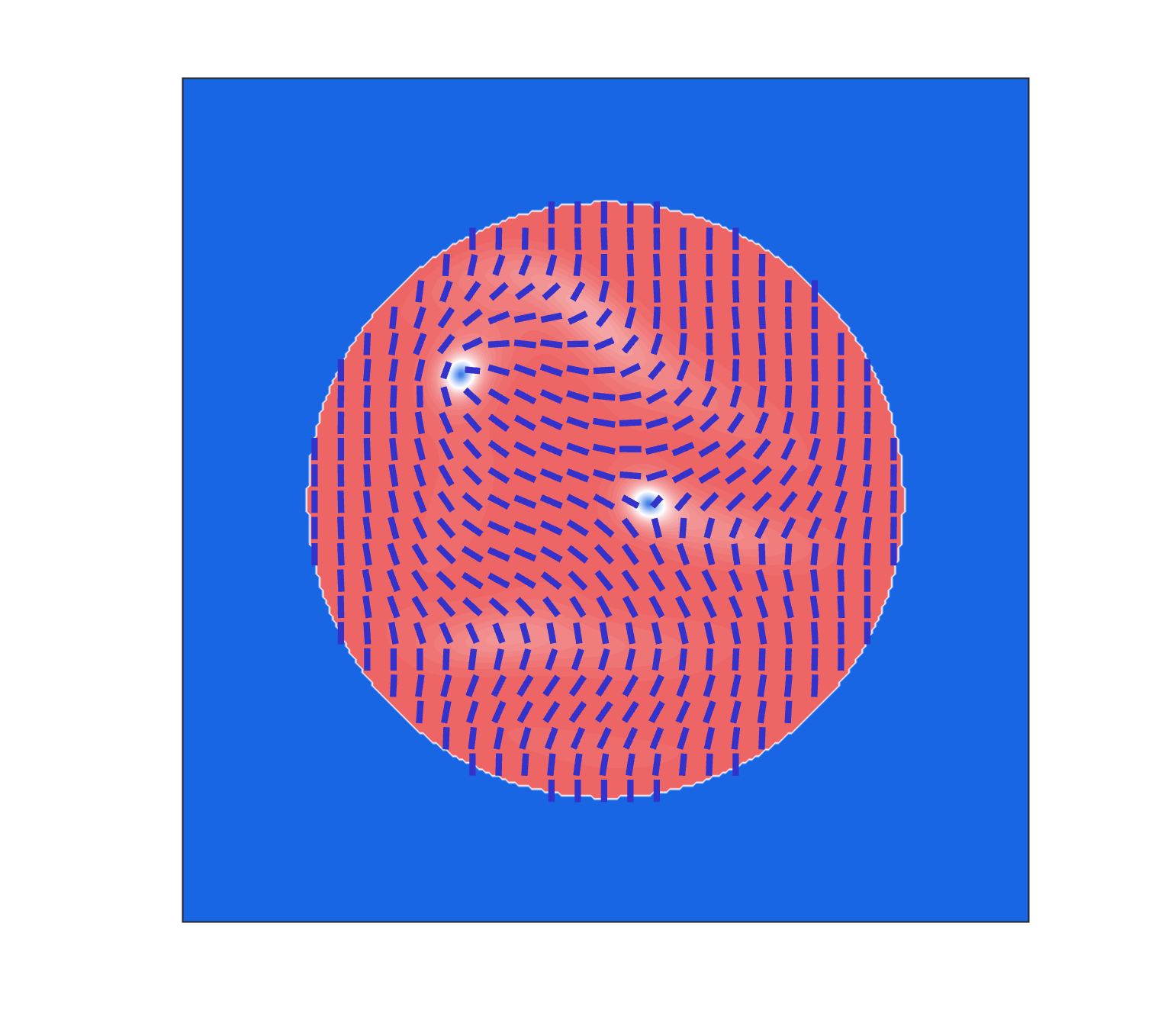}
			\includegraphics[width=0.18\textwidth]{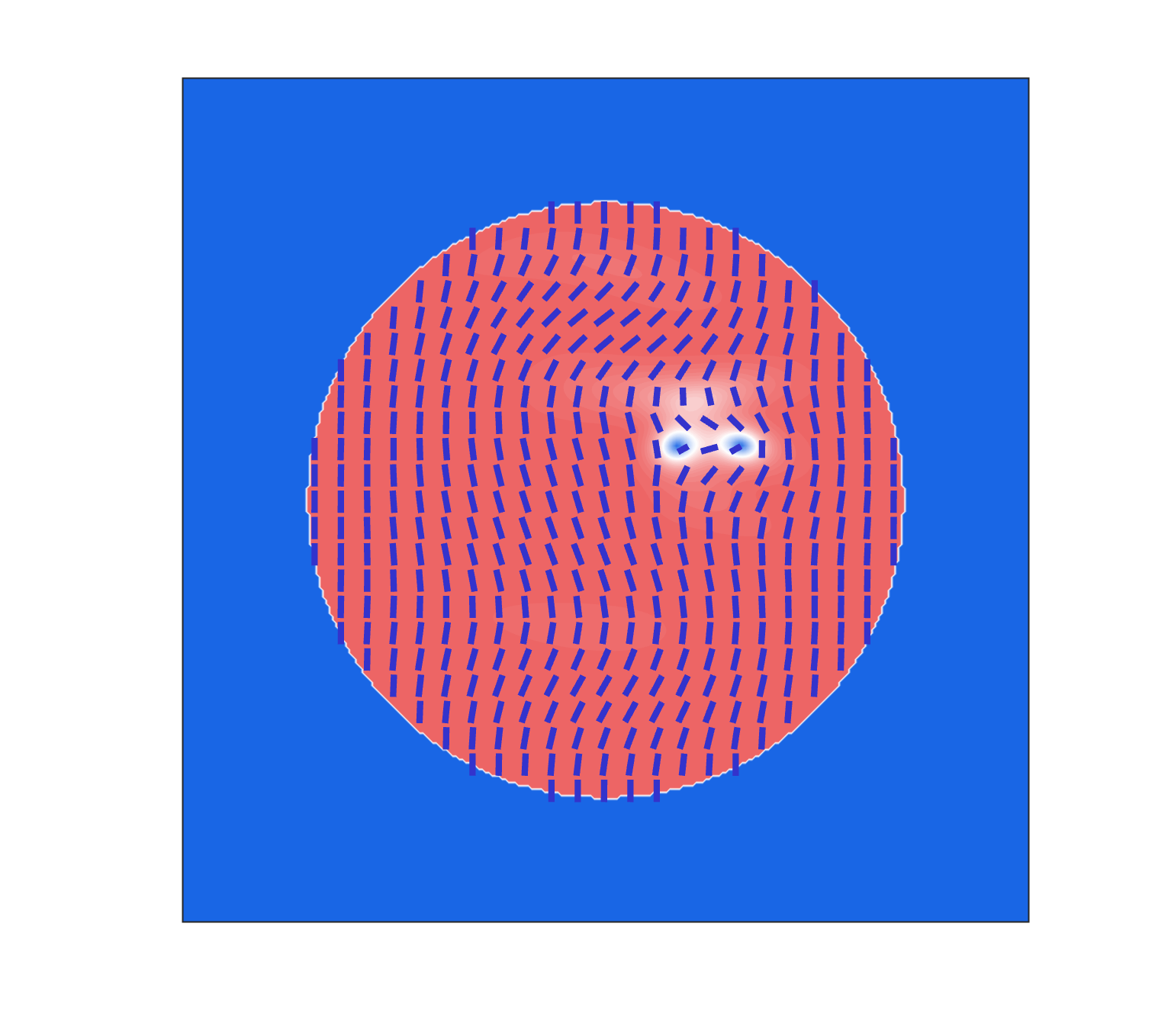}
			\includegraphics[width=0.18\textwidth]{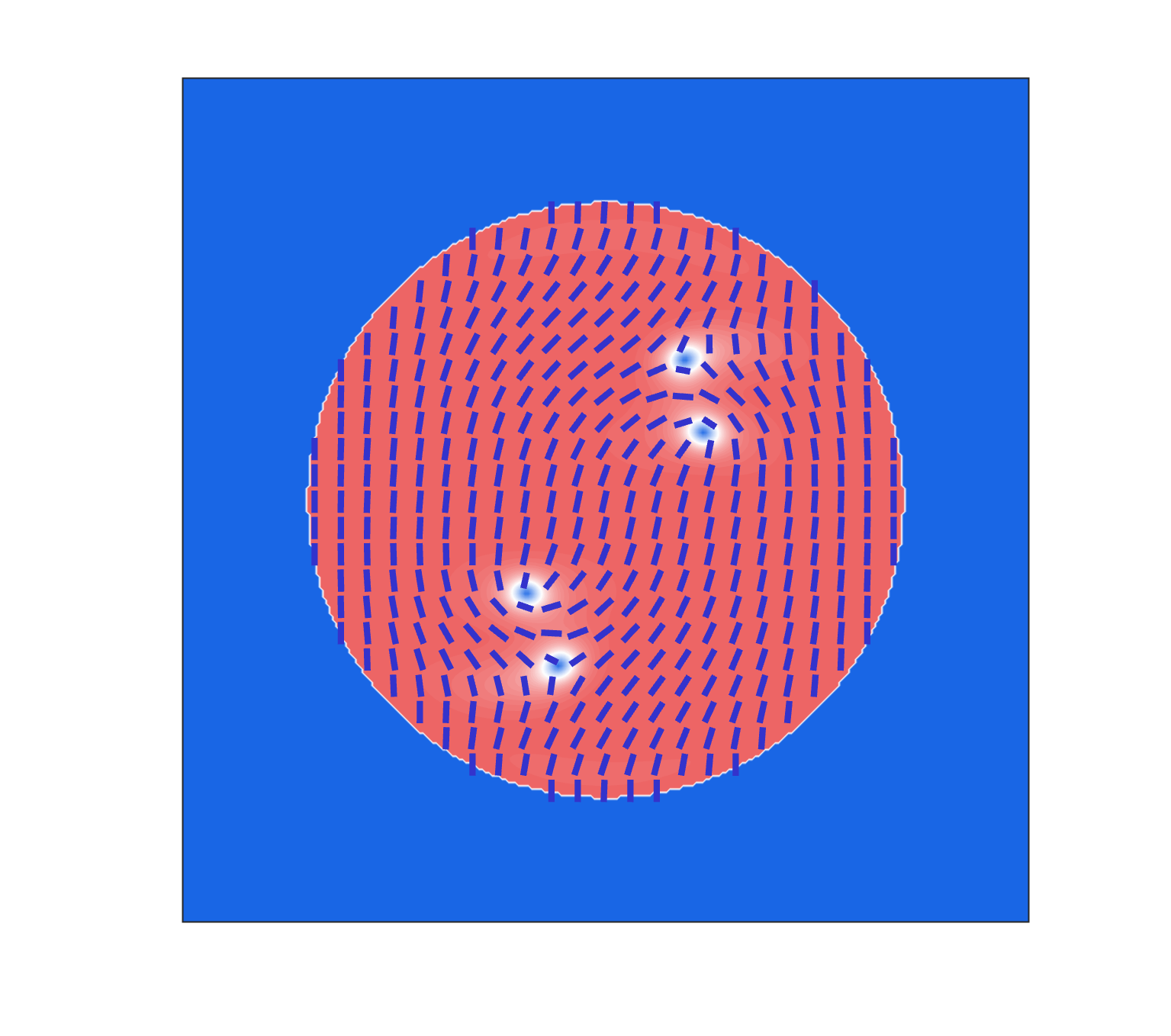}
			\includegraphics[width=0.18\textwidth]{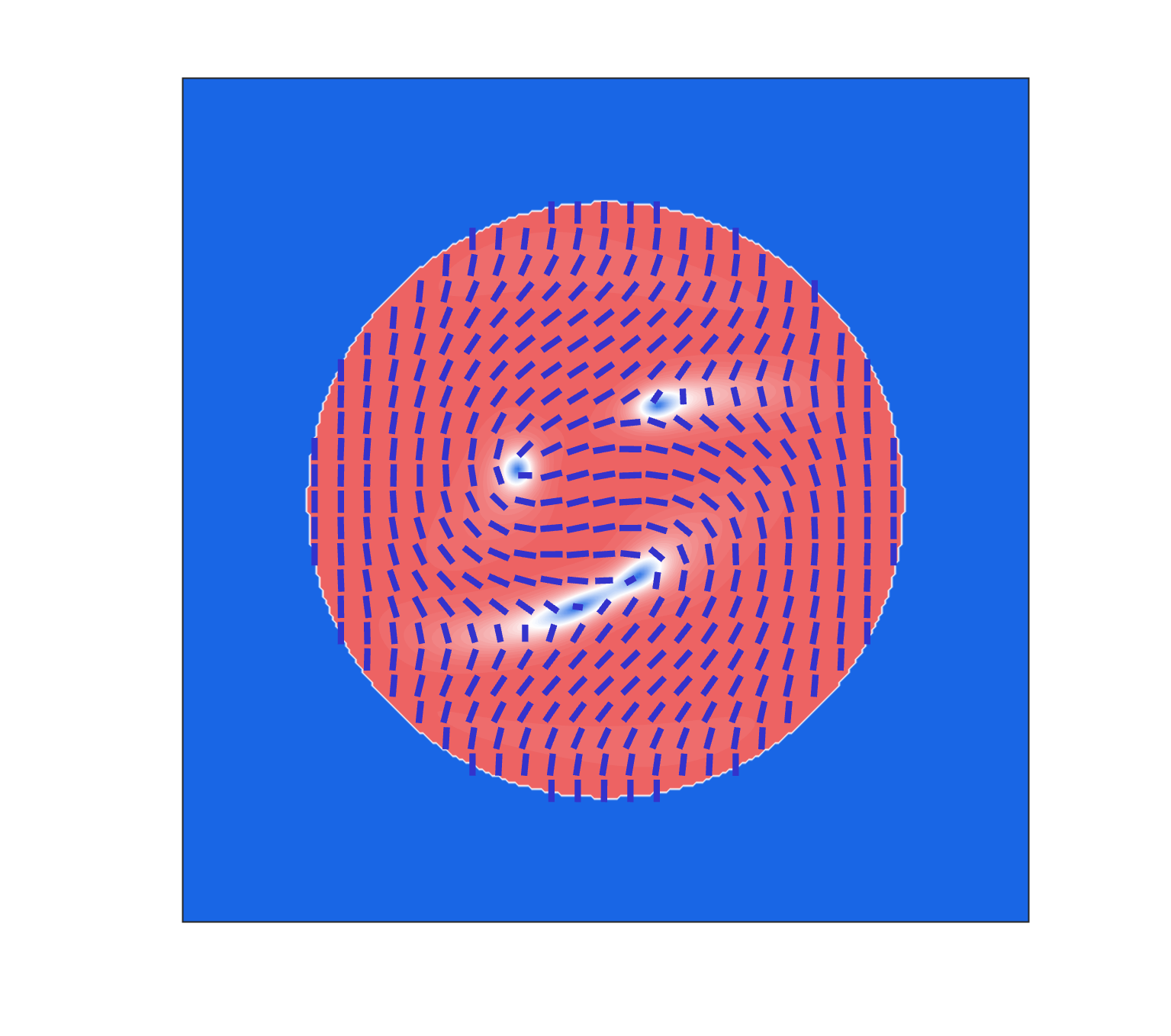}
			\includegraphics[width=0.18\textwidth]{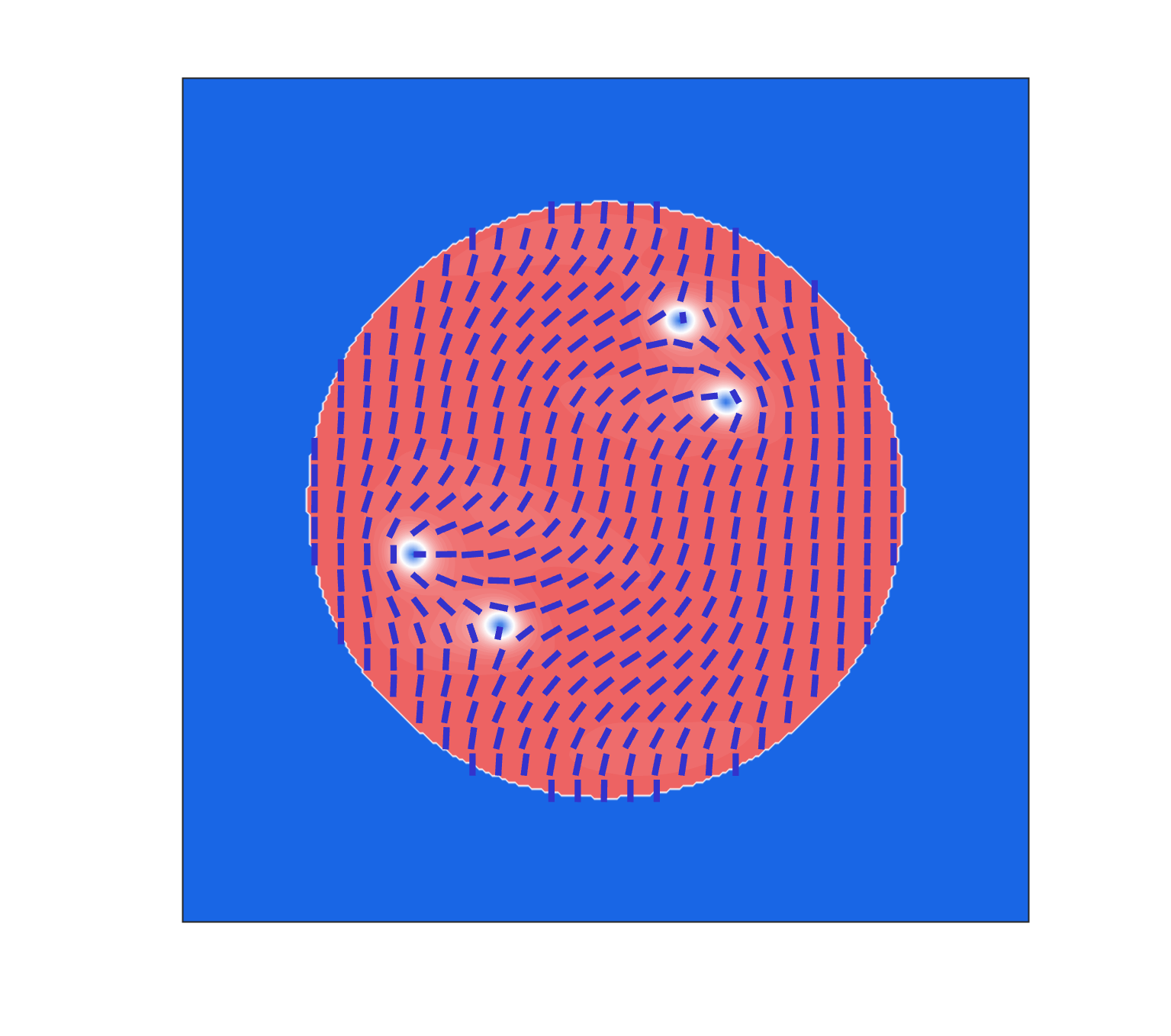}
		\end{minipage}
	}%
	\vspace{1em}
	\fbox{%
		\begin{minipage}{0.96\textwidth}
			\centering
			\textbf{$\chi_{fluid} = -10$, $\xi_{fluid} = 0.1$} \\[0.5em]
			\includegraphics[width=0.18\textwidth]{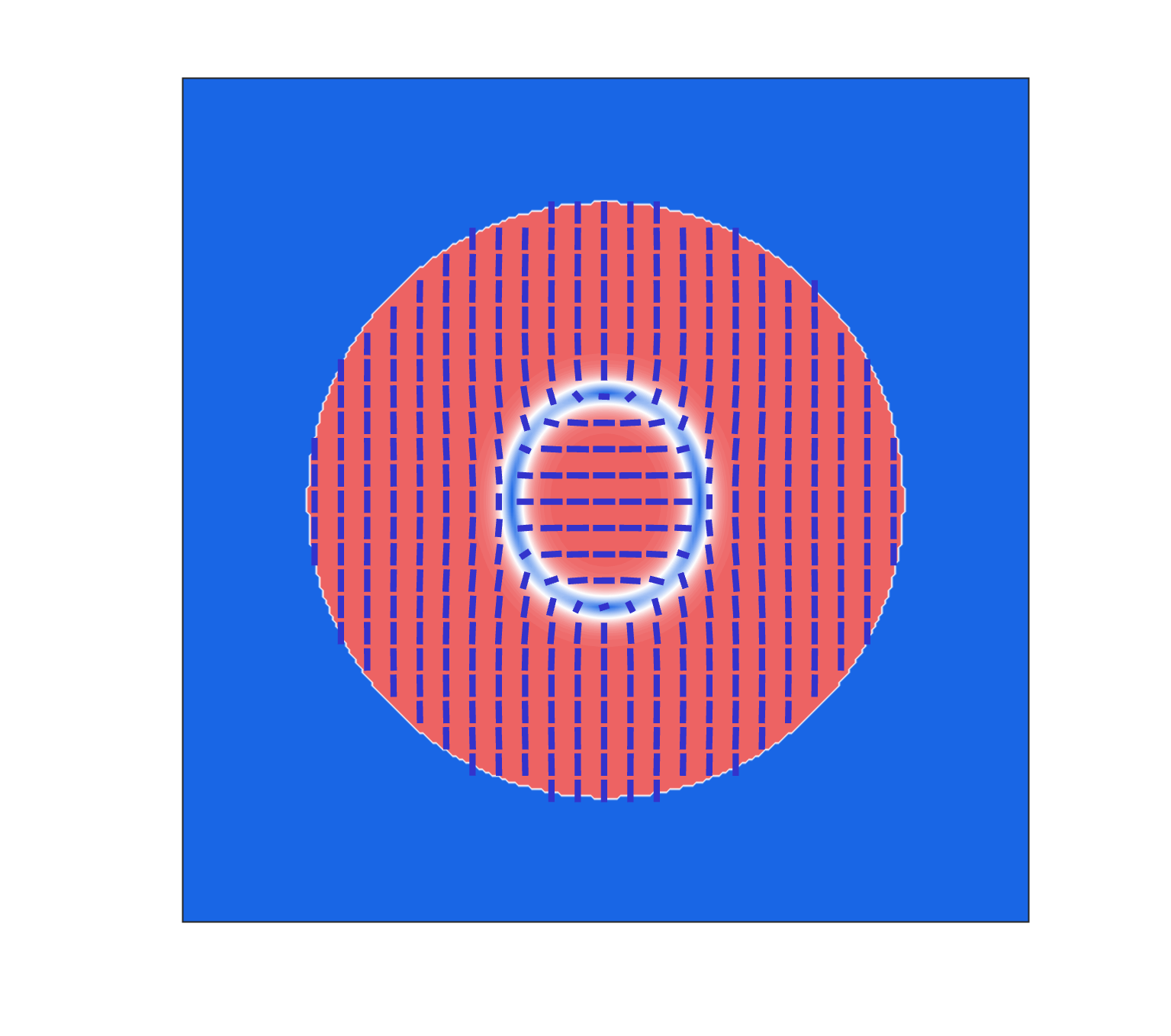}
			\includegraphics[width=0.18\textwidth]{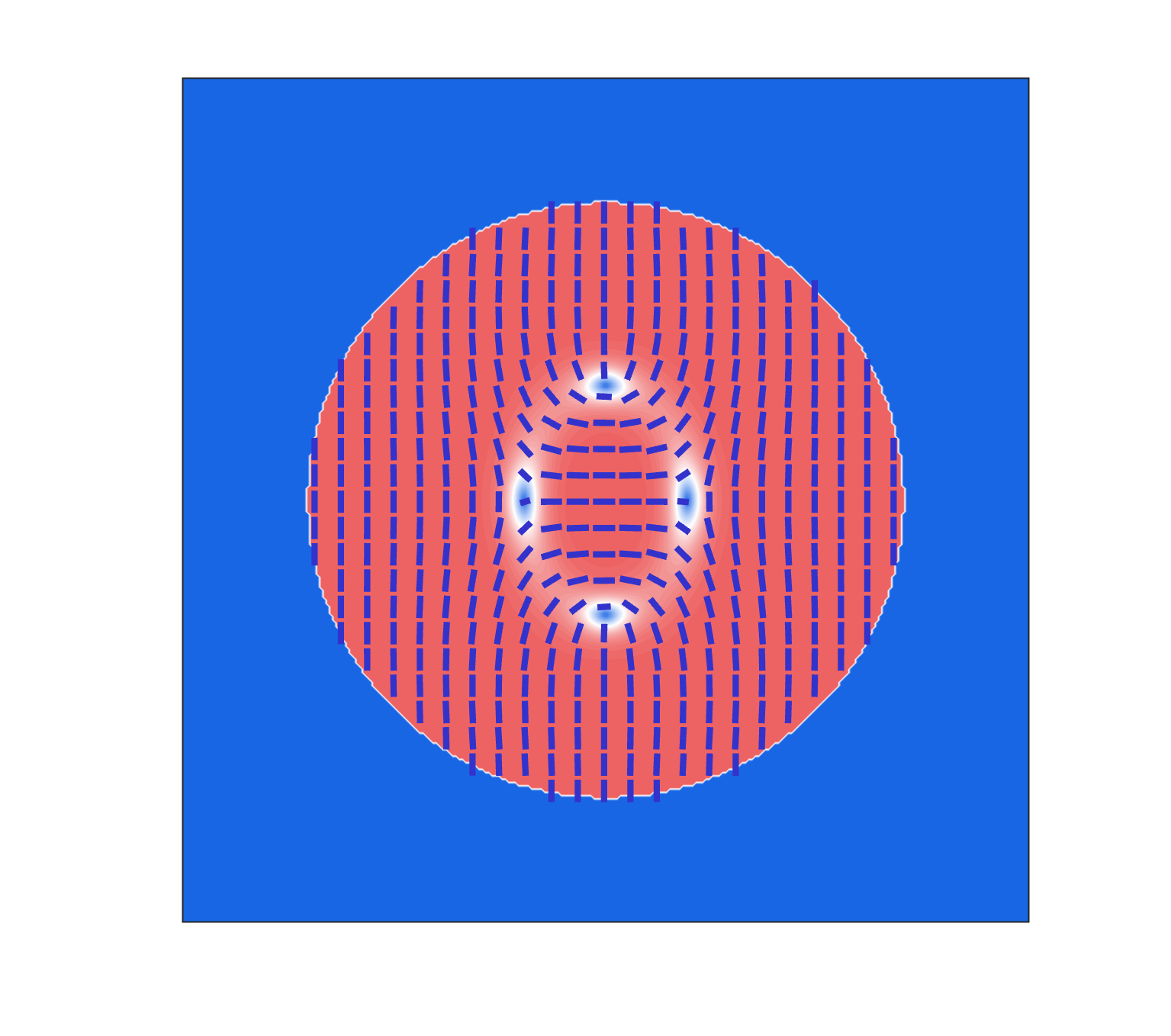}
			\includegraphics[width=0.18\textwidth]{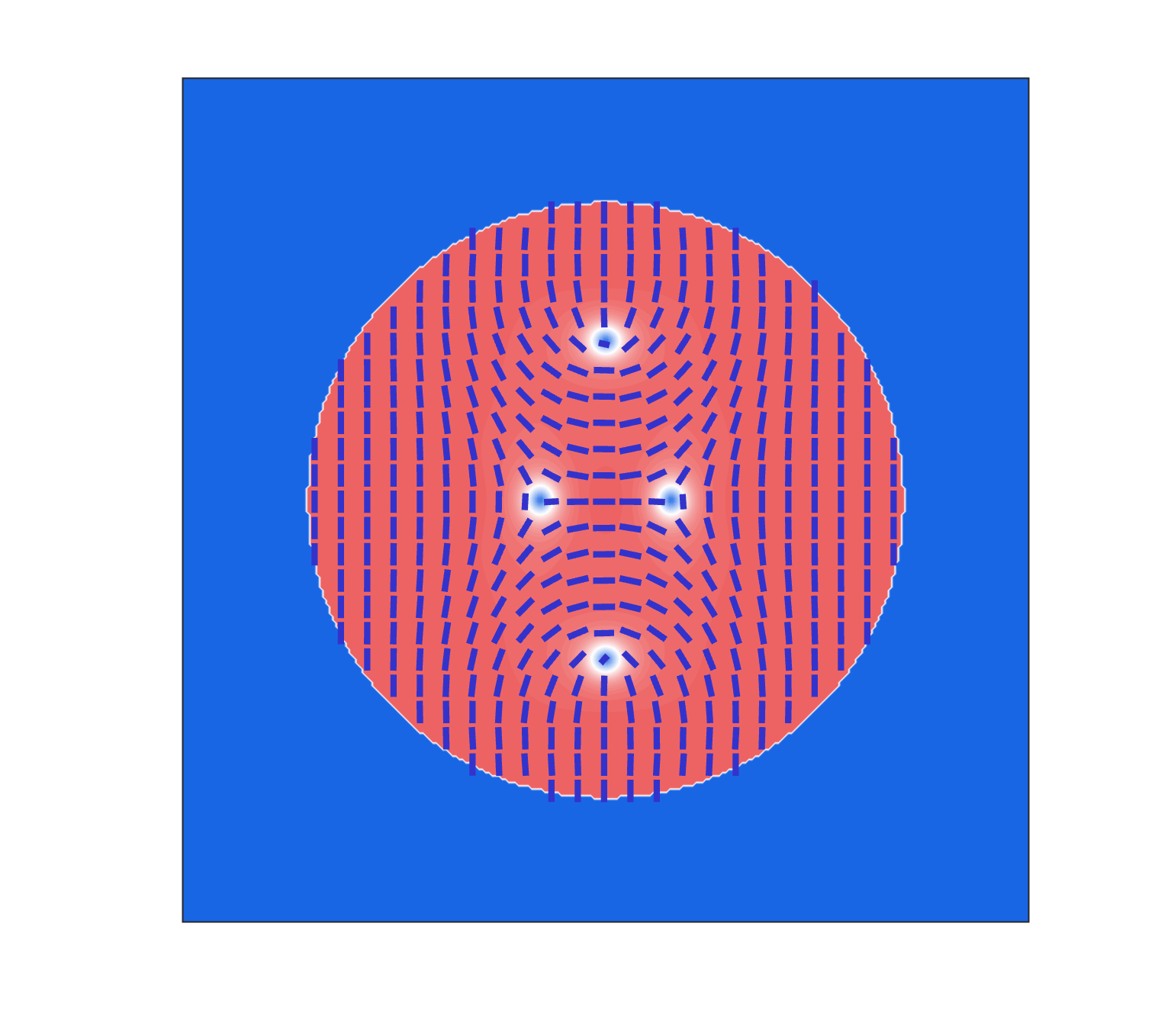}
			\includegraphics[width=0.18\textwidth]{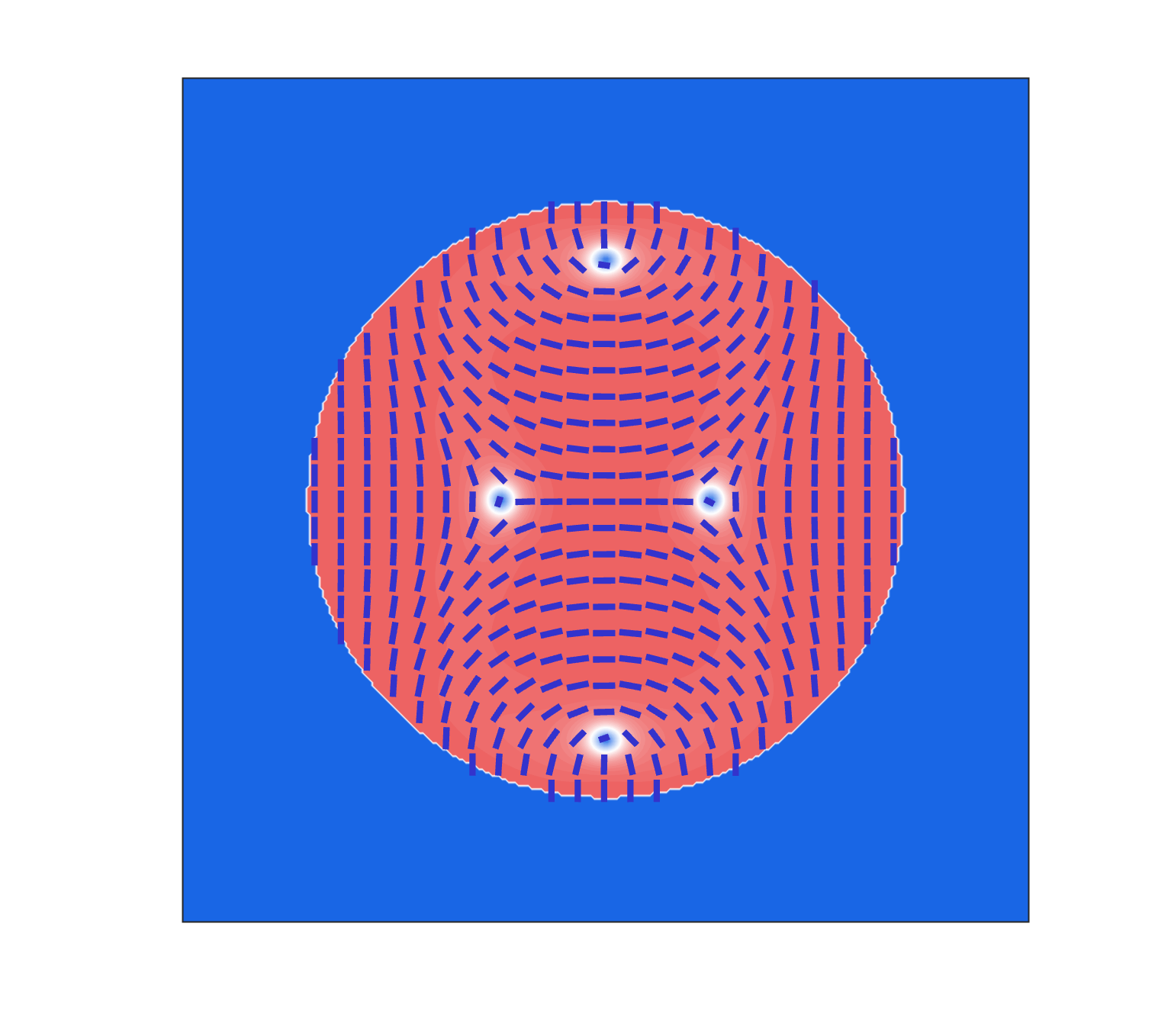}
			\includegraphics[width=0.18\textwidth]{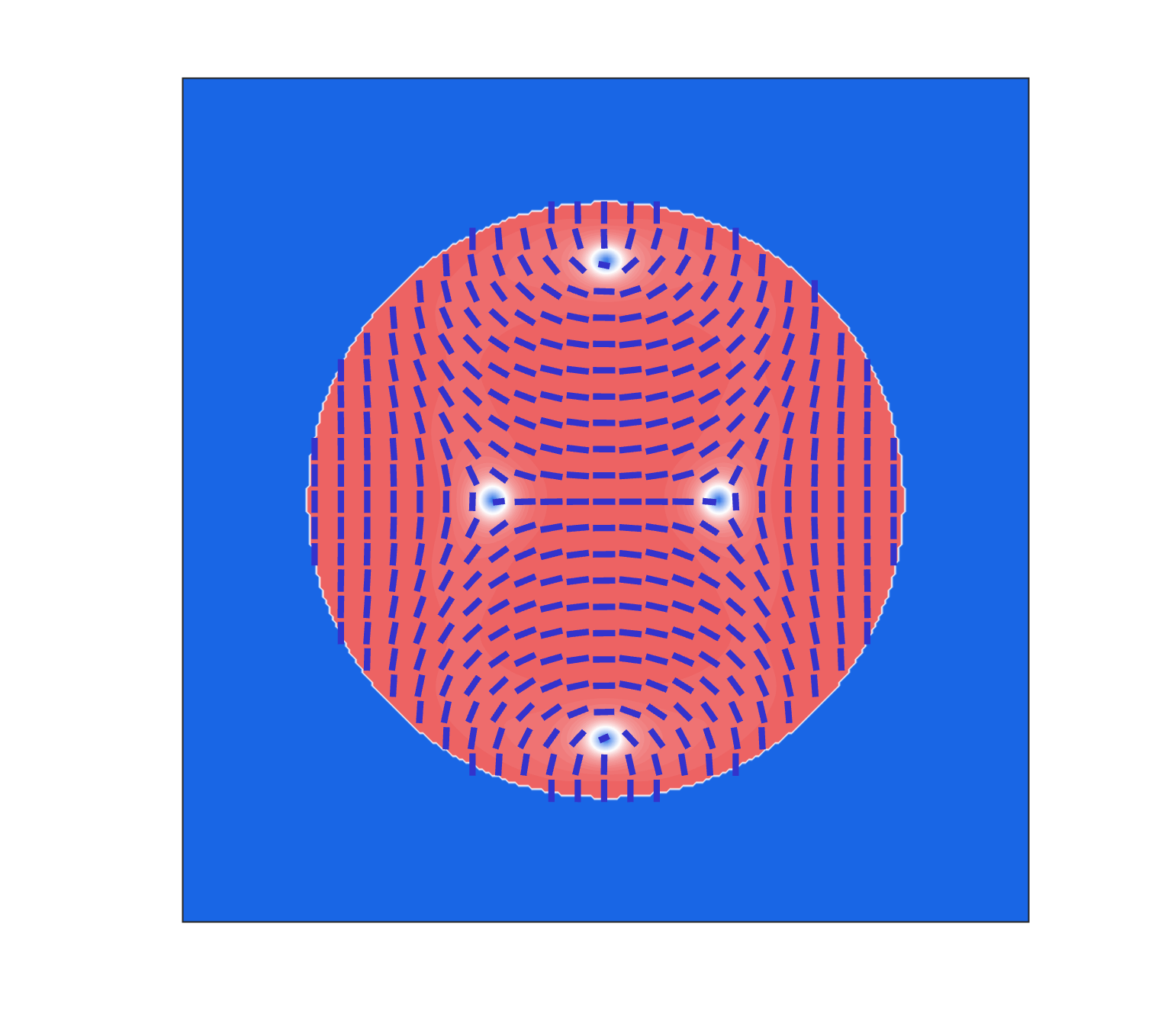}
			\includegraphics[width=0.18\textwidth]{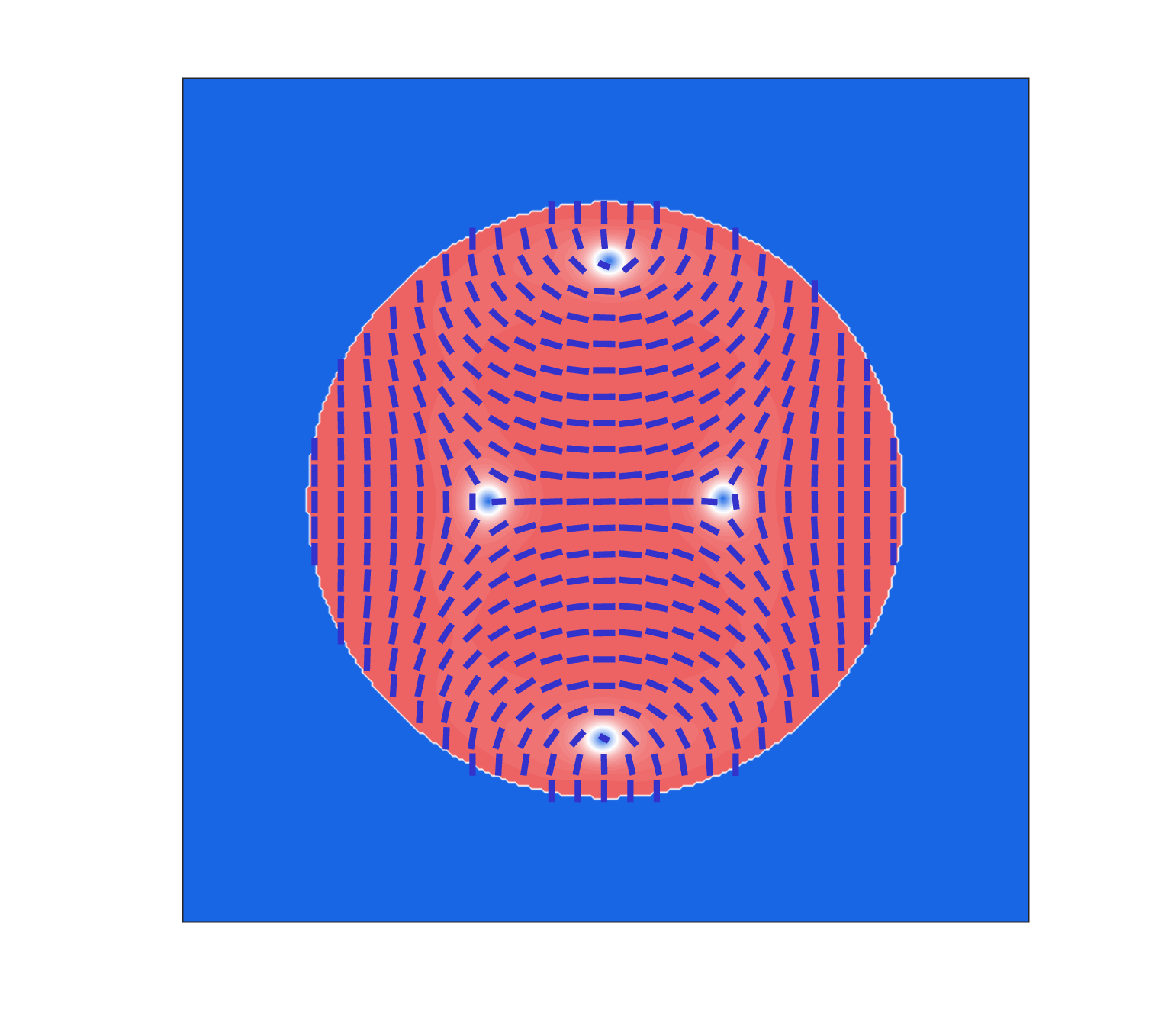}
			\includegraphics[width=0.18\textwidth]{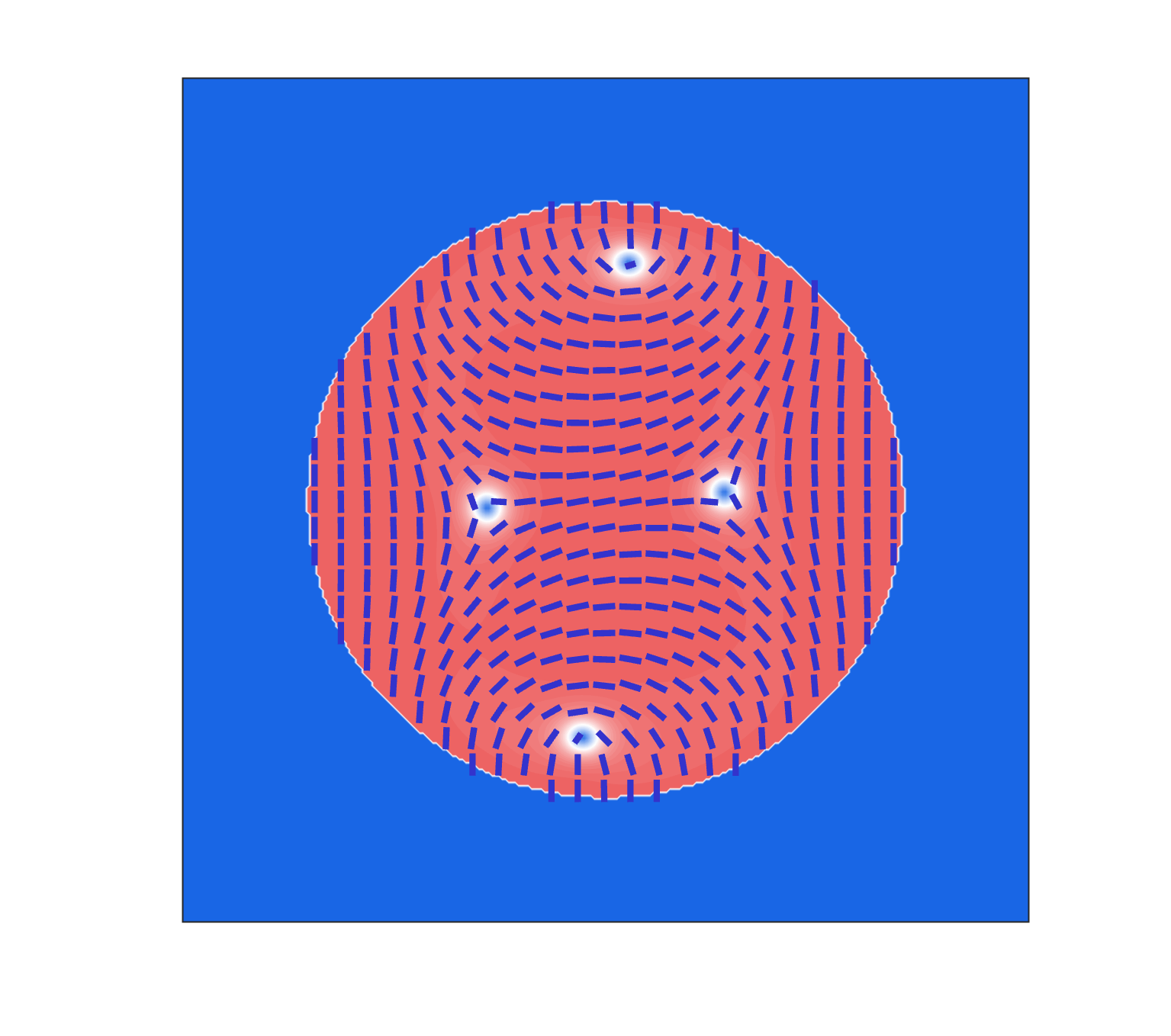}
			\includegraphics[width=0.18\textwidth]{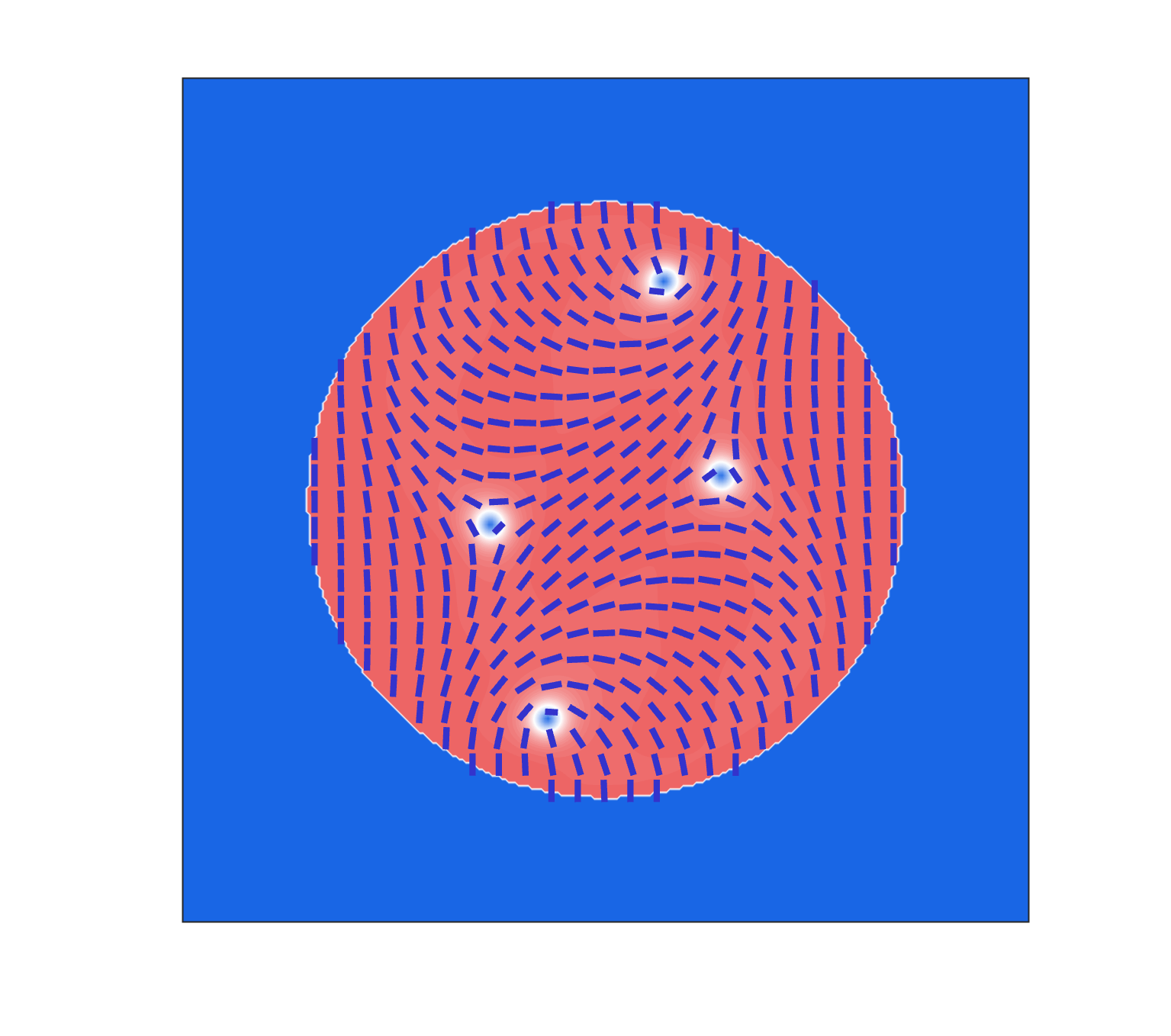}
			\includegraphics[width=0.18\textwidth]{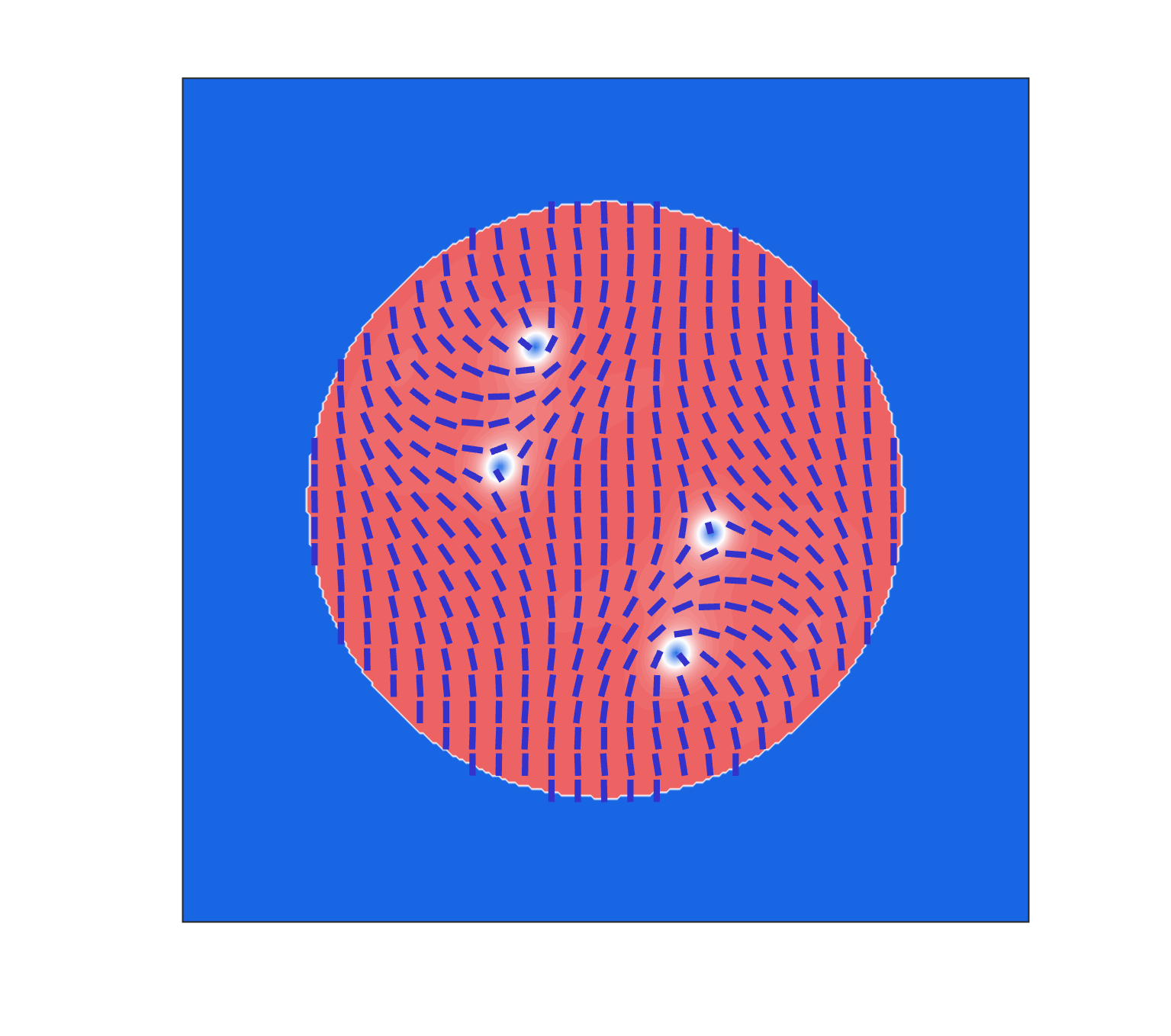}
			\includegraphics[width=0.18\textwidth]{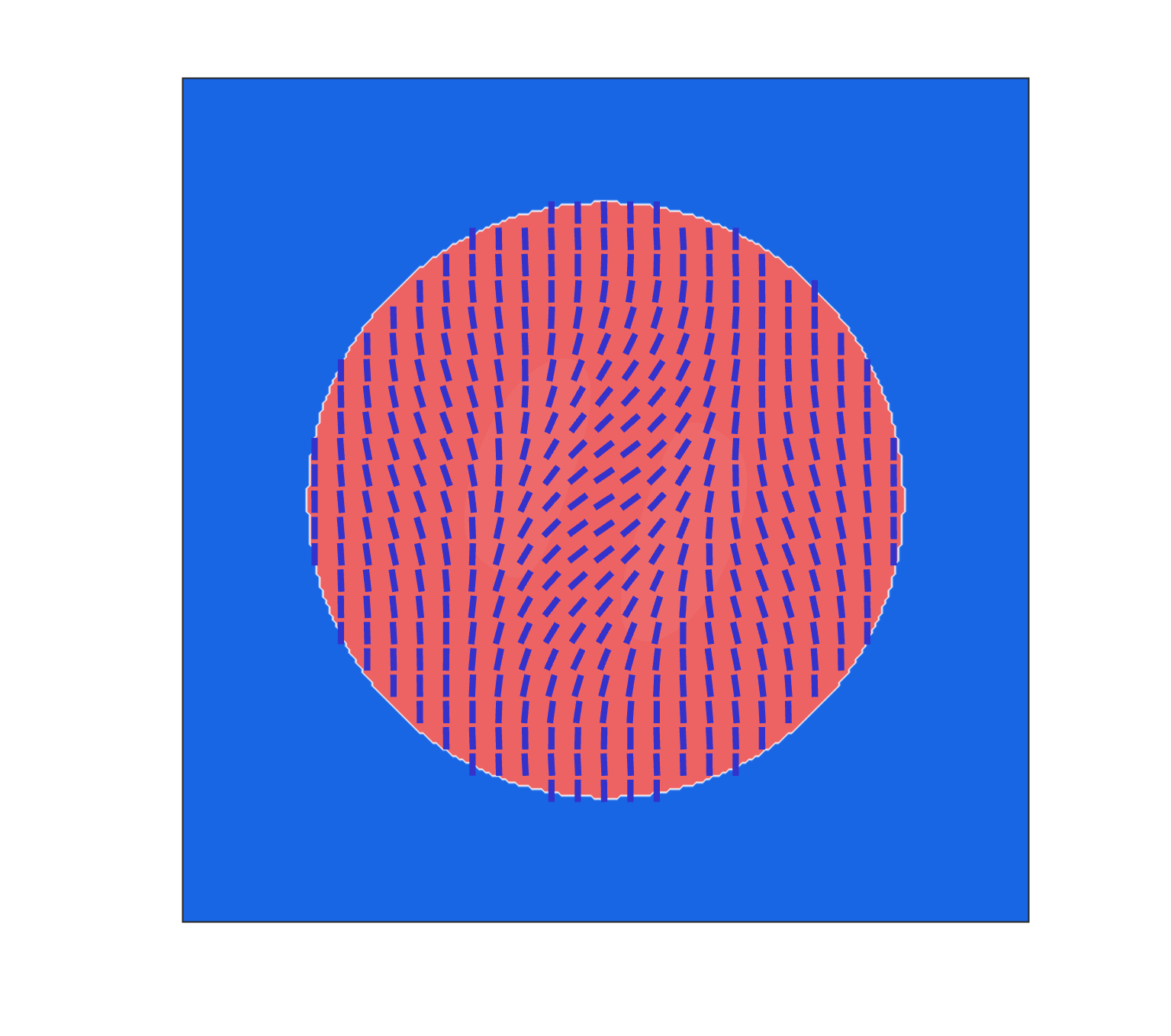}
		\end{minipage}
	}

	\caption{Comparison of defect dynamics for active liquid crystals
		with opposite signs of the activity in a circular domain. Each
		panel shows the contour of the principal eigenvalue and director
		field at various times. Both panels record the solutions at
		$t=0.1$, $0.2$, $0.5$, $5$, $7$, $10$, $13$, $15$, $17$, $20$,
		respectively.}
	\label{fig:incircle-compare}
\end{figure}

The dynamics of defects exhibit notable differences depending on
the shape of the  obstacle. In the circular obstacle case, the two
$+\tfrac{1}{2}$ defects initially move apart due to active forces.
Subsequently, their motion becomes irregular as a result of
interactions among defects, boundary anchoring, and strong active
stresses. Due to high activities, the system does not approach a
steady state in this scenario. As a result, we see dynamic defects
form and go.

\begin{figure}[H]
	\centering

	\fbox{%
		\begin{minipage}{0.96\textwidth}
			\centering
			\textbf{$\chi_{fluid} = 10$, $\xi_{fluid} = -0.1$} \\[0.5em]
			\includegraphics[width=0.15\textwidth]{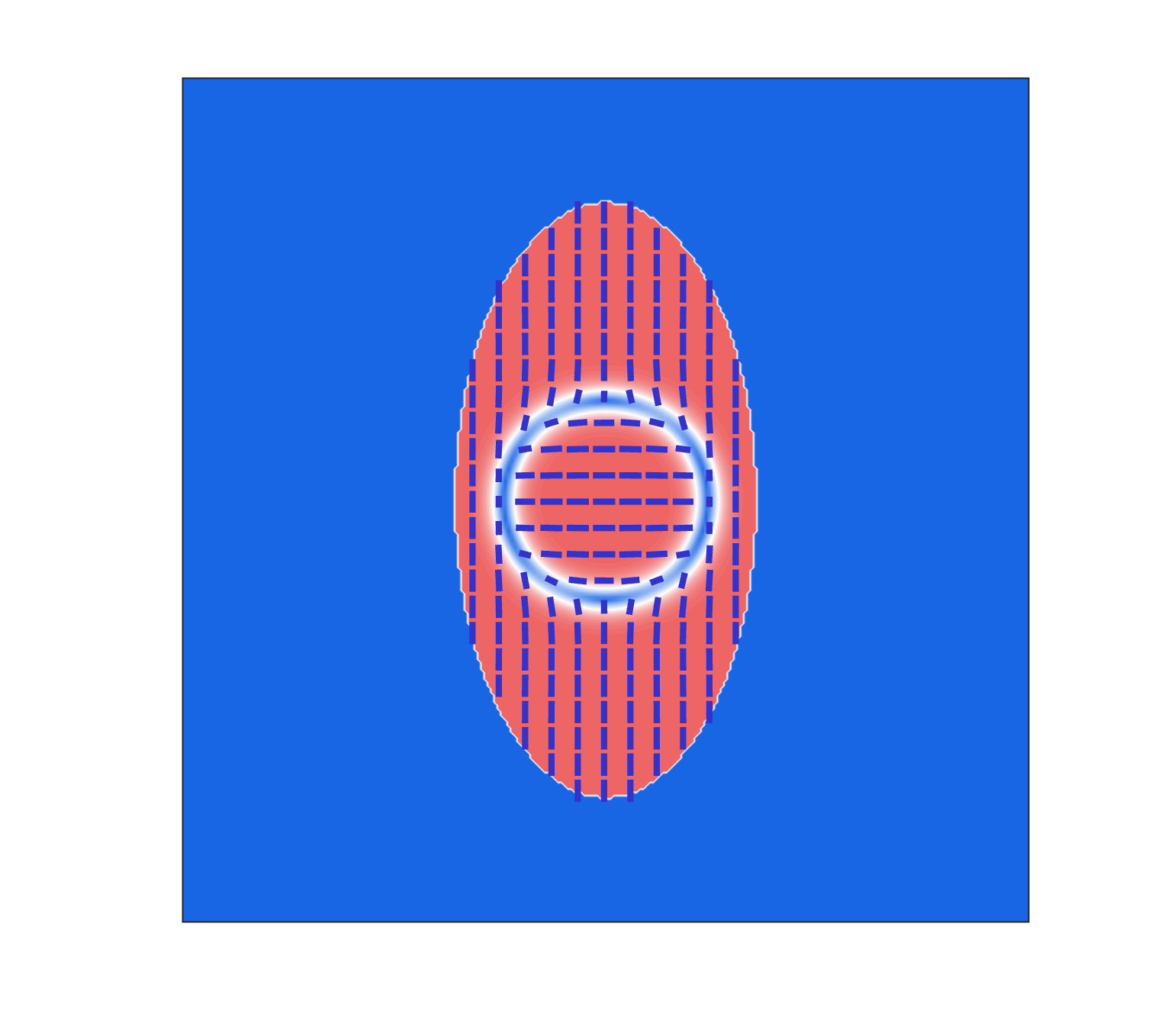}
			\includegraphics[width=0.15\textwidth]{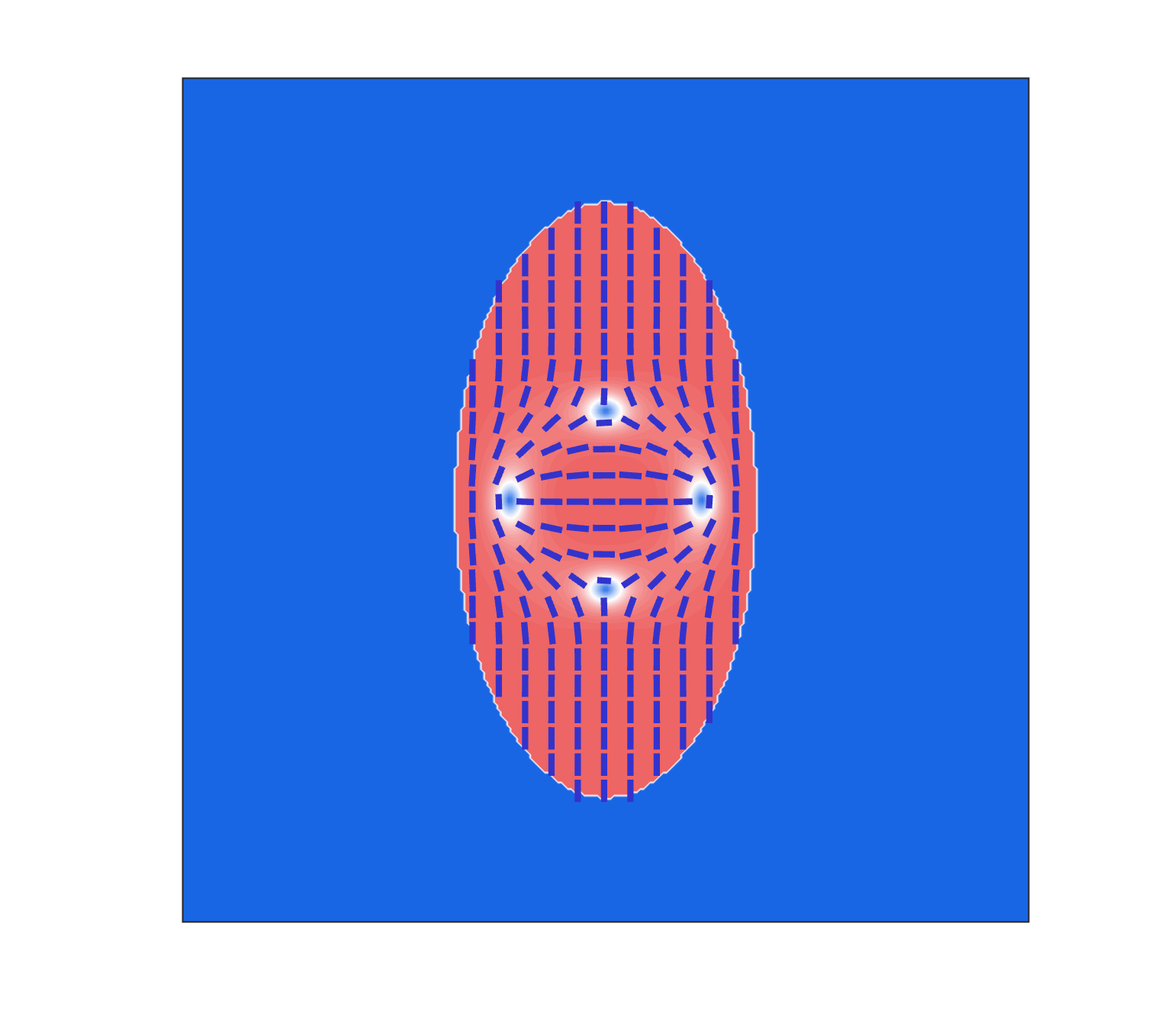}
			\includegraphics[width=0.15\textwidth]{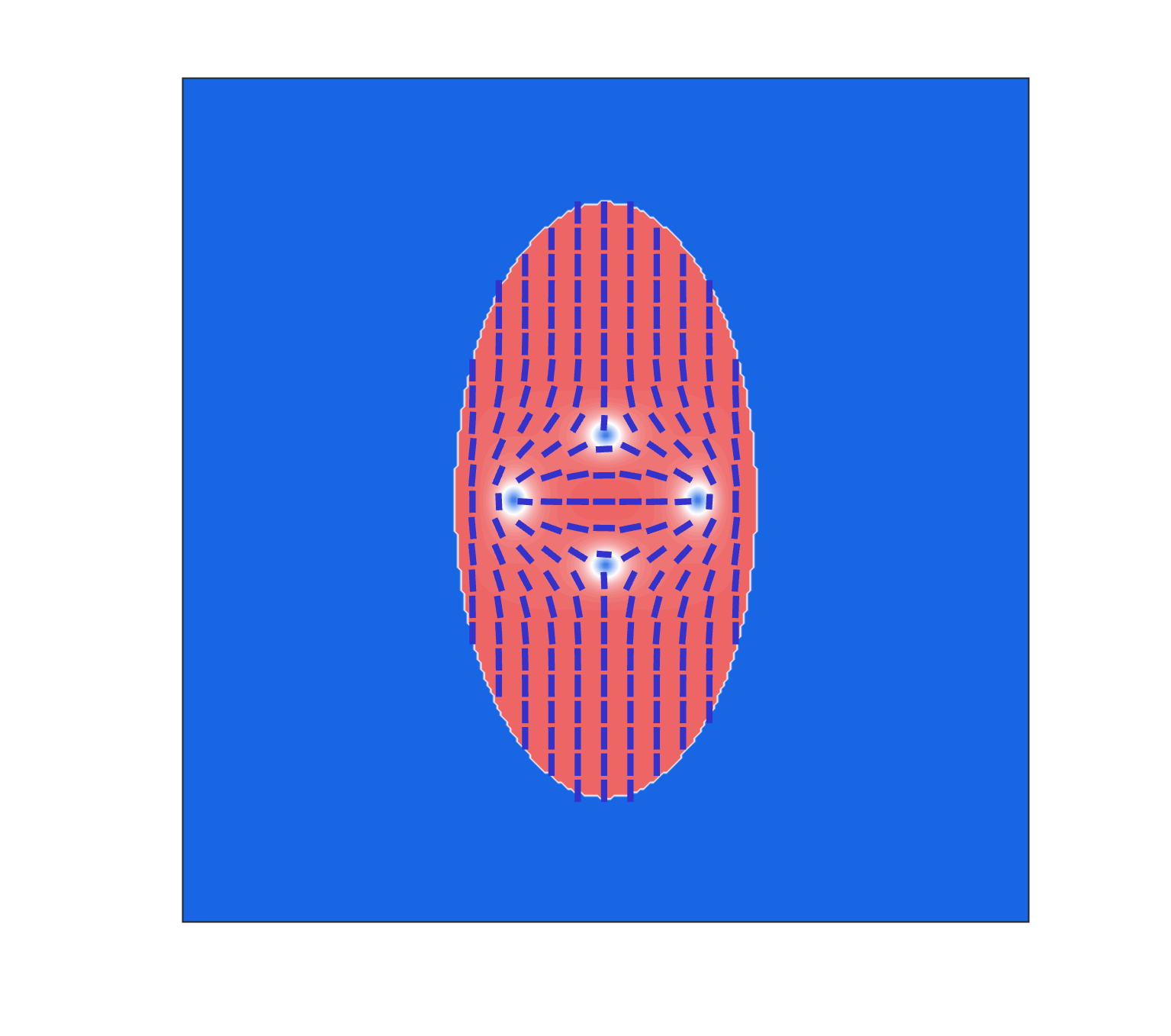}
			\includegraphics[width=0.15\textwidth]{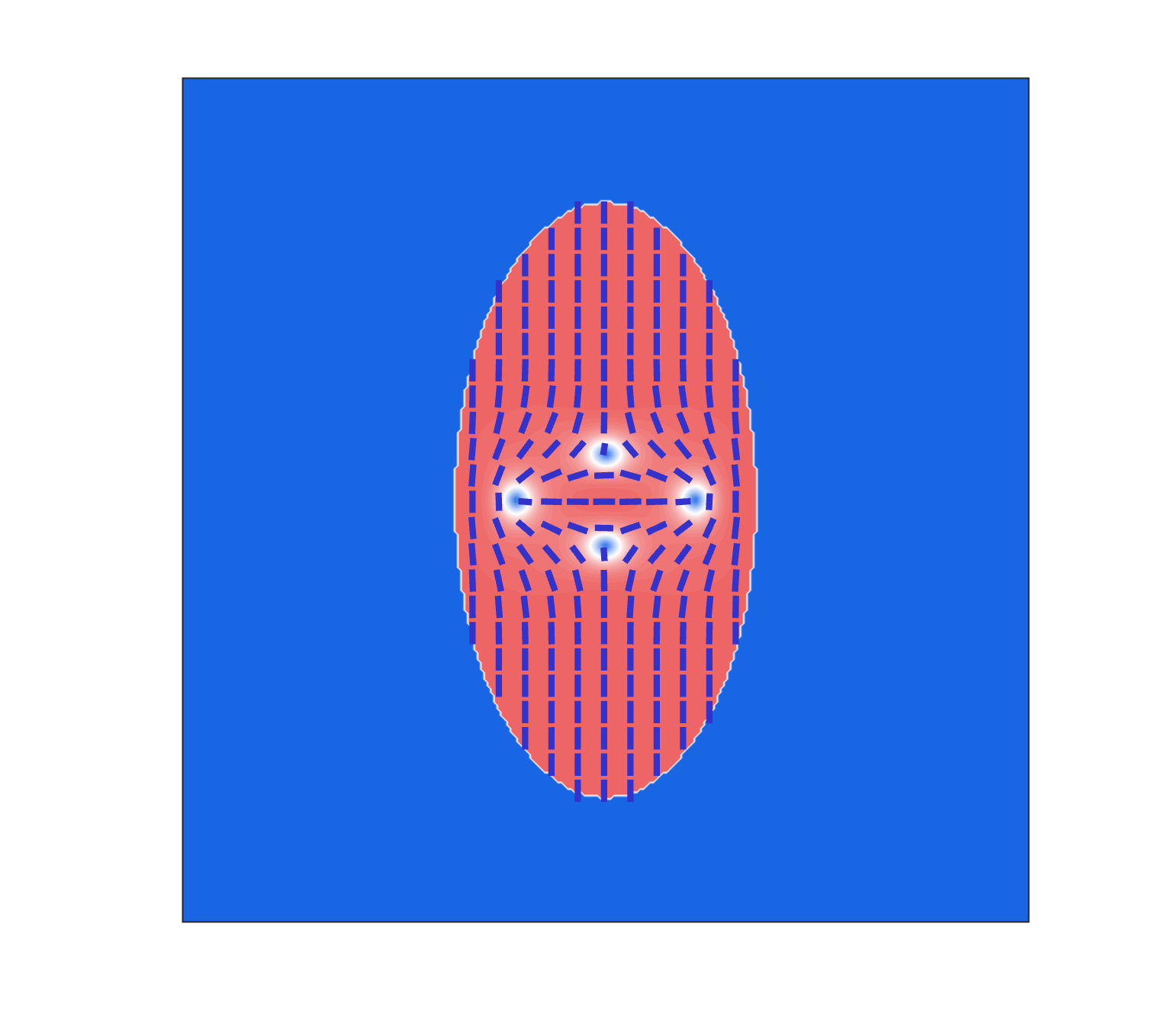}
			\includegraphics[width=0.15\textwidth]{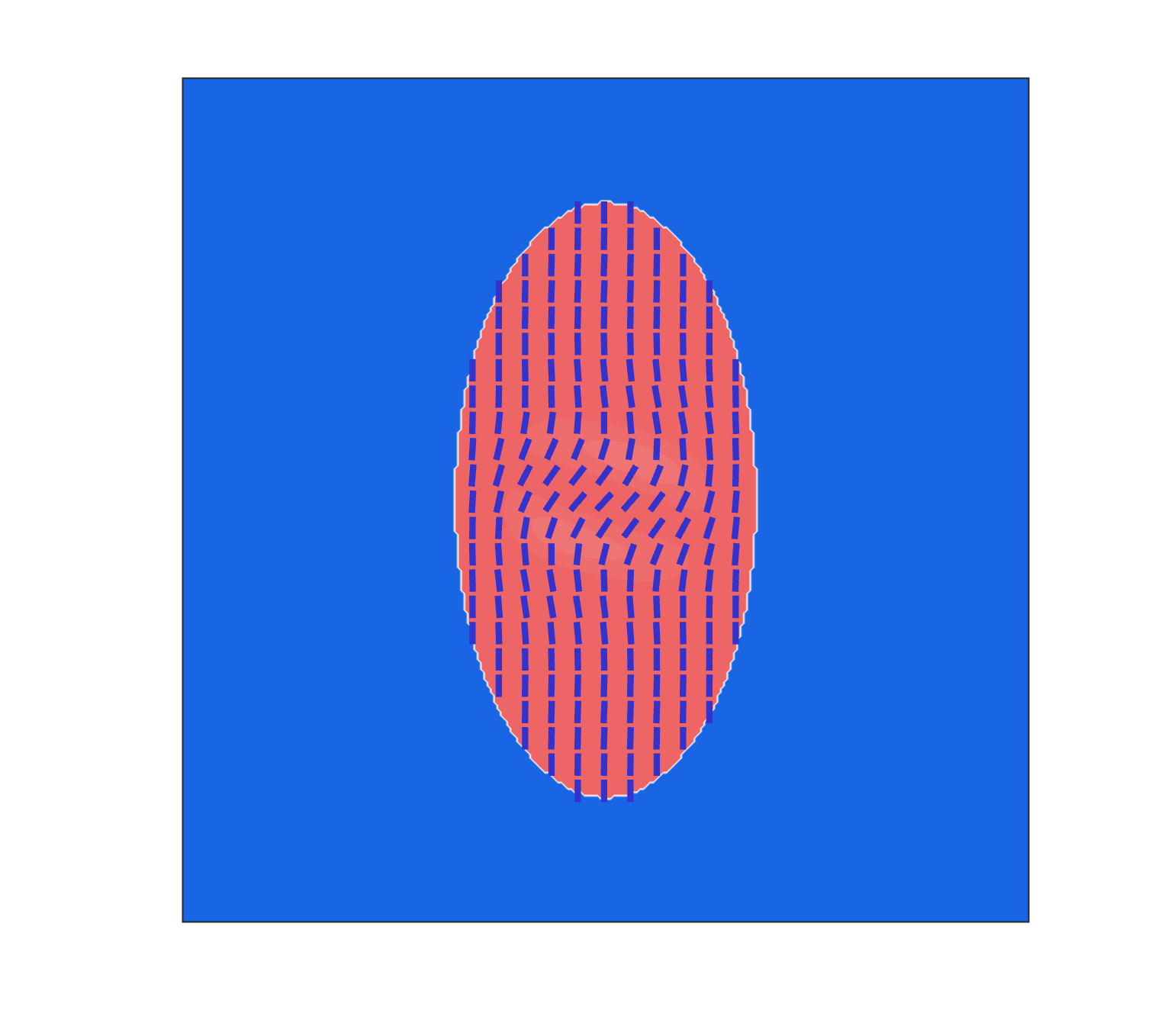}
			\includegraphics[width=0.15\textwidth]{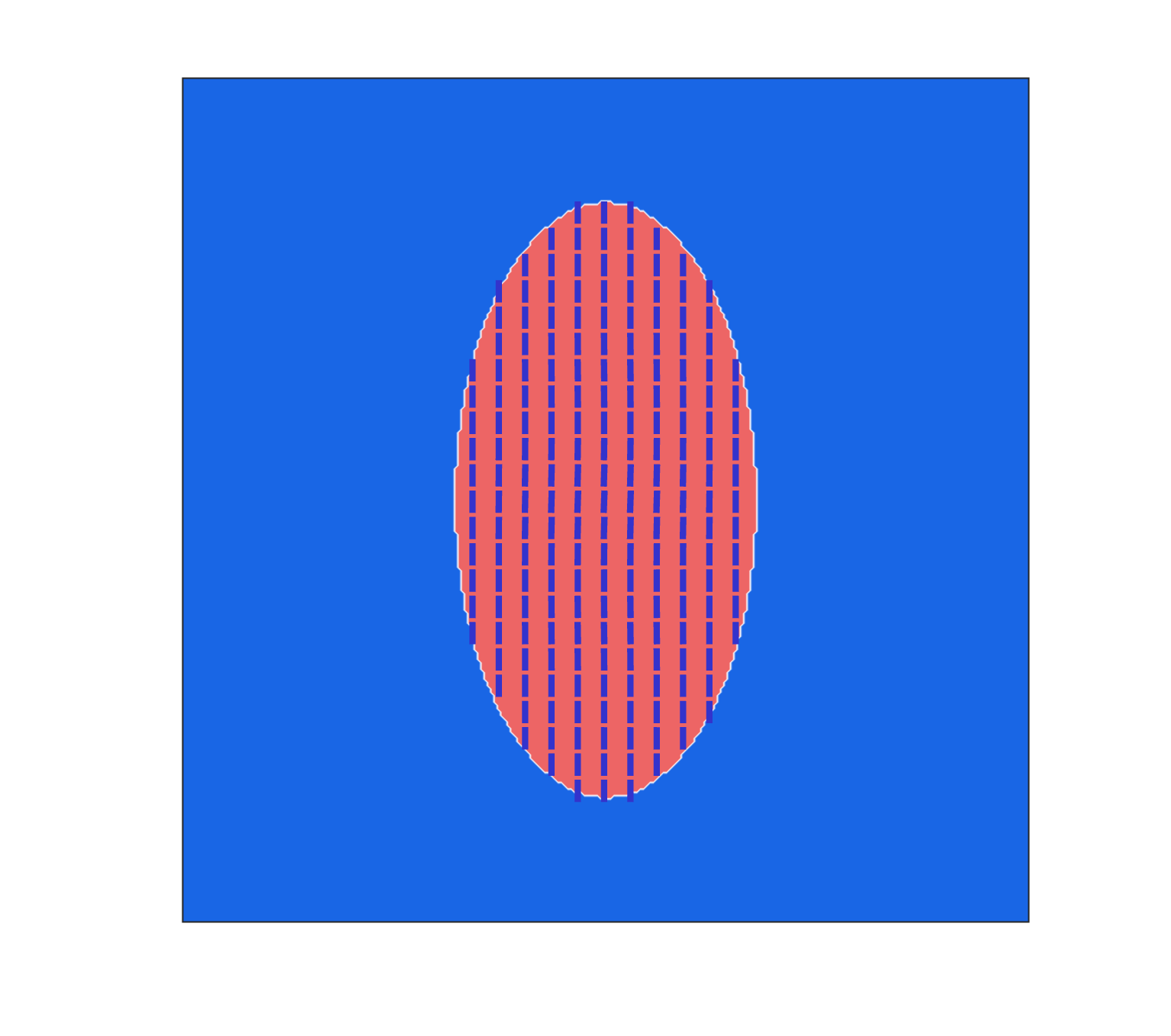}
		\end{minipage}
	}%
	\vspace{1em}
	\fbox{%
		\begin{minipage}{0.96\textwidth}
			\centering
			\textbf{$\chi_{fluid} = -10$, $\xi_{fluid} = 0.1$} \\[0.5em]
			\includegraphics[width=0.15\textwidth]{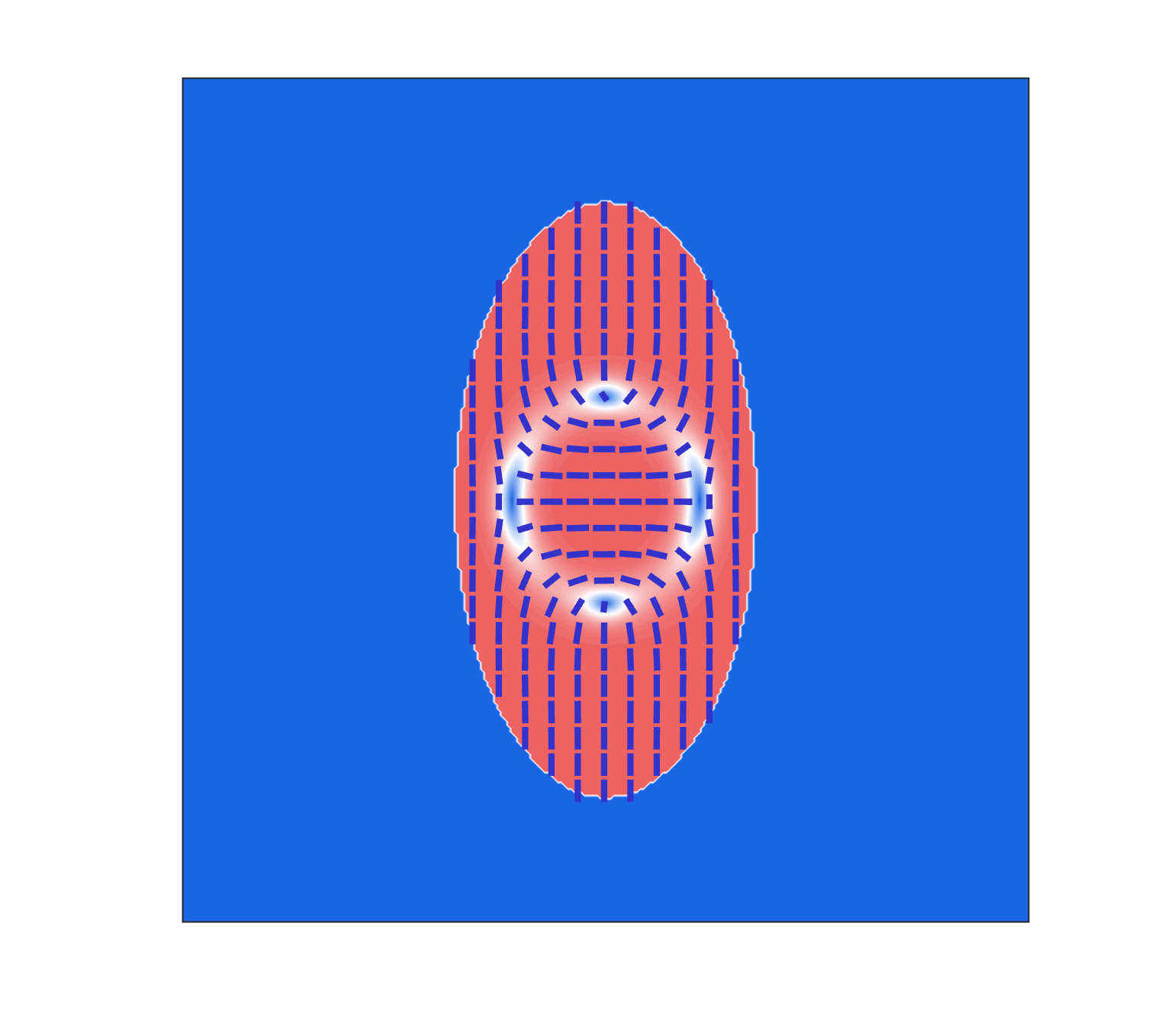}
			\includegraphics[width=0.15\textwidth]{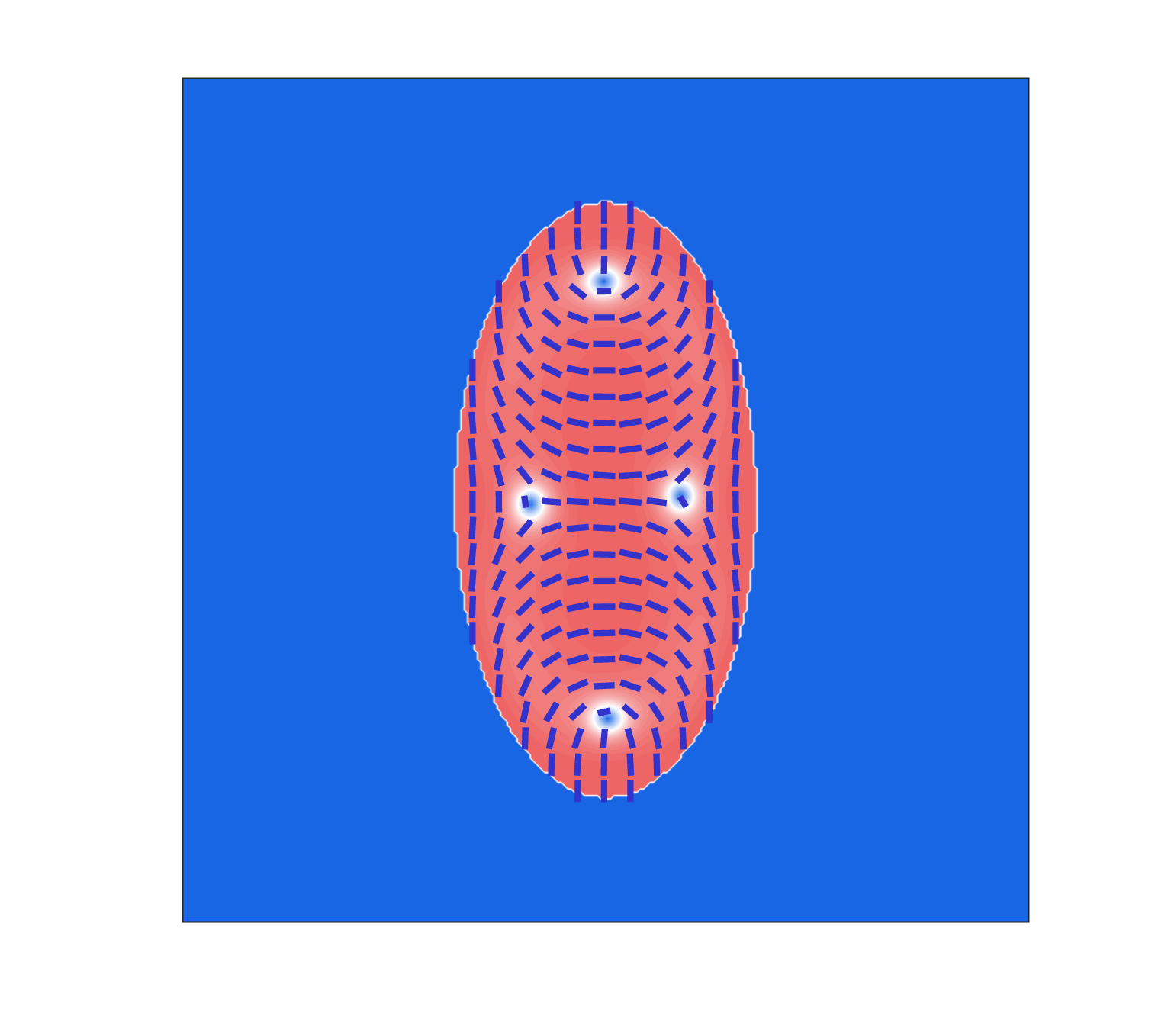}
			\includegraphics[width=0.15\textwidth]{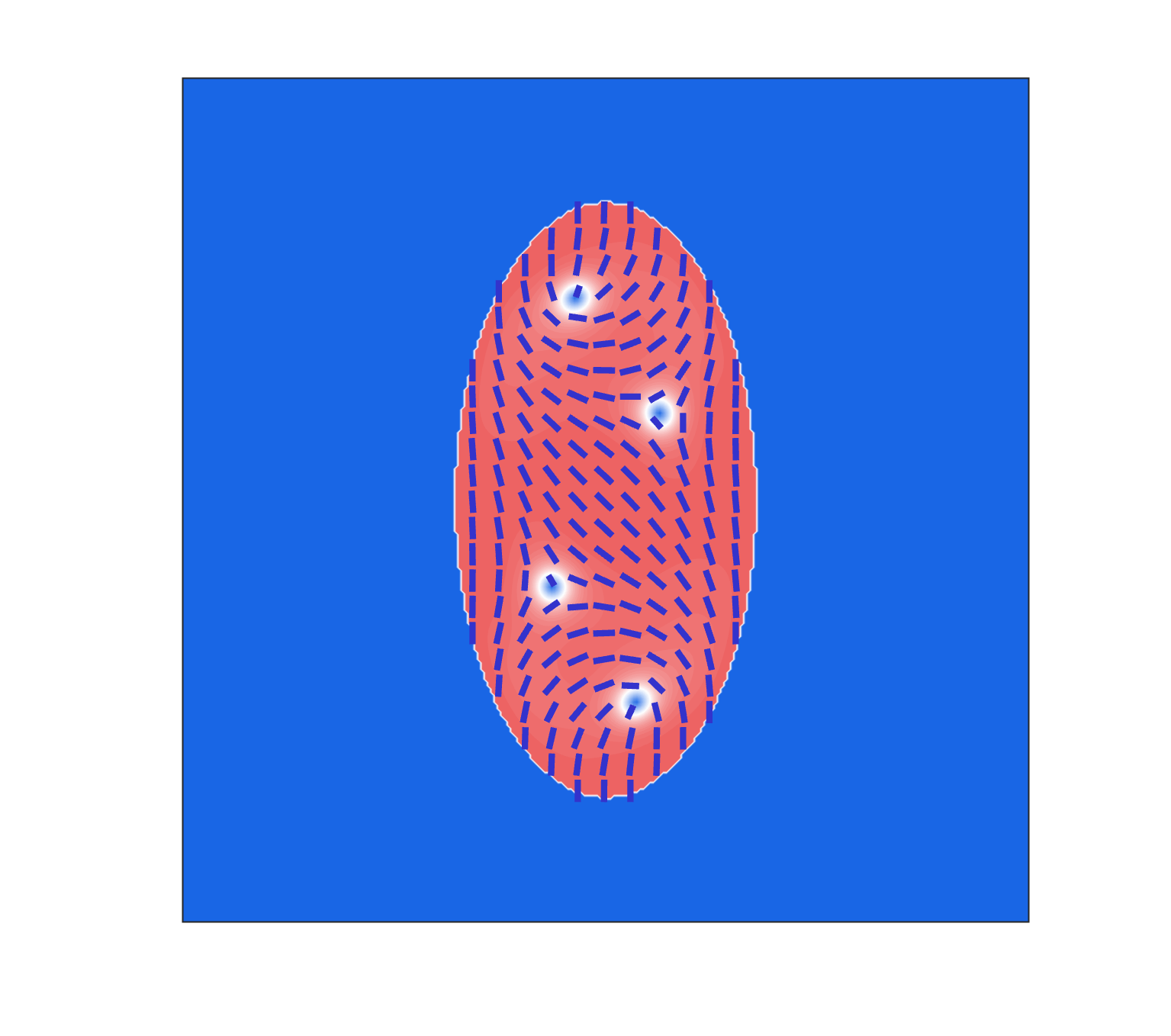}
			\includegraphics[width=0.15\textwidth]{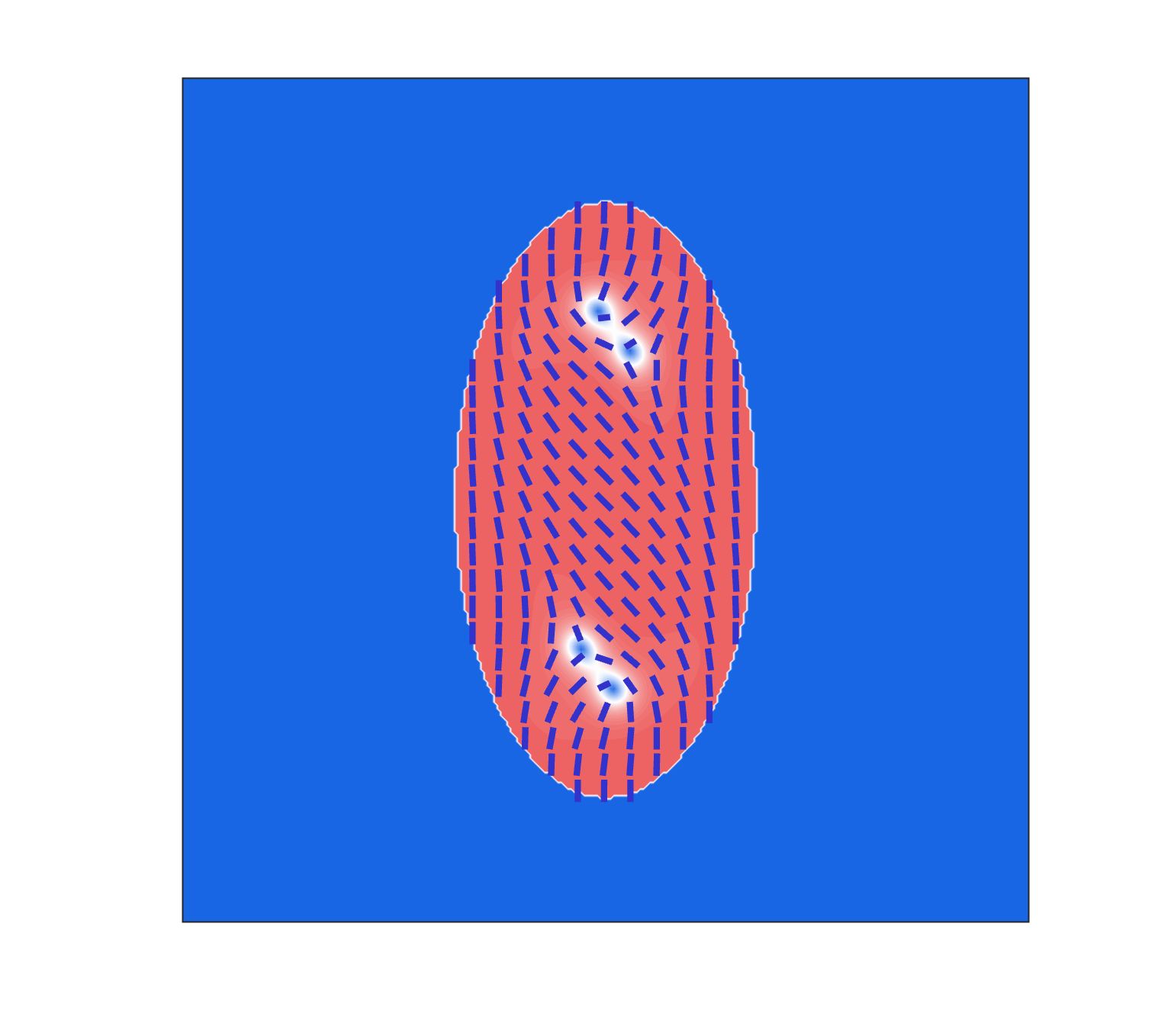}
			\includegraphics[width=0.15\textwidth]{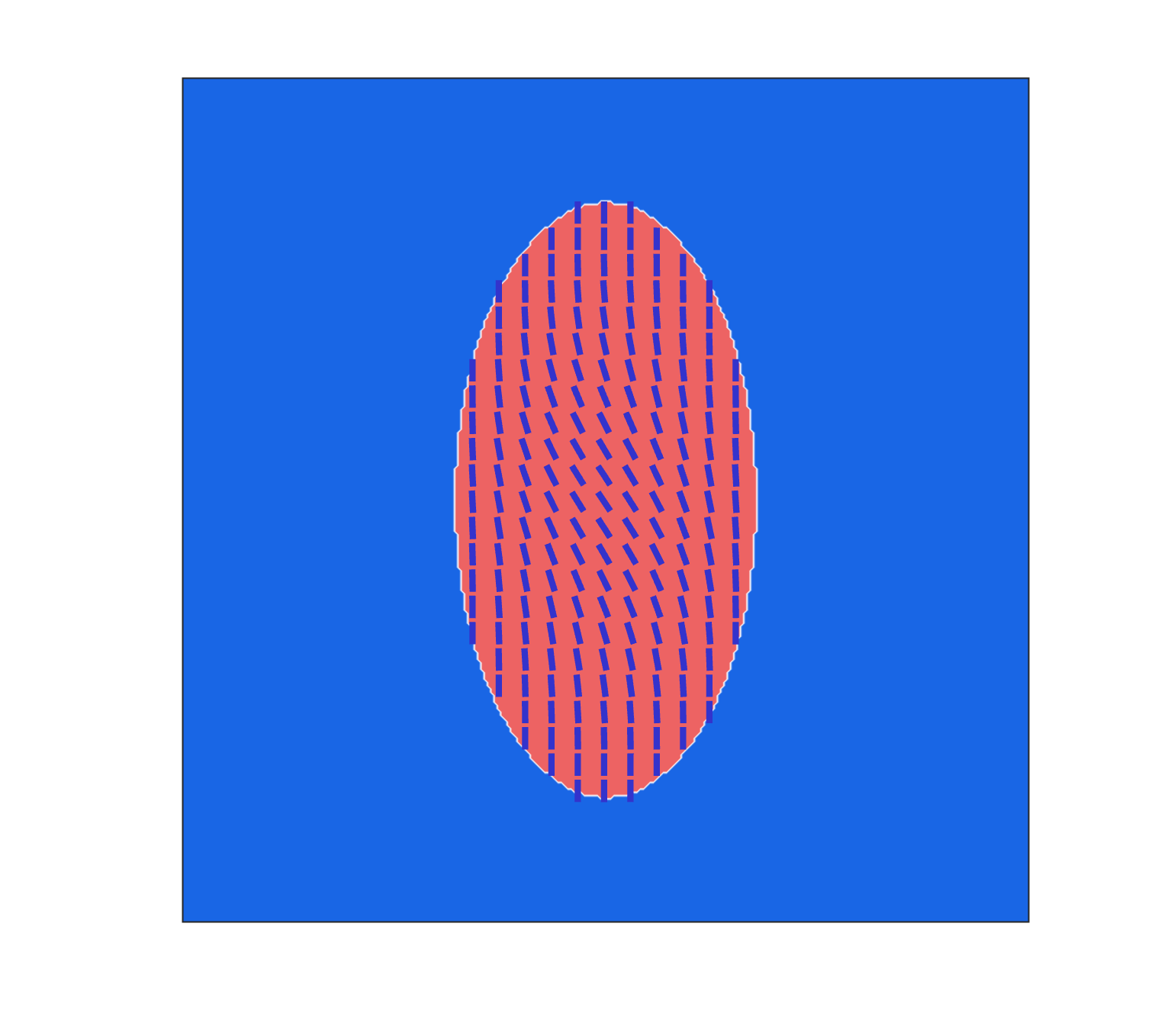}
			\includegraphics[width=0.15\textwidth]{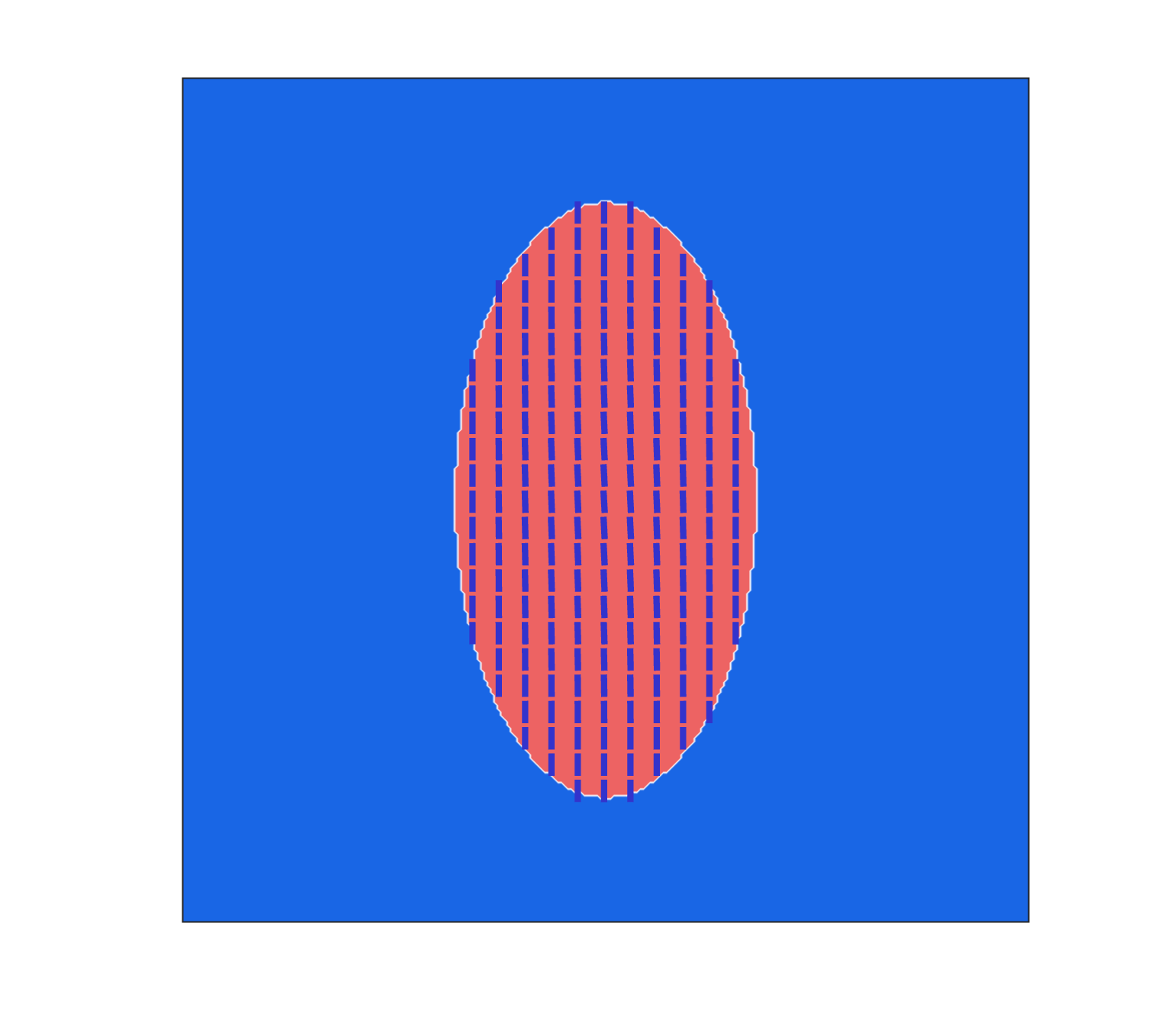}
		\end{minipage}
	}

	\caption{Comparison of defect dynamics for active liquid crystals
		with opposite signs of the activity in a vertical, elliptical
		domain. Each panel shows the contour of the principal eigenvalue
		and director field at various times. The top panel records the
		solutions at $t=0.1$, $0.3$, $1$, $3$, $5$, $10$, the bottom
		panel shows the solution at $t = 0.2$, $7$, $10$, $11$, $12$,
		$20$, respectively.}
	\label{fig:inellipse1-compare}
\end{figure}
In the vertically oriented elliptical obstacle, when $\chi_{fluid}
	> 0$, the two $+\tfrac{1}{2}$ defects cannot separate significantly
because of the geometric confinement. Eventually, the mutual
attraction between the four defects leads them to merge and
annihilate, eliminating all defects and driving the system toward a
defect-free steady state.

\begin{figure}[H]
	\centering

	\fbox{%
		\begin{minipage}{0.96\textwidth}
			\centering
			\textbf{$\chi_{fluid} = 10$, $\xi_{fluild} = -0.1$} \\[0.5em]
			\includegraphics[width=0.15\textwidth]{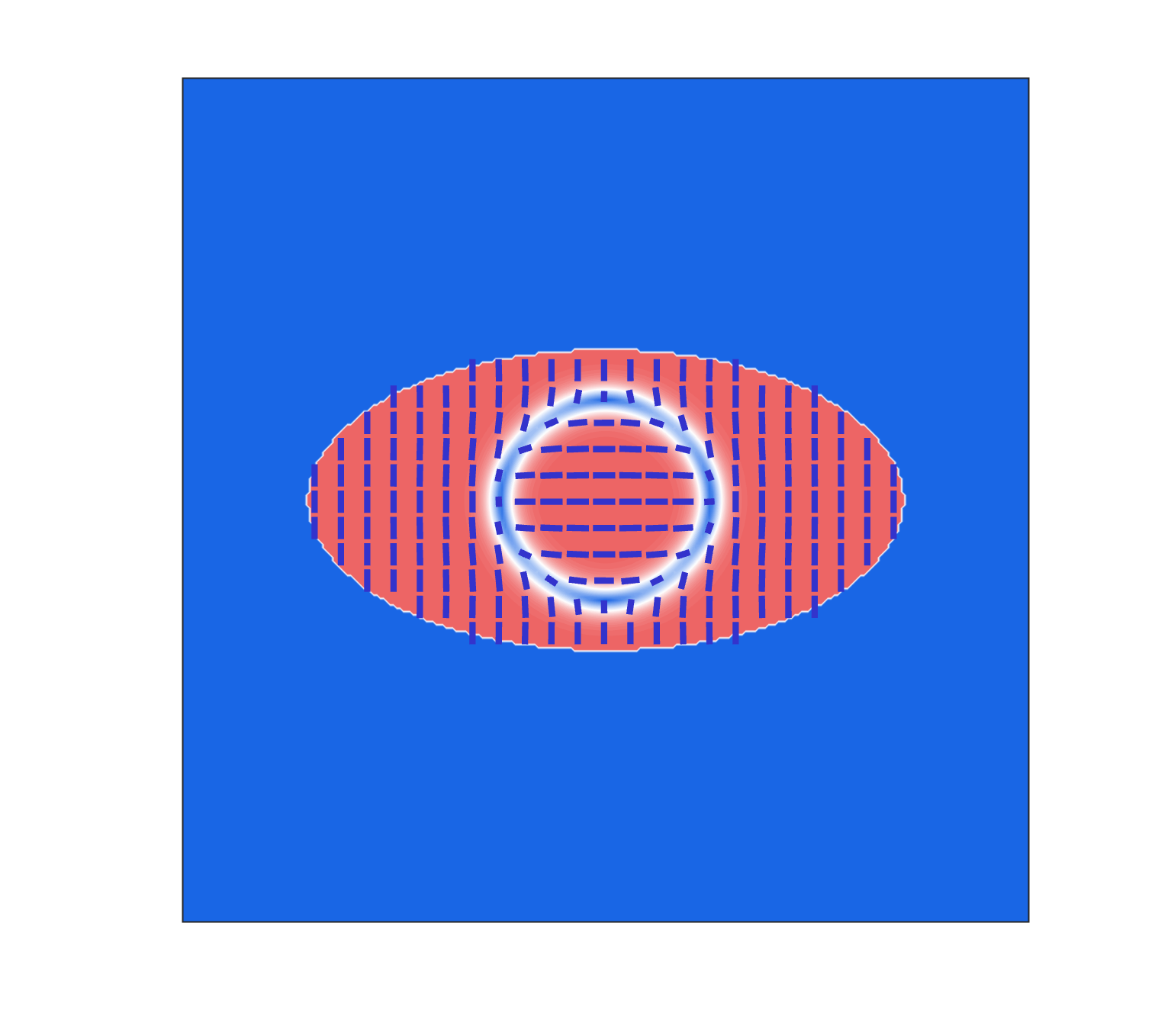}
			\includegraphics[width=0.15\textwidth]{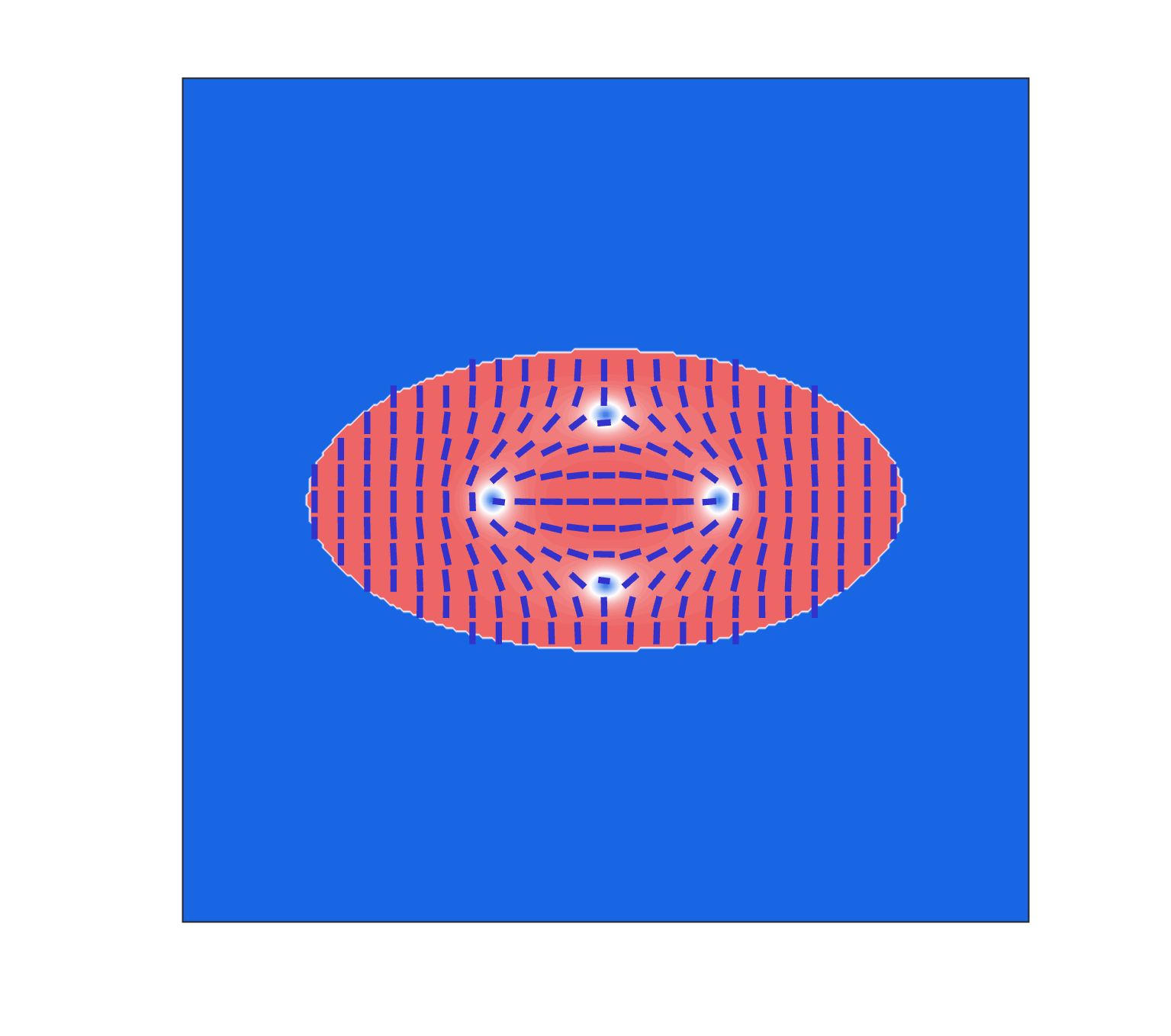}
			\includegraphics[width=0.15\textwidth]{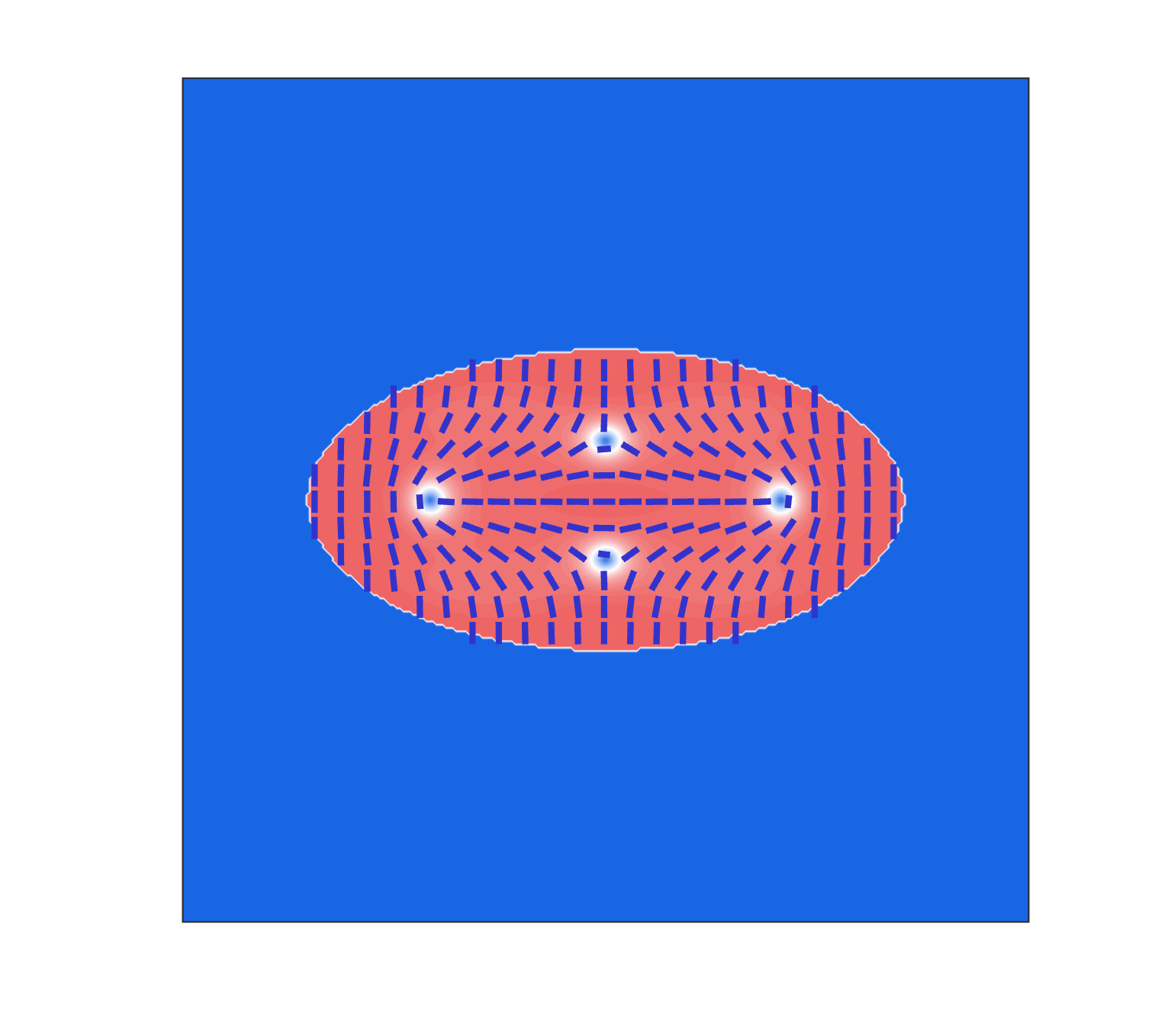}
			\includegraphics[width=0.15\textwidth]{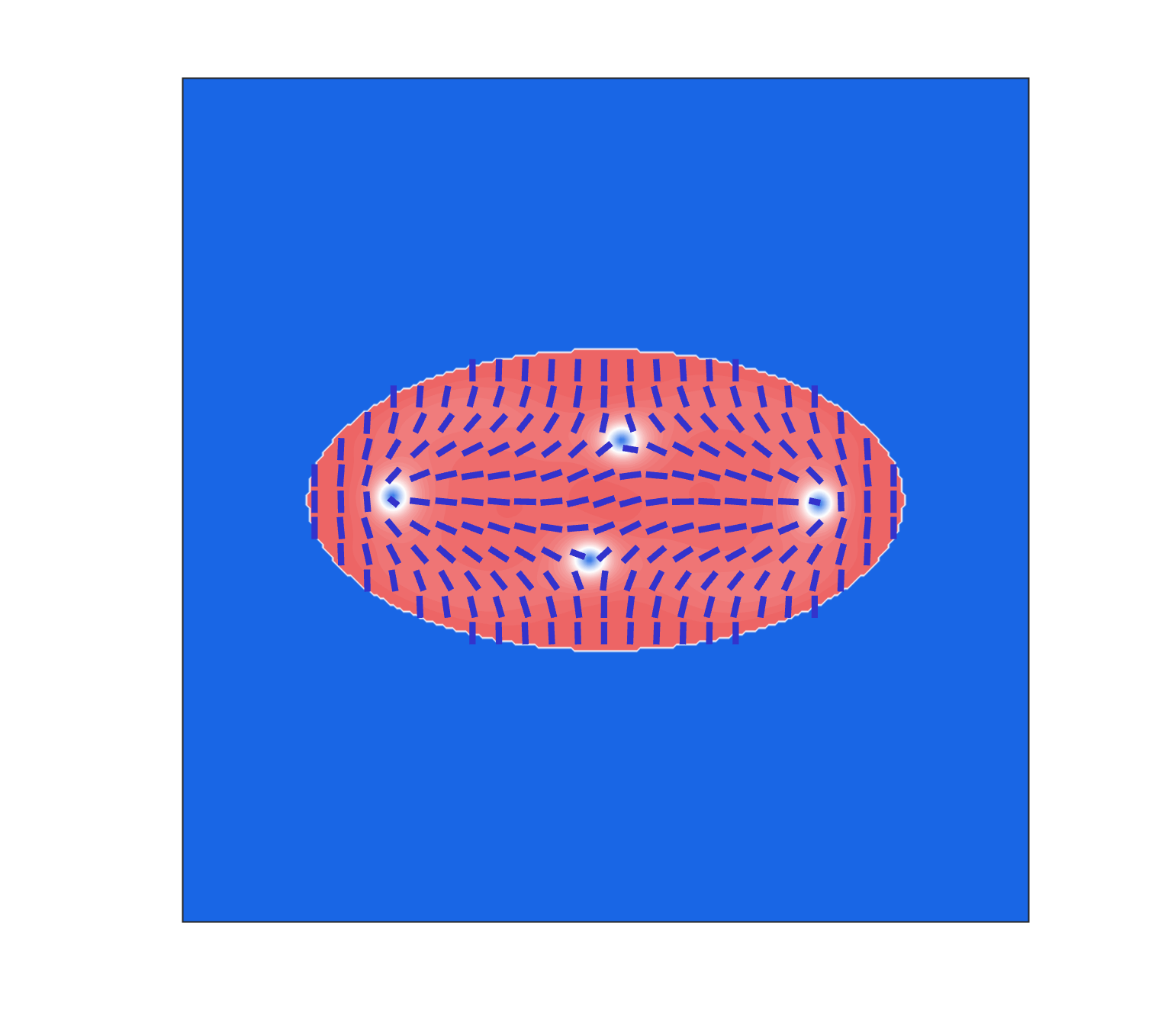}
			\includegraphics[width=0.15\textwidth]{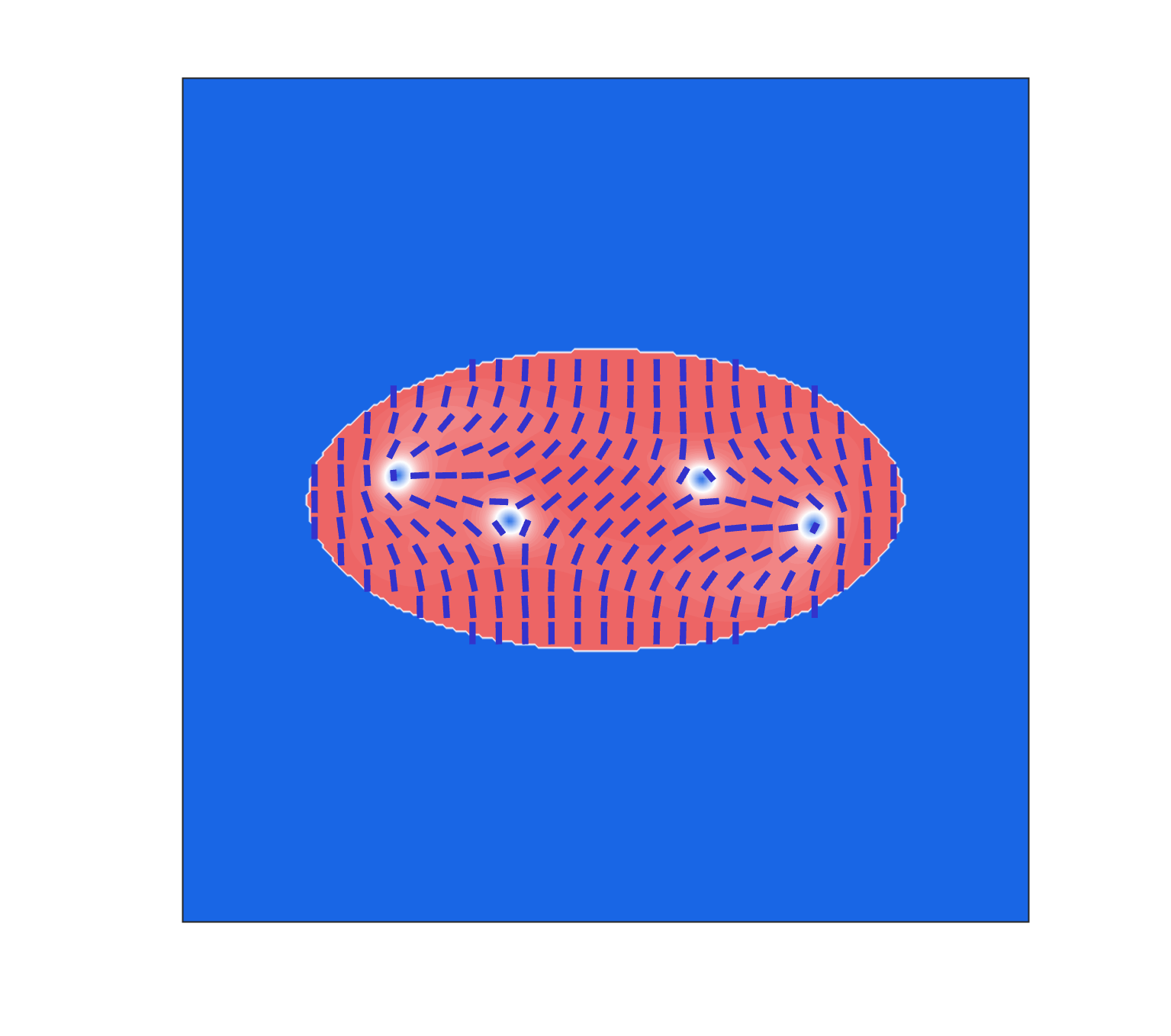}
			\includegraphics[width=0.15\textwidth]{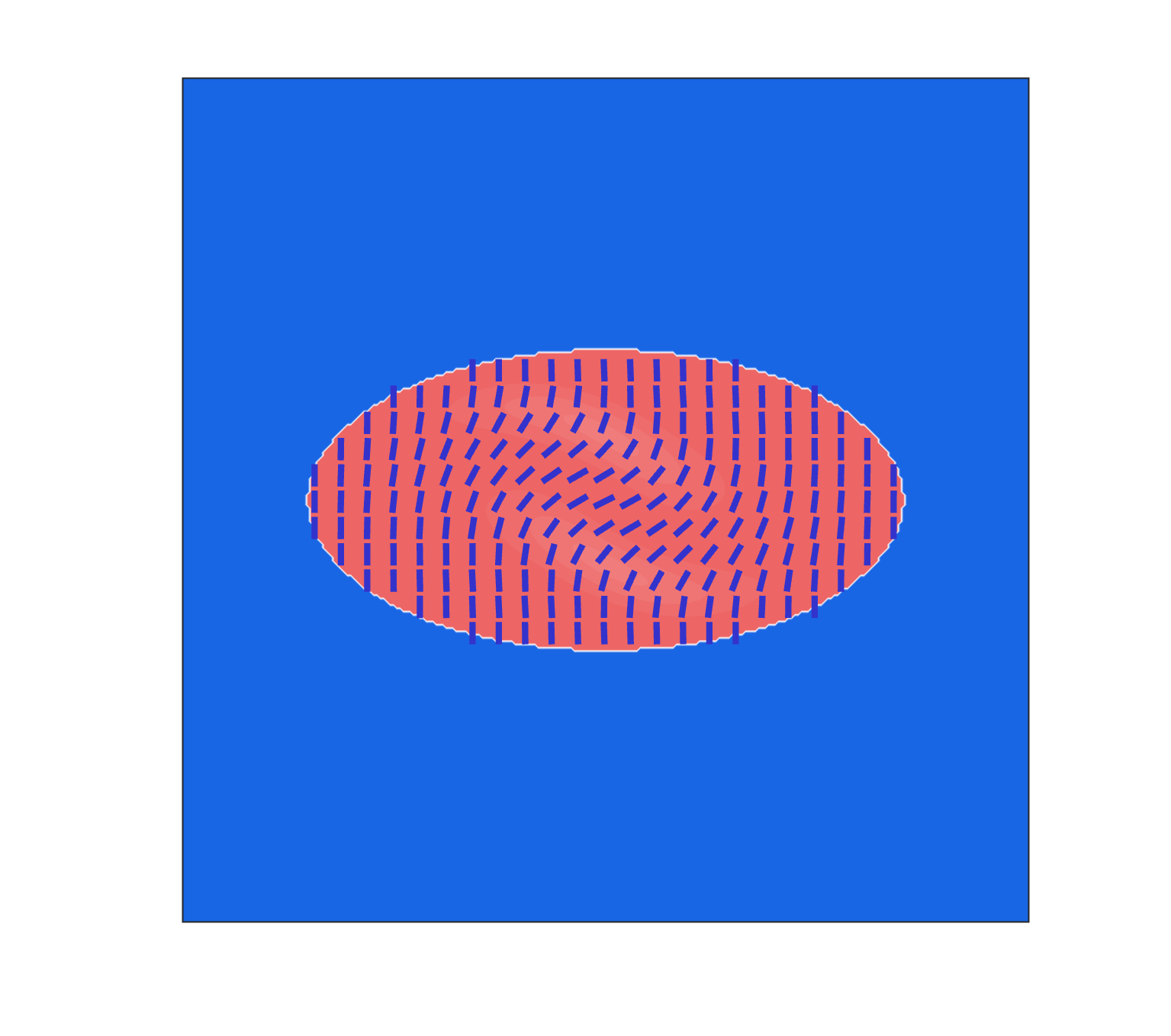}
		\end{minipage}
	}%
	\vspace{1em}
	\fbox{%
		\begin{minipage}{0.96\textwidth}
			\centering
			\textbf{$\chi = -10\psi$, $\xi = 0.1\psi$} \\[0.5em]
			\includegraphics[width=0.15\textwidth]{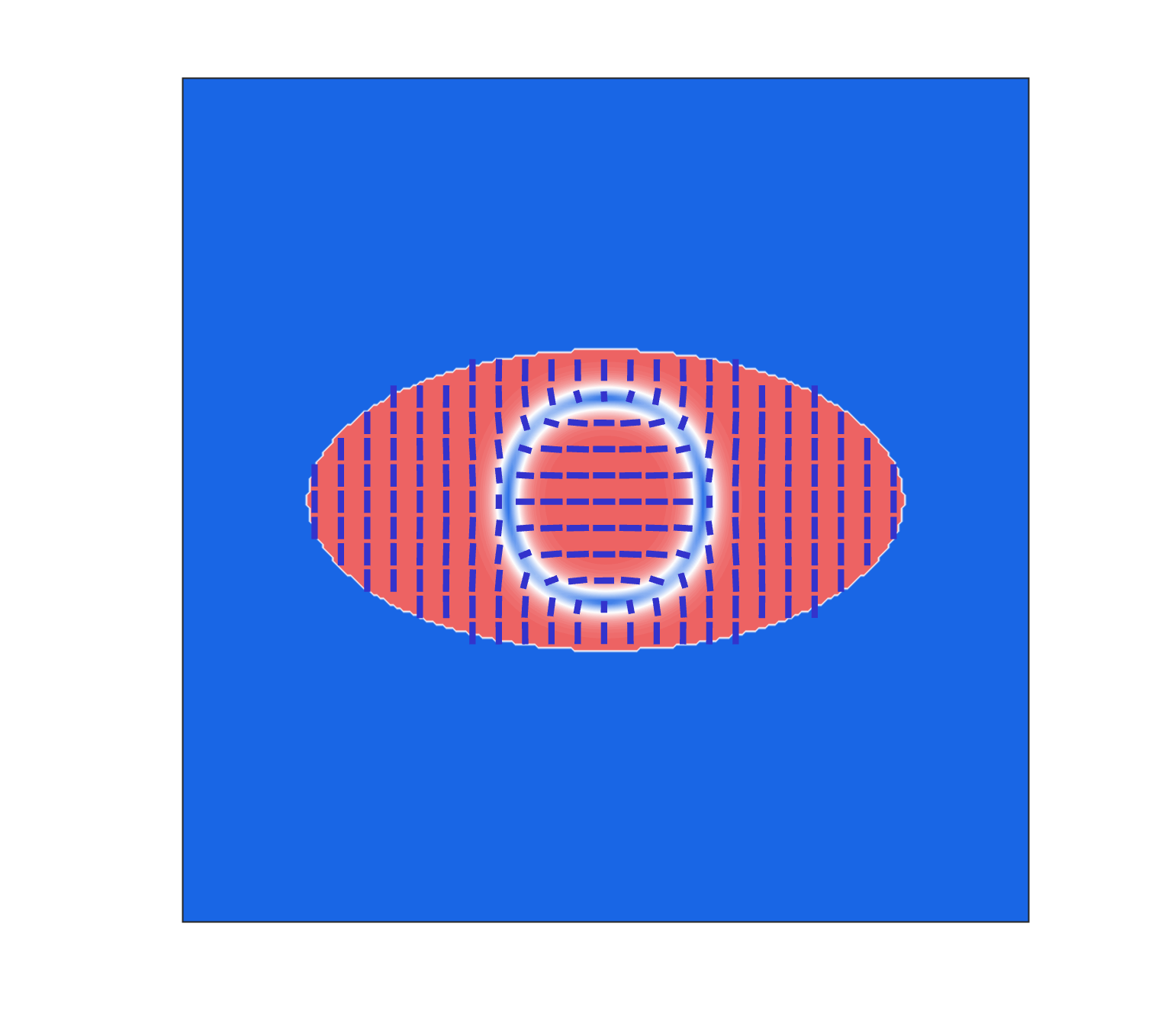}
			\includegraphics[width=0.15\textwidth]{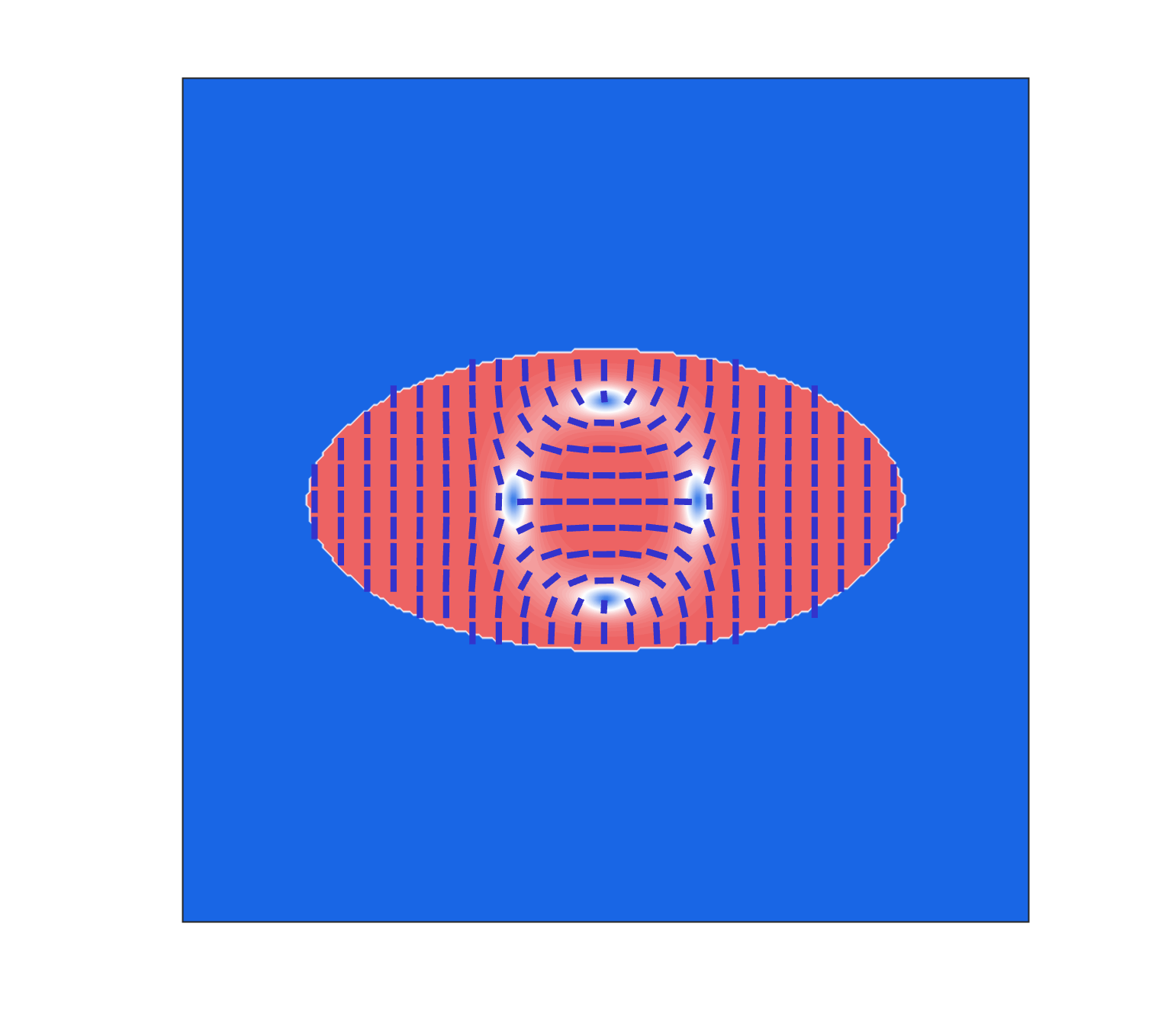}
			\includegraphics[width=0.15\textwidth]{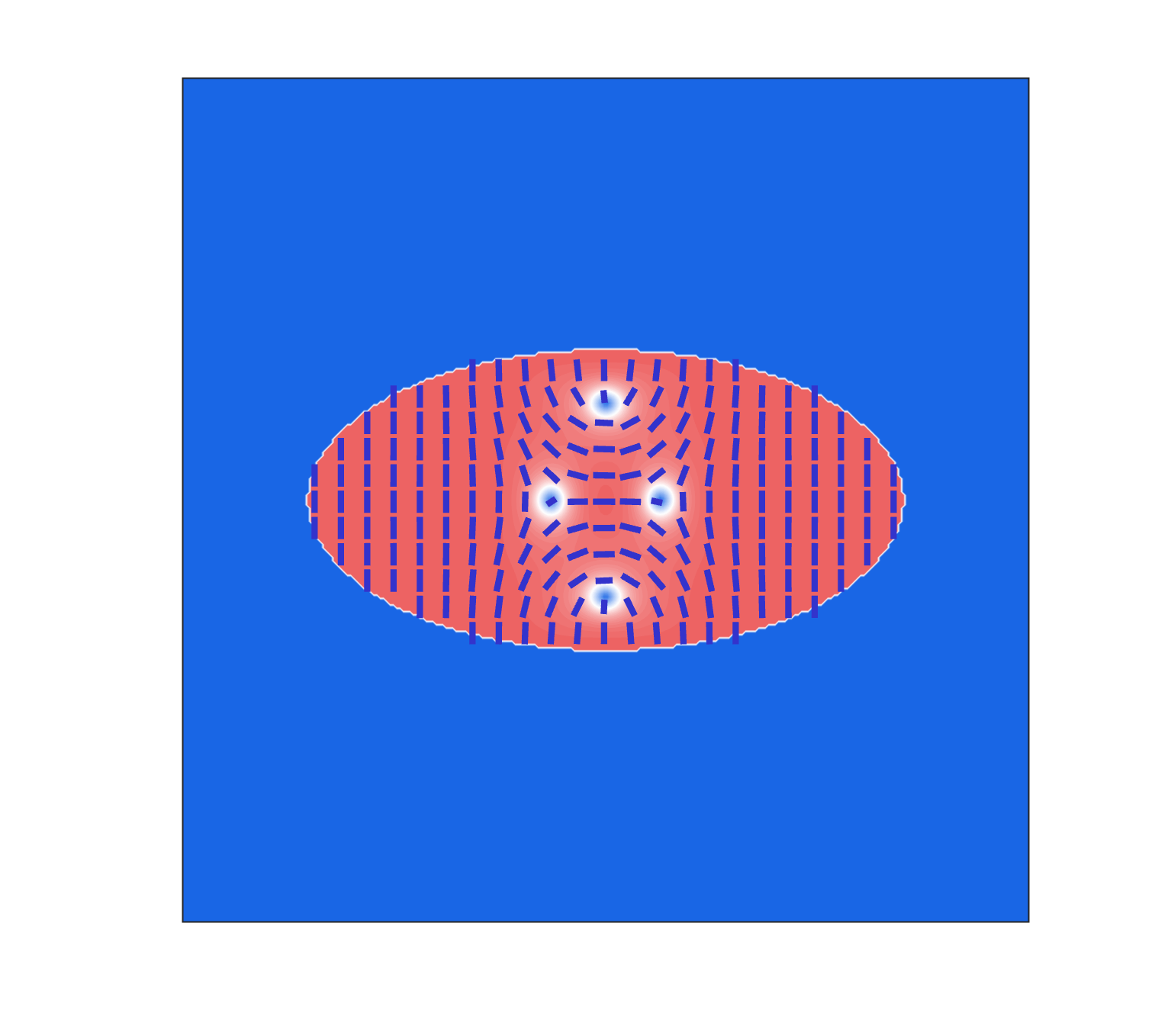}
			\includegraphics[width=0.15\textwidth]{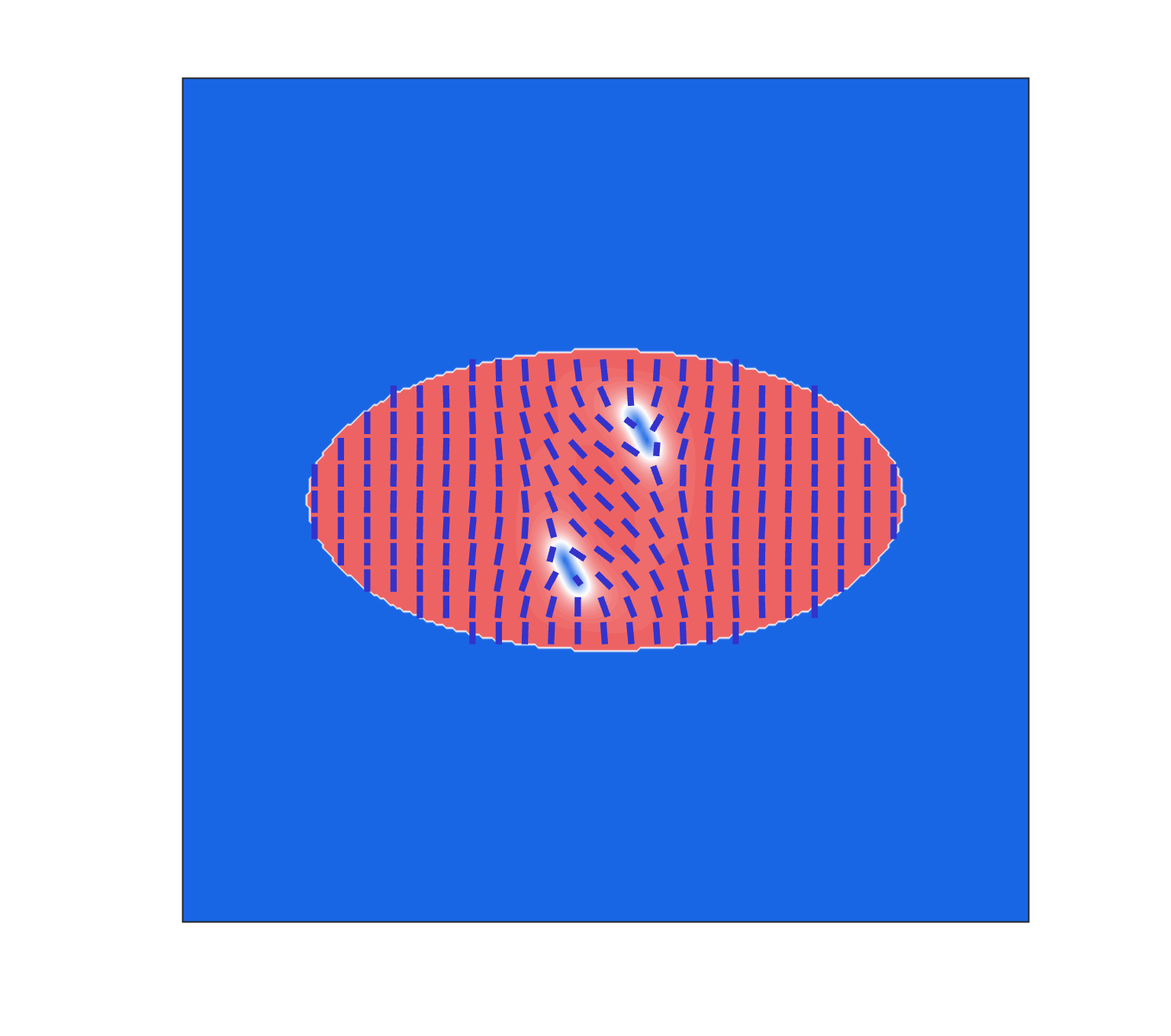}
			\includegraphics[width=0.15\textwidth]{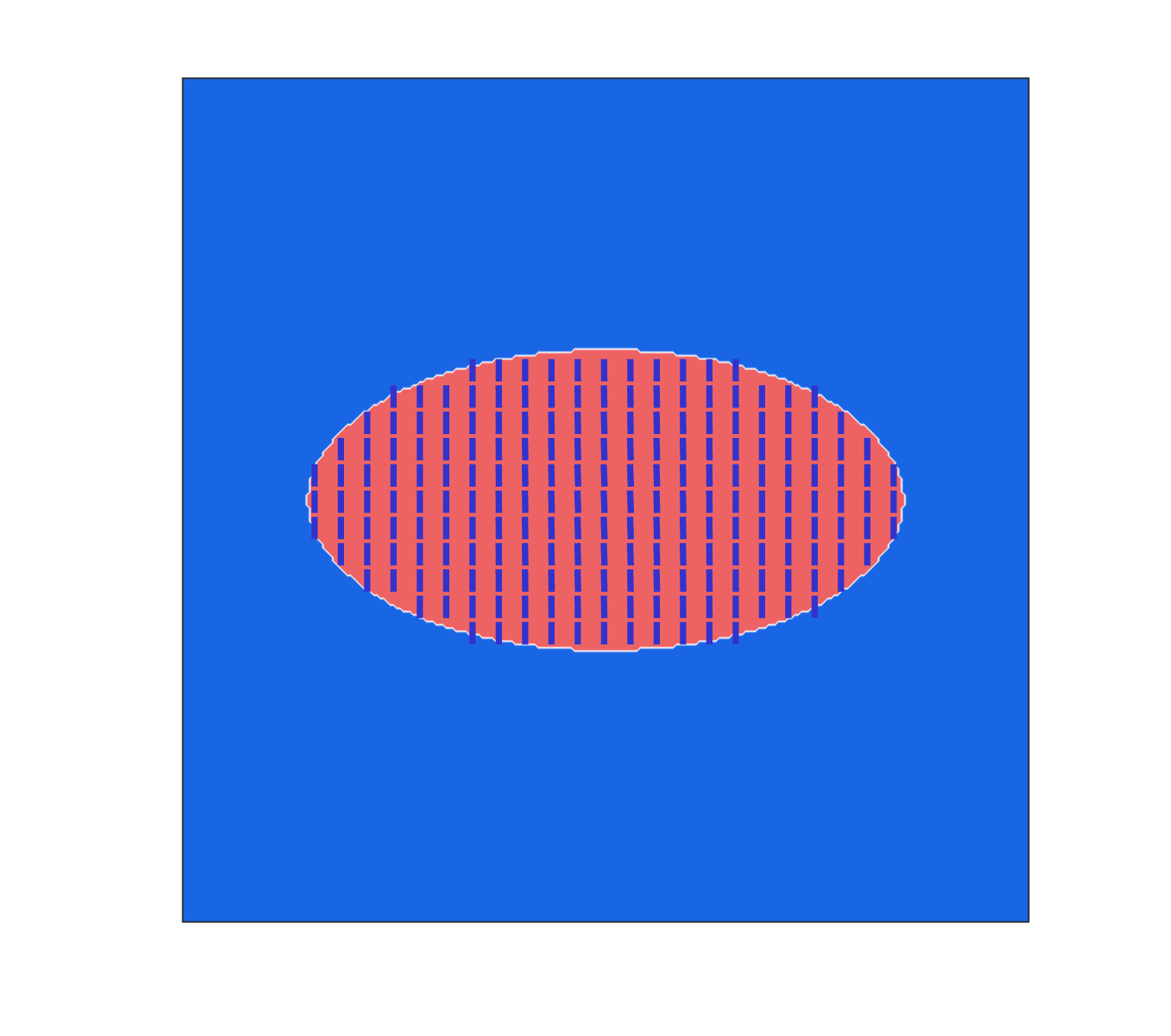}
			\includegraphics[width=0.15\textwidth]{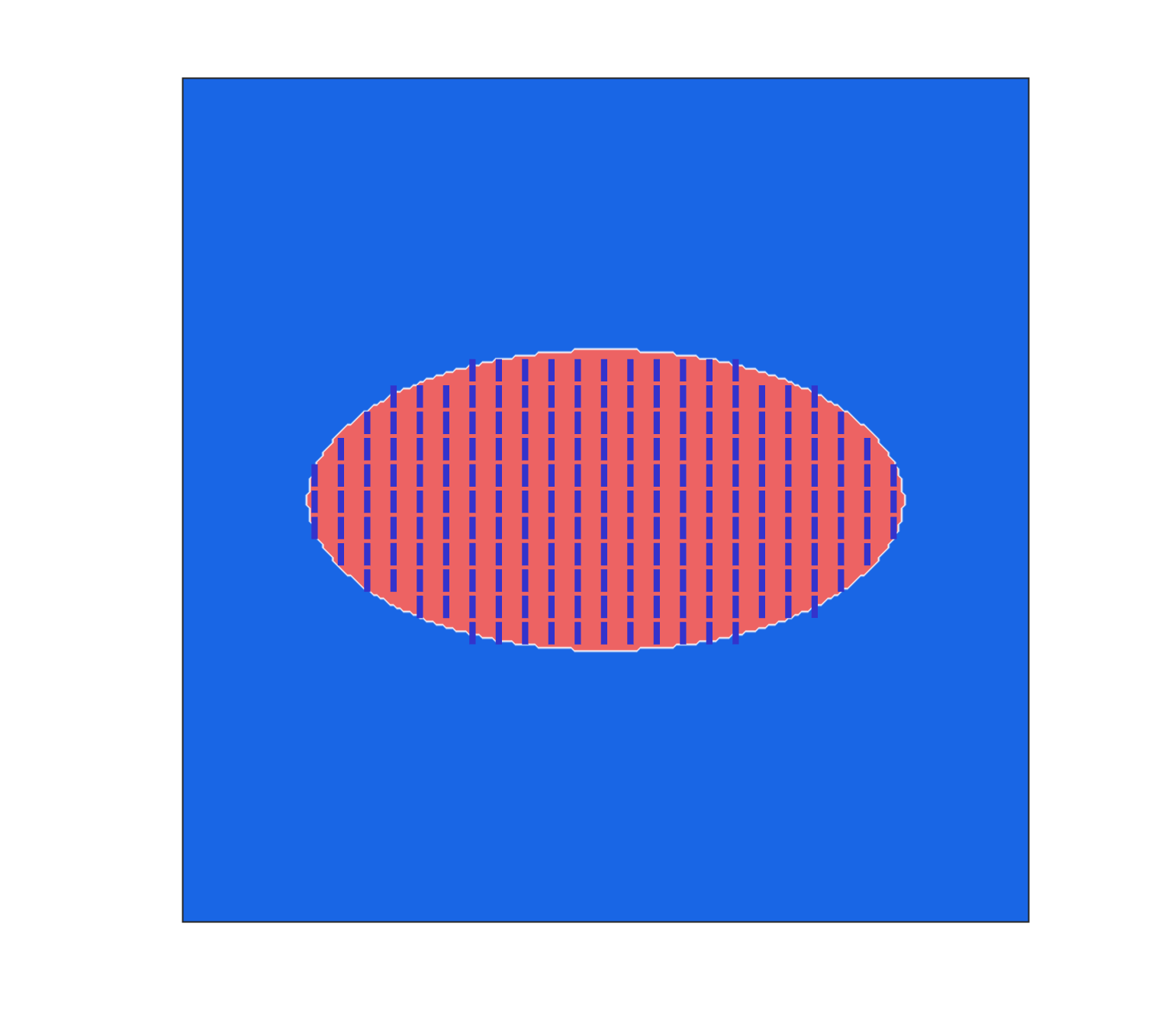}
		\end{minipage}
	}

	\caption{Comparison of defect dynamics for active liquid crystals
		with opposite signs of the activity in a horizontal, elliptical
		domain. Each panel shows the contour of the principal eigenvalue
		and director field at various time times. The top panel records
		the solutions at $t=0.1$, $0.3$, $1$, $5$, $7$, $10$, the bottom
		panel shows the solution at $t = 0.1$, $0.2$, $5$, $7.6$, $10$,
		$12$, respectively.}
	\label{fig:inellipse2-compare}
\end{figure}
For the horizontally oriented elliptical obstacle, when
$\chi_{fluid} > 0$, the $+\tfrac{1}{2}$ defects again initially
move apart due to the active forcing, after which defect pairs
annihilate and the system attains a steady state. For $\chi_{fluid}
	< 0$, a similar defect annihilation scenario occurs, but the motion
of the $+\tfrac{1}{2}$ defects is impeded by the solid boundary,
leading to the rapid stabilization.
\begin{figure}[H]
	\centering

	\fbox{%
		\begin{minipage}{0.96\textwidth}
			\centering
			\textbf{$\chi_{fluid} = 10$, $\xi_{fluid} = -0.1$} \\[0.5em]
			\includegraphics[width=0.15\textwidth]{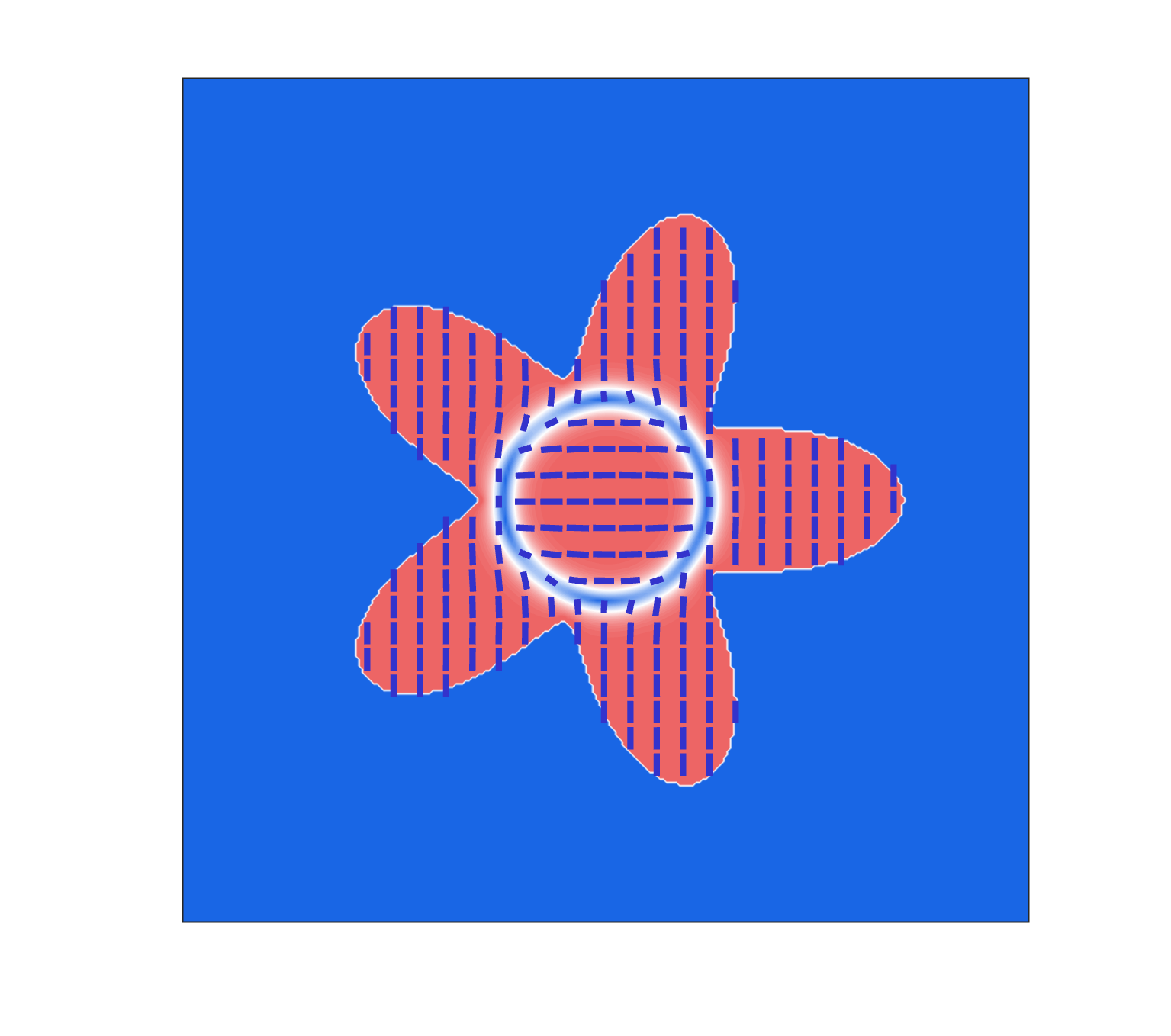}
			\includegraphics[width=0.15\textwidth]{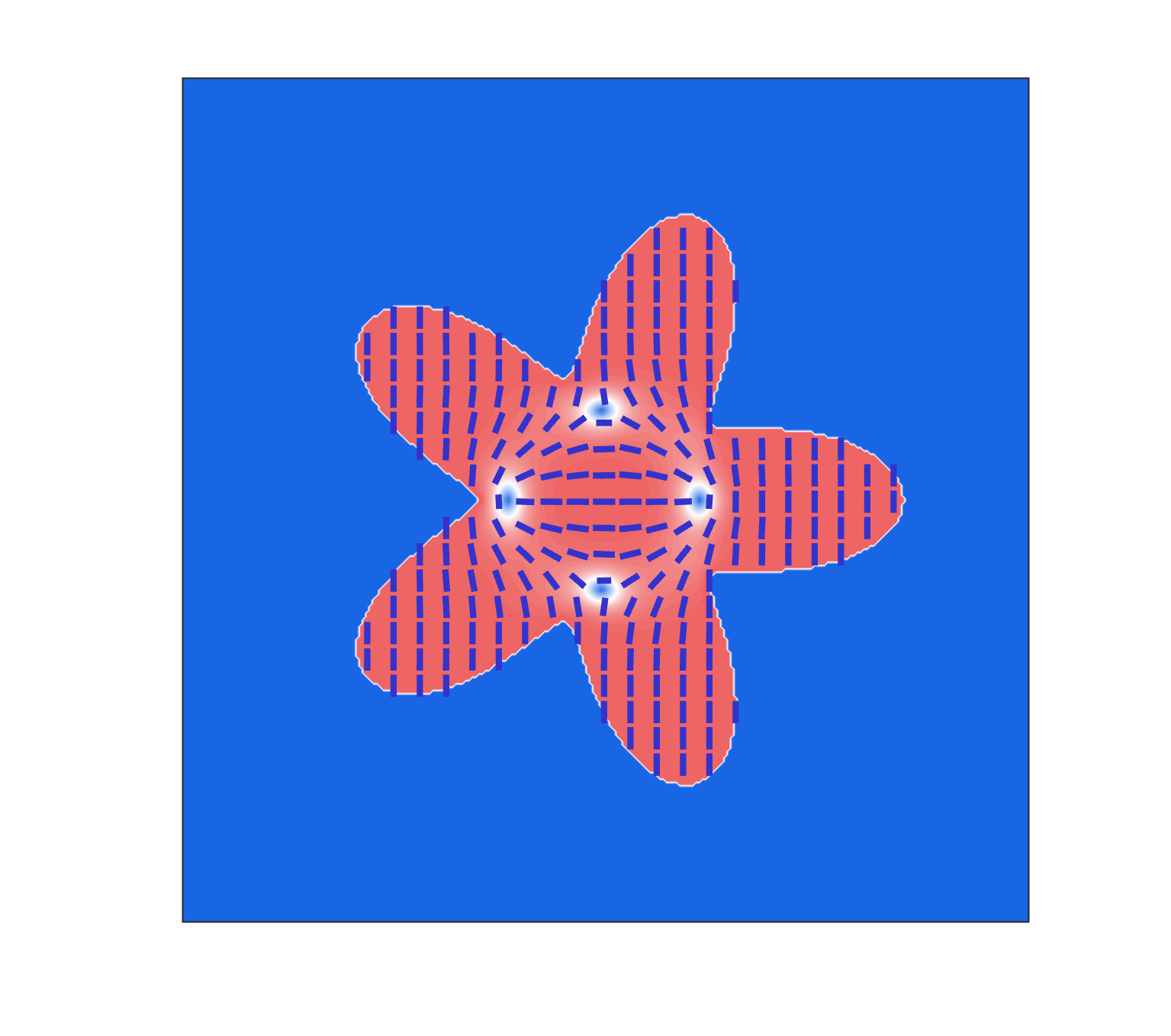}
			\includegraphics[width=0.15\textwidth]{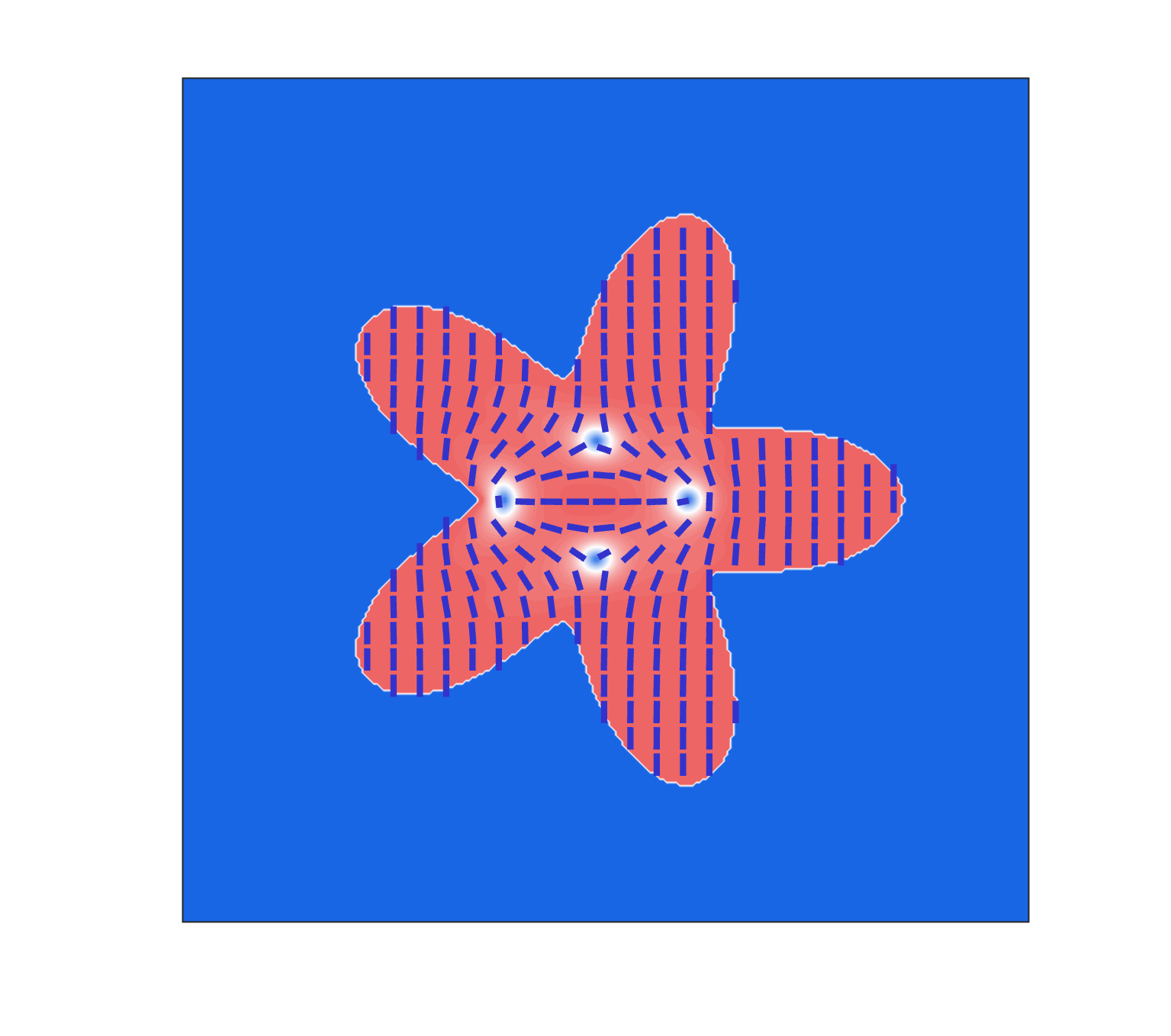}
			\includegraphics[width=0.15\textwidth]{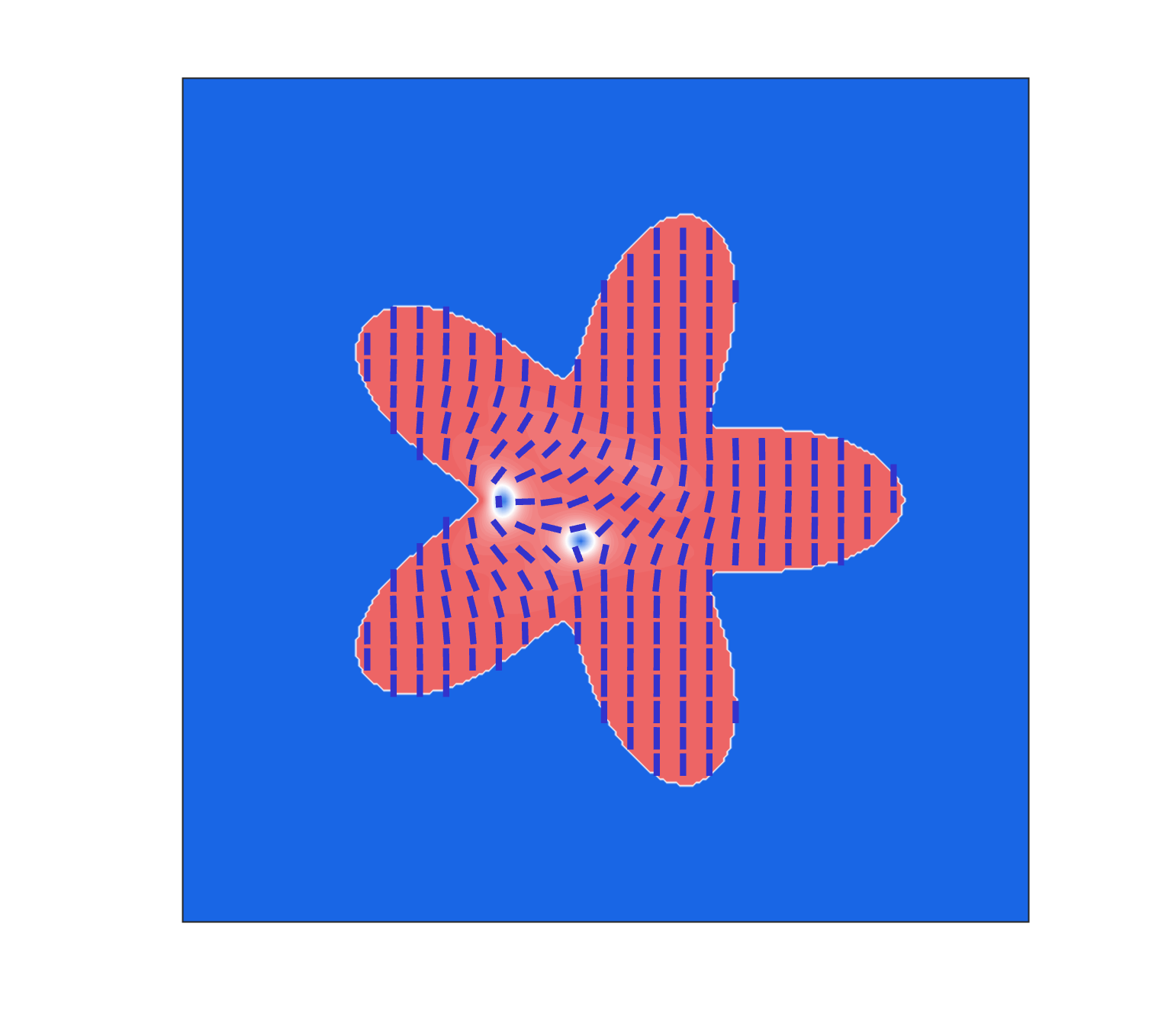}
			\includegraphics[width=0.15\textwidth]{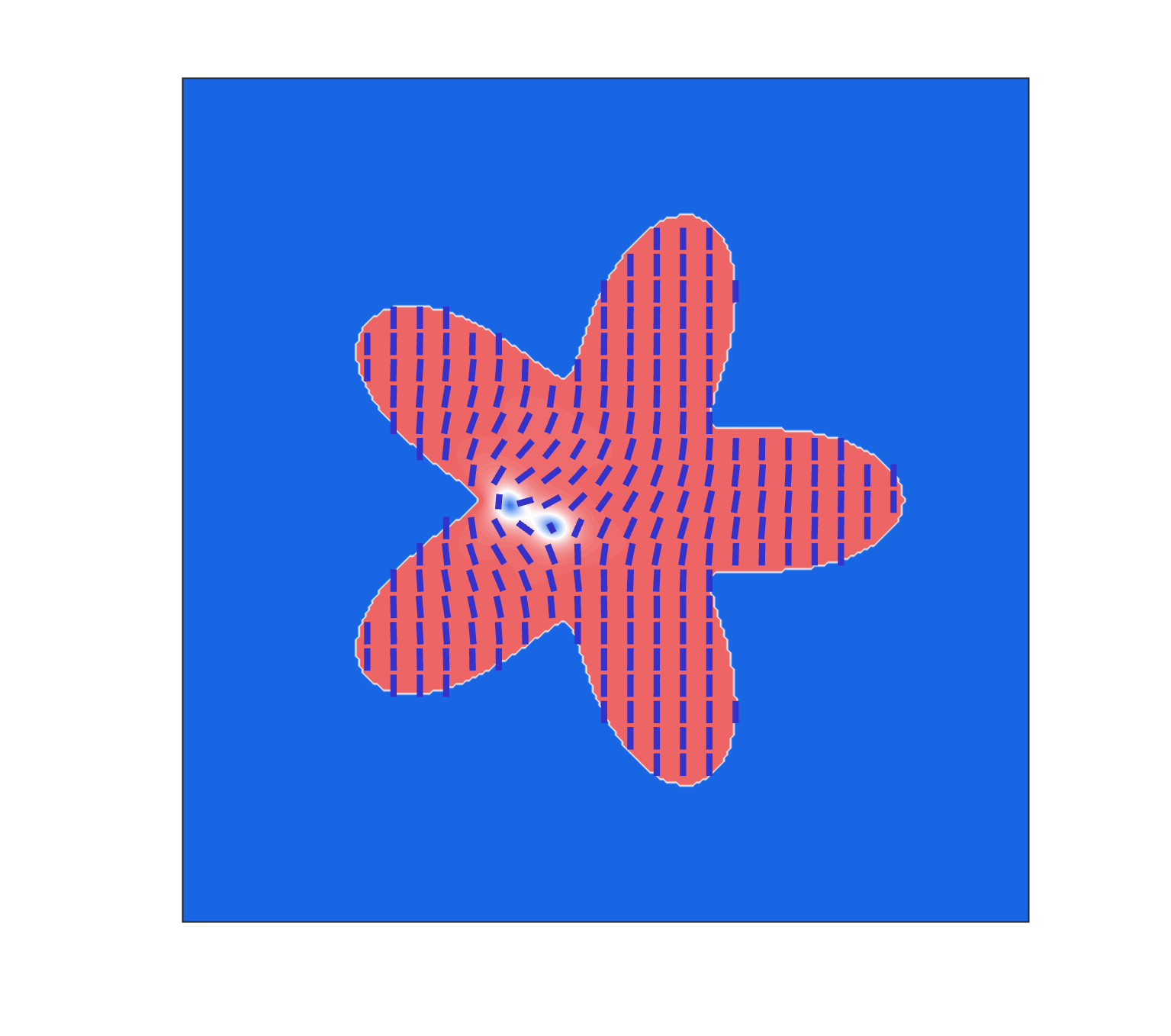}
			\includegraphics[width=0.15\textwidth]{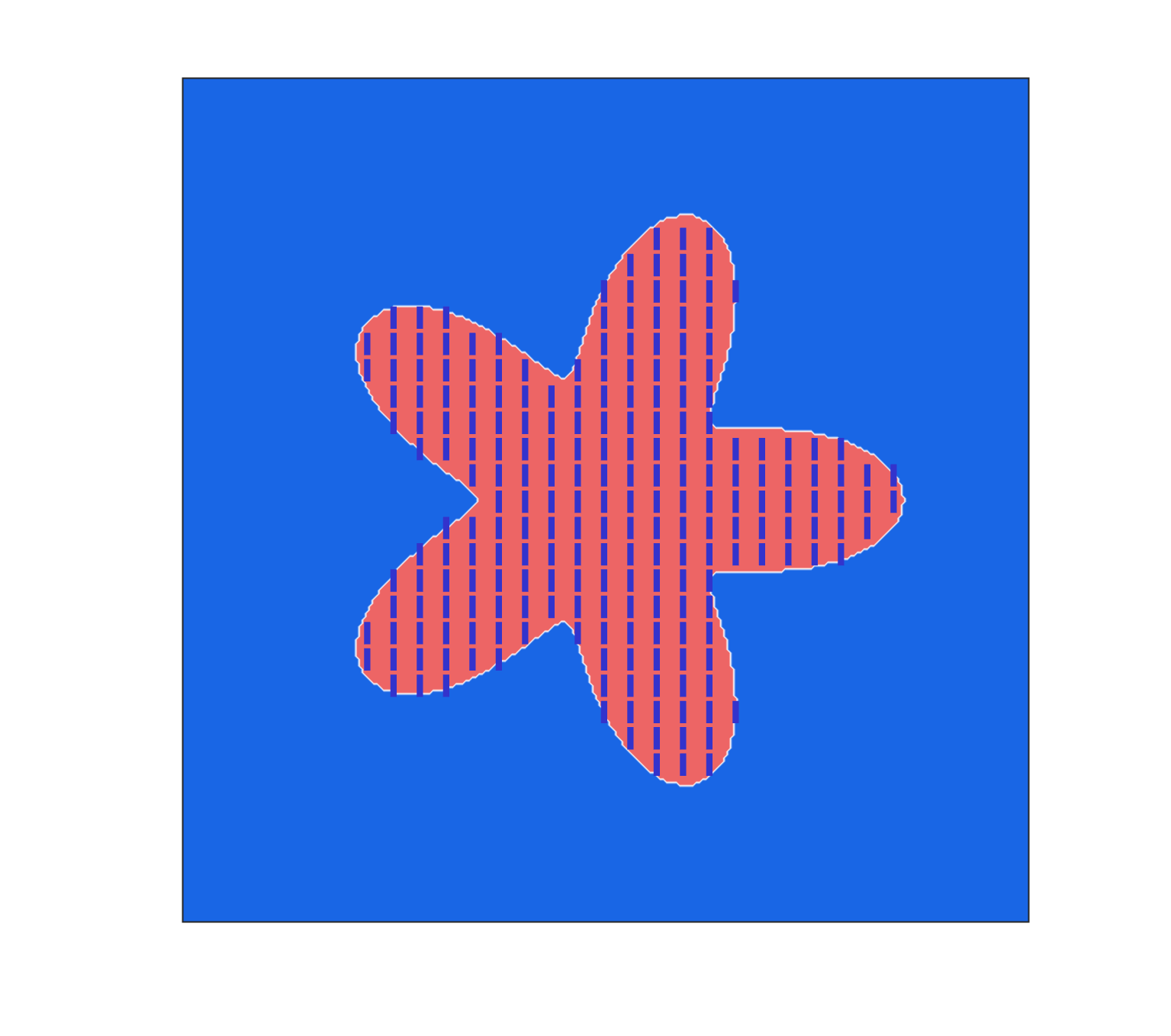}
		\end{minipage}
	}%
	\vspace{1em}
	\fbox{%
		\begin{minipage}{0.96\textwidth}
			\centering
			\textbf{$\chi_{fluid} = -10$, $\xi_{fluid} = 0.1$} \\[0.5em]
			\includegraphics[width=0.15\textwidth]{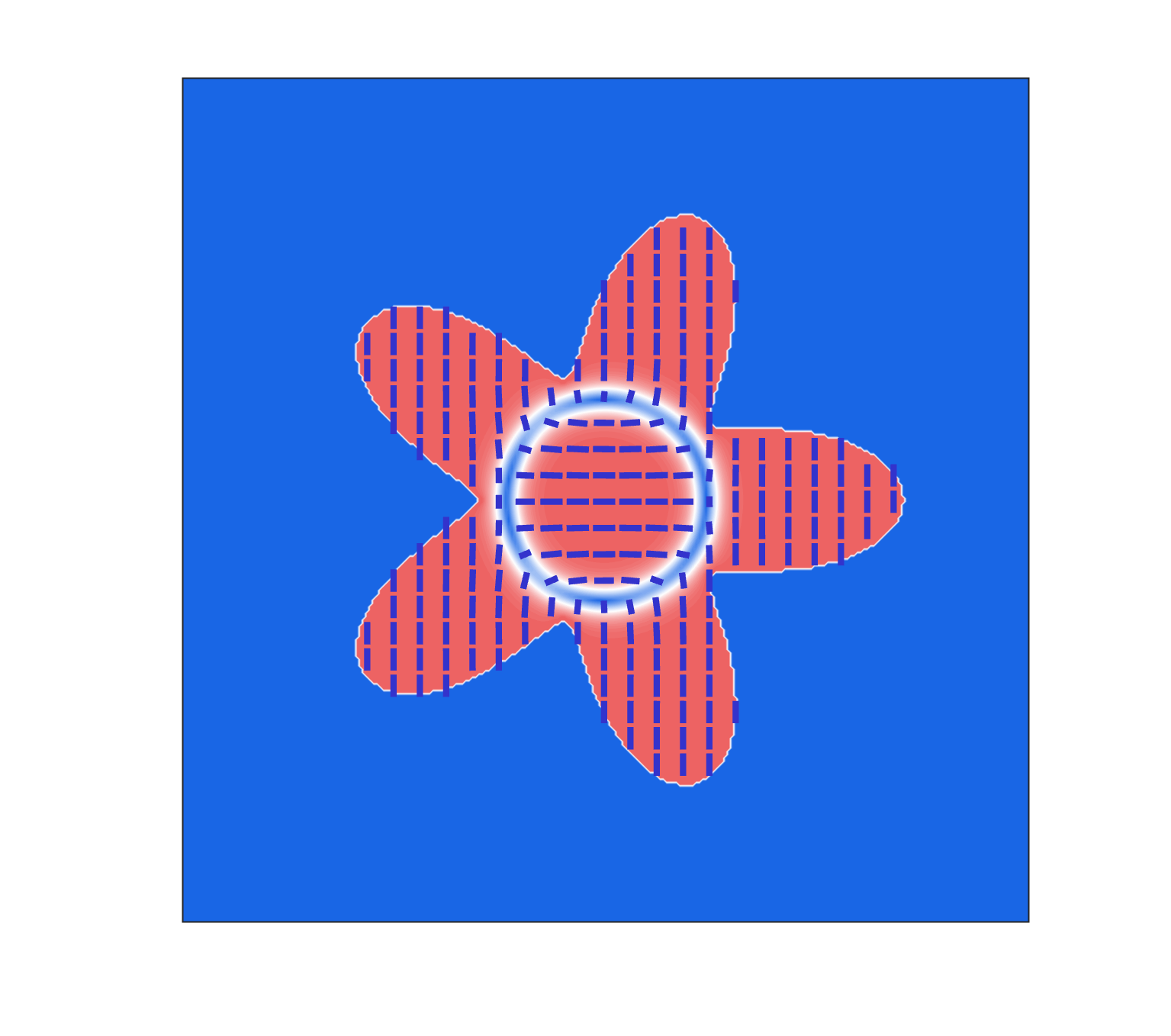}
			\includegraphics[width=0.15\textwidth]{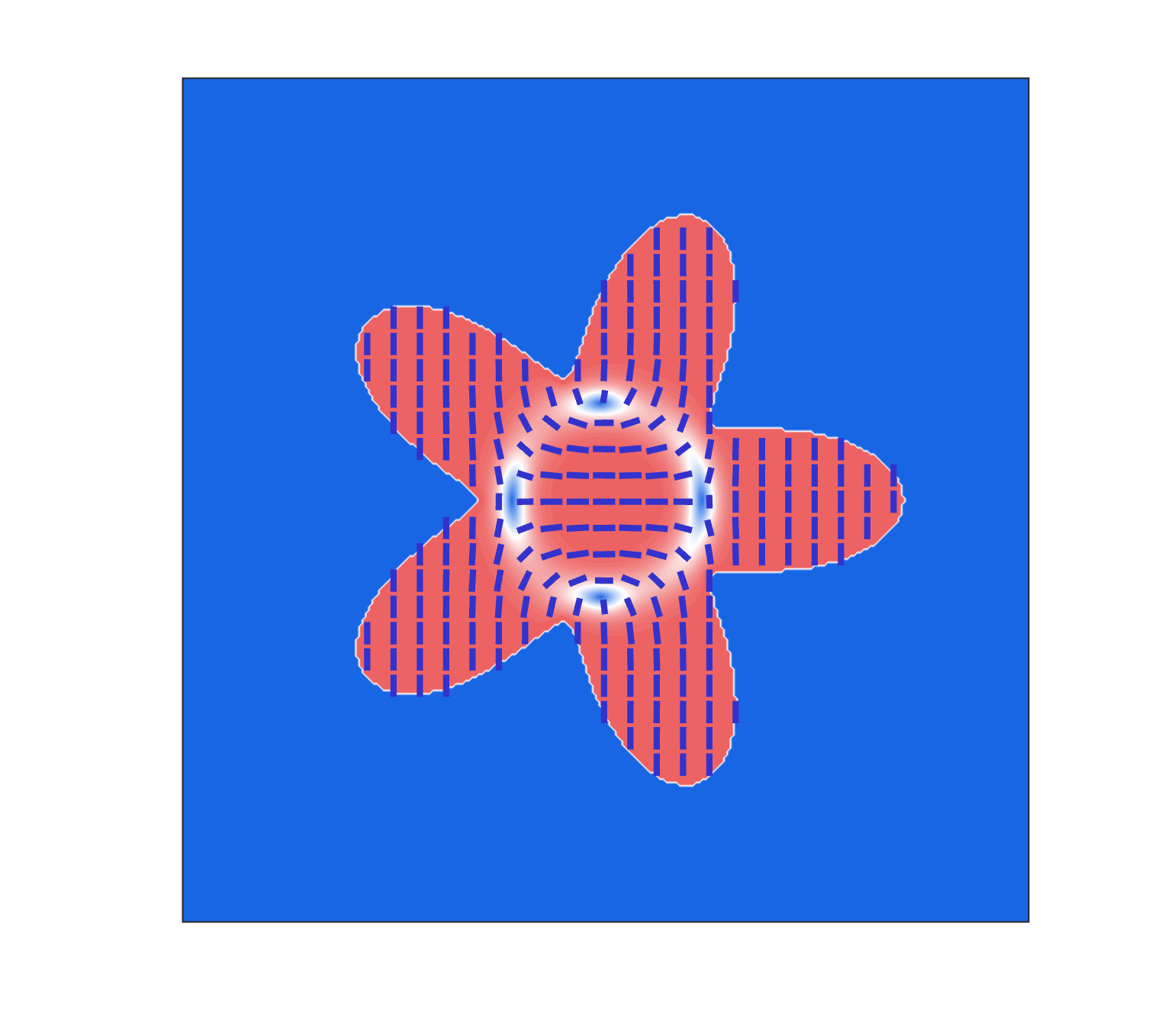}
			\includegraphics[width=0.15\textwidth]{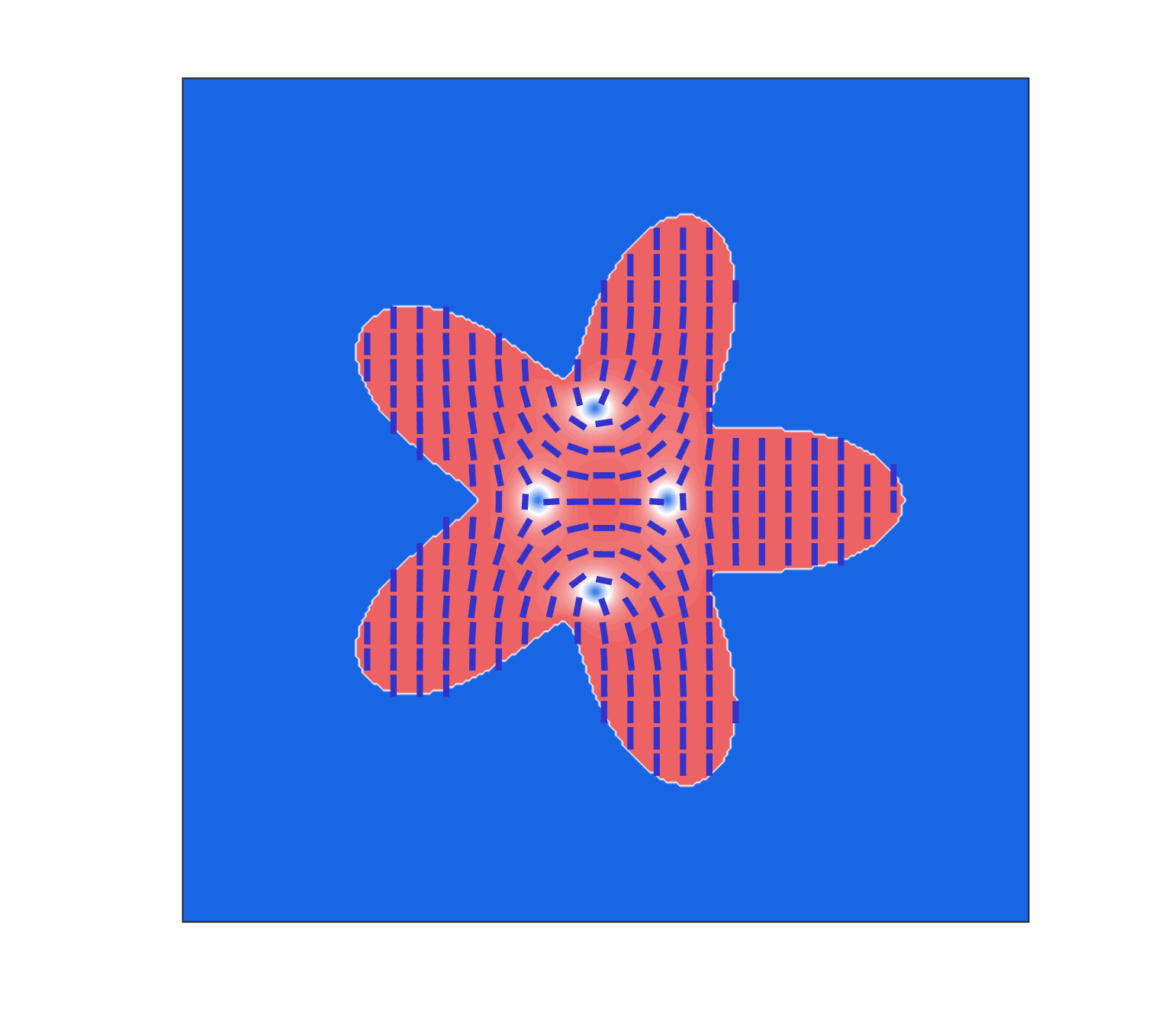}
			\includegraphics[width=0.15\textwidth]{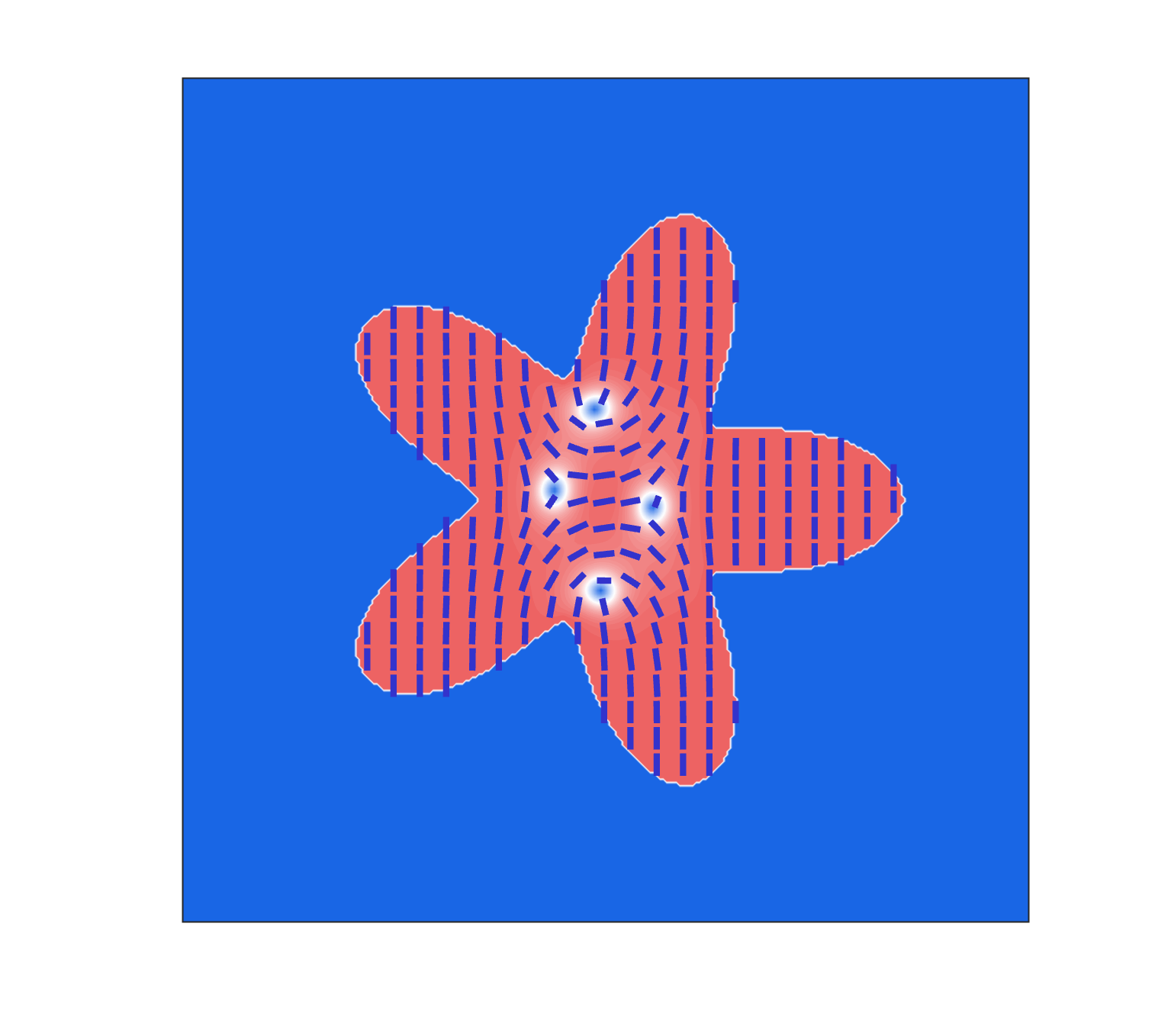}
			\includegraphics[width=0.15\textwidth]{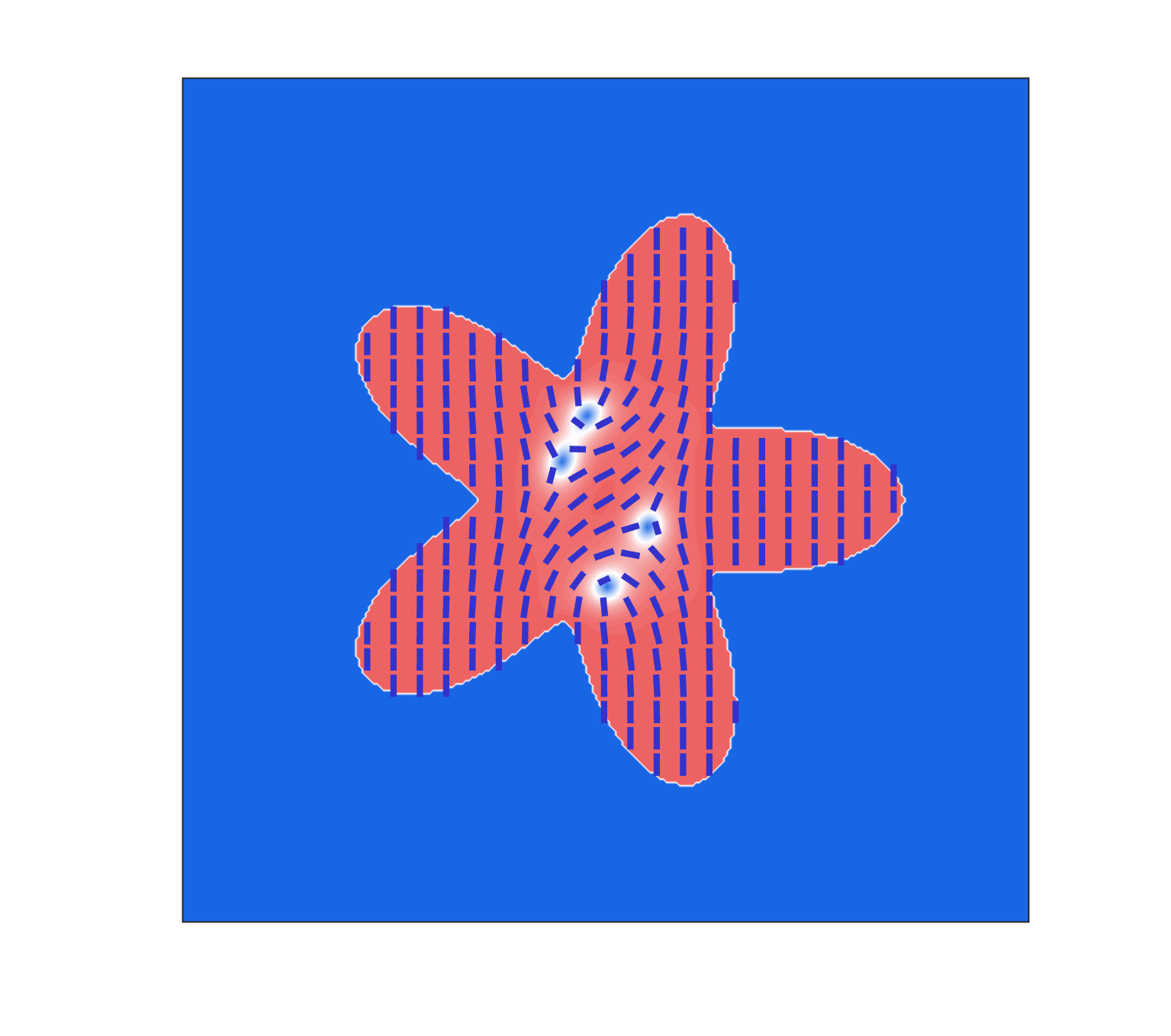}
			\includegraphics[width=0.15\textwidth]{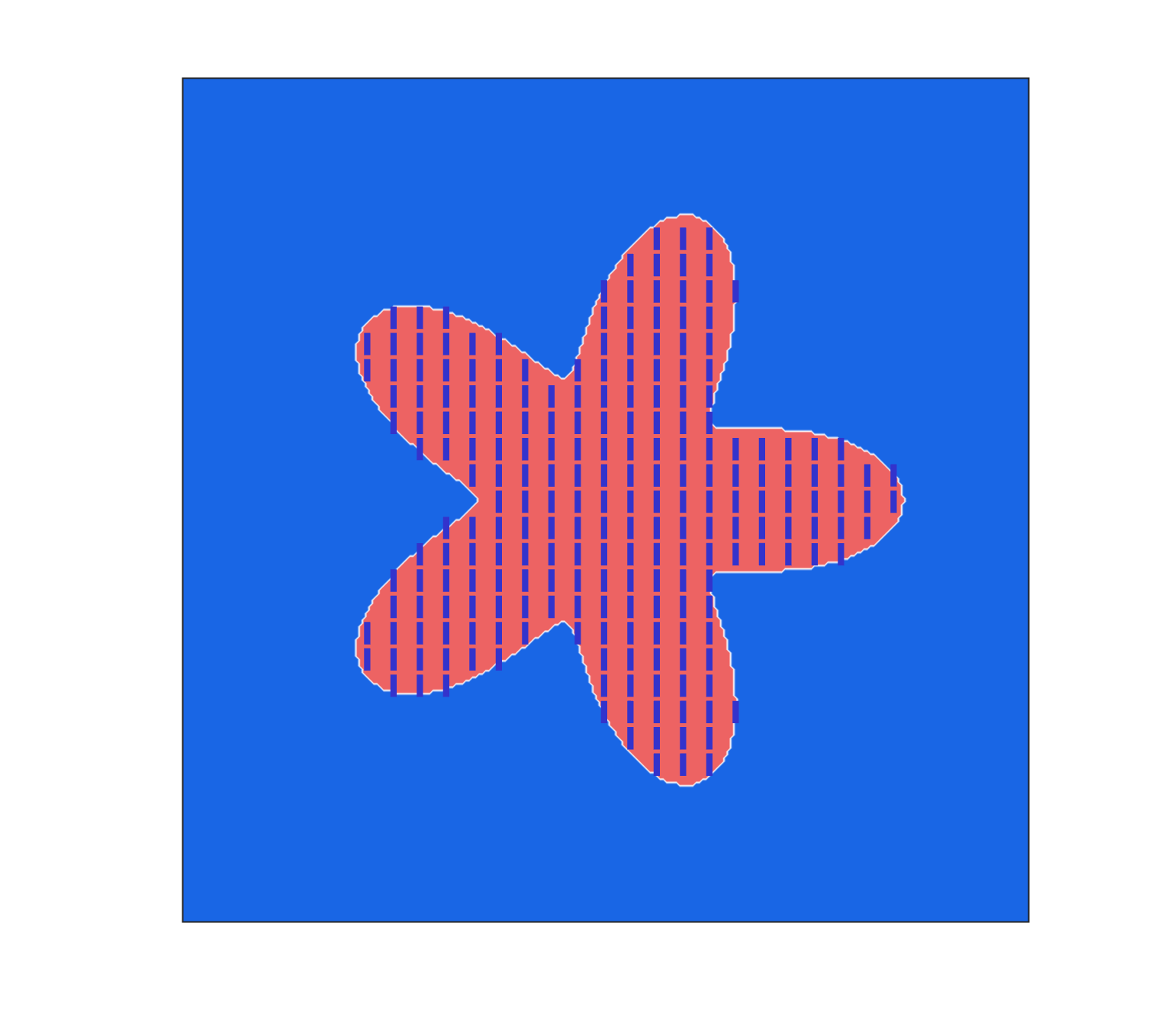}
		\end{minipage}
	}

	\caption{Comparison of defect dynamics for active liquid crystals
		with opposite signs of the activity in a star-shaped domain. Each
		panel shows the contour of the principal eigenvalue and director
		field at various time times. The top panel records the solutions
		at $t=0.1$, $0.3$, $1$, $3$, $3.5$, $20$ the bottom panel shows
		the solution at $t = 0.1$, $0.2$, $1$, $3$, $3.5$, $20$, respectively.}
	\label{fig:star-compare}
\end{figure}

Finally, in the star-shaped domain, the dynamics resemble the
horizontal ellipse case, rapidly approaching equilibrium.

\begin{figure}[H]
	\begin{center}
		\includegraphics[width=0.36\textwidth]{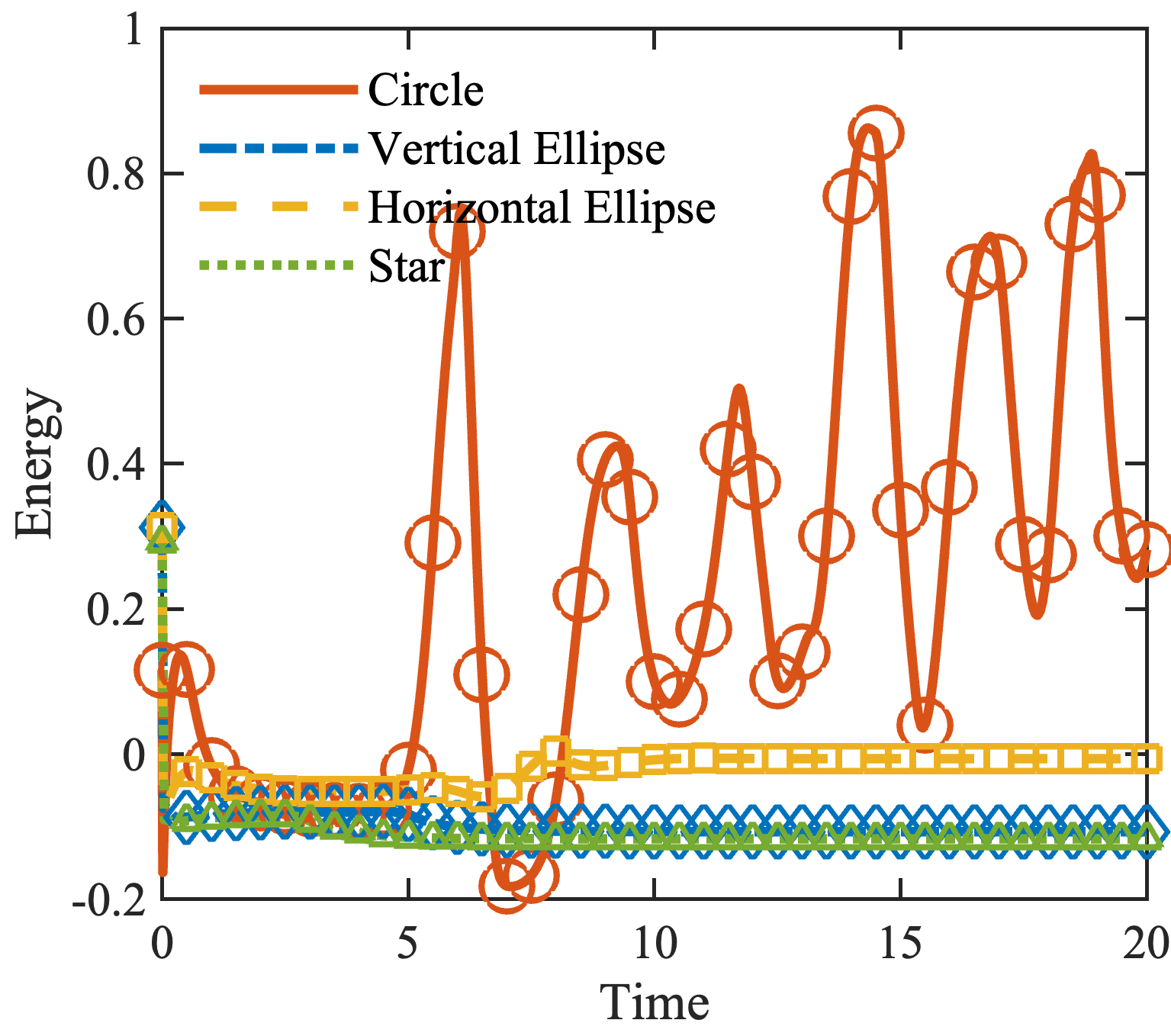}
		\includegraphics[width=0.36\textwidth]{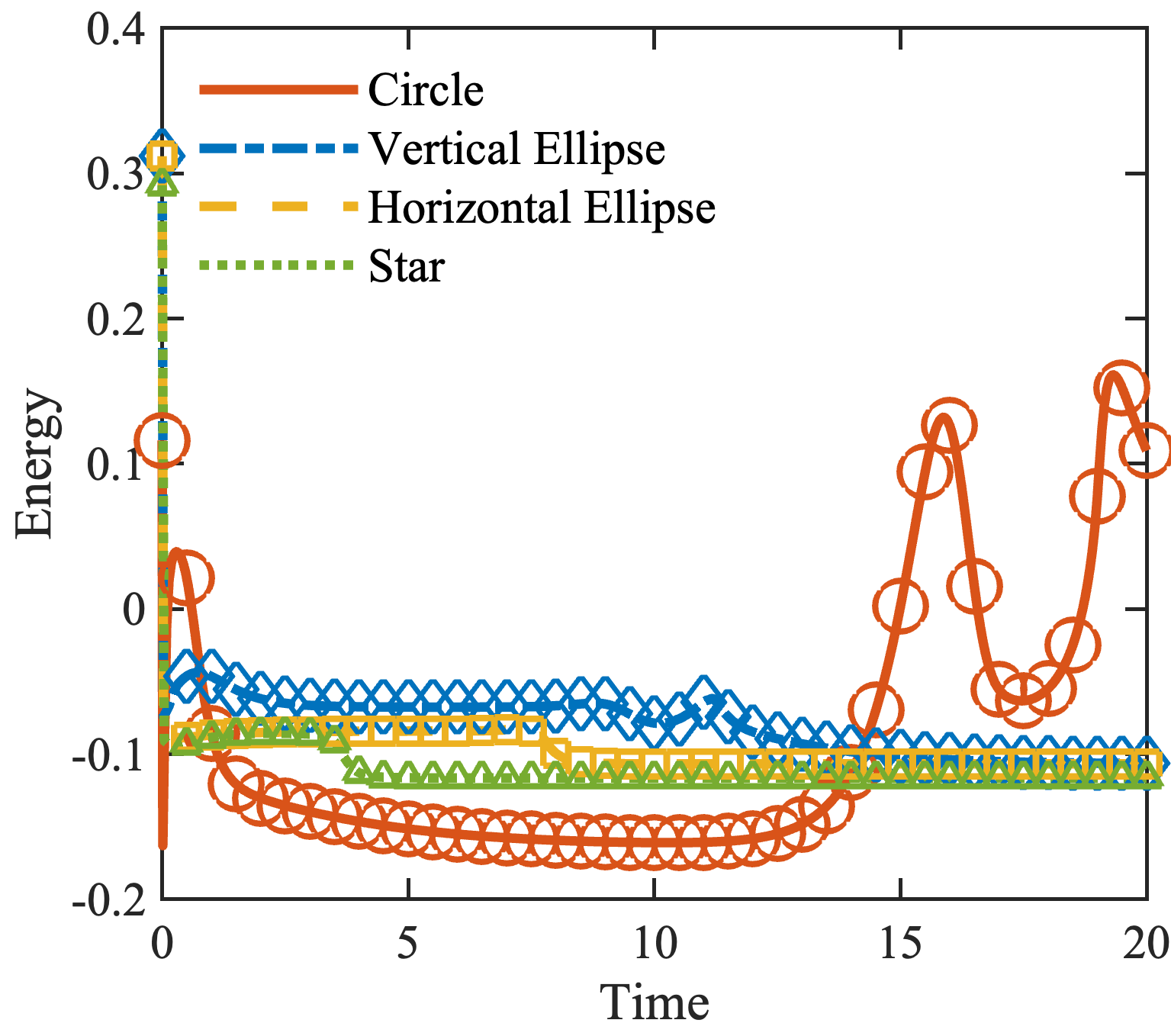}
	\end{center}
	\caption{Temporal evolution of the free energy for active liquid
		crystals in domains with different obstacle shapes. The left
		panel corresponds to the case with $\chi_{fluid} = 10$ and
		$\xi_{fluid} = -0.1$, while the right panel shows the results for
		$\chi_{fluid} = -10$ and $\xi_{fluid} = 0.1$.}\label{fig:energy_shape}
\end{figure}

Figure~\ref{fig:energy_shape} depicts the time evolution of the
free energy for the aforementioned four cases, the left panel shows
the result of $\chi_{fluid} > 0$ and the right panel depicts the
result of $\chi_{fluid} < 0$. One observes that in this case, due
to large  active parameter values, the energy  is clearly no longer
monotonically decreasing, especially in the case of the circular
obstacle. Due to the geometric domain left for the defects to move
around, their motion is primarily driven by the active stress. As a
result, the system does not reach a steady state at the end of the
computation. For other cases, they ultimately arrives at a steady
state due to the effect between $+\tfrac{1}{2}$ and $-\tfrac{1}{2}$
defects and the confined geometry.

\section{Conclusion}
In this study, we have developed a set  of thermodynamically consistent
numerical algorithms for a novel
$\bQ$-tensor-based two-phase hydrodynamic model of active nematic liquid
crystals and a solid substrate in a region of arbitrary geometry.
The discrete equation system resulted from each
algorithm consists of linear, decoupled solvers for the Stokes and
Poisson systems, ensuring computational efficiency. Convergence
tests confirm their first- and
second-order temporal accuracy, respectively. For the passive model,
the SGE-BDF2 scheme is shown in
Appendix~\ref{app:sge-bdf2-stability} to dissipate a modified energy
under the boundedness and stabilization assumptions stated in the
analysis, while the SGE-PDG scheme retains an unconditional
second-order energy stability property. Through extensive numerical
studies, we demonstrated how activity parameters and obstacle
geometries influence defect dynamics and steady-state
configurations. The proposed model and algorithms not only provide
deeper physical insights into the hydrodynamics of the active
nematic system but also
establish a robust computational framework for future research in
active matter. In particular, they furnish an efficient suite of
tools to simulate two-phase active liquid crystals in domains with
arbitrary solid boundaries or embedded obstacles of general shapes.

\section*{Acknowledgements}
Xuelong Gu's research is  supported  by
NSF award  OIA-2242812.  and Qi Wang's research is  partially supported  by
NSF awards  OIA-2242812 and DMS-2038080, DOE award DE-SC0025229,
and an SC GAIN-CRP award.
	{\color{magenta}
		\appendix
		\section{Energy stability of the SGE--BDF2 scheme}
		\label{app:sge-bdf2-stability}

		This appendix provides the energy stability proof for the SGE--BDF2 scheme in the limit of passive liquid crystals.
		For any sequence $f^n$, we write
		\begin{equation*}
			\begin{aligned}
				d_t f^{n+1} & = f^{n+1}-f^n,\qquad
				d_{2t} f^{n+1} = 3f^{n+1}-4f^n+f^{n-1},\qquad
				d_{tt} f^n = f^{n+1}-2f^n+f^{n-1},
			\end{aligned}
		\end{equation*}
		and
		\begin{equation*}
			\begin{aligned}
				D_t f^{n+1} & = \frac{1}{\tau}d_t f^{n+1},\qquad
				D_{2t} f^{n+1} = \frac{1}{2\tau}d_{2t} f^{n+1},\qquad
				D_{tt} f^n = \frac{1}{\tau^2}d_{tt} f^n.
			\end{aligned}
		\end{equation*}
		We assume that \eqref{eq:ass-stabilization} holds. Consequently, for any
		$\mathbf A$ and $\mathbf B$ satisfying \eqref{eq:ass-stabilization}, there
		exists $L_f>0$ such that
		\begin{equation*}
			\|f_{\rm bulk}(\mathbf A)-f_{\rm bulk}(\mathbf B)\|
			\le L_f\|\mathbf A-\mathbf B\|.
		\end{equation*}
		The BDF2 molecular field is given by
		\begin{equation*}
			\begin{aligned}
				\mathbf G_{\mathbf Q}^{n+1}
				={} & -\nabla\cdot\left(K\nabla\mathbf Q^{n+1}\right)
				+W(\mathbf Q^{n+1}-\mathbf Q_{\star})
				+\frac{\tau\kappa}{2}d_{2t}\mathbf Q^{n+1}
				+B_2(f_N(\mathbf Q^n)).
			\end{aligned}
		\end{equation*}

		\begin{lem}
			The following BDF2 identities hold:
			\begin{equation}\label{eq:primitive-bdf2-id-1}
				\begin{aligned}
					2(d_{2t} f^{n+1},f^{n+1})
					={} & \|f^{n+1}\|^2+\|B_2(f^{n+1})\|^2
					-\|f^n\|^2-\|B_2(f^n)\|^2
					+\|d_{tt}f^n\|^2,
				\end{aligned}
			\end{equation}
			\begin{equation}\label{eq:primitive-bdf2-id-2}
				\begin{aligned}
					\|d_{2t}f^{n+1}\|^2
					=6\|d_t f^{n+1}\|^2-2\|d_t f^n\|^2
					+3\|d_{tt}f^n\|^2.
				\end{aligned}
			\end{equation}
		\end{lem}

		\begin{lem}
			For the BDF2 extrapolated nonlinear force, one has
			\begin{equation}\label{eq:primitive-taylor-claim}
				\begin{aligned}
					d_{2t}F_{\rm bulk}(\mathbf Q^{n+1})
					\le{} & \left(d_{2t}\mathbf Q^{n+1},
					B_2(f_{\rm bulk}(\mathbf Q^n))\right)
					+3L_f\|d_t\mathbf Q^{n+1}\|^2
					+3L_f\|d_t\mathbf Q^n\|^2.
				\end{aligned}
			\end{equation}
		\end{lem}

		\begin{proof}
			A direct calculation gives
			\begin{equation*}
				\begin{aligned}
					 & \left(d_{2t}\mathbf Q^{n+1},B_2(f_{\rm bulk}(\mathbf Q^n))\right) \\
					 & \quad=\left(3d_t\mathbf Q^{n+1}-d_t\mathbf Q^n,
					f_{\rm bulk}(\mathbf Q^n)+d_t f_{\rm bulk}(\mathbf Q^n)\right)       \\
					 & \quad=3(d_t\mathbf Q^{n+1},f_{\rm bulk}(\mathbf Q^n))
					+3(d_t\mathbf Q^{n+1},d_t f_{\rm bulk}(\mathbf Q^n))                 \\
					 & \qquad -(d_t\mathbf Q^n,f_{\rm bulk}(\mathbf Q^n))
					-(d_t\mathbf Q^n,d_t f_{\rm bulk}(\mathbf Q^n)).
				\end{aligned}
			\end{equation*}
			The Taylor's formula with an integral remainder yields
			\begin{equation*}
				\begin{aligned}
					d_tF_{\rm bulk}(\mathbf Q^{n+1})
					={} & (f_{\rm bulk}(\mathbf Q^n),d_t\mathbf Q^{n+1}) \\
					    & +\int_0^1(1-s)
					\left(f_{\rm bulk}'(\mathbf Q^n+s d_t\mathbf Q^{n+1})
					d_t\mathbf Q^{n+1},d_t\mathbf Q^{n+1}\right)\,ds.
				\end{aligned}
			\end{equation*}
			Therefore,
			\begin{equation*}
				(f_{\rm bulk}(\mathbf Q^n),d_t\mathbf Q^{n+1})
				\ge d_tF_{\rm bulk}(\mathbf Q^{n+1})
				-\frac{L_f}{2}\|d_t\mathbf Q^{n+1}\|^2.
			\end{equation*}
			Similarly,
			\begin{equation*}
				\begin{aligned}
					d_tF_{\rm bulk}(\mathbf Q^n)
					={} & (d_t\mathbf Q^n,f_{\rm bulk}(\mathbf Q^n)) \\
					    & -\int_0^1s
					\left(f_{\rm bulk}'(\mathbf Q^n-s d_t\mathbf Q^n)d_t\mathbf Q^n,
					d_t\mathbf Q^n\right)\,ds,
				\end{aligned}
			\end{equation*}
			hence
			\begin{equation*}
				-(f_{\rm bulk}(\mathbf Q^n),d_t\mathbf Q^n)
				\ge -d_tF_{\rm bulk}(\mathbf Q^n)
				-\frac{L_f}{2}\|d_t\mathbf Q^n\|^2.
			\end{equation*}
			Moreover,
			\begin{equation*}
				\|d_t f_{\rm bulk}(\mathbf Q^n)\|
				\le L_f\|d_t\mathbf Q^n\|.
			\end{equation*}
			It follows that
			\begin{equation*}
				\begin{aligned}
					3(d_t\mathbf Q^{n+1},d_t f_{\rm bulk}(\mathbf Q^n))
					 & \ge -3L_f\|d_t\mathbf Q^{n+1}\|\,\|d_t\mathbf Q^n\|, \\
					-(d_t\mathbf Q^n,d_t f_{\rm bulk}(\mathbf Q^n))
					 & \ge -L_f\|d_t\mathbf Q^n\|^2.
				\end{aligned}
			\end{equation*}
			Combining the preceding estimates, we have
			\begin{equation*}
				\begin{aligned}
					 & \left(d_{2t}\mathbf Q^{n+1},B_2(f_{\rm bulk}(\mathbf Q^n))\right) \\
					 & \quad\ge d_{2t}F_{\rm bulk}(\mathbf Q^{n+1})
					-\frac{3L_f}{2}\|d_t\mathbf Q^{n+1}\|^2
					-3L_f\|d_t\mathbf Q^{n+1}\|\,\|d_t\mathbf Q^n\|
					-\frac{3L_f}{2}\|d_t\mathbf Q^n\|^2.
				\end{aligned}
			\end{equation*}
			Since
			\begin{equation*}
				2\|d_t\mathbf Q^{n+1}\|\,\|d_t\mathbf Q^n\|
				\le \|d_t\mathbf Q^{n+1}\|^2+\|d_t\mathbf Q^n\|^2,
			\end{equation*}
			the estimate \eqref{eq:primitive-taylor-claim} follows.
		\end{proof}

		\begin{thm}
			Let
			\begin{equation*}
				\overline{\Gamma}=\|\Gamma_{\mathbf Q}\|_{L^\infty(\Omega)},
				\qquad
				\Theta=\frac12\sqrt{\frac{\kappa}{\overline{\Gamma}}}.
			\end{equation*}
			Assume that
			\begin{equation}\label{eq:primitive-kappa-condition}
				\kappa\ge \frac94\,\overline{\Gamma}L_f^2.
			\end{equation}
			Define the BDF2 energy by
			\begin{equation}\label{eq:primitive-modified-energy}
				\begin{aligned}
					E_{\rm BDF2}^{n+1}
					={} & \frac14\left(\|\mathbf v^{n+1}\|^2
					+\|B_2(\mathbf v^{n+1})\|^2\right)                                   \\
					    & +\frac14\left(\|\sqrt K\nabla\mathbf Q^{n+1}\|^2
					+\|\sqrt K\nabla B_2(\mathbf Q^{n+1})\|^2\right)                     \\
					    & +\frac14\left(\|\sqrt W(\mathbf Q^{n+1}-\mathbf Q_{\star})\|^2
					+\|\sqrt W B_2(\mathbf Q^{n+1}-\mathbf Q_{\star})\|^2\right)         \\
					    & +\frac32F_{\rm bulk}(\mathbf Q^{n+1})
					-\frac12F_{\rm bulk}(\mathbf Q^n)
					+\left(2\Theta+\frac{3L_f}{2}\right)
					\|d_t\mathbf Q^{n+1}\|^2.
				\end{aligned}
			\end{equation}
			Then the SGE--BDF2 scheme in the limit of passive liquid crystals satisfies
			\begin{equation}\label{eq:primitive-final-stability}
				\begin{aligned}
					 & E_{\rm BDF2}^{n+1}-E_{\rm BDF2}^n
					+(4\Theta-3L_f)\|d_t\mathbf Q^{n+1}\|^2
					+\tau\left(2\|\sqrt\eta\mathbf D^{n+1}\|^2
					+\|\sqrt b\mathbf v^{n+1}\|^2\right)     \\
					 & \quad +3\Theta\|d_{tt}\mathbf Q^n\|^2
					+\frac14\|d_{tt}\mathbf v^n\|^2
					+\frac14\|\sqrt K\nabla d_{tt}\mathbf Q^n\|^2
					+\frac14\|\sqrt W d_{tt}\mathbf Q^n\|^2
					\le 0.
				\end{aligned}
			\end{equation}
			In particular, \eqref{eq:primitive-kappa-condition} implies
			$E_{\rm BDF2}^{n+1}\le E_{\rm BDF2}^n$ for every $\tau>0$.
		\end{thm}

		\begin{proof}
			In the passive case, $B_2(\mathbf f_{\mathbf v}^n)=0$ and
			$B_2(\mathbf F_{\mathbf Q}^n)=0$. The SGE--BDF2 update is
			\begin{equation}\label{eq:primitive-equation}
				\left\{
				\begin{aligned}
					D_{2t}\mathbf v^{n+1}
					={} & -\nabla p^{n+1}+2\nabla\cdot(\eta\mathbf D^{n+1})
					-b\mathbf v^{n+1}
					+\zeta^{n+1}B_2(\mathbf c_{\mathbf v}^n)
					-\omega^{n+1}B_2(\mathbf g_{\mathbf v}^n),              \\
					\nabla\cdot\mathbf v^{n+1}
					={} & 0,                                                \\
					D_{2t}\mathbf Q^{n+1}
					={} & -\Gamma_{\mathbf Q}\mathbf G_{\mathbf Q}^{n+1}
					+\zeta^{n+1}B_2(\mathbf C_{\mathbf Q}^n)
					-\omega^{n+1}B_2(\mathbf G_{\mathbf Q}^n).
				\end{aligned}
				\right.
			\end{equation}
			Take the $L^2$ inner product of the first equation in
			\eqref{eq:primitive-equation} with $\tau\mathbf v^{n+1}$ and the
			$L^2$ inner product of the third equation in
			\eqref{eq:primitive-equation}
			with $\tau\mathbf G_{\mathbf Q}^{n+1}$. The pressure term
			vanishes because
			\begin{equation*}
				-\tau(\nabla p^{n+1},\mathbf v^{n+1})
				=\tau(p^{n+1},\nabla\cdot\mathbf v^{n+1})=0.
			\end{equation*}
			The viscous and drag terms yield
			\begin{equation*}
				2\tau(\nabla\cdot(\eta\mathbf D^{n+1}),\mathbf v^{n+1})
				=-2\tau\|\sqrt\eta\mathbf D^{n+1}\|^2,
			\end{equation*}
			and
			\begin{equation*}
				-\tau(b\mathbf v^{n+1},\mathbf v^{n+1})
				=-\tau\|\sqrt b\mathbf v^{n+1}\|^2.
			\end{equation*}
			The tensor dissipation gives
			\begin{equation*}
				-\tau(\Gamma_{\mathbf Q}\mathbf G_{\mathbf Q}^{n+1},
				\mathbf G_{\mathbf Q}^{n+1})
				=-\tau\|\sqrt{\Gamma_{\mathbf Q}}\mathbf G_{\mathbf Q}^{n+1}\|^2.
			\end{equation*}
			The SGE terms cancel exactly. Indeed,
			\begin{equation*}
				\begin{aligned}
					 & \zeta^{n+1}\left[(\mathbf v^{n+1},B_2(\mathbf c_{\mathbf v}^n))
						              +(\mathbf G_{\mathbf Q}^{n+1},B_2(\mathbf C_{\mathbf
								              Q}^n))\right]         \\
					 & \quad -\omega^{n+1}\left[(\mathbf v^{n+1},B_2(\mathbf
						                      g_{\mathbf v}^n))
						                      +(\mathbf G_{\mathbf Q}^{n+1},B_2(\mathbf G_{\mathbf
								                      Q}^n))\right] \\
					 & =\zeta^{n+1}\omega^{n+1}
					\left(\|B_2(\mathbf g_{\mathbf v}^n)\|^2
					+\|B_2(\mathbf G_{\mathbf Q}^n)\|^2\right)                                 \\
					 & \quad -\omega^{n+1}\zeta^{n+1}
					\left(\|B_2(\mathbf g_{\mathbf v}^n)\|^2
					+\|B_2(\mathbf G_{\mathbf Q}^n)\|^2\right)=0.
				\end{aligned}
			\end{equation*}
			Adding the two tested equations yields
			\begin{equation}\label{eq:primitive-tested-sum}
				\begin{aligned}
					 & \frac12\left(d_{2t}\mathbf v^{n+1},\mathbf v^{n+1}\right)
					+\frac12\left(d_{2t}\mathbf Q^{n+1},
					\mathbf G_{\mathbf Q}^{n+1}\right)                           \\
					 & \quad =-\tau\left(
					2\|\sqrt\eta\mathbf D^{n+1}\|^2
					+\|\sqrt b\mathbf v^{n+1}\|^2
					+\|\sqrt{\Gamma_{\mathbf Q}}\mathbf G_{\mathbf Q}^{n+1}\|^2
					\right).
				\end{aligned}
			\end{equation}
			By the definition of $\mathbf G_{\mathbf Q}^{n+1}$,
			\begin{equation}\label{eq:primitive-G-expanded}
				\begin{aligned}
					 & \frac12\left(d_{2t}\mathbf Q^{n+1},
					\mathbf G_{\mathbf Q}^{n+1}\right)            \\
					 & =\frac12\left(d_{2t}\mathbf Q^{n+1},
					-\nabla\cdot(K\nabla\mathbf Q^{n+1})
					+W(\mathbf Q^{n+1}-\mathbf Q_{\star})\right)  \\
					 & \quad +\frac12\left(d_{2t}\mathbf Q^{n+1},
					B_2(f_N(\mathbf Q^n))\right)
					+\frac{\tau\kappa}{4}\|d_{2t}\mathbf Q^{n+1}\|^2.
				\end{aligned}
			\end{equation}
			Applying \eqref{eq:primitive-bdf2-id-1} to the velocity gives
			\begin{equation}\label{eq:primitive-velocity-identity}
				\begin{aligned}
					\frac12\left(d_{2t}\mathbf v^{n+1},\mathbf v^{n+1}\right)
					={} & \frac14d_t\left(\|\mathbf v^{n+1}\|^2
					+\|B_2(\mathbf v^{n+1})\|^2\right)
					+\frac14\|d_{tt}\mathbf v^n\|^2.
				\end{aligned}
			\end{equation}
			Using integration by parts and the boundary conditions, the
			linear tensor part satisfies
			\begin{equation}\label{eq:primitive-linear-Q-identity}
				\begin{aligned}
					 & \frac12\left(d_{2t}\mathbf Q^{n+1},
					-\nabla\cdot(K\nabla\mathbf Q^{n+1})
					+W(\mathbf Q^{n+1}-\mathbf Q_{\star})\right)                \\
					 & =\frac14d_t\left(\|\sqrt K\nabla\mathbf Q^{n+1}\|^2
					+\|\sqrt K\nabla B_2(\mathbf Q^{n+1})\|^2\right)
					+\frac14\|\sqrt K\nabla d_{tt}\mathbf Q^n\|^2               \\
					 & \quad +\frac14d_t\left(\|\sqrt W(\mathbf Q^{n+1}-\mathbf
					Q_{\star})\|^2
					+\|\sqrt W B_2(\mathbf Q^{n+1}-\mathbf Q_{\star})\|^2\right)
					+\frac14\|\sqrt W d_{tt}\mathbf Q^n\|^2.
				\end{aligned}
			\end{equation}
			By \eqref{eq:primitive-taylor-claim},
			\begin{equation}\label{eq:primitive-bulk-in-energy}
				\begin{aligned}
					 & \frac12\left(d_{2t}\mathbf Q^{n+1},
					B_2(f_{\rm bulk}(\mathbf Q^n))\right)                   \\
					 & \quad \ge \frac12d_{2t}F_{\rm bulk}(\mathbf Q^{n+1})
					-\frac{3L_f}{2}\|d_t\mathbf Q^{n+1}\|^2
					-\frac{3L_f}{2}\|d_t\mathbf Q^n\|^2                     \\
					 & \quad =d_t\left(\frac32F_{\rm bulk}(\mathbf Q^{n+1})
					-\frac12F_{\rm bulk}(\mathbf Q^n)\right)
					-\frac{3L_f}{2}\|d_t\mathbf Q^{n+1}\|^2
					-\frac{3L_f}{2}\|d_t\mathbf Q^n\|^2.
				\end{aligned}
			\end{equation}
			Combining
			\eqref{eq:primitive-tested-sum}--\eqref{eq:primitive-bulk-in-energy}
			gives
			\begin{equation}\label{eq:primitive-preliminary}
				\begin{aligned}
					 & \widetilde E^{n+1}-\widetilde E^n
					+\tau\left(
					2\|\sqrt\eta\mathbf D^{n+1}\|^2
					+\|\sqrt b\mathbf v^{n+1}\|^2
					+\|\sqrt{\Gamma_{\mathbf Q}}\mathbf G_{\mathbf Q}^{n+1}\|^2
					\right)                                                   \\
					 & \quad +\frac{\tau\kappa}{4}\|d_{2t}\mathbf Q^{n+1}\|^2
					+\frac14\|d_{tt}\mathbf v^n\|^2
					+\frac14\|\sqrt K\nabla d_{tt}\mathbf Q^n\|^2
					+\frac14\|\sqrt W d_{tt}\mathbf Q^n\|^2                   \\
					 & \le \frac{3L_f}{2}\|d_t\mathbf Q^{n+1}\|^2
					+\frac{3L_f}{2}\|d_t\mathbf Q^n\|^2,
				\end{aligned}
			\end{equation}
			where $\widetilde E^{n+1}$ denotes \eqref{eq:primitive-modified-energy}
			without the last term
			$\left(2\Theta+\frac{3L_f}{2}\right)\|d_t\mathbf Q^{n+1}\|^2$.

			The Allen--Cahn stabilization is used through the following
			square estimate.
			For the dissipative tensor block, after the SGE contribution
			has canceled in
			the summed energy identity,
			\begin{equation*}
				D_{2t}\mathbf Q^{n+1}=-\Gamma_{\mathbf Q}\mathbf G_{\mathbf Q}^{n+1}.
			\end{equation*}
			Consequently,
			\begin{equation*}
				\tau\|\sqrt{\Gamma_{\mathbf Q}}\mathbf G_{\mathbf Q}^{n+1}\|^2
				=\frac{1}{4\tau}
				\left\|\Gamma_{\mathbf Q}^{-1/2}d_{2t}\mathbf Q^{n+1}\right\|^2.
			\end{equation*}
			Since $\Gamma_{\mathbf Q}\le\overline{\Gamma}$, we obtain
			\begin{equation}\label{eq:primitive-FTY-square}
				\begin{aligned}
					 & \tau\|\sqrt{\Gamma_{\mathbf Q}}\mathbf G_{\mathbf Q}^{n+1}\|^2
					+\frac{\tau\kappa}{4}\|d_{2t}\mathbf Q^{n+1}\|^2                  \\
					 & \quad \ge
					\left(\frac{1}{4\tau\overline{\Gamma}}+\frac{\tau\kappa}{4}\right)
					\|d_{2t}\mathbf Q^{n+1}\|^2
					\ge \Theta\|d_{2t}\mathbf Q^{n+1}\|^2.
				\end{aligned}
			\end{equation}
			Using \eqref{eq:primitive-FTY-square} in
			\eqref{eq:primitive-preliminary}
			gives
			\begin{equation}\label{eq:primitive-after-square}
				\begin{aligned}
					 & \widetilde E^{n+1}-\widetilde E^n
					+\Theta\|d_{2t}\mathbf Q^{n+1}\|^2
					+\tau\left(2\|\sqrt\eta\mathbf D^{n+1}\|^2
					+\|\sqrt b\mathbf v^{n+1}\|^2\right)          \\
					 & \quad +\frac14\|d_{tt}\mathbf v^n\|^2
					+\frac14\|\sqrt K\nabla d_{tt}\mathbf Q^n\|^2
					+\frac14\|\sqrt W d_{tt}\mathbf Q^n\|^2       \\
					 & \le \frac{3L_f}{2}\|d_t\mathbf Q^{n+1}\|^2
					+\frac{3L_f}{2}\|d_t\mathbf Q^n\|^2.
				\end{aligned}
			\end{equation}
			Finally, by \eqref{eq:primitive-bdf2-id-2},
			\begin{equation}\label{eq:primitive-BDF2-norm-Q}
				\begin{aligned}
					\|d_{2t}\mathbf Q^{n+1}\|^2
					=6\|d_t\mathbf Q^{n+1}\|^2
					-2\|d_t\mathbf Q^n\|^2
					+3\|d_{tt}\mathbf Q^n\|^2.
				\end{aligned}
			\end{equation}
			Substituting \eqref{eq:primitive-BDF2-norm-Q} into
			\eqref{eq:primitive-after-square} and moving all terms to the left-hand
			side yield
			\begin{equation*}
				\begin{aligned}
					 & \widetilde E^{n+1}-\widetilde E^n
					+\left(6\Theta-\frac{3L_f}{2}\right)\|d_t\mathbf Q^{n+1}\|^2
					-\left(2\Theta+\frac{3L_f}{2}\right)\|d_t\mathbf Q^n\|^2 \\
					 & \quad +3\Theta\|d_{tt}\mathbf Q^n\|^2
					+\frac14\|d_{tt}\mathbf v^n\|^2
					+\frac14\|\sqrt K\nabla d_{tt}\mathbf Q^n\|^2
					+\frac14\|\sqrt W d_{tt}\mathbf Q^n\|^2                  \\
					 & \le -2\tau\|\sqrt\eta\mathbf D^{n+1}\|^2
					-\tau\|\sqrt b\mathbf v^{n+1}\|^2.
				\end{aligned}
			\end{equation*}
			Add and subtract
			$\left(2\Theta+\frac{3L_f}{2}\right)\|d_t\mathbf Q^{n+1}\|^2$.
			The difference of the added terms is precisely the last part of
			the modified
			energy \eqref{eq:primitive-modified-energy}, and the remaining
			coefficient of
			$\|d_t\mathbf Q^{n+1}\|^2$ is
			\begin{equation*}
				\begin{aligned}
					6\Theta-\frac{3L_f}{2}-\left(2\Theta+\frac{3L_f}{2}\right)
					=4\Theta-3L_f.
				\end{aligned}
			\end{equation*}
			This proves \eqref{eq:primitive-final-stability}. Finally,
			\eqref{eq:primitive-kappa-condition} is equivalent to $4\Theta\ge 3L_f$;
			hence the modified energy is non-increasing for arbitrary $\tau>0$.
		\end{proof}
		\begin{rmk}
			The condition, \eqref{eq:primitive-kappa-condition}, is a sufficient
			condition to ensure the energy stability of the
			SGE--BDF2 scheme. As in the BDF1 case, such a requirement is mainly
			introduced to ensure the validity of the theoretical analysis and is
			not expected to be necessary in practical computations. In practice,
			the stabilization parameter can be chosen in a more practical way,
			potentially in conjunction with the time step size, while
			preserving  energy stable behavior.
		\end{rmk}
	}

\bibliographystyle{elsarticle-num}
\bibliography{Ref}

\end{document}